\documentclass[11pt]{aims}
\usepackage{amsmath, amssymb, mathrsfs}
  \usepackage{paralist}
  \usepackage{graphics} 
  \usepackage{epsfig} 
\usepackage{graphicx}
  \usepackage{epstopdf}
 \usepackage[colorlinks=true]{hyperref}
 \hypersetup{urlcolor=blue, citecolor=red}
 \usepackage[toc,page]{appendix}
\usepackage{chngcntr}

\usepackage{titlesec}
\titleformat{\section}{\Large\bfseries}{\thesection.}{4pt}{}
\titleformat{\subsection}{\large\bfseries}{\thesection.\arabic{subsection}.}{4pt}{}
\titleformat{\subsubsection}{\bfseries}{\thesection.\arabic{subsection}.\arabic{subsubsection}.}{4pt}{}
\titleformat*{\paragraph}{\bfseries}
\titleformat*{\subparagraph}{\bfseries}
 \usepackage[margin=1.1in]{geometry}

  \def\currentyear{}
   \def\currentmonth{}
    \def\ppages{}
     \def\DOI{}

\newcommand{\be}{\begin{equation}}
\newcommand{\ee}{\end{equation}}

\usepackage[notref,notcite]{showkeys}

\newcommand{\e}{\varepsilon} 
\newcommand{\la}{\langle} 
\newcommand{\ra}{\rangle} 
\newcommand{\pa}{\partial} 
 
\newcommand{\lab}{\label} 
\newcommand{\ba}{\begin{array}}
\newcommand{\ea}{\end{array}}
\newcommand{\bee}{\begin{eqnarray*}}
\newcommand{\eee}{\end{eqnarray*}}
\newcommand{\bea}{\begin{eqnarray}}
\newcommand{\eea}{\end{eqnarray}}
\newcommand{\non}{\nonumber}
\newcommand{\lb}{\lambda}

\newcommand{\Lb}{\Lambda}

\def\fref#1{{\rm (\ref{#1})}}
\newcommand{\donothing}[1]{{}}

\newcommand{\tL}{\tilde{\mathcal L}}

\newtheorem{theorem}{Theorem}

\newtheorem{corollary}[theorem]{Corollary}

\newtheorem{definition}[theorem]{Definition}

\newtheorem{lemma}[theorem]{Lemma}

\newtheorem{proposition}[theorem]{Proposition}
\newtheorem{remark}[theorem]{Remark}

\numberwithin{equation}{section}

\newcommand{\lY}{\langle Y \rangle}
\newcommand{\lZ}{\langle Z \rangle}

\newcommand{\tX}{\tilde X}

\begin{document}

\title{Singularity formation for Burgers equation with transverse viscosity}

\author[C. Collot]{Charles Collot}
\address{Department of Mathematics, New York University in Abu Dhabi, Saadiyat Island, P.O. Box 129188, Abu Dhabi, United Arab Emirates.}
\email{cc5786@nyu.edu}
\author[T.-E. Ghoul]{Tej-Eddine Ghoul}
\address{Department of Mathematics, New York University in Abu Dhabi, Saadiyat Island, P.O. Box 129188, Abu Dhabi, United Arab Emirates.}
\email{teg6@nyu.edu}
\author[N. Masmoudi]{Nader Masmoudi}
\address{Department of Mathematics, New York University in Abu Dhabi, Saadiyat Island, P.O. Box 129188, Abu Dhabi, United Arab Emirates. Courant Institute of Mathematical Sciences, New York University, 251 Mercer Street, New York, NY 10012, USA,}
\email{nm30@nyu.edu}


\thispagestyle{empty}

\begin{abstract} 
We consider Burgers equation with  transverse viscosity $$\pa_tu+u\pa_xu-\pa_{yy}u=0, \ \ (x,y)\in \Bbb R^2, \ \ u:[0,T)\times \mathbb R^2\rightarrow \mathbb R.$$ We construct and describe precisely a family of solutions which become singular in finite time by having their gradient becoming unbounded. To leading order, the solution is given by a backward self-similar solution of Burgers equation along the $x$ variable, whose scaling parameters evolve according to parabolic equations along the $y$ variable, one of them being the quadratic semi-linear heat equation. We develop a new framework adapted to this mixed hyperbolic/parabolic blow-up problem, revisit the construction of flat blow-up profiles for the semi-linear heat equation, and the self-similarity in singularities of the inviscid Burgers equation.
\end{abstract}

\maketitle


\section{Introduction}

\subsection{Setting of the problem and motivations}

We consider Burgers equation with transverse viscosity
\be \lab{eq:burgers}
\left\{\begin{array}{l}
\pa_tu+u\pa_xu-\pa_{yy}u=0, \ \ (x,y)\in \mathbb R^2, \ \ \\
u_{t=0}=u_0,
\end{array}\right.
\end{equation}
for $u:[0,T)\times \mathbb R^2\rightarrow \mathbb R$. The present study is motivated by the following. This model reduces to the classical inviscid Burgers equation for solutions that are independent of the  transverse variable $u(t,x,y)=U(t,x)$, which is a classical example of a nonlinear hyperbolic equation for which initially smooth solutions can become singular in finite time, see for example \cite{D,S}. The effects of viscosity in the streamwise direction, namely the equation $\pa_tu+u\pa_xu-\epsilon \pa_{xx}u=0$, have been extensively studied, see \cite{H,I} and references therein. This work aims at understanding precisely the consequence of an additional effect, here the  transverse viscosity, on a blow-up dynamics that it does not prevent. Moreover, this new effect changes the nature of the equation which is of a mixed hyperbolic/parabolic type. Handling these two issues, our result then extends known ones for blow-ups in a new direction, as well as raising new interesting problems, see the comments after the main Theorem \ref{th:main}.\\

\noindent More importantly, the study of \fref{eq:burgers} is motivated by fluid dynamics, from the fact that it is a simplified version of Prandtl's boundary layer equations. Solutions to Prandtl's equations might blow up in finite time \cite{CS,EE,KVW} but a precise description of the singularity formation is still lacking. The present work is a step towards that goal and this issue will be investigated in a forthcoming work. Finally, let us mention that there has been recent progress on other models for singular solutions in fluid dynamics, see \cite{CHKLSY,CINT,W} and references therein. \\

\noindent The existence of smooth enough solutions to \fref{eq:burgers} follows from classical arguments. For example, relying on a fixed point argument and energy estimates, one can show that the equation is locally well-posed in $H^s(\mathbb R^2)$ for $s\geq 3$. There then holds the following blow-up criterion (again from energy estimates because of the identity $|\int uvv_x|\lesssim \| u_x\|_{L^{\infty}}\int v^2$): the solution $u$ blows up at time $T>0$ if and only if
\be \lab{eq:critere explo}
\limsup_{t\uparrow T} \| \pa_x u \|_{L^{\infty}(\mathbb R^2)}=+\infty.
\ee
The existence of global kinetic solutions $u\in L^{\infty}([0,+\infty),L^1(\mathbb R^2))$ has been showed by Chen and Perthame \cite{CP} following the framework of Lions, Perthame and Tadmor \cite{LPT}. We refer to \cite{S} for an introduction on kinetic solutions for scalar conservation laws. We expect singularities for such low regularity solutions to be different than the solutions in the present paper, as regularity plays a key role in the blow-up mechanism we describe. Before stating the main theorem, let us give the structure of the singularities of Burgers equation, and of the ones of the parabolic system encoding the effects of the  transverse viscosity.


\subsection{Self-similarity in shocks for Burgers equation} \lab{subsec:Burgers}

Burgers equation
\be \lab{eq:burgers2}
\left\{\begin{array}{l}
\pa_tU+U\pa_xU=0, \ \ x\in \mathbb R, \\
U_{t=0}=U_0,
\end{array}\right.
\end{equation}
admits solutions becoming singular in finite time in a self-similar way:
$$
U(t,x)=\mu^{-1}(T-t)^{\frac{1}{2i}}\Psi_i \left(\mu \frac{x}{(T-t)^{1+\frac{1}{2i}}} \right)
$$
where $(\Psi_i)_{i\in \mathbb N^*}$ is a family of analytic profiles (see \cite{EF} for example), and where $\mu>0$ is a free parameter. They are at the heart of the shock formation, a fact that is rarely emphasised, which lead us to give a precise and concise study in Section \ref{sec:burgers}. Self-similar and discretely self-similar blow-up profiles for Burgers equation are classified in Proposition \ref{pr:clas}. Different scaling laws are thus possible, depending on the initial condition via its behaviour near the characteristic where the shock will form, which has to do with the fact that the scaling group of \fref{eq:burgers2} is two-dimensional, see Section \fref{sec:burgers}. This possibility of several scaling exponents is referred to as self-similarity of the second kind \cite{B}. For each $i\in \mathbb N^*$, $\Psi_i$, defined in Proposition \ref{prop:smoothselfsim}, is an odd decreasing profile, which is nonnegative and concave on $(-\infty,0]$ and such that $\pa_X \Psi_i$ is minimal at the origin with asymptotic $\Psi_i(X)=-X+X^{2i+1}$ as $X\rightarrow 0$. One has in particular the formula
\be \lab{eq:def Psi1}
\Psi_1(X):= \left(-\frac X2 +\left(\frac{1}{27}+\frac{X^2}{4}\right)^{\frac 12} \right)^{\frac 13}+ \left(-\frac X2 -\left(\frac{1}{27}+\frac{X^2}{4}\right)^{\frac 12} \right)^{\frac 13},
\ee
for the fundamental one \cite{CSW}. As in the above formula, all these profiles are unbounded at infinity but they emerge nonetheless from well localised initial data. A precise description of these profiles is given in Proposition \ref{prop:smoothselfsim}. Any regular enough non-degenerate solution $v$ to \fref{eq:burgers2} that forms a shock at $(T,x_0)$ is equivalent to leading order near the singularity to a self-similar profile $\Psi_i$ up to the symmetries of the equation
\be \lab{eq:def selfsim burgers}
U(t,x)\sim(T-t)^{\frac{1}{2i}} \mu^{-1}\Psi_i \left(\mu\frac{x-(x_0-c(T-t))}{(T-t)^{1+\frac{1}{2i}}} \right)+c \ \ \text{as} \ (t,x)\rightarrow (T,x_0),
\ee
see Proposition \ref{pr:nbassin}. The blow-up dynamics involving the concentration of $\Psi_1$ is a stable one for smooth enough solutions. The scenario corresponding to the concentration of $\Psi_i$ for $i\geq 2$ is an unstable one. For a suitable topological functional space, the set of initial conditions leading to the concentration of $\Psi_i$ for $i\geq 2$ is located at the boundary of the set of initial condition leading to the concentration of $\Psi_1$, and admits $2(i-1)$ instability directions yielding one or several shocks formed by $\Psi_j$ for $j<i$. The linearised dynamics is described in Proposition \ref{pr:H}.


\subsection{Blow-up for the reduced parabolic system}

For a solution $u$ to \fref{eq:burgers} that is odd in $x$, the behaviour on the  transverse axis $\{x=0 \}$ is encoded by a closed system, which is the motivation for this symmetry assumption. It admits solutions blowing up simultaneously with a precise behaviour. Indeed, assume $\pa_x^j u_0(0,y)=0$ for all $y\in \mathbb R$ for $2\leq j \leq 2i$ for some integer $i\in \mathbb N$. This remains true for later times and the trace of the derivatives
\be \lab{def:xi}
\xi(t,y):=-\pa_x u (t,0,y) \ \ \text{and} \ \ \zeta (t,y):=\pa_x^{2i+1} u (t,0,y)
\ee
solve the parabolic system
\be \lab{eq:NLH}
\left\{ \begin{array}{l l} (NLH) \ \ \xi_t-\xi^2-\pa_{yy}\xi=0, \\
(LFH) \ \ \zeta_t-(2i+2)\xi \zeta-\pa_{yy}\zeta=0.
\end{array} \right.
\ee
Solutions to the nonlinear heat equation $(NLH)$ might blow up in finite time, a dynamics that can be detailed precisely, see \cite{QS} for an overview. There exists a stable fundamental blow-up \cite{BK,BK2,HV,MZ}, and a countable number of unstable "flatter" blow-ups \cite{BK2,HV2}, all driven to leading order by the ODE $f'=f^2$. For the present work, we had to show additional weighted estimates than those showed in these articles. In particular, we revisited the proof in \cite{BK2,HV2,MZ} and obtained a true improvement for the "flat" unstable blow-ups, see the comment below. For the solutions $\xi$ to $(NLH)$ below the solution to the linearly forced heat equation $(LFH)$ may also blow-up in finite time with precise asymptotic that we detail later on.

\begin{theorem} \lab{th:NLHinstable}

Let $J\in \mathbb N$. There exists an open set for a suitable topology of even solutions to $(NLH)$ blowing up in finite time $T>0$ with
$$
\xi (t,y)=\frac{1}{T-t} \ \frac{1}{1+\frac{y^2}{8(T-t)|\log (T-t)|}}+\tilde \xi,
$$
where the remainder $\tilde \xi$ satisfies for $0\leq j \leq J$ for some constant $C>0$:
$$
|\pa_y^j \tilde \xi |\leq \frac{C}{(T-t)|\log (T-t)|^{\frac 14}} \ \frac{1}{\left(1+\frac{y^2}{8(T-t)|\log (T-t)|}\right)^{\frac 34}} \frac{1}{\left(\sqrt{(T-t)|\log (T-t)|}+|y| \right)^j}.
$$
For any $k\in \mathbb N$, $k\geq 2$, $a>0$, there exists $T^*>0$, such that for any $0<T<T^*$ there exists an even solution to \fref{eq:NLH} blowing up at time $T$ with
\be \lab{id:decomposition NLH}
\xi (t,y)=\frac{1}{T-t+ay^{2k}} +\tilde \xi,
\ee
where the remainders $\tilde \xi$ satisfies for $j=0,...,J$ for some constant $C>0$:
\be \lab{bd:remainder NLH}
|\pa_y^j \tilde \xi | \leq C \left( (T-t)^{\frac{1}{2k}}+|y|\right)^{\frac 12-(2k+j)}.
\ee

\end{theorem}

\vspace{0.3cm}

{\it Comments on the result}.

\vspace{0.2cm}

The even assumption is not necessary, it is here to fit the even assumption on \fref{eq:burgers}. The construction that we give here for the second case of the unstable blow-ups is not a copy of the seminal previous ones \cite{BK2,HV2,MZ}, but a bit simpler and more precise. In particular, we extensively use the fact that these profiles are perturbations of the smooth unstable self-similar profiles of the quadratic equation $f_t=f^2$, and that away from the origin in self-similar variables the problem is a perturbation of the renormalised quadratic equation $f_s+f-(Z/2k)f_Z-f^2=0$. We avoid the use of maximum principle as in \cite{HV2} or of Feynman-Kac formula as in \cite{BK2,MZ}, and obtain a sharp estimate. Namely, the convergence of the solution to the blow-up profile holds in a spatial region that is of size one in original $y$ variables which is the estimate \fref{bd:remainder NLH}. For example, this estimate directly implies the existence of a profile at blow-up time $u(t,y)\rightarrow U^*(y)$ as $t\rightarrow T$ for $y\neq 0$, and that it satisfies $U^*(y)\sim (ay^{2k})^{-1}$ as $y\rightarrow 0$ (this fact would not be obtained directly in previous works).

\begin{proof}[Proof of Theorem \ref{th:NLHinstable}]

The second part, concerning the unstable blow-ups, is proved in Section \ref{sec:NLH}. The proof of the first part for the stable blow-up is very similar, and though our method is a bit simpler than \cite{BK2,MZ} it does not yield truly improved estimates, hence we just sketch the proof in Section \ref{sec:stable}.

\end{proof}


\subsection{Statement of the result}

The main result of this paper shows how, in a case with symmetries, the viscosity affects the shock formation of Burgers equation, resulting in a concentration of a self-similar shock $\Psi_i$ along the vertical axis $\{x=0\}$, with scaling parameters that are related to the solution of the parabolic system $(NLH)-(LFH)$. As a consequence, any blow-up solutions to the two one-dimensional equations can be combined to obtain a two-dimensional blow-up. Note that the solutions below can be chosen initially with compact support, and that we are only able to construct them around an initially concentrated blow-up profile. The first theorem involves the stable blow-up of $(NLH)$. The blow-up pattern is stable in a Banach space $\mathcal B$ of $C^4$ regularity with polynomial weight associated to the norm:
\be \label{bd:def mathcal B}
\| u \|_{\mathcal B}:=\sum_{j_1+j_2=0}^4 \left\| \frac{\langle x\rangle^{j_1}\langle y\rangle^{j_2}\pa_x^{j_1}\pa_y^{j_2}u}{\langle x \rangle^4(\langle y\rangle^3+\langle x \rangle)^{-\frac{11}{3}}}\right\|_{L^{\infty}(\mathbb R^2)}.
\ee

\begin{theorem} \lab{th:mainstable}

For any $i\in \mathbb N^*$ and $b>0$, there exists a Schwartz class solution $u$ to \fref{eq:burgers}, blowing up at time $T$ with
$$
u(t,x,y)= b^{-1} \lambda^{\frac{1}{2i+1}}(t,y)\Psi_i\left(b\frac{x}{\lambda (t,y)}\right)+ \tilde u(t,x,y)
$$
where $\Psi_i$ is defined by \fref{id:implicit} and the  transverse scale satisfies
$$
\lambda (t,y)=(T-t)^{1+\frac{1}{2i}} \ \left(1+\frac{y^2}{8(T-t)|\log (T-t)|}\right)^{1+\frac{1}{2i}}
$$
and one has the convergence in self-similar variables $(X,Z)$ 
\be \lab{main:bdtildeu'}
(T-t)^{-\frac{1}{2i}}u \left((T-t)^{1+\frac{1}{2i}}X,\sqrt{(T-t)|\log(T-t)|}Z\right) \rightarrow b^{-1}(1+Z^2/8)^{\frac{1}{2i}}\Psi_i \left(b\frac{X}{(1+Z^{2}/8)^{1+\frac{1}{2i}}}\right)
\ee
in $C^1$ on compacts sets and for some constants $C>0$ the remainder satisfies
\be \lab{main:bdtildeu2'}
\| \pa_x \tilde u \|_{L^{\infty}}\leq C (T-t)^{-1}|\log (T-t)|^{-\frac 14}.
\ee
For $i=1$, there exists a ball in $\mathcal B$ around $u(t=0)$ such that any other solution with initial datum in that set blows up with the same behaviour.

\end{theorem}

For $i=1$, we did not pursue optimality for the weighted space $\mathcal B$. Other choices for the weight may be possible, but the $C^4$ regularity is essential and could only be lowered to $C^{3+}$ by adapting the proof. Importantly, $\mathcal B$ allows for unbounded perturbations\footnote{The local well-posedness for equation \fref{eq:burgers2} in $\mathcal B$ follows from local well-posedness for localised data, using the finite speed of propagation of the equation along the $x$ variable.}, highlighting the fact that only the control of derivatives is of importance in this problem. The second theorem involves the unstable "flat" blow-ups of $(NLH)$.

\begin{theorem} \lab{th:main}

For any $k,i\in \mathbb N^*$, $k\geq 2$, $a,b>0$, there exists $T^*>0$, such that for any $0<T<T^*$ there exists a Schwartz class solution $u$ to \fref{eq:burgers} odd in $x$ and even in $y$, blowing up at time $T$ with
\be \lab{main:decompositionu}
u(t,x,y)= b^{-1} \lambda^{\frac{1}{2i+1}}(t,y)\Psi_i\left(b\frac{x}{\lambda (t,y)}\right)+ \tilde u(t,x,y)
\ee
where $\lambda (t,y)=(T-t+ay^{2k})^{1+\frac{1}{2i}}$ and one has the convergence in self-similar variables $(X,Z)$
\be \lab{main:bdtildeu}
(T-t)^{-\frac{1}{2i}}u \left((T-t)^{1+\frac{1}{2i}}X,(T-t)^{\frac{1}{2k}}Z\right) \rightarrow b^{-1}(1+aZ^{2k})^{\frac 12}\Psi_i \left(b\frac{X}{(1+aZ^{2k})^{\frac 32}}\right)
\ee
in $C^1$ on compact sets and for some constants $C,\eta>0$ the remainder satisfies
\be \lab{main:bdtildeu2}
\| \pa_x \tilde u \|_{L^{\infty}}\leq C (T-t)^{-1+\eta}.
\ee

\end{theorem}

\begin{proof}

Theorem \ref{th:main} is proved in Section \ref{sec:main}. It is a consequence of Proposition \ref{main:pr:bootstrap} and Lemma \ref{main:lem:pointwise e}. The proof of Theorem \ref{th:mainstable} is very similar, and is just sketched in Section \ref{sec:stable}.
 
\end{proof}

\subsection{Comments on the result and open problems}

\vspace{0.2cm}
\noindent 
{\it 1. Stability/Instability in the symmetric case}. The solutions of Theorem \ref{th:mainstable} with $i\geq 2$, or that of Theorem \ref{th:main} are instable within the class of odd in $x$ and even in $y$ solutions. For $i\geq 2$, these solutions are such that $\pa_x^3 u_{|x=0}=0$, and we expect a generic perturbation of this third order derivative on the axis to lead to the blow-up behaviour described in Proposition \ref{pr:LHinstable} for $k\geq 2$ and \fref{stable:bd:tildeg} for $k=1$. The sign of such perturbation is important as we believe $b<0$ would create shocks outside the vertical axis. For $k\geq 2$, these solutions are such that $\pa_x u_{|x=0}$ is an instable blow-up solution of (NLH) and generic perturbations lead to the stable blow-up $k=1$ \cite{HV3}. We do not believe that solutions of Theorem \ref{th:mainstable} with $i\geq 1$, though being stable, provide a generic blow-up behaviour in this symmetry class: the blow-up might occur at another point than the origin where such symmetry around that point would fail.\\

\noindent 
{\it 2. Symmetry breaking}. We expect all solutions to Theorem \ref{th:mainstable} and \ref{th:main} to be instable by symmetry breaking. Formally, our blow-up profile \fref{eq:def Theta} admits non-symmetrical analogues, but the anisotropic viscosity make them fail to be approximate solutions, so our Ansatz does not adapt. The symmetry class we use allow for a control of the viscosity effects. But we wonder wether outside this symmetry class, viscosity, via some kind of hypo-elliptic effect, might prevent blow-up. This can be seen formally by considering solutions of the form $u=V(t,x-\epsilon y)$. For $\epsilon=0$, $V$ solves Burgers equation \fref{eq:burgersusual} and $u$ might blow-up. By tilting slightly the symmetry axis $\epsilon>0$, $V$ solves the viscous Burgers equation $V_t+VV_x-\epsilon^2 V_{xx}=0$ and is global. Investigating the generic behaviour is thus an interesting open problem.\\

\noindent{\it 3. Anisotropy}. Very few results concerning a precise description of anisotropic singularity formation exist, despite its fundamental relevance in fluid dynamics. We see that here a wide range of different scaling laws in the $x$ and $y$ variables are possible. The formation of shocks for two-dimensional extensions of Burgers equation is studied in \cite{PLGG}. Let us also mention that in \cite{CMR,MRS} anistropic blow-ups were constructed for the energy supercritical semi-linear heat equation.\\

\noindent{\it 4. Connections between self-similar blow-ups}. \fref{eq:burgers} appears to be a good candidate to study connexions between self-similar profiles. As the concentration of $\Psi_i$ for $i\geq 2$ for Burgers is unstable with instabilities yielding to the concentration of $\Psi_j$ for $j<i$, and as the same should hold for the unstable blow-ups of $(NLH)$ (see \cite{HV3} for the genericity result), one interesting result would be to prove rigorously that solutions to \fref{eq:burgers} concentrating the $i$-th profile of Burgers and the $k$-th of $(NLH)$ are unstable with instabilities yielding to the concentration of the $j$-th profile of Burgers and the $\ell$-th of $(NLH)$ for $(j,\ell)<(i,k)$.\\

\noindent{\it 5. Continuation after blow-up}. The inviscid Burgers equation possesses global weak solutions that can be obtained using a viscous approximation and that are unique under a suitable entropy condition. The investigation of the analogous problem for \fref{eq:burgers} is natural. In particular, if the solution can be continued and has jumps, what is the set of points with discontinuities and its dynamics?


\subsection{Ideas of the proof and Organisation of the paper}

The result relies on the extension of a lower-dimensional blow-up along a new spatial direction, as in \cite{CMR,MRS}. Self-similar blow-up in Burgers equation is completely studied via direct computations, without technical difficulties. It is an easy setting to understand properties of blow-ups, for example regularity and stability issues and discretely self-similar singularities. The extension along the transverse direction is studied through modulation equations \fref{eq:NLH}, which for the first time are non-trivial PDEs. To obtain weighted estimates for $(NLH)$ we adapt \cite{MZ} and use a new exterior Lyapunov functional in Lemma \ref{NLH:lem:improvedoustide}, see the comments below Theorem \ref{th:NLHinstable}. The blow-up of the solution to $(NLF)$ can then be studied in the same analytical framework. The core of the paper is the 2-d analysis. The ideas are somewhat similar to those used in other contexts of blow-up through a prescribed profile, but are specific to the problem at hand and we hope that they will have other applications in transport and mixed hyperbolic/parabolic problems. We derive a blow-up profile with well-understood properties and linearisation, and build an approximate blow-up profile using modulation to neutralise growing modes. We then construct a solution in its vicinity via a bootstrap argument. We use solely weighted energy estimates, which are robust and reminiscent of a duality method for the asymptotic linear operator, and derivatives are taken along adapted vector fields to commute well with the equation.\\

\noindent The paper contains two independent sections devoted to Burgers equation and the modulation system, and another one proving the main theorem which can also be read separately as it uses their results as a black box. Section \ref{sec:burgers} concerns the self-similarity in the blow-ups of Burgers equation. Section \ref{sec:main} is devoted to the proof of Theorem \ref{th:main}, assuming some results for the derivatives on the vertical axis, Theorem  \ref{th:NLHinstable} and Proposition \ref{pr:LHinstable}. The blow-up profile and the linearised dynamics are studied in Lemma \ref{lem:Psi1} and Proposition \ref{pr:mathcalL}, and the heart of the proof is a bootstrap argument in Proposition \ref{main:pr:bootstrap}. Section \ref{sec:NLH} deals with the two Propositions \ref{th:NLHinstable} and \ref{pr:LHinstable} admitted in Section \ref{sec:main}, and concerns in particular the flat blow-ups for the semi-linear heat equation. Finally in Section \ref{sec:stable} we sketch how the proof of Theorem \ref{th:main} can be adapted to prove Theorem \ref{th:mainstable}.\\


\subsection{Notations}

\noindent We use the Japanese bracket notation
$$
\lY= (1+Y^2)^{\frac 12}.
$$
For functions having in argument a rescaling of the variable $X$, we use the general notation $\tilde X$ for their variable, as in $(\tilde X +\tilde X^2)(cX)=cX+(cX)^2$ for example. Depending on the context, $\tilde X$ will also refer to the main renormalised variable
\be \lab{eq:deftildeX}
\tilde X=\frac{X}{(1+Z^{2k})^{\frac 32}}
\ee
and there should not be confusions. We write $a\lesssim b$ if there exists a constant $C$ independent of the other constants of the problem such that $a\leq C b$. We write $a\approx b$ if $a\lesssim b$ and $b\lesssim a$. Generally, $C$ will denote a constant that is independent of the parameters used in the proof, whose value can change from one line to another. When its value depends on some parameter $p$, we will specify it by the notation $C(p)$. To perform localisations, the function $\chi$ is a smooth nonnegative cut-off function, $\chi=1$ on $[-1,1]$ and $\chi=0$ outside $[-2,2]$.

\subsection*{Acknowledgements}  C. Collot is supported by the ERC-2014-CoG 646650 SingWave. N. Masmoudi is supported by NSF grant DMS-1716466.



\section{Self similarity in shocks for Burgers equation} \lab{sec:burgers}

This section is devoted to the formation of shocks for Burgers equation
\be \lab{eq:burgersusual}
U_t+UU_x=0
\ee
This simple equation appears as a toy model for blow-up issues involving self-similar behaviours. However, we did not find works in which this was emphasised apart from \cite{EF} (though implicit in some other works) where the existence of smooth self-similar singularities and their linearised dynamics are briefly studied, the usual point of view being geometrical \cite{CEHL}. Everything is explicit, which is convenient as the picture described in Subsection \ref{subsec:Burgers} shares many similarities with other equations. In particular one sees the link between the regularity of the solution and its blow-up behaviours (this issue appearing in other hyperbolic equations as in \cite{KST}).


\subsection{Invariances}

If $U(t,x)$ is a solution to \eqref{eq:burgersusual}, then the following function is again a solution by time and space translation, galilean transformation, space and time scaling invariances:
$$
\frac{\mu}{\lambda}U\left( \frac{t-t_0}{\lambda},\frac{x-x_0-ct}{\mu}\right)+c.
$$
In particular for $\lambda \in \mathbb R^*_+$ and $\alpha \in \mathbb R$, $\lambda^{\alpha-1}U\left(t/\lambda,x/\lambda^{\alpha}\right)$ is also a solution. 
The associated infinitesimal generators of the above transformations are\footnote{Here $\text{Id}$ stands for the identity and $1$ for the function with constant value $1$.}
\be \lab{burgers:def:infinitesimal}
\Lb_\mu :=\text{Id}-x\pa_x, \ \ \tilde \Lb^{(\alpha)}_\lb :=-(1-\alpha)\text{Id}-t\pa_t-\alpha x\pa_x, \ \ \tilde \Lb_c:=-t\pa_x +1, \ \ \Lb_{x_0}:=-\pa_x, \ \ \Lambda_{t_0}:=-\pa_t
\ee
and there holds the commutators relations
\be \lab{id:comutators}
[\Lb _\mu,\tilde \Lb_{\lb}^{(\alpha)}]=0, \ \ [\tilde \Lb _c,\tilde \Lb_{\lb}^{(\alpha)}]=-(\alpha-1) \tilde \Lb_c+1, \ \ [\Lb _{x_0},\tilde \Lb_{\lb}^{(\alpha)}]=-\alpha \Lb_{x_0}, \ \ [\Lb_{t_0},\tilde \Lb_{\lb}^{(\alpha)}]=-\Lb_{t_0}.
\ee
The tilde comes from the fact that we will use their spatial counterparts:
\be \lab{burgers:def:infinitesimalspatial}
\Lb_{\alpha}:= (1-\alpha)\text{Id}+\alpha X\pa_X, \ \ \Lb_c:= \pa_X +1, \ \ \Lb := -1+X\pa_X
\ee


\subsection{Self-similar and discretely self-similar solutions}

Important solutions are those who constantly reproduce themselves to smaller and smaller scales. To measure their regularity, let us define the following H\"older spaces. For $i\in \mathbb N$ one takes $C^i(\mathbb R)$ to be the usual space of real-valued functions $i$ times continuously differentiable on $\mathbb R$. For $i\in \mathbb N$ and $\delta\in (0,1)$, $C^{i+\delta}$ is the set of functions $f\in C^{i}(\Omega)$ such that
$$
\lim_{x\uparrow x_0} \frac{\pa_x^i f(x)-\pa_x^if(x_0)}{|x-x_0|^{\delta}} \ \ \text{and} \ \ \lim_{x\downarrow x_0} \frac{\pa_x^i f(x)-\pa_x^if(x_0)}{|x-x_0|^{\delta}}
$$
are well-defined for all $x_0\in \mathbb R$. We then use the notation $C^{r+}=\cup_{r'>r} C^{r'}$ and $C^{r-}=\cup_{r'<r}C^{r'}$. Assume $U$ is a $C^1$ solution to Burgers equation becoming singular at a singularity point $(t_0,x_0)$. Then one can always use gauge invariance to map it to a solution defined on some domain $(T,0)\times \mathbb R$ with $T<0$, that becomes singular at $(0,0)$ and such that $U(t,0)=0$ for all $t\in (T,0)$. In particular, $U_x(t,\cdot)$ is minimal at the origin with $U_x(t,0)=-1/t$. The subgroup of the invariances $\mathbb R^3\times (\mathbb R^*)^2$ preserving these properties is
$$
g=(\lambda,\mu) \in \mathcal G:=(0,+\infty)^2, \ \ g.U:(t,x)\mapsto \frac{\mu}{\lambda}U\left(\frac{t}{\lambda},\frac{x}{\mu} \right).
$$
Let $\Omega:=(-\infty,0)\times \mathbb R$. The stabiliser of $U\in C^1(\Omega)$ is the subgroup $\mathcal G_s(U):=\{g\in \mathcal G, \ g.U=U\}$. Solutions with invariances can be classified according to their regularity.

\begin{proposition}[Classification of self-similar solutions] \lab{pr:clas}

Let $U\in C^1(\Omega)$ be a solution to \fref{eq:burgersusual} with $U(-1,0)=0$, $\inf_{\mathbb R} U_x(-1,\cdot)=U_x(-1,0)=-1$ and such that $\mathcal G_s$ is nontrivial. Then three scenarios only are possible and exist, the profiles $\Psi \in C^1(\mathbb R)$ below being defined in Propositions \ref{prop:smoothselfsim} and \ref{prop:roughdselfsim} and in \fref{burgers:def:roughselfsim}.

\begin{itemize}
\item[-] \emph{Analytic self-similarity:} $U$ is analytic and there exists $i\in \mathbb N$ and $\mu>0$ such that
$$
U(t,x)=\mu^{-1} (-t)^{\frac{1}{2i}} \Psi_i \left(\mu\frac{x}{(-t)^{1+\frac{1}{2i}}} \right),
$$
or $U(t,x)= \Psi_{\infty}(x/(-t))=x/t$.
\item[-] \emph{Non-smooth self-similarity:} There exists $i,\mu,\mu'>0$ with $i\notin \mathbb N$ and $\mu=\mu'$ (resp. $i>0$ and $\mu\neq \mu'$) such that
$$
U(t,x)=(-t)^{\frac{1}{2i}} \Psi_{(i,\mu,\mu')}\left(\frac{x}{(-t)^{1+\frac{1}{2i}}} \right).
$$
where $\Psi_{(i,\mu,\mu')}$ is defined by \fref{burgers:def:roughselfsim}, and $\Psi_{(i,\mu,\mu')}\in C^{1+2i}(\mathbb R)$, $\Psi_{(i,\mu,\mu')}\notin C^{1+2i+}(\mathbb R)$ (resp. $\Psi_{(i,\mu,\mu')}\in C^{1+2i-}(\mathbb R)$, $\Psi_{(i,\mu,\mu')}\notin C^{1+2i}(\mathbb R)$).
\item[-] \emph{Non-smooth discrete self-similarity:} There exists $i>0$ and $\lambda>1$ such that $U\notin C^{1+2i}(\Omega)$ (there exist such solutions with any regularity bewteen $C^1$ and $C^{1+2i-}$), that for all $k\in \mathbb Z$:
$$
U(t,x)=\lambda^{\frac{k}{2i}} U\left(\frac{t}{\lambda^k},\frac{x}{\lb^{k\left(1+\frac{1}{2i}\right)}}\right),
$$
and that there exists $(t,x)\in \Omega$ such that $U(t,x)\neq (-t)^{1/(2i)}U(-1,x/(-t)^{1+1/(2i)})$.
\end{itemize}

\end{proposition}

Before proving the above Proposition \ref{pr:clas}, let us present the self-similar and discretely self-similar solutions.

\begin{proposition}[Self-similar solutions \cite{EF}] \lab{prop:smoothselfsim}

There exists a set $ \{ \Psi_i, \ i \in \mathbb N^* \}\cup \{ \Psi_{\infty}\}$ of analytic functions on $\mathbb R$ with the following properties. One has $\Psi_{\infty}(X)=-X$. For $i\in \mathbb N^*$, the function $\Psi_i$ is odd, decreasing, and concave on $(-\infty,0]$, satisfy the implicit equation
\be \lab{id:implicit}
X=-\Psi_i(X)-(\Psi_i(X))^{1+2i}
\ee
and have the following asymptotic expansions:
\be \lab{id:dvptUic}
\Psi_i^{(i)}(X)=-X+X^{2i+1}+\sum_{k=2}^{+\infty}c_{i,k}X^{2ki+1} \ \ \text{as} \ X\rightarrow 0,
\ee
\be \lab{id:asUiinfty}
\Psi_i(X)= -\text{sgn}(X) |X|^{\frac{1}{1+2i}}+\text{sgn}(X)\frac{|X|^{-1+\frac{2}{2i+1}}}{2i+1}+O(|X|^{-2+\frac{3}{2i+1}}) \ \ \text{as} \ |X|\rightarrow +\infty.
\ee
Moreover, it solves the equation
\be \lab{eq:smoothselfsim}
-\frac{1}{2i} \Psi_i +\frac{2i+1}{2i} X\pa_X \Psi_i+\Psi_i\pa_X\Psi_i=0
\ee
and any other globally defined $C^1$ solution is of the form $\Psi=\mu^{-1} \Psi_i(\mu X)$ for some $\mu>0$ or is $-X$ or $0$. The functions $U^{(\infty)}(t,x)=x/t$ and $U^{(i,\mu)}(t,x)=\mu^{-1}(-t)^{1/(2i)}\Psi_i (\mu x/(-t)^{1+1/(2i)})$ are solutions to \fref{eq:burgersusual}.

\end{proposition}

\begin{proof}

Consider the function $\phi (\Psi)=-\Psi-\Psi^{2i+1} $ which is an analytic diffeomorphism on $\mathbb R$. Its inverse $\Psi_i:=\phi^{-1}$ satisfies \fref{id:implicit}, \fref{id:dvptUic}, \fref{id:asUiinfty} and the other properties of the proposition from direct computations. Since
\begin{align*}
&\frac{2i}{(\phi^{-1})'(X)}\left(-\frac{1}{2i}\phi^{-1}(X)+\frac{2i+1}{2i}X(\phi^{-1})'(X)+\phi^{-1}(X)(\phi^{-1})'(X)\right)\\
&\qquad \qquad \quad =-\phi^{-1}(X)\phi'(\phi^{-1}(X))+(2i+1)X+2i\phi^{-1}(X)\\
&\qquad \qquad \qquad \qquad \qquad =-\phi^{-1}(X)(-1-(2i+1)(\phi^{-1}(X))^{2i})+(2i+1)X+2i\phi^{-1}(X)\\
&\qquad \qquad \qquad \qquad \qquad \qquad \qquad \qquad =-(2i+1)(-\phi^{-1}(X)-(\phi^{-1}(X))^{2i+1})+(2i+1)X=0,
\end{align*}
it solves the equation \fref{eq:smoothselfsim}. Since it solves this equation, $U(t,x)=(-t)^{1/(2i)}\Psi_i(x/(-t)^{1+1/(2i)})$ solves \fref{eq:burgersusual} since introducing $\alpha_i=1+1/(2i)$:
\bee
U_t+UU_x&=&-(\alpha_i-1)(-t)^{\alpha_i-2}\Psi_i\left(\frac{x}{(-t)^{\alpha_i}}\right)+\alpha_i(-t)^{\alpha_i-2}\frac{x}{(-t)^{\alpha_i}}\pa_X \Psi_i\left(\frac{x}{(-t)^{\alpha_i}}\right)\\
&&+(-t)^{\alpha_i-1}\Psi_i\left(\frac{x}{(-t)^{\alpha_i}}\right)(-t)^{-1}\pa_X \Psi_i\left(\frac{x}{(-t)^{\alpha_i}}\right) \\
&=&(-t)^{\alpha_i-2}\left(-(\alpha_i-1)\Psi_i+\alpha_i y U_y^{(i)}+\Psi_i\pa_X \Psi_i \right)\left(\frac{x}{(-t)^{\alpha_i}}\right)=0.
\eee
The same reasoning applies for $\mu^{-1}\Psi_i(\mu X)$ since \fref{eq:smoothselfsim} is invariant under the transformation $\Psi \mapsto \mu^{-1}\Psi(\mu X)$. If $\Psi$ is another solution to 
\fref{eq:smoothselfsim} with $-1<\Psi(1)<0$ then using this invariance $\Psi=\mu^{-1}\Psi (\mu X)$ for some $\mu>0$. If $\Psi(1)<-1$ or $\Psi(1)>0$ it is easy to check that the solution is not globally defined.

\end{proof}

There also exist solutions reproducing themselves to a smaller scale, but in a somewhat periodic manner, unlike self-similar solutions. They have a fractal behaviour near the origin and are never smooth.

\begin{proposition}[Non-smooth discretely self-similar blow-up] \lab{prop:roughdselfsim}

Let $\alpha >1$, $\lambda>1$, $X_0,X_1\in (-\infty,0)$ with $\lb^{1-\alpha}X_0<X_1<\lb^{-\alpha}X_0$ and consider a function $V\in C^1([X_0,X_1),\mathbb R)$ satisfying\footnote{The set of such functions is non empty and it contains profiles which do not satisfy $(1-\alpha) V +\alpha XV_X+VV_X=0$.}
$$
X_1=\lambda^{-\alpha} X_0+(\lambda^{-\alpha}-\lambda^{1-\alpha})V(X_0),
$$
$$
V(X)\in (0,-X) \ \ \text{and} \ \ V_X(X)\in (-1,0) \ \ \text{on } \ \ [X_0,X_1),
$$
and
\be \lab{id:cond v}
\lim_{X\rightarrow X_1}V(X)=\lambda^{1-\alpha}V(X_0), \ \ \lim_{X\rightarrow X_1}V_X(X)=\frac{\lambda V_X(X_0)}{1-(\lambda-1)V_X(X_0)}.
\ee
Then there exists a unique odd function $W\in \mathcal C^1(\mathbb R)$ such that for all $X\in \mathbb R$,
\be \lab{burgers:eq:dss}
W(X)=\lambda^{1-\alpha}W\left(\lambda^{\alpha}X+(\lambda^{\alpha}-\lambda^{\alpha-1})W(X)\right)
\ee
and $W=V$ on $[X_0,X_1)$. One has $W(X)\in (0,-X)$ and $W_X(X)\in (-1,0)$ for all $X\in (-\infty,0)$, and its derivative is minimal at the origin with value $W_X(0)=-1$. Let $i=1/(2(\alpha-1))$. Then
$$
0<\liminf_{X\uparrow 0} \frac{-W(X)-X}{|X|^{1+2i}}\leq \limsup_{X\uparrow 0} \frac{-W(X)-X}{|X|^{1+2i}}<+\infty 
$$
with equality if and only if $W(X)=\mu^{-1}\Psi_i(\mu X)$ for some $\mu>0$ where $\Psi_i$ is given by \fref{id:implicit}. Therefore, unless $W=\mu^{-1}\Psi_i(\mu X)$ one has $W\notin C^{2i+1}$. There exist such solutions of regularity $C^{2i+1-\epsilon}$ for any $\epsilon>0$. Moreover, the solution $U$ defined on $(-\infty,0)\times \mathbb R$ as the solution to \fref{eq:burgersusual} with $U(-1,x)=W(x)$ satisfies
\be \lab{id:dss}
U(t,X)=\frac{1}{\lambda^{k(1-\alpha)}}U\left(\frac{t}{\lambda^k},\frac{X}{\lambda^{k\alpha}}\right)
\ee
for all $(t,X,k)\in (-\infty,0)\times \mathbb R \times \mathbb Z$.

\end{proposition}

\begin{remark}

If $(1-\alpha) W +\alpha XW_X+WW_X\neq 0$, then $U$ is not of the form $U=(-t)^{\alpha-1}W( x/(-t)^{\alpha})$, implying that the set of all $k\in \mathbb R$ such that \fref{id:dss} hold is isomorphic to $\mathbb Z$ and that the solution is not continuously self-similar.

\end{remark}

\begin{proof}

We proceed in two steps. First we extend $V$ in a periodic manner, and then we show the regularity properties.\\

\noindent \textbf{Step 1} \emph{Construction}. Consider the mapping $\phi:[X_0,X_1)\rightarrow \mathbb R$ defined by
$$
\phi (X)=\lambda^{-\alpha} X+(\lambda^{-\alpha}-\lambda^{1-\alpha})V(X).
$$
One has $\phi(X_0)=X_1$ and since $\lambda,\alpha>1$ and $V_X\in (-1,0)$ one computes
$$
\phi_X (X) = \lambda^{-\alpha}+(\lambda^{-\alpha}-\lambda^{1-\alpha})V_X(X)> \lambda^{-\alpha}>0 
$$
and hence $\phi$ is a $C^1$ diffeomorphism onto its image. Define
$$
X_2=\lim_{X\rightarrow X_1} \phi(X)=\lambda^{-\alpha} X_1+(\lambda^{-\alpha}-\lambda^{1-\alpha})\lb^{1-\alpha}V(X_0).
$$
and for $X\in [X_1,X_2)$ extend $V$ by
$$
W(X)=\lambda^{1-\alpha}V(\phi^{-1}(X)).
$$
\textbf{Claim}: One has $X_1<X_2<0$, that $W$ is $C^1$ on $[X_0,X_2)$ and that restricted to $[X_1,X_2)$ it satisfies the condition of the proposition. Moreover for all $X\in [X_1,X_2)$, $\lambda^{\alpha} X+(\lambda^{\alpha}-\lambda^{\alpha-1})W(X)\in [X_0,X_1)$ and
$$
W(X)=\lambda^{1-\alpha}W\left(\lambda^{\alpha}X+(\lambda^{\alpha}-\lambda^{\alpha-1})W(X)\right).
$$
The proof of this claim involves only basic computations that we  omit here. \\

From the Claim we see that we can repeat the construction a countable number of time. If $(X_k)_{k\in \mathbb N}$ denotes the set of points coming from the construction then by induction
$$
X_k=\lb^{-k\alpha} X_0+(\lambda^{-k\alpha}-\lambda^{k(1-\alpha)})W(X_0)
$$
hence $X_k\rightarrow 0$. The construction then provides a $C^1$ extension $W$ of $V$ on $(X_0,0)$ such that for all $X$ in this set, $0<W(X)<-X$, $-1<W_X<0$, and for all $X_1\leq X<0$, $\lambda^{\alpha} X+(\lambda^{\alpha}-\lambda^{\alpha-1})W(X)\in [X_0,0)$ and
$$
W(X)=\lambda^{1-\alpha}W\left(\lambda^{\alpha}X+(\lambda^{\alpha}-\lambda^{\alpha-1})W(X)\right).
$$

\noindent \textbf{Step 2:} \emph{Properties}. From the definition of the extensions one has
$$
\sup_{X\in [X_{k+1},X_{k+2}]} |W(X)|= \lb^{1-\alpha}\sup_{X\in [X_{k},X_{k+1}]} |W(X)|
$$
and therefore $\lim_{X\rightarrow 0}W(X)=0$. From \fref{burgers:eq:dss} one sees that
$$
\pa_X W \left(\lb^{-\alpha}X+(\lb^{-\alpha}-\lb^{1-\alpha})W(X) \right)=f(\pa_XW(X)), \ \ f(a)=\frac{\lb a}{1-(1-\lb)a}.
$$
$f$ has two fixed points $-1$ and $0$, is increasing on $(-1,0)$ with $-1<f(a)<a$. Therefore,
$$
-1 < \inf_{[X_{k+1},X_{k+2}]} \pa_X W=f(\inf_{[X_{k},X_{k+1}]} \pa_X W) \leq \sup_{[X_{k+1},X_{k+2}]} \pa_X W= f(\sup_{[X_{k},X_{k+1}]} \pa_X W )\rightarrow -1
$$
implying that $\pa_X W(X)\rightarrow -1$ as $X\rightarrow 0$, and in particular $\pa_X W$ is minimal at the origin with $\pa_X W(0)=-1$.
We now prove the absence of regularity at the origin. Take any $z_0\in [X_0,X_1)$ and define the sequence $z_k$ by induction following
$$
z_k=\lb^{-\alpha}z_{k-1}+(\lb^{-\alpha}-\lb^{1-\alpha})W(z_{k-1}).
$$
It follows that $W(z_k)=\lambda^{1-\alpha}W(z_{k-1})$. By induction,
$$
z_k=\lb^{-k\alpha} z_0+(\lambda^{-k\alpha}-\lambda^{k(1-\alpha)})W(z_0)=-\lb^{k(1-\alpha)}W(z_0)(1+O(\lb^{-k}))
$$
as $k\rightarrow +\infty$ since $\lb,\alpha>1$, with the constant in the $O$ uniform in $z_0\in [X_0,X_1)$. By induction,
$$
-z_{k+1}-W(z_{k+1})=\lb^{-k\alpha}(-z_0-W(z_0)).
$$
Therefore,
$$
\frac{-z_k-W(z_k)}{|z_k|^{\frac{\alpha}{\alpha-1}}}\rightarrow \frac{-z_0-W(z_0))}{|W(z_0)|^{\frac{\alpha}{\alpha-1}}}>0
$$
as $k\rightarrow +\infty$. One then deduces that since the convergence is uniform for $z_0$ taken in $[X_0,X_1)$,
$$
0<\liminf_{X\rightarrow 0} \frac{-X-W(X)}{|X|^{\frac{\alpha}{\alpha-1}}}\leq \limsup_{X\rightarrow 0} \frac{-X-W(X)}{|X|^{\frac{\alpha}{\alpha-1}}}<+\infty .
$$
Therefore the solution is not $C^{\frac{\alpha}{\alpha-1}}$ if the equality does not hold. Assume now the equality. This means that there exist a constant $c>0$ such that for any $X\in [X_0,X_1)$ one has
$$
\frac{-X_0-W(X_0))}{|W(X_0)|^{\frac{\alpha}{\alpha-1}}}=c, \ \ \text{i.e.} \ \ X=-W-c|W|^{\frac{\alpha}{\alpha-1}}.
$$
$W$ is then the self-similar profile\footnote{For $\alpha\neq 1+1/(2i)$ for $i\in \mathbb N$ the profiles $\Psi_{1/(2(\alpha-1))}$ defined in Proposition \ref{prop:smoothselfsim} still exist and have all the corresponding properties, they just are no longer smooth.} of Proposition \ref{prop:smoothselfsim}, and is not discretely self-similar. \\

One can apply the same extension technique to define $W$ on the other side $(-\infty,X_0)$. The uniqueness of the extension follows from an induction, using the fact that if $W$ is given on some $[X_{k},X_{k+1})$ then it has to be given on $[X_{k+1},X_{k+2})$ by the construction we provided. We leave to the reader to prove that if $V\in C^{\gamma}$ for some $1<\gamma<\alpha/(\alpha-1)$ then so is $W$.\\

\end{proof}

Self-similar and discretely self-similar solutions having been presented in Propositions \ref{prop:smoothselfsim} and \ref{prop:roughdselfsim}, we can now give the proof of the classification Proposition \ref{pr:clas}.

\begin{proof}[Proof of Proposition \ref{pr:clas}]

We only sketch the proof, since either the computations involved are rather easy or they are very similar to what can be found in the proofs of Proposition \ref{prop:smoothselfsim} and \ref{prop:roughdselfsim}. The stabiliser of $U$ is closed in $\mathcal G$ from the regularity of $U$. One identifies $\mathcal G_s$ to a closed subgroup of $\mathbb R^2$ via $(z_1,z_2)=(\log (\lambda),\log (\mu))$, and recall that closed subgroups of $\mathbb R^2$ are isomorphic to one of the following groups: $\mathbb Z$, $\mathbb R$, $\mathbb Z \times \mathbb Z$, $\mathbb R\times \mathbb Z$ or $\mathbb R^2$.\\

\noindent \textbf{Case 1, $\mathcal G_s \simeq \mathbb Z$:} in that case $\mathcal G_s=\{ (\lambda^k,\mu^k), \ k\in \mathbb Z\}$ for $(\lambda,\mu)\neq (1,1)$ meaning that 
$$
U(t,x)=\frac{\mu^k}{\lambda^{k}}U \left(\frac{t}{\lambda^k},\frac{x}{\mu^{ k}}\right), \ \ \forall k\in \mathbb Z.
$$
One can check that if $\lb=1$ then $u=c(t)x$, and if $\mu=1$ then $U=0$, which are contradictions. Hence $\lambda,\mu\neq 1$ and we define $\alpha \in \mathbb R$, by $\mu=\lambda^{\alpha}$ giving $\mathcal G_s=\{ (\lambda^k,\lambda^{k\alpha}), \ k\in \mathbb Z\}$. For all $k\in \mathbb Z$,
$$
U(t,x)=\frac{1}{\lambda^{(1-\alpha)k}}U \left(\frac{t}{\lambda^k},\frac{x}{\lambda^{\alpha k}}\right)
$$
and since $\mathcal G_s \not\simeq \mathbb R$ there exists $(t,x)$ such that $U(t,x)\neq (-t)^{\alpha-1}U(-1,x/(-t)^{\alpha})$. One can always take $\lambda>1$. We take $t=-1$, $k=1$, to obtain
$$
U\left( \frac{-1}{\lambda},x\right)=\lambda^{1-\alpha}U(-1,\lambda^{\alpha}x).
$$
From the relation on characteristics
$$
U(-1,x)=U\left( -\frac{1}{\lambda},x+\left(1-\frac{1}{\lambda}\right)U(-1,x)\right).
$$
Introducing the profile $W(X):=U(-1,X)$ one deduces that it satisfies
\be \lab{id:eq U dss clas}
W(X)=\lambda^{1-\alpha}W\left(\lambda^{\alpha}X+(\lambda^{\alpha}-\lambda^{\alpha-1})W(X)\right)
\ee
and that $W$ is $C^1$ with $W(0)=0$ and $W_X$ minimal at zero with $W_y(0)=-1$. We claim that if $\alpha>1$ then $W$ is a function as described in Proposition \ref{prop:roughdselfsim} in which the above functional equation was studied. We claim that the case $\alpha=1$ is impossible and that if $\alpha<1$ the function is not $C^1$ by looking at its behaviour at the origin. Therefore Case 1 corresponds to Proposition \ref{prop:roughdselfsim}.\\

\noindent \textbf{Case 2, $\mathcal G_s \simeq \mathbb R$:} in that case $\mathcal G_s=\{ (\lambda^a,\mu^a), \ a\in \mathbb R\}$ for $(\lambda,\mu)\neq (1,1)$ meaning that
$$
U(t,x)=\frac{\mu^a}{\lambda^{a}}U \left(\frac{t}{\lambda^a},\frac{x}{\mu^{ a}}\right), \ \ \forall a \in \mathbb R.
$$
This group of transformation contains the cases $a\in \mathbb Z$, and we have seen in case 1 that one cannot have $\lambda=1$ or $\mu=1$. Hence $\lambda \neq 1$ and $\mu \neq 1$. Define $\alpha$ by $\mu=\lambda^{\alpha}$ giving (up to an abuse of notation) $\mathcal G_s=\{ (\lambda,\lambda^{\alpha}), \ \lambda >0\}$ and that for all $\lambda >0$,
$$
U(t,x)=\frac{1}{\lambda^{(1-\alpha)}}U \left(\frac{t}{\lambda},\frac{x}{\lambda^{\alpha}}\right).
$$
In particular, $u$ is invariant by the transformation
$$
U(t,x)=\frac{1}{\lambda^{k(1-\alpha)}}U \left(\frac{t}{\lambda^k},\frac{x}{\lambda^{k\alpha}}\right).
$$
for any fixed $\lambda>1$ and $k\in \mathbb Z$. We have seen in the study of Case 1 that one cannot have $\alpha<1$ for such an invariance, hence $\alpha>1$. We now write
$$
U(t,x)=\frac{1}{(-t)^{(1-\alpha)}}U \left(-1,\frac{x}{(-t)^{\alpha}}\right).
$$
Hence the profile $W(X)=U(-1,X)$ satisfies the equation
$$
(1-\alpha)W+\alpha XW_X+WW_X=0.
$$
Solutions to this equation with $W(0)=0$, $W_X$ minimal at $0$ with $W_X(0)=-1$ have been classified when $1/(2(\alpha-1))\in \mathbb N^*$ in Proposition \ref{prop:smoothselfsim}. It is straightforward to check that if $1/(2(\alpha-1))\notin \mathbb N^*$ the profiles $\Psi_{1/(2(\alpha-1))}$ defined in Proposition \ref{prop:smoothselfsim} exist, have all the corresponding properties, and are $C^{\alpha/(\alpha-1)}$. Any self-similar shocks can then be written in the form
\be \lab{burgers:def:roughselfsim}
\Psi_{(i,\mu,\mu')}(X)= \left\{ \begin{array}{l l} \mu^{-1}\Psi_i(\mu X) \ \ \text{if} \ X\leq 0, \\ \mu^{'-1}\Psi_i(\mu'X) \ \ \text{if} \ X\geq 0,  \end{array} \right.
\ee
for $i\in \mathbb R$, $i>0$ where $\Psi_i$ is given by \fref{id:implicit}. When $\alpha=1$, the only solution to $XW_X+WW_X=0$ with $W(0)=0$ and $W_X(0)=-1$ is $W(X)=-X$ which is a contradiction.\\

\noindent \textbf{Case 3} If $\mathcal G_s \simeq \mathbb Z^2$ or $\mathcal G_s \simeq \mathbb Z\times \mathbb R$, in this case that there exists a subgroup of $\mathcal G_s$ of the form $\{ (\lambda^k,\lambda^{k\alpha}), \ k\in \mathbb Z \}$ with $\lb>1$ and $\alpha<0$. Indeed, the mapping $(\lambda,\mu)\mapsto (\log\lambda,\log\mu)$ transforms $\mathcal{G}_s$ into a subgroup of $\mathbb{R}^2$. This new subgroup is also isomorphic to $\mathbb{Z}^2$ or $\mathbb{Z}\times\mathbb{R}$. Any such subgroup must contains a point $(z_1,z_2)$ in the bottom right quadrant $z_1>0$ and $z_2<0$, then $\lambda=e^{z_1}$ and $\alpha=\frac{z_2}{z_1}$ gives the desired subgroup. From the study of Case 1, such an invariance is impossible. If $\mathcal G_s \simeq \mathbb R^2$, one can check that $u(t,x)=x/t$.

\end{proof}


\subsection{Stability and convergence at blow-up to self-similar solutions}

The suitable framework for the stability of $\Psi_i$ is that of self-similar variables where the linearised operator is
\be \lab{def:H}
H_X:=\Lambda_{\alpha_i}+\Psi_i\pa_X+\pa_X \Psi_i=(1-\alpha_i)+\pa_X \Psi_i+(\alpha_iX+\Psi_i)\pa_X.
\ee

\begin{proposition}[Spectral properties of $H_X$ \cite{EF}] \lab{pr:H}

The point spectrum of $H_X$ on smooth functions is 
\be \lab{def:nuk}
\Upsilon (H_X)=\left\{\frac{j-2i-1}{2i}, \ \ j \in \mathbb N \right\}.
\ee
The eigenfunctions related to symmetries are
\be \label{burgers:id:eigenfunctions commutators}
H_X \Lambda_{x_0}\Psi_i=-\alpha_i \Lambda_{x_0} \Psi_i, \ \ H_X(\Lb_{\alpha_i}\Psi_i)=-\Lb_{\alpha_i}\Psi_i, \ \ H_X(\Lambda_c \Psi_i)=-(\alpha_i-1)(\Lambda_c \Psi_i), \ \ H_X\Lambda_{\mu} \Psi_i=0.
\ee
More generally, the eigenfunctions are given by the formula:
\be \lab{eq:def phiXj}
H_X(\phi_{X,j})=\frac{j-2i-1}{2i}\phi_{X,j}, \ \ \phi_{X,j}:=\frac{(-1)^k\Psi_i^j }{1+(2i+1)\Psi_i^{2i}}.
\ee
They have the following asymptotic behaviour:
\be \lab{id:as phik0}
\phi_{X,j}(X)=X^{j}-(j+2i+1)X^{j+2i}+O(X^{j+4i}) \ \ \text{as} \ \ X \rightarrow 0,
\ee
\be \lab{id:as phikinfty}
\phi_{X,j}(X)=\frac{1}{2i+1}|X|^{\frac{j-2i}{2i+1}}+O(|X|^{\frac{jf-2i}{2i+1}-\frac{2i}{2i+1}}) \ \ \text{as} \ \ X \rightarrow +\infty.
\ee

\end{proposition}

\begin{proof}

\noindent \textbf{Step 1} \emph{Proof of \fref{burgers:id:eigenfunctions commutators}}. Let $U(t,x):=(-t)^{\alpha_i-1}\Psi_i(x/(-t)^{\alpha_i})$ which solves \fref{eq:burgersusual} and by invariance, $(\tau^{(3)}_cU)_t=-(\tau^{(3)}_cU)\pa_x(\tau^{(3)}_cU)$ for any $c\in \mathbb R$. Differentiating with respect to $c$ one obtains $(\tilde \Lb_c U)_t=-\tilde \Lb_cU \pa_x U-U\pa_x(\tilde \Lb_c U)$, which evaluated at $t=-1$ yields:
\be \lab{burgers:id:eigenfunctions commutators 1}
\pa_t \left(\tilde \Lb_c U \right)(-1,\cdot)=-\Psi_i \pa_X (\Lb_c \Psi_i)-\pa_X \Psi_i \Lb_c \Psi_i.
\ee
Self-similarity implies from \fref{burgers:def:infinitesimal} that $\tilde \Lb_\lb^{(\alpha_i)}u=0$ hence $\tilde \Lb_\lb^{(\alpha_i)} \tilde \Lb_c u+[\tilde \Lb_c,\tilde \Lb_\lb^{(\alpha_i)}]u=0$. This identity reads from the commutator relation \fref{id:comutators}:
$$
\tilde \Lb_\lb^{(\alpha_i)} \tilde \Lb_c u =(\alpha_i-1)\tilde \Lb_c u.
$$
At time $t=-1$ the above identity yields from \fref{burgers:def:infinitesimal} and \fref{burgers:def:infinitesimalspatial}:
$$
\pa_t (\tilde \Lb_c u)(-1,\cdot)-(1-\alpha_i)\Lb_c \Psi_i-\alpha_i X\pa_X \Lb_c \Psi_i=(\alpha_i-1) \Lb_c \Psi_i.
$$
From \fref{burgers:id:eigenfunctions commutators 1} the left hand side in this identity is $-H_X\Lambda_c \Psi_i$, ending the proof of \fref{burgers:id:eigenfunctions commutators}. The proof for the eigenfunctions related to the other symmetries \fref{burgers:id:eigenfunctions commutators} is exactly the same.\\

\noindent \textbf{Step 2} \emph{Proof of \fref{def:nuk} and \fref{eq:def phiXj}}. Assume $f$ solves $H_Xf=\nu f$. Then using the implicit equation \fref{id:implicit} one obtains:
$$
\frac{\pa f }{\pa \Psi_i} =  f\left[\frac{\alpha_i +\nu +(\alpha_i-1+\nu)(2i+1)\Psi_i^{2i}}{(\alpha_i-1)\Psi_i+\alpha_i \Psi_i^{2i+1}} \right] =  f\left[\frac{2i+1+2i\nu+(1+2i\nu)(2i+1)\Psi_i^{2i}}{\Psi_i+(2i+1)\Psi_i^{2i+1}} \right] \\
$$
whose solution is of the form
$$
f \in \text{Span} \left(\frac{\Psi_i^{2i+1+2i\nu }}{1+(2i+1)\Psi_i^{2i}} \right)
$$
From \fref{id:dvptUic} the above formula defines a smooth function if and only if $\nu=(j-2i-1)/(2i)$ for some $j \in \mathbb N$.

\end{proof}

The smooth self-similar profiles are the asymptotic attractors of all smooth and non-degenerate shocks in the following sense.

\begin{proposition} \lab{pr:nbassin}

Let $U_0\in C^{\infty}(\mathbb R)$ be such that $\pa_x U_0$ is minimal at $x_0$ with 
\be \lab{n:id:cond init}
U_0(x_0)=c, \ \ \pa_xU_0(x_0)<0, \ \ \pa_x^j U_0 (x_0)=0 \ \ \text{for} \ j=2,...,2i, \ \ \text{and} \ \ \pa_x^{2i+1}U_0(x_0)>0
\ee
for some $i\in \mathbb N^*$. Then $u$ blows up at time $T=-1/U_x(x_0)$ at the point $x_{\infty}=x_0+cT$ with:
$$
U(t,x)=\mu^{-1}(T-t)^{\frac{1}{2i}} \Psi_i \left(\mu\frac{x-x_0-ct}{(T-t)^{1+\frac{1}{2i}}} \right)+c+w(t,x)
$$
where $\Psi_i$ is defined by Proposition \ref{prop:smoothselfsim}, where $ \mu=\left(\frac{\pa_x^{2i+1}U(x_0)}{(2i+1)!(-\pa_x U(x_0))^{2i+2}}\right)^{\frac{1}{2i}}$ and where
\be \lab{burgers:cvseflsim}
\frac{w}{(T-t)^{\frac{1}{2i}} \Psi_i \left(\mu\frac{x-x_0-ct}{(T-t)^{1+\frac{1}{2i}}} \right)}\rightarrow 0 \ \ \text{as} \ (x,t)\rightarrow (x_{\infty},T).
\ee

\end{proposition}

\begin{proof}

Without loss of generality, up to the symmetries of the equation we consider the case $x_0=0$, $U(0)=0$, $U_x(0)=-1$ and $\pa_x^{2i+1}U_0(0)=(2i+1)!$, i.e. $T=1=b$, $c=0$. For $0\leq t <1$ and $x\in \mathbb R$ we have the formula using characteristics for $|y|\leq 1$:
\be \lab{burgers:characteristicsh} 
U(x,t)=U_0(\phi_t^{-1}(x)), \ \ \phi_t(y)=y+tU_0(y)=(1-t)y+y^{2i+1}+O(y^{2i+2})+O(|y|^{2i+1}|1-t|),
\ee
$\phi_t$ defining a diffeomorphism on $\mathbb R$ for all $0\leq t<1$. Given $(t,x)$ close to $(1,0)$ we look for an inverse $\phi_t^{-1}(x)$ of the form $-(1-t)^{1/(2i)}\Psi_i(x(1+h)/(1-t)^{1+1/(2i)})$. Since $|\Psi_i(x)|\lesssim |x|^{1/(2i+1)}$ for $x\in \mathbb R$, we compute using \fref{id:implicit}:
\bee
&& \phi_t\left(-(1-t)^{\frac{1}{2i}}\Psi_i\left(\frac{x(1+h)}{(1-t)^{1+1/(2i)}}\right)\right)\\
&=& -(1-t)^{1+\frac{1}{2i}}\left( \Psi_i+\Psi_i^{2i+1}\right)\left(\frac{x(1+h)}{(1-t)^{1+1/(2i)}}\right)+O\left((1-t)^{1+\frac{2}{2i}}\Psi_i^{2i+2}\left(\frac{x(1+h)}{(1-t)^{1+1/(2i)}} \right)\right) \\
&&+O\left((1-t)^{2+\frac{1}{2i}}|\Psi_i^{2i+1}|\left(\frac{x(1+h)}{(1-t)^{1+1/(2i)}} \right)\right)\\
&=& x(1+h)+O\left((1-t)^{1+\frac{2}{2i}}\left(\frac{|x||1+h|}{(1-t)^{1+1/(2i)}} \right)^{1+\frac{1}{2i+1}}\right) +O\left((1-t)^{2+\frac{1}{2i}} \frac{|x||1+h|}{(1-t)^{1+1/(2i)}} \right)\\
&=&x(1+h)+O(|x|^{1+\frac{1}{2i+1}}|1+h|)+O((1-t)|x||1+h|).
\eee
From the intermediate values theorem, there exists $ h=O(|x|^{1/(2i+1)}+(1-t))$ such that there holds the inverse formula
\be \lab{burgers:id:inverseh}
\phi_t\left(-(1-t)^{\frac{1}{2i}}\Psi_i\left(\frac{x(1+h)}{(1-t)^{1+1/(2i)}}\right)\right)=x
\ee
and there holds
\bee
\Psi_i\left(\frac{x(1+h)}{(1-t)^{1+1/(2i)}}\right) & = & \Psi_i\left(\frac{x}{(1-t)^{1+1/(2i)}}\right)+\int_{\mu=1}^{1+h} \frac{d\mu}{\mu} (\tilde X\pa_{\tX} \Psi_i)\left(\frac{x\mu}{(1-t)^{1+1/(2i)}} \right) \\
 & = & \Psi_i\left(\frac{x}{(1-t)^{1+1/(2i)}}\right)+O\left(|h|\left|\Psi_i\left(\frac{x}{(1-t)^{1+1/(2i)}}\right)\right| \right). \\
\eee
Injecting $U_0(y)=-y+O(|y|^{2i+1})$ in \fref{burgers:characteristicsh}, using \fref{burgers:id:inverseh}, the bound on $h$, the above bound and $(1-t)^{1/(2i)}|\Psi_i|(x/(1-t)^{1+1/(2i)})\lesssim |x|^{1/(2i+1)}$, one obtains \fref{burgers:cvseflsim}.

\end{proof}


\section{Proof of the main Theorem \ref{th:main}} \lab{sec:main}

To ease notations we consider the case $i=1$ corresponding to the $\Psi_1$ profile for Burgers, the proof being the same for $i\geq 2$. Recall the notation for the derivatives on the  transverse axis \fref{def:xi} and the corresponding system \fref{eq:NLH} that they solve under the odd in $x$ and even in $y$ symmetry assumption. Solutions to $(NLH)$ in \fref{eq:NLH} might blow up according to a dynamic described in Theorem \ref{th:NLHinstable}. The following proposition then describes how the singularity formation for $\xi$ makes some solutions to the other equation $(LFH)$ in \fref{eq:NLH} blow up in finite time with a precise behaviour. Its proof and that of Theorem \ref{th:NLHinstable} are relegated to Section \ref{sec:NLH} and we prove here Theorem \ref{th:main} admitting them.

\begin{proposition} \lab{pr:LHinstable}

Let $i=1$. For any $k\in \mathbb N$ with $k\geq 2$, $a,b>0$ and $J\in \mathbb N$, there exists $T^*>0$ such that for any $0<T<T^*$, there exists $\xi$ a solution to \fref{eq:NLH} satisfying \fref{id:decomposition NLH} and \fref{bd:remainder NLH}, and $\zeta_0$ such that the corresponding solution to $(LFH)$ blows up at time $T$ with
$$
\zeta =\frac{b}{(T-t+ay^{2k})^4}  +\tilde  \zeta,
$$
where the remainders $\tilde \zeta$ satisfy for $j=0,...,J$ for some constant $C(a,b)>0$:
\be \label{bd:remainder LFH}
| \pa_y^j\tilde \zeta |\leq C \left( (T-t)^{\frac{1}{2k}}+|y|\right)^{\frac 12-(8k+j)}.
\ee

\end{proposition}

\begin{proof}[Proof of Theorem \ref{th:NLHinstable} and of Proposition \ref{pr:LHinstable}]

Section \ref{sec:NLH} is devoted to their proof. Proposition \ref{pr:bootstrap} and the estimates \fref{bd:eLinfty} for $\xi$, and Proposition \ref{pr:bootstrap2} and the estimates \fref{2bd:Linfty} for $\zeta$, indeed imply that Theorem \ref{th:NLHinstable} and Proposition \ref{pr:LHinstable} hold for one particular value of $a>0$ and of $b>0$. One then obtains the general result for any value of $a$ and $b$ by using the symmetries of the equation. Namely, $(NLH)$ and $(LFH)$ are invariant by time translation, \fref{eq:NLH} is invariant by the scaling transformation $\xi \mapsto \lambda^2\xi (\lambda^2t,\lambda y)$ for any $\lambda>0$ and $(LFH)$ is invariant by homothety since it is linear.

\end{proof}

We assume throughout the section that $\xi$ and $\zeta$ satisfy the conclusions of Proposition \ref{pr:LHinstable}.

\subsection{Self-similar variables}
First, the behaviour as $t\rightarrow T$ of $\xi$ and $\zeta$, in Proposition \ref{pr:LHinstable}, suggests that the typical scale along the $y$ variable is $|y|\sim (T-t)^{1/(2k)}$. The typical scale for diffusive effects for a blow-up at the origin at time $T$ is $|y|\sim \sqrt{T-t}$. Formally, since $k\geq 2$ the diffusive effects are negligible. As $|\xi |\sim (T-t)^{-1}$ and $|\zeta|\sim (T-t)^{-4}$, this suggests the scale $|x|\sim (T-t)^{3/2}$. We introduce:
\be \lab{eq:def XYZ}
X:=\sqrt{\frac b6} \frac{x}{(T-t)^{\frac 32}}, \ \ Y:=a^{\frac{1}{2k}}\frac{y}{\sqrt{T-t}}, \ \ s:=-\log (T-t), \ \ Z:=e^{-\frac{k-1}{2k}s}Y=\frac{a^{\frac{1}{2k}}y}{(T-t)^{\frac{1}{2k}}}
\ee
and
$$
u(t,x,y)=\sqrt{\frac 6b} (T-t)^{\frac 12} v\left(s,X,Y \right)
$$
where the renormalisation factors $\sqrt{b/6}$ and $a^{1/2k}$ will simplify notations. To ease the analysis, since the value of $a$ and $b$ will never play a role in this section, we take
\be \lab{eq:def a b}
a=1=b
\ee
without loss of generality for the argument. Then $v$ solves from the choice \fref{eq:def a b}:
\be \lab{main:eqvautosim}
v_s-\frac 12 v +\frac{3}{2}X\pa_X v+\frac 12 Y \pa_Y v +v\pa_X v- \pa_{YY}v=0.
\ee
We define accordingly
\be \lab{main:eq:def fg}
f(s,Y):=-\pa_X v(s,0,Y)=(T-t)\xi (t,y), \ \ g(s,Y):= \pa_X^3 v(s,0,Y)=(T-t)^4\frac{6}{b}\zeta (t,y),
\ee
which from Theorem \ref{th:NLHinstable} satisfy:
\be \lab{eq:def tildeF}
f(s,Y)=F_k(Z)+\tilde f, \ \ F_k(Z):=\frac{1}{1+Z^{2k}}, \ \ |\pa_Z^j \tilde f|\lesssim e^{-\frac{1}{4k}s}(1+|Z|)^{\frac 12 -2k-j}, \ \ j=0,...,J,
\ee
\be \lab{eq:def tildeG}
g(s,Y)=G_k+\tilde g, \ \ G_k:=\frac{6}{(1+Z^{2k})^4}, \ \ |\pa_Z^j \tilde g|\lesssim e^{-\frac{1}{4k}s}(1+|Z|)^{\frac 12 -8k-j}, \ \ j=0,...,J,
\ee
and solve the system from \fref{eq:NLH} and \fref{eq:def a b}:
\be \lab{eq:f}
\left\{ \begin{array}{l l} f_s+f+\frac Y2 \pa_Y f-f^2 -\pa_{YY}f=0, \\
g_s+4g+\frac Y2 \pa_Y g-4fg-\pa_{YY} g=0. \end{array} \right.
\ee
In \fref{eq:def XYZ}, the variable $Y$ is adapted to the viscosity effects whereas the variable that is adapted to the blow-up profile is $Z$. The renormalised function $w(s,X,Z)=v(s,X,Y)$ solves in fact
\be \lab{eq:w}
w_s-\frac 12 w+\frac{3}{2}X\pa_X w+\frac{1}{2k} Z \pa_Z w +w\pa_X w- e^{-\frac{k-1}{k}s} \pa_{ZZ}w=0.
\ee

\subsection{2D Blow-up profile and spectral analysis}

The infinitesimal behaviour near the origin along the  transverse axis being understood by \fref{eq:def tildeF} and \fref{eq:def tildeG}, we need to "extend" it along the $x$ variable. A reasonable guess is that the blow-up of a solution to \fref{eq:burgers} is given by a shock of Burgers equation $\lambda^{1/2}\mu \Psi_1(\lambda^{-3/2}\mu^{-1}x)$ whose two parameters are dictated by \fref{eq:NLH}. Let us first give additional properties of $\Psi_1$ than those contained in Subsection \ref{subsec:Burgers}. From Proposition \ref{prop:smoothselfsim} it solves
\be \lab{main:eq:Psi}
- \frac 12 \Psi_1+\frac 32 X \pa_X \Psi_1+\Psi_1\pa_X \Psi_1=0
\ee
and has the asymptotic behaviour
\be \lab{main:eqPsi1origine}
\Psi_1(X)\underset{X\rightarrow 0}{=}-X+X^3+O(X^5) , \ \ \ \Psi_1(X)\underset{|X|\rightarrow +\infty}{=}-\text{sgn}(X)|X|^{\frac{1}{3}}+O(|X|^{-\frac 13}).
\ee
Since $w$ is a global solution to \fref{eq:w} whose derivatives up to third order on the axis $\{ X=0\}$ converge to some fixed profiles from \fref{eq:def tildeF} and \fref{eq:def tildeG} one can believe that $w$ converges as $s\rightarrow +\infty$ to a profile $w_\infty$ which then has to solve the asymptotic stationary self-similar\footnote{Self-similarity is here with respect to the equation \fref{eq:burgers} without viscosity.} equation
\be \lab{eq:Q}
-\frac 12 w_{\infty}+\frac{3}{2}X\pa_X w_{\infty}+\frac{1}{2k} Z \pa_Z w_{\infty} +w_{\infty}\pa_X w_{\infty}=0.
\ee

\begin{lemma} \lab{lem:Psi1}

For any $a,b>0$, equation \fref{eq:Q} admits the following solution that is odd in $X$ and even in $Z$:
$$
\Theta [a,b] (X,Z):=b^{-1} F_k^{-\frac{1}{2}}(aZ)\Psi_1\left(b F_k^{\frac 32}(aZ)X \right)
$$

\end{lemma}

\begin{proof}

This is a direct computation. The equation is invariant by the scaling $z\mapsto aZ$, $x\mapsto bX$ and $w_{\infty}\mapsto b^{-1}w_{\infty}$, so that we take $a=b=1$ without loss of generality. From \fref{eq:def tildeF} and \fref{main:eq:Psi}:
\bee
&&-\frac 12 \Theta[1,1]+\frac{3}{2}\tX\pa_X \Theta[1,1]+\frac{1}{2k} Z \pa_Z \Theta[1,1] +\Theta[1,1]\pa_X \Theta[1,1] \\
&=& F_k^{-\frac 12}(Z)\left(-\frac 12 \Psi_1+\frac 32 \tX \pa_{\tX} \Psi_1 \right) (F_k^{\frac 32}(Z)X) \\
&&+\frac{1}{2k} Z\pa_Z F_k(Z)F_k^{-\frac 32}(Z)\left(-\frac 12 \Psi_1+\frac 32 \tX \pa_{\tX} \Psi_1\right)(F_k^{\frac 32}(Z)X)+F_k^{\frac 12}(Z)(\Psi_1 \pa_{\tX} \Psi_1)(F_k^{\frac 32}(Z)X)\\
&=& F_k^{-\frac 12} \left(-\frac 12 \Psi_1+\frac 32 \tX \pa_{\tX} \Psi_1 \right) (F_k^{\frac 32}X) \\
&&+(F_k^{\frac 12}-F_k^{-\frac 12})\left(-\frac 12 \Psi_1+\frac 32 \tX \pa_{\tX} \Psi_1\right)(F_k^{\frac 32}X)-F_k^{\frac 12}\left(-\frac 12 \Psi_1+\frac 32\tX \pa_{\tX} \Psi_1\right)(F_k^{\frac 32}X) =0.
\eee

\end{proof}

The choice \fref{eq:def a b} implies that the good candidate for \fref{eq:Q} is 
\be \lab{eq:def Theta}
\Theta(X,Z):= \Theta [1,1](X,Z)=F_k^{-\frac 12}(Z)\Psi_1\left(F_k^{\frac 32}(Z)X \right).
\ee
The linearised operator corresponding to \fref{eq:w} near $\Theta$ neglecting the transversal viscosity is
\bee
\mathcal L_Z &:=&-\frac 12+\frac 32 X\pa_X +\frac{1}{2k} Z \pa_Z+\Theta \pa_X+\pa_X \Theta \\
&=& -\frac 12 +\frac 32 X \pa_X +\frac{1}{2k} Z \pa_Z +F_k^{-\frac 12 }\Psi_1 \left( F_k^{\frac 32}(Z)X\right)\pa_X+F_k(Z) \pa_X \Psi_1 \left( F_k^{\frac 32}(Z)X\right).
\eee
We claim that its spectral structure can be understood trough the spectral analysis of two linearised operators, $H_X$ for Burgers equation studied in Proposition \ref{pr:H} and $H_Z$ for the semi-linear heat equation studied in Proposition \ref{pr:Fk}.

\begin{proposition} \lab{pr:mathcalL}

Let $k\in \mathbb N$, $k\geq 2$. For any $(j,\ell)\in \mathbb N^2$, $(j-3)/2+\ell/(2k)$ is an eigenvalue of the operator $\mathcal L_Z:\mathcal C^1(\mathbb R^2)\rightarrow \mathcal C^0(\mathbb R^2)$ associated to the eigenfunction
\be \lab{id:phij0} 
\varphi_{j,\ell}(X,Z)=\phi_{Z,\ell}(Z)  F_k^{-1-\frac j2 }(Z) \phi_{X,j}\left( F_k^{\frac 32}(Z)X\right)  = Z^\ell F_k^{1-\frac j2}(Z) \times \frac{(-1)^j \Psi_1^j \left( F_k^{\frac 32}(Z)X\right)}{1+3\Psi_1^2\left( F_k^{\frac 32}(Z)X\right)},
\ee
where $\phi_{X,j}$ and $\phi_{Z,\ell}$ are defined by \fref{eq:def phiXj} and \fref{eq:def phiXell}.

\end{proposition}

\begin{proof}

This is a direct computation. From \fref{NLH:eq:Fk}, \fref{eq:def phiXj} and \fref{eq:def phiXell} one has:
\bee
\mathcal L_Z \varphi_{j,k}&=& \phi_{Z,\ell} F_k^{-1-\frac j2 } \left(-\frac 12 \phi_{X,j}+\frac 32 \tX \pa_{\tX} \phi_{X,j}\right)( F_k^{\frac 32}X) +\frac{1}{2k} Z\pa_Z \phi_{Z,\ell} F_k^{-1 -\frac j 2} \phi_{X,j} ( F_k^{\frac 32}X) \\
&&+\frac{1}{2k} Z\pa_Z F_k \phi_{Z,\ell} F_k^{-2 -\frac j2} \left((-1-\frac j2) \phi_{X,j}+\frac 32 \tX \pa_{\tX} \phi_{X,j}\right)( F_k^{\frac 32}X)\\
&&+ \phi_{Z,\ell} F_k^{-\frac j2} (\pa_{\tX} \Psi_1 \phi_{X,j}+\Psi_1\pa_X \phi_{X,j}) ( F_k^{\frac 32}X )\\
&=&\phi_{Z,\ell} F_k^{-1-\frac j2 } \left(-\frac 12 \phi_{X,j}+\frac 32 \tX \pa_{\tX} \phi_{X,j}\right)( F_k^{\frac 32}X)\\
&& +\left((\frac{\ell-2k}{2k}-1)\phi_{Z,\ell}+2F_k \phi_{Z,\ell}\right)F_k^{-1 -\frac j 2} \phi_{X,j} ( F_k^{\frac 32}X)  \\
&&+ \phi_{Z,\ell} (F_k^{-\frac j2}-F_k^{-1-\frac j2})\left((-1-\frac j2) \phi_{X,j}+\frac 32 \tX \pa_{\tX} \phi_{X,j}\right)( F_k^{\frac 32}X)\\
&&+  \phi_{Z,k} F_k^{-\frac j2} \left((\frac{j-3}{2}+\frac 12)\phi_{X,j}-\frac 32 \tX \pa_{\tX} \phi_{X,j} \right) ( F_k^{\frac 32}X)\\
&=& \left(\frac{j-3}{2}+\frac{\ell}{2k} \right) \phi_{Z,\ell} F_k^{-1-\frac j2 } \phi_{X,j}( F_k^{\frac 32}X) .
\eee

\end{proof}

\subsection{Linear estimates}

A maximum principle holds for the linear transport operator $\mathcal L_Z$. Also, we will not use this estimate as is, we believe it is of interest.

\begin{lemma}
Assume $\e_0\in L^{\infty}_{\text{loc}}(\mathbb R^2)$ is such that $| \e_0|\leq C |\varphi_{j,\ell}|$ on $\mathbb R^2\backslash \{X=0 \}$, for some $C>0$ and $(j,\ell)\in \mathbb N^2$. Then the solution to $\pa_s \e+\mathcal L_Z\e=0$ with initial datum $\e_0$ satisfies:
\be \lab{main:bd:estimationtransport}
\left\| \frac{\e }{\varphi_{j,\ell}} \right\|_{L^{\infty}(\mathbb R^2\backslash \{X=0 \})} \leq e^{-\left(\frac{j-3}{2}+\frac{\ell}{2k}\right)s}\left\| \frac{\e_0}{\varphi_{j,\ell}} \right\|_{L^{\infty}(\mathbb R^2\backslash \{X=0 \})}.
\ee

\end{lemma}

\begin{proof}

This is a straightforward computation along the characteristics, that we omit.

\end{proof}

We now investigate the linear dynamics in the presence of dissipation, for the full operator:
$$
\mathcal L:= -\frac 12+\pa_X \Theta +\left(\frac 32 X +\Theta \right)\pa_X +\frac 12 Y\pa_Y-\pa_{YY}=\mathcal L_Z+\frac{k-1}{2k}Z\pa_Z-\pa_{YY},
$$
and find an energy estimate that mimics \fref{main:bd:estimationtransport}. We will use eigenfunctions of the form $\phi_{j,\ell}$ with $j> 3$ and $\ell =0$ to ensure time decay, and because the choice $\ell>0$ would produce a singularity near $Z=0$ that is incompatible with the viscosity. We replace the $L^\infty$ norm by a weighted $L^q$ norm with $q$ large enough, also in order to be compatible with the viscous term.

\begin{lemma} \lab{main:lem:lineaire}
Let $0\leq j<i_0$. For any $\kappa>0$, there exists $q^*\in \mathbb N$ such that for all $q\in \mathbb N$ with $q\geq q^*$ there exists $s^*\geq 0$ such that the following holds. For any $s^*\leq s_0< s_1$, if $\e$ and $\Xi$ are in the Schwartz class and satisfy on $[s_0,s_1]$ that:
\be \lab{eq:evolinee}
\e_s +\mathcal L \e =\Xi,
\ee
that for $i=0,...,i_0$ one has the cancellation on the axis $\{ X=0\}$
\be \lab{eq:condevolinee}
\pa_X^i \e (s,0,Y)=0 \qquad \mbox{and} \qquad \pa_X^i \Xi (s,0,Y)=0,
\ee
then the following energy identity holds:
\begin{align}\lab{id:estimationlineaire} 
&\frac{d}{ds} \left( \frac{1}{2q} \int_{\mathbb R^2} \frac{\e^{2q}}{\varphi_{j,0}^{2q}(X,Z)}\frac{dXdY}{|X| \lY} \right)\\
\non & \leq  -\left(\frac{j-3}{2}-\frac{\kappa}{2}\right) \int_{\mathbb R^2} \frac{\e^{2q}}{\varphi_{j,0}^{2q}(X,Z)}\frac{dXdY}{|X| \lY}-\frac{2q-1}{q^2}\int \frac{|\pa_Y (\e^q)|^2}{\varphi_{j,0}^{2q}(X,Z)}\frac{dXdY}{|X|\lY}+\int \frac{\e^{2q-1}\Xi}{\varphi_{j,0}^{2q}(X,Z)}\frac{dXdY}{|X|\lY}.
\end{align}

\end{lemma}

\begin{proof}

This is a direct computation. One computes from the evolution equation \fref{eq:evolinee}, performing integration by parts,
\bee
&&\frac{d}{ds} \left( \frac{1}{2q} \int_{\mathbb R^2} \frac{\e^{2q}}{\varphi_{j,0}^{2q}(X,Z)}\frac{dXdY}{|X| \lY} \right)  = \int_{\mathbb R^2} \frac{\e^{2q-1}\pa_s \e}{\varphi_{j,0}^{2q}(X,Z)}\frac{dXdY}{|X| \lY} - \int_{\mathbb R^2} \frac{\e^{2q}\pa_s\varphi_{j,0}(X,Z) }{\varphi_{j,0}^{2q+1}(X,Z)}\frac{dXdY}{|X| \lY} \\
&=& \int \frac{\e^{2q-1}}{\varphi_{j,0}^{2q}(X,Z)}\Big[\frac 12 \e -\frac 32 X \pa_X \e -\frac 12 Y\pa_Y \e -F_k^{-\frac 12}\Psi_1\left(F_k^{\frac 32}(Z)X\right)\pa_X \e +\pa_{YY}\e \\
&&-F_k(Z) \pa_X\Psi_1 \left(F_k^{\frac 32}(Z)X \right)\e+\Xi\Big]\frac{dXdY}{|X|\lY }+\frac{k-1}{2k}  \int \frac{\e^{2q}}{\varphi_{j,0}^{2q+1}(X,Z)}Z\pa_Z \varphi_{j,0}(X,Z)\frac{dXdY}{|X| \lY} \\
&=& \int \frac{\e^{2q}}{\varphi_{j,0}^{2q+1}(X,Z)}\Big[\frac 12 \varphi_{j,0} -\frac 32 X \pa_X \varphi_{j,0}-\frac{1}{2k} Z\pa_Z \varphi_{j,0}-F_k^{-\frac 12}(Z)\Psi_1(F_k^{\frac 32}(Z)X)\pa_X \varphi_{j,0}\\
&&-F_k (Z)\pa_X\Psi_1(F_k^{\frac 32}(Z)X) \varphi_{j,0} \Big] \frac{dXdY}{|X|\lY }\\
&& +\frac{1}{2q} \int \frac{\e^{2q}}{\varphi_{j,0}^{2q}}\left( \frac 12 \pa_Y \left(\frac{Y}{\lY }\right)\frac{1}{|X|}+\pa_X \left(\frac{F_k^{-\frac 12}(Z)\Psi_1(F_k^{\frac 32}(Z)X)}{X}\right)\frac{1}{\lY}\right)dXdY \\
&&-\frac{2q-1}{q^2}\int \frac{|\pa_Y (\e^q)|^2}{\varphi_{j,0}^{2q}}\frac{dXdY}{|X|\lY}+\frac{1}{2q} \int \e^{2q}\pa_{YY} \left(\frac{1}{\varphi_{j,0}^{2q}}\frac{1}{\lY}\right)\frac{dXdY}{|X|}+\int \frac{\e^{2q-1}\Xi}{\varphi_{j,0}^{2q}}\frac{dXdY}{|X|\lY}\\
\non &=& -\frac{j-3}{2} \int_{\mathbb R^2} \frac{\e^{2q}}{\varphi_{j,0}^{2q}(X,Z)}\frac{dXdY}{|X| \lY}-\frac{2q-1}{q^2}\int \frac{|\pa_Y (\e^q)|^2}{\varphi_{j,0}^{2q}}\frac{dXdY}{|X|\lY}+\int \frac{\e^{2q-1}\Xi}{\varphi_{j,0}^{2q}}\frac{dXdY}{|X|\lY}  \\
\non && +\frac{1}{2q} \int \frac{\e^{2q}}{\varphi_{j,0}^{2q}}\left( \frac 12 \pa_Y \left(\frac{Y}{\lY }\right)\frac{1}{|X|}+\pa_X \left(\frac{F_k^{-\frac 12}(Z)\Psi_1(F_k^{\frac 32}(Z)X)}{|X|}\right)\frac{1}{\lY}\right)dXdY\\
\non &&+\frac{1}{2q} \int \e^{2q}\pa_{YY} \left(\frac{1}{\varphi_{j,0}^{2q}}\frac{1}{\lY}\right)\frac{dXdY}{|X|}
\eee
where we used Proposition \fref{pr:mathcalL}. The integrations by parts are legitimate near the axis $\{X=0 \}$ because of the cancellation \fref{eq:condevolinee} and since $\Psi_1(X)\sim -X$ and $\varphi_{j,0}\sim X^j$ as $X\rightarrow 0$ from \fref{main:sizephi}. The last terms are lower order ones. Indeed, one has:
$$
\left| \pa_Y \left(\frac{Y}{\lY}\right) \right| = \frac{1}{\lY ^3} 
$$
and
$$
\pa_X \left(\frac{F_k^{-\frac 12}(Z)\Psi_1(F_k^{\frac 32}(Z)X)}{|X|}\right)=\frac{F_k (Z)\pa_X \Psi_1(F_k^{\frac 32}(Z)X)}{|X|}+\frac{F_k^{-\frac 12} (Z) \Psi_1(F_k^{\frac 32}(Z)X)}{X|X|}.
$$
For the first term in the above identity, one has that $|F_k(Z)|=(1+Z^{2k})^{-1}\leq 1$ and that $|\pa_X \Psi_1|=|1/(1+3\Psi_1^2)|\leq 1$. For the second, one has that $|\Psi_1 (X)|\leq |X|$. Therefore,
$$
\left| \pa_X \left(\frac{F_k^{-\frac 12}(Z)\Psi_1(F_k^{\frac 32}(Z)X)}{|X|}\right) \right|\leq \frac{2}{|X|}.
$$
Next, since $|\pa_Z^j \varphi_{4,0}(X,Z)|\lesssim (1+|Z|)^{-j}|\varphi_{4,0}(X,Z)|$ from \fref{id:phij0} and $\pa_Y=e^{-(k-1)s/(2k)}\pa_Z$:
$$
\left| \pa_{YY} \left(\frac{1}{\varphi_{4,0}^{2q}}\frac{1}{\lY}\right)\right| \lesssim \frac{(1+q^2e^{-\frac{k-1}{k}s})}{\varphi_{4,0}^{2q}}\frac{1}{\lY}.
$$
Therefore,
\bee
\non && \Big| \frac{1}{2q} \int \frac{\e^{2q}}{\varphi_{j,0}^{2q}}\left( \frac 12 \pa_Y \left(\frac{Y}{\lY }\right)\frac{1}{|X|}+\pa_X \left(\frac{F_k^{-\frac 12}(Z)\Psi_1(F_k^{\frac 32}(Z)X)}{|X|}\right)\frac{1}{\lY}\right)dXdY\\
\non &&+\frac{1}{2q} \int \e^{2q}\pa_{YY} \left(\frac{1}{\varphi_{j,0}^{2q}}\frac{1}{\lY}\right)\frac{dXdY}{|X|} \Big| \leq \frac{C}{q}(1+q^2e^{-\frac{k-1}{k}s}) \int \frac{\e^{2q}}{\varphi_{j,0}^{2q}}\frac{dXdY}{|X|\lY}.
\eee
which, injected in the previous energy identity yields the desired result upon choosing $q$ large enough and then $s^*$ large enough.

\end{proof}

\subsection{Bootstrap analysis}

We are now ready to prove Theorems \ref{th:mainstable} and \ref{th:main}. Throughout the analysis, the functions $F_k$, $\Psi_1$, $\Theta$ and $\varphi_{j,0}$ will be extensively used. In particular, from \fref{eq:def Theta}, a relevant variable for the stream direction is $\tilde X$ defined by \fref{eq:deftildeX} with:
$$
 |\tilde X|\approx |X|(1+|Z|)^{-3k},
$$
and from \fref{eq:def tildeF}, \fref{eq:def Theta} and \fref{id:phij0}  their size is encoded by the following estimates (which adapt to derivatives)
\be \lab{main:sizeFkPsi1}
F_k (Z)\approx (1+|Z|)^{-2k}, \ \ |\Psi_1 (X)|\approx |X|(1+|X|)^{\frac 13 -1},
\ee
\be \lab{main:sizeTheta}
|\Theta (X,Z)|\approx |X| \left((1+|Z|)^{3k}+|X| \right)^{\frac 13 -1} \approx (1+|Z|)^k |\tilde X|(1+|\tilde X|)^{\frac 13 -1}
\ee
\be \lab{main:sizephi}
|\varphi_{j,0}(X,Z)|\approx |X|^j \left((1+|Z|)^{3k}+|X|\right)^{\frac{j-2}{3}-j} \approx (1+|Z|)^{k(j-2)} |\tilde X|^j(1+|\tilde X|)^{\frac{j-2}{3} -j},  \\
\ee
The strategy is to show that there exists global solutions to \fref{main:eqvautosim} converging to $\Theta$ defined by \fref{eq:def Theta} as $s\rightarrow +\infty$. We will use an approximate blow-up profile, i.e. refine $\Theta$ to show this. To obtain decay in the linear estimate \fref{id:estimationlineaire}, one needs $j>3$ which from \fref{main:sizephi} in turn requires that $\e=O(|X|^j)$ as $X\rightarrow 0$. The linearised operator $\mathcal L$ has then a damping effect on functions vanishing up to order $3$ on the vertical axis. Consequently, we use the profile $\mu^{-1} \lambda^{-1/2}\Psi_1(\mu\lb^{3/2} X)$ at each line $\{ Y=Cte\}$, to match the solution at order $1$ and $3$ near the vertical axis $\{X=0 \}$. Far away, such a decomposition ceases to make sense since we are no more in the blow-up zone, and the appropriate profile is $0$ rather than $\Psi_1$. We set for $d>0$ a cut-off function (note that $|Y|\leq de^{s/2}$ is equivalent to $|y|\leq d$ and $|Z|\leq e^{s/2k}$),
$$
\chi_d(s,Y):=\chi \left(\frac{Y}{de^{\frac s2}} \right)
$$
and then decompose our solution to \fref{main:eqvautosim} according to:
\be \lab{def:e}
v(s,X,Y)= Q+\e, \qquad \qquad Q=\chi_d(s,Y) \tilde \Theta +(1-\chi_d(s,Y))\Theta_{e} \\
\ee
where $\tilde \Theta$ is the approximate blow-up profile in the interior zone
\be \lab{def:tildeTheta}
\tilde \Theta(s,X,Y):= \mu ^{-1}(s,Y)f^{-\frac 12}(s,Y) \Psi_1\left(f^{\frac 32}(s,Y) \mu(s,Y) X \right) = \sqrt 6 g^{-\frac 12}f^{\frac 32} \Psi_1 \left( \frac{g^{\frac 12}f^{-\frac 12}}{\sqrt 6} X \right)
\ee
where $f$ and $g$ are defined in \fref{main:eq:def fg} and
\be \lab{eq:def b}
\mu(s,Y):=\left(\frac{g(s,Y)}{6f^4(s,Y)} \right)^{\frac 12},
\ee
(notice that for $d$ small enough and for $Y$ in the support of $\chi_d(s,\cdot)$, the functions $f$ and $g$ do not vanish from \fref{eq:def tildeF} and \fref{eq:def tildeG}, and hence $\mu$ and $\mu^{-1}$ are well-defined), and where $\Theta_e$ is the profile for the external zone
\be \lab{def:Thetae}
\Theta_e(s,X,Y):=  \left(-X f(s,Y)+X^3\frac{g(s,Y)}{6} \right)e^{-\tilde X^4}.
\ee
The profile $\Theta_e$ matches with $v$ up to third order near the vertical axis, what allows for a unified control of the remainder $\varepsilon$, inside and outside the blow-up zone simultaneously \fref{main:weighteddecaye} and \fref{main:weighteddecaye2}. It is not a very precise approximation, what does not matter since the weights we use penalise the exterior zone. Other choices for $\Theta_e$ are thus possible.

To estimate the remainder $\e$, we will use weighted Sobolev norms, and to control its derivatives we will use vector fields that commute well with $\pa_s+\mathcal L_Z$:
\be \lab{main:def:A}
A:=\left(\frac 32 X +F_k^{-\frac 32}(Z)\Psi_1(F_k^{\frac 32}(Z)X) \right)\pa_X, \ \ \pa_Z \ \ \text{and} \ \ Z\pa_Z
\ee
and that are equivalent to usual vector fields, see Lemma \ref{an:lem:equiv}.

\begin{definition}[Trapped solutions]

Let constants $\kappa,d>0 $, $ q\in\mathbb{N}$, $K_{j_1,j_2}>0$ for nonnegative integers $j_1,j_2$ with $0\leq j_1+j_2\leq 2$, and $\tilde K_{j_1,j_2}>0$ for nonnegative integers $j_1,j_2$ with $0\leq j_1+j_2\leq 2$ and $j_2\geq 1$, and $s_0<s_1$. We say a solution to \fref{main:eqvautosim} is trapped on $[s_0,s_1]$, if, decomposed according to \fref{def:e}, it satisfies on $[s_0,s_1]$:
\be \lab{main:weighteddecaye}
\left( \int_{\mathbb R^2} \frac{((\pa_Z^{j_1}A^{j_2}\e)^{2q}}{\varphi_{4,0}^{2q}(X,Z)}\frac{dXdY}{|X|\lY}\right)^{\frac{1}{2q}}\leq K_{j_1,j_2} e^{-\left(\frac 12 -\kappa \right)s},
\ee
and for $0\leq j_1+j_2\leq 2$ and $j_2\geq 1$:
\be \lab{main:weighteddecaye2}
\left( \int_{\mathbb R^2} \frac{(((Y\pa_Y)^{j_1}A^{j_2}\e)^{2q}}{\varphi_{4,0}^{2q}(X,Z)}\frac{dXdY}{|X|\lY}\right)^{\frac{1}{2q}}\leq \tilde K_{j_1,j_2} e^{-\left(\frac 12 -\kappa \right)s},
\ee
and write $u \in \mathcal T(\kappa,q,s_0,s_1,(K_{j_1,j_2})_{0\leq j_1+j_2\leq 2},(\tilde K_{j_1,j_2})_{0\leq j_1+j_2\leq 2, 1\leq j_1})$ for the set of such solutions.
\end{definition}

We claim that $\e$ decays thanks to the following bootstrap argument, which is the heart of our proof of Theorem \ref{th:main}.

\begin{proposition} \lab{main:pr:bootstrap}
Let $\xi$ and $\zeta$ be given by Proposition \ref{pr:LHinstable}. Let $0<\kappa<1/2$ and $q\in \mathbb N^*$ such that Lemma \ref{main:lem:lineaire} holds true. Then there exist positive constants $d$, $(K_{j_1,j_2})_{0\leq j_1+j_2\leq 2}$ and $(\tilde K_{j_1,j_2})_{0\leq j_1+j_2\leq 2, 1\leq j_1}$, and $s^*\geq 0$, such that if at anytime $s_0\geq s^*$ the solution is given by \fref{def:e} with $\e(s_0)=\e_0$ satisfies
\be \lab{main:weighteddecayeinit}
\sum_{0\leq j_1+j_2\leq 1} \int_{\mathbb R^2} \frac{ (\pa_Z^{j_1}A^{j_2}\e_0)^{2q}+ ((Y\pa_Y)^{j_1}A^{j_2}\e_0)^{2q}}{\varphi_{4,0}^{2q}(X,Z)}\frac{dXdY}{|X|\lY} \leq e^{-2q\left(\frac 12 -\kappa \right)s_0}
\ee
then the solution to \fref{main:eqvautosim} is global and is trapped on $[s_0,+\infty)$:
$$
u \in \mathcal T(\kappa,q,s_0,+\infty,(K_{j_1,j_2})_{0\leq j_1+j_2\leq 2},(\tilde K_{j_1,j_2})_{0\leq j_1+j_2\leq 2, 1\leq j_1}).
$$
\end{proposition}

\begin{remark}[On the constants in the bootstrap analysis] \label{rk:constants}
To track dependencies of the constants in the proof we use the following:
\begin{itemize}
\item The functions $\xi$ and $\zeta$ are chosen in advance as solutions to \fref{eq:NLH} satisfying \fref{bd:remainder NLH} and \fref{bd:remainder LFH}. The constants in these estimates are considered as universal and are independent of the bootstrap constants.
\item The parameter $d>0$ is fixed so that one has for $s$ large enough for $|Z|\leq 2e^{s/2k}$ :
\be \label{id:choiced}
\left| \frac{f(s,Z)}{F_k(Z)}-1 \right|\leq \frac{1}{2} \qquad \mbox{and} \qquad \left| \frac{g(s,Z)}{G_k(Z)}-1 \right|\leq \frac{1}{2} 
\ee
which is always possible from \fref{eq:def tildeF} and \fref{eq:def tildeG}.
\item $\kappa$ and $q$ are fixed, and $s^*$ is large enough, so that Lemma \ref{main:lem:lineaire} holds.
\item Thus, the constants which are not fixed at this stage are $K_{j_1,j_2}$, $\tilde{K}_{j_1,j_2}$ and $s^*$. Their choice is made in the following order. $K_{0,0}$ is chosen bigger than a universal constant see Lemma \ref{main:lem:energy0}, then $K_{0,1}$ is chosen bigger than a constant depending only on $K_{0,0}$ see \eqref{ChoiceK01}, then $K_{1,0}$ is chosen bigger than a constant depending only on $K_{0,0}$ and $K_{0,1}$, then $\tilde K_{1,0}$ is chosen depending on $(K_{0,0},K_{0,1})$, and so on. This is first because a given derivative of $\e$ sees lower order derivatives as forcing terms from Leibniz formula. Second, $A$ commutes with the full transport operator in $\pa_s+\mathcal L_Z$ \fref{id:commutationA}, while $\partial_Z$ and $Y\partial_Y$ commute with the $\partial_Z$ of this transport operator but not with the $\pa_X$ part. Hence, we first control $\e$, then to take derivatives we first control $A\e$, then $\pa_Z \e$ and $Y \partial_Y\e$, and then we move to higher order derivatives and so on.
\item $s^*$ is chosen last.
\item Constants $C$ in forthcoming estimates stand for constants that are independent of the $K$'s and $\tilde K$'s constants, unless explicitly mentioned. We shall write $A \lesssim B$ if $A\leq CB$ for such a constant $C$.
\end{itemize}
\end{remark}

The proof of the above Proposition \ref{main:pr:bootstrap} follows a classical bootstrap reasoning. Namely, throughout the remaining part of this section we assume that $v$ is a solution to \fref{main:eqvautosim} defined on $[s_0,s_1]$ and such that the decomposition \fref{def:e} satisfies \fref{main:weighteddecayeinit}, \fref{main:weighteddecaye} and \fref{main:weighteddecaye2}. All the results below will show that \fref{main:weighteddecaye} and \fref{main:weighteddecaye2} are in fact strict at time $s_1$, what will allow us to conclude the proof of Proposition \ref{main:pr:bootstrap} by a continuity argument at the end of this section.\\

\noindent First, notice that the bounds of Proposition \ref{main:pr:bootstrap} imply pointwise control by weighted Sobolev embedding.

\begin{lemma} \lab{main:lem:pointwise e}

There holds on $[s_0,s_1]$ with constants in the inequalities depending on the bootstrap constants $K_{j_1,j_2}$ and $\tilde K_{j_1,j_2}$:
\be \lab{main:bde}
| \e| \lesssim e^{-\left(\frac 12 -\kappa \right)s} (1+|Z|)^{2k}|\tilde X|^4 (1+|\tilde X|)^{\frac{2}{3}-4}\lesssim e^{-\left(\frac 12 -\kappa \right)s} |X|^4((1+|Z|)^{3k}+|X|)^{\frac 23 -4}\lesssim e^{-\left(\frac 12 -\kappa \right)s}|X|,
\ee
\be \lab{main:bde2}
|\pa_X \e| \lesssim e^{-\left(\frac 12 -\kappa \right)s} (1+|Z|)^{-k}|\tilde X|^3 (1+|\tilde X|)^{\frac{2}{3}-4}\lesssim e^{-\left(\frac 12 -\kappa \right)s} |X|^3((1+|Z|)^{3k}+|X|)^{\frac 23 -4}\lesssim e^{-\left(\frac 12 -\kappa \right)s}
\ee
\bea \lab{main:bde3}
|\pa_Z \e| &\lesssim & e^{-\left(\frac 12 -\kappa \right)s} (1+|Z|)^{2k-1}|\tilde X|^4 (1+|\tilde X|)^{\frac{2}{3}-4}\\
\non &\lesssim &e^{-\left(\frac 12 -\kappa \right)s} (1+|Z|)^{-1}|X|^3((1+|Z|)^{3k}+|X|)^{\frac 23 -4}\\
\non &\lesssim & e^{-\left(\frac 12 -\kappa \right)s} |X|(1+|Z|)^{-1}
\eea

\end{lemma}

\begin{proof}

\noindent Recall that $Y\pa_Y=Z\pa_Z$. \textbf{Step 1} \emph{Proof of \fref{main:bde}}. From the identity
\be \lab{main:pointwise intermediaire}
|\lY \pa_Y \e |\lesssim |\pa_Y \e|+|Y\pa_Y \e|=e^{-\frac{k-1}{2k}s}|\pa_Z \e|+|Z\pa_Z \e|
\ee
and the equivalence between vector fields \fref{an:equivalence3} we infer that
$$
|\e|+|X\pa_X\e|+|\la Y\ra \pa_Y \e|\lesssim |\e|+|A\e|+|\pa_Z \e|+|Z\pa_Z\e|
$$
and therefore the bootstrap bounds \fref{main:weighteddecaye} and \fref{main:weighteddecaye2} imply in particular that:
$$
\int_{\mathbb R^2} \frac{\e^{2q}}{\varphi_{4,0}^{2q}(X,Z)}\frac{dXdY}{|X| \lY}+\int_{\mathbb R^2} \frac{(X \pa_X \e)^{2q}}{\varphi_{4,0}^{2q}(X,Z)}\frac{dXdY}{|X| \lY}+\int_{\mathbb R^2} \frac{(\lY \pa_Y\e)^{2q}}{\varphi_{4,0}^{2q}(X,Z)}\frac{dXdY}{|X| \lY} \lesssim e^{-2q\left(\frac 12 -\kappa \right)s}
$$
and the weighted Sobolev embedding \fref{eq:weightedSobo} implies that $|\e|\lesssim e^{-(1/2-\kappa)s}|\varphi_{4,0}|$ which gives \fref{main:bde} using  \fref{main:sizephi}.\\

\noindent \textbf{Step 2} \emph{Proof of \fref{main:bde2}}. The very same reasoning as in Step 1 show that the bootstrap bounds \fref{main:weighteddecaye} and \fref{main:weighteddecaye2} imply $|A\e|\lesssim e^{-(1/2-\kappa)s}\varphi_{4,0}$. From \fref{an:equivalence1} and \fref{an:equivalence3} we infer that $|A\e|\approx |X\pa_X \e|$, implying then that $|\pa_X\e|\lesssim e^{-(1/2-\kappa)s} |X|^{-1}|\varphi_{4,0}|$, which yields \fref{main:bde2} using \fref{main:sizephi}.\\

\noindent \textbf{Step 3} \emph{Proof of \fref{main:bde3}}. Using \fref{main:pointwise intermediaire} and \fref{an:equivalence5} we obtain that
$$
|\lY \pa_Y\pa_Z \e|\lesssim |\pa_Z^2 \e|+|Z\pa_Z^2 \e|\lesssim |\pa_Z^2 \e|+|Z^2\pa_Z^2 \e| \lesssim |\pa_Z^2 \e|+|Z\pa_Z \e|+|(Z\pa_Z)^2\e|
$$
and from \fref{an:equivalence3} that $|X\pa_X\pa_Z \e|\lesssim |A\e|+|\pa_ZA\e|$. Therefore, we infer from \fref{main:weighteddecaye} and \fref{main:weighteddecaye2} that
$$
\int_{\mathbb R^2} \frac{(\pa_Z\e)^{2q}}{\varphi_{4,0}^{2q}(X,Z)}\frac{dXdY}{|X| \lY}+\int_{\mathbb R^2} \frac{(X \pa_X \pa_Z\e)^{2q}}{\varphi_{4,0}^{2q}(X,Z)}\frac{dXdY}{|X| \lY}+\int_{\mathbb R^2} \frac{(\lY \pa_Y\pa_Z\e)^{2q}}{\varphi_{4,0}^{2q}(X,Z)}\frac{dXdY}{|X| \lY} \lesssim e^{-2q\left(\frac 12 -\kappa \right)s}
$$
implying using \fref{eq:weightedSobo} that $|\pa_Z \e|\lesssim e^{-(1/2-\kappa)s}|\varphi_{4,0}|$. The very same reasoning applies for the function $Z\pa_Z \e$, yielding $|Z\pa_Z \e|\lesssim e^{-(1/2-\kappa)s}|\varphi_{4,0}|$. Therefore, $|\pa_Z \e|\lesssim e^{-(1/2-\kappa)s} (1+|Z|)^{-1}|\varphi_{4,0}|$ which yields \fref{main:bde3} using \fref{main:sizephi}.

\end{proof}

We start by investigating the infinitesimal behavior of $\e$ near the line $\{ X=0 \}$. This corresponds to establishing the so-called modulation equation for the parameters describing the blow-up profile $\Psi_1$ and the external profile $\Theta_e$ on each fixed line $\{ Y=Cte\}$. We claim that $\e$ vanishes on the axis up to the third order.

\begin{lemma}

For all $s\geq 0$ and $Y\in \mathbb R$ one has that
\be \lab{id:modulation}
\pa_X^j \e (s,0,Y)=0, \ \ j=0,1,2,3,4.
\ee

\end{lemma}

\begin{proof}

This is a direct computation. First since the profile is odd in $X$ and even in $Y$ one has that $\e$, $\pa_X^2\e$ and $\pa_X^4\e$ vanish on the vertical axis $\{X=0 \}$. Then, one has by definition \fref{main:eq:def fg} of $f$ that $\pa_X v(s,0,Y)=-f(s,Y)$ and from \fref{def:e}, \fref{def:tildeTheta} and \fref{def:Thetae} that
$$
\pa_Xv(s,0,Y)=-\chi_d f-(1-\chi_d)f+\pa_X \e(s,0,Y)=-f+\pa_X \e(s,0,Y).
$$
Therefore $\pa_X \e(s,0,Y)=0$ for all $s\geq s_0$ and $Y\in \mathbb R$. Similarly, by \fref{main:eq:def fg}, $\pa_X^3 v(s,0,Y)=g(s,Y)$ and from \fref{def:e}, \fref{def:tildeTheta} and \fref{def:Thetae} one has
$$
\pa_X^3 v(s,0,Z)=\chi_d 6b^2f^4+(1-\chi_d) g+\pa_X^3 \e(s,0,Z)=g+\pa_X^3\e(s,0,Z).
$$
Therefore $\pa_X^3 \e(s,0,Z)=0$ for all $s\geq s_0$ and $Y\in \mathbb R$ which ends the proof of the lemma.

\end{proof}

The time evolution of $\e$ is given by:
\be \lab{main:evolutione}
\e_s+\mathcal L \e+\tilde{\mathcal L}\e+R+\e \pa_X \e=0
\ee
where
$$
\mathcal L:= -\frac 12+\pa_X \Theta +\left(\frac 32 X +\Theta \right)\pa_X +\frac 12 Y\pa_Y-\pa_{YY}=\mathcal L_Z+\frac{k-1}{2k}Z\pa_Z-\pa_{YY}, \ \ 
$$
\be \lab{main:def:tL}
\tL \e=(Q-\Theta)\pa_X \e +(\pa_X Q-\pa_X \Theta)\e,
\ee
and
\be \lab{eq:def R}
R=Q_s-\frac 12 Q +\frac 32 X \pa_X Q+\frac 12 Y \pa_Y Q +Q\pa_X Q-  \pa_{YY}Q.
\ee

\begin{lemma} \lab{main:lem:R}

Let $q\in \mathbb N$, $q\geq 1$, and $3+1/k\leq j \leq 5-2/k$. Then one has the estimate
\be \lab{main:estimationR} 
\sum_{0\leq j_1+j_2\leq 1} \int_{\mathbb R^2} \frac{ (\pa_Z^{j_1}A^{j_2}R)^{2q}+ ((Y\pa_Y)^{j_1}A^{j_2}R)^{2q}}{\varphi_{4,0}^{2q}(X,Z)}\frac{dXdY}{|X|\lY}  \lesssim se^{-q s}.
\ee

\end{lemma}

\begin{proof}

Recall that $Y\pa_Y=Z\pa_Z$. From \fref{eq:def R}, \fref{def:tildeTheta} and \fref{def:Thetae} we compute:
$$
R=R_1+R_2+R_3+R_4+R_5
$$
where
$$
R_1:=\chi_d \left(\tilde \Theta_s-\frac 12 \tilde \Theta+\frac 32 X \pa_X \tilde \Theta+\frac 12 Y \pa_Y \tilde \Theta+\tilde \Theta \pa_X \tilde \Theta-  \pa_{YY}\tilde \Theta \right),
$$
$$
R_2:=(1-\chi_d) \left(\pa_s\Theta_e-\frac 12 \Theta_e+\frac 32 X \pa_X \Theta_e+\frac 12 Y \pa_Y \Theta_e+ \Theta_e \pa_X \Theta_e-  \pa_{YY} \Theta_e \right),
$$
$$
R_3:=(\tilde \Theta-\Theta_e)(\pa_s \chi_d +\frac 12 Y \pa_Y \chi_d-  \pa_{YY}\chi_d), \ \ R_4:=-2  \pa_Y \chi_d \pa_Y (\tilde \Theta-\Theta_e),
$$
$$
R_5:=\chi_d(1-\chi_d)(\tilde \Theta-\Theta_e)\pa_X (\tilde \Theta-\Theta_e).
$$
We now prove the corresponding bounds for all terms $R_i$.\\

\noindent \textbf{Step 1} \emph{Estimate for $R_1$}. All the computations are performed in the domain of $\chi_d$, $|Y|\lesssim 2de^{s/2}$, where $f,g>0$. We compute from \fref{def:tildeTheta}, \fref{eq:def b}, \fref{eq:f} and Lemma \ref{lem:Psi1} that only some viscosity terms remain in $R_1$:
\bee
&&\tilde \Theta_s-\frac 12 \tilde \Theta+\frac 32 X \pa_X \tilde \Theta+\frac 12 Y \pa_Y \tilde \Theta+\tilde \Theta \pa_X \tilde \Theta-  \pa_{YY}\tilde \Theta \\
&=& -  \sqrt 6 (\pa_{Y}g)^2 g^{-\frac 52}f^{\frac 32} \left[(-\frac 32 +\frac 12 \tX \pa_{\tX})(-\frac 12 \Psi_1+\frac 12 \tX\pa_{\tX} \Psi_1)\right]\left( \frac{g^{\frac 12}f^{-\frac 12}}{\sqrt 6} X \right)\\
&&- \sqrt{6} \pa_Y g \pa_Y f g^{-\frac 32}f^{\frac 12} \left[-\frac 32\Psi_1+\frac 32 \tX \pa_{\tX} \Psi_1 -\frac 12 \tX^2 \pa_{\tX}^2 \Psi_1\right]\left( \frac{g^{\frac 12}f^{-\frac 12}}{\sqrt 6} X \right)\\
&&- \sqrt 6 g^{-\frac 12} (\pa_Y f)^2 f^{-\frac 12}\left[(\frac 12 -\frac 12 \tX \pa_{\tX})(\frac 32 -\frac 12 \tX \pa_{\tX})\Psi_1\right]\left( \frac{g^{\frac 12}f^{-\frac 12}}{\sqrt 6} X \right)\\
\eee
We only treat the first term, the proof being the same for the others. First, from \fref{eq:def tildeF} and \fref{eq:def tildeG}:
$$
(1+|Z|)^{-3k} \lesssim \left| g^{\frac 12}(s,Y)f^{-\frac 12}(s,Y)  \right| \lesssim (1+|Z|)^{-3k}, \ \ |(Y\pa_Y)^j \chi_d|\lesssim 1
$$
on the support of $\chi_d$. For any $j\in \mathbb N$, from \fref{main:eqPsi1origine} one has that
$$
\left| \left[(\tX\pa_{\tX})^j (-\frac 32 +\frac 12 \tX \pa_{\tX})(-\frac 12 +\frac 12 \tX \pa_{\tX})\Psi_1\right](\tX) \right|\lesssim |\tX|^5(1+|\tX|)^{\frac 13 -5}
$$
and from \fref{eq:def tildeF} and \fref{eq:def tildeG} and for $j_1+j_1'\leq J$:
$$
\frac{|(Y\pa_Y)^j\pa_Z^{j'}(g^{\frac 12}f^{-\frac 12})|}{g^{\frac 12}f^{-\frac 12}}\lesssim (1+|Z|)^{-j'}, \ \ |(Y\pa_Y)^j\pa_Z^{j'}((\pa_Y g)^2 g^{-\frac 52}f^{\frac 32})|\lesssim e^{-\frac{k-1}{k}s}(1+|Z|)^{k-2-j'}
$$
since $\pa_Y=e^{-(k-1)s/(2k)}\pa_Z$. We therefore infer that:
\bee
&& \left||(Y\pa_Y)^{j_1}\pa_Z^{j'_1}(X\pa_{X})^{j_2}\left(\left[(-\frac 32 +\frac 12 \tX \pa_{\tX})(-\frac 12 +\frac 12 \tX \pa_{\tX})\Psi_1\right]\left( \frac{g^{\frac 12}f^{-\frac 12}}{\sqrt 6} X \right)\right)\right|\\
&\lesssim & \left| \frac{g^{\frac 12}f^{-\frac 12}}{\sqrt 6} X \right|^5\left(1+ \left| \frac{g^{\frac 12}f^{-\frac 12}}{\sqrt 6} X \right| \right)^{\frac 13 -5} \lesssim |\tilde X|^5 (1+|\tilde X|)^{\frac 13 -5} \lesssim \frac{|X|^5((1+|Z|)^{3k}+|X|)^{\frac 13 -5}}{(1+|Z|)^{k}},
\eee
and in turn, using \fref{main:sizephi}:
\bee
&& \frac{\left|(Y\pa_Y)^{j_1}\pa_Z^{j_1'}(X\pa_X)^{j_2}\left(\left[(\pa_Y g)^2 g^{-\frac 52}f^{\frac 32}(-\frac 32 +\frac 12 \tX \pa_{\tX})(-\frac 12 +\frac 12 \tX \pa_{\tX})\Psi_1 \right]\left( \frac{g^{\frac 12}f^{-\frac 12}}{\sqrt 6} X \right)\right)\right|}{|\varphi_{j,0}(X,Z)|} \\
& \lesssim & e^{-\frac{k-1}{k}s}|X|^{5-j}(1+|Z|)^{-2}((1+|Z|)^{3k}+|X|)^{j-4-\frac j3}.
\eee
Therefore, one has the following estimate, performing two changes of variables, the first one being $X=(1+Z^{2k})^{3/2}\tilde X$ and the second one $Z=e^{-(k-1)s/(2k)}Y$, with $dX/|X|=d\tilde X/|\tilde X|$, $dY/|Y|=dZ/|Z|$, since $|Y|/\lY \leq 1$ and $|(Y\pa_Y)^{j}\pa_Z^{j'} \chi_d|\lesssim 1$:
\bee
&& \int_{\mathbb R^2} \left( \frac{(Y\pa_Y)^{j_1}\pa_Z^{j_1'}(X\pa_X)^{j_2}\chi_d(\pa_Y g)^2 g^{-\frac 52}f^{\frac 32}(-\frac 32 X +\frac 12 X \pa_X)(-\frac 12 +\frac 12 X \pa_X)\Psi_1\left( \frac{g^{\frac 12}f^{-\frac 12}}{\sqrt 6} X \right)}{|\varphi_{j,0}(X,Z)|} \right)^{2q} \frac{dXdY}{|X|\lY} \\
&\lesssim & e^{-2q \frac{k-1}{k}s} \int_{\mathbb R^2} \left(|X|^{5-j}(1+|Z|)^{-2}((1+|Z|)^{3k}+|X|)^{j-4-\frac j3}\right)^{2q}\frac{dXdY}{|X|\lY} \\
&\lesssim & e^{-2q \frac{k-1}{k}s} \int_{\mathbb R} \left(\int_{\mathbb R} \left((1+|Z|)^{3k(1-\frac j3)-2}|\tilde X|^{5-j}(1+|\tilde X|)^{j-4-\frac j3}\right)^{2q} \frac{d\tilde X}{|\tilde X|}\right) \frac{d Y}{\lY}\\
 &\lesssim & e^{-2q \frac{k-1}{k}s}  \int_{\mathbb R} \left((1+|Z|)^{3k(1-\frac j3)-2}\right)^{2q}\frac{ d Y}{\lY} \\
 &\lesssim & e^{-2q \frac{k-1}{k}s}  \int_{|Y|\leq e^{\frac{k-1}{2k}s}} \frac{ d Y}{\lY} +\int_{|Y|\geq e^{\frac{k-1}{2k}s}} (|Z|^{3k(1-\frac j3)-2})^{2q}\frac{ d Y}{\lY}\\
 &\lesssim & e^{-2q \frac{k-1}{k}s}  \left(s +\int_{|Z|\geq 1} (|Z|^{3k(1-\frac j3)-2})^{2q}\frac{ d Z}{|Z|}\frac{|Y|}{\lY}\right)\lesssim se^{-2q \frac{k-1}{k}s} 
\eee
provided that $3<j<5$. We claim that the very same estimates for the other terms in the expression of $R_1$ holds, and that they can be proved performing the same computation, thanks to the same fundamental cancellations
$$
\left| \left[-\frac 32\Psi_1+\frac 32 \tX \pa_{\tX} \Psi_1 -\frac 12 \tX^2 \pa_{\tX}^2 \Psi_1\right]( \tilde X)\right|+\left|\left[(\frac 12 -\frac 12 \tX \pa_{\tX})(\frac 32 -\frac 12 \tX \pa_{\tX})\Psi_1\right](\tilde X)\right|\lesssim  |\tilde X|^5(1+|\tilde X|)^{\frac 13 -5},
$$
implying that for $j=4$, since $k\geq 2$:
$$
\sum_{0\leq j_1+j_2\leq 1} \int_{\mathbb R^2} \frac{((Y\pa_Y)^{j_1}(X\pa_X)^{j_2}R_1)^{2q}+(\pa_Z^{j_1}(X\pa_X)^{j_2}R_1)^{2q}}{\varphi_{4,0}^{2q}(X,Z)}\frac{dXdY}{|X|\lY}  \lesssim se^{-2q \frac{k-1}{k}s}  \lesssim s e^{-qs},
$$
which can be rewritten using the equivalence \fref{an:equivalence5}, \fref{an:equivalence1} and \fref{an:equivalence2}:
$$
\sum_{0\leq j_1+j_2\leq 1} \int_{\mathbb R^2} \frac{((Y\pa_Y)^{j_1}A^{j_2}R_1)^{2q}+(\pa_Z^{j_1}A^{j_2}R_1)^{2q}}{\varphi_{4,0}^{2q}(X,Z)}\frac{dXdY}{|X|\lY}  \lesssim s e^{-qs}.
$$

\noindent \textbf{Step 2} \emph{Estimate for $R_2$}: Note that $|Z|\geq de^{s/(2k)}\gg 1$ on the support of $1-\chi_d$ . We first compute using \fref{def:Thetae}, \fref{eq:def tildeF} and \fref{eq:def tildeG}:

\bea
\lab{main:idR2} &&\pa_s\Theta_e-\frac 12 \Theta_e+\frac 32 X \pa_X \Theta_e+\frac 12 Y \pa_Y \Theta_e+ \Theta_e \pa_X \Theta_e-  \pa_{YY} \Theta_e  \\
\non &=&\frac{1}{12}X^5g^2e^{-2 \tilde X ^4}+ (-Xf+X^3\frac g6)(\pa_s+\frac 32 X \pa_X +\frac 12 Y \pa_Y-  \pa_{YY} )\left(e^{-\tilde X^4} \right)\\
\non &&+ (Xf^2-4\frac{X^3g}{6}f)e^{- \tilde X ^4}\left(e^{- \tilde X ^4}-1\right) \\
\non &&+(-Xf+\frac{X^3}{6}g)^2 e^{-\tilde X^4}\pa_X \left(e^{- \tilde X^4}\right)-2 \pa_Y (-Xf+\frac{X^3}{6}g) \pa_Y \left(e^{-\tilde X^4}\right)\\
\eea
In what follows $0<\gamma\ll 1$ denotes a small constant whose value can change from one line to another. For the first term, \fref{eq:def tildeF} and \fref{eq:def tildeG} imply that on the support of $1-\chi_d$ for $j_1,j_1',j_2\in \mathbb N$ with $j_1+j_1'\leq J$:
$$
\left| (Y\pa_Y)^{j_1}\pa_Z^{j_1'} (X\pa_X)^{j_2}\left(X^5g^2 e^{-2 \tilde X^4}\right) \right| \lesssim e^{-\frac{1}{2k}s}(1+|Z|)^{1-k-j_1'}|\tilde X|^5 e^{-\gamma \tilde X^4}.
$$
Therefore, using \fref{main:sizephi}:
$$
\frac{\left| (Y\pa_Y)^{j_1}\pa_Z^{j_1'}(X\pa_X)^{j_2} \left(X^5g^2e^{-2 \tilde X^4}\right) \right|}{|\varphi_{j,0}(X,Z)|}\lesssim e^{-\frac{1}{2k}s}(1+|Z|)^{1-k(j-3)}|\tilde X|^{5-j}e^{-\gamma \tilde X^4}.
$$
Since $dX/|X|=d\tilde X/|\tilde X|$ and $|Z|\gtrsim e^{\frac{1}{2k}s}$ on the support of $1-\chi_d$ and $|(Y\pa_Y)^j\pa_Z^{j'} \chi_d|\lesssim 1$ one then infers that:
\bee
&&\int_{\mathbb R^2} \left(\frac{(Y\pa_Y)^{j_1} \pa_Z^{j_1'}(X\pa_X)^{j_2} \left((1-\chi_d)X^5g^2 e^{-2 \tilde X^4}\right)}{|\varphi_{j,0}(X,Z)|}\right)^{2q}\frac{dXdY}{|X|\lY}\\
&\lesssim & e^{-\frac{q}{k}s} \int_{|Z|\geq e^{\frac{1}{2k}s}} \left((1+|Z|)^{1-k(j-3)}|\tilde X|^{5-j}e^{-\gamma \tilde X^4}\right)^{2q} \frac{d\tilde XdZ}{|\tilde X||Z|} \lesssim e^{-\frac{q}{k}s} \int_{|Z|\geq e^{\frac{1}{2k}s}}  |Z|^{(1-k(j-3))2q} \frac{dZ}{|Z|}\\
&\lesssim & e^{-(j-3)qs}
\eee
provided that $3+1/k<j<5$. We now turn to the second term in the expression of $R_2$. One has that:
$$
\left|(Y\pa_Y)^{j_1} \pa_Z^{j_1'}(X\pa_X)^{j_2} \left((\pa_s+\frac 32 X \pa_X +\frac 12 Y \pa_Y-  \pa_{YY} )\left(e^{-\left(\tilde X\right)^4} \right)\right)\right| \lesssim |\tilde X|^{4}e^{-\gamma \tilde X^4}.
$$
Therefore, from \fref{eq:def tildeF}, \fref{eq:def tildeG} and \fref{main:sizephi} we obtain that:
\bee
&& \left| \frac{(Y\pa_Y)^{j_1} \pa_Z^{j_1'}(X\pa_X)^{j_2} \left((-Xf+X^3\frac g6)(\pa_s+\frac 32 X \pa_X +\frac 12 Y \pa_Y-  \pa_{YY} )\left(e^{-\left(\tilde X\right)^4} \right)\right) }{\varphi_{4,0}(X,Z)}\right| \\
& \lesssim  &e^{-\frac{1}{4k}s}(1+|Z|)^{\frac 12 -k(j-3)}|\tilde X|^{5-j}e^{-\gamma \tilde X^4}
\eee
Since $dX/|X|=d\tilde X/|\tilde X|$, $|Z|\gtrsim e^{\frac{1}{2k}s}$ and $|(Y\pa_Y)^j \chi_d|\lesssim 1$ on the support of $1-\chi_d$ one then infers that:
\bee
&&\int \left|\frac{(Y\pa_Y)^{j_1}\pa_Z^{j_1'} (X\pa_X)^{j_2}\left((1-\chi_d)(-Xf+X^3\frac g6)(\pa_s+\frac 32 X \pa_X +\frac 12 Y \pa_Y-  \pa_{YY} )\left(e^{-\tilde X^4} \right)\right)}{\varphi_{j,0}(X,Z)} \right|^{2q}\frac{dXdY}{|X|\lY}\\
&\lesssim & e^{-\frac{q}{2k}s} \int_{|Z|\geq e^{\frac{1}{2k}s}} \left((1+|Z|)^{\frac 12-k(j-3)}|\tilde X|^{5-j}e^{-\gamma \tilde X^4}\right)^{2q} \frac{d\tilde XdZ}{|\tilde X||Z|}\\
&\lesssim & e^{-\frac{q}{2k}s} \int_{|Z|\geq e^{\frac{1}{2k}s}}  |Z|^{(\frac 12-k(j-3))2q} \frac{dZ}{|Z|}\lesssim e^{-(j-3)qs}
\eee
provided that $3+1/(2k)<j<5$. We claim that all the other remaining terms in \fref{main:idR2} can be treated verbatim the same way, yielding for $j=4$:
$$
\sum_{0\leq j_1+j_2\leq 1} \int_{\mathbb R^2} \frac{((Y\pa_Y)^{j_1}(X\pa_X)^{j_2}R_2)^{2q}+(\pa_Z^{j_1}(X\pa_X)^{j_2}R_2)^{2q}}{\varphi_{4,0}^{2q}(X,Z)}\frac{dXdY}{|X|\lY}  \lesssim se^{-2q \frac{k-1}{k}s}  \lesssim s e^{-qs},
$$
which can be rewritten using the equivalence \fref{an:equivalence5}, \fref{an:equivalence1} and \fref{an:equivalence2}:
$$
\sum_{0\leq j_1+j_2\leq 1} \int_{\mathbb R^2} \frac{((Y\pa_Y)^{j_1}A^{j_2}R_2)^{2q}+(\pa_Z^{j_1}A^{j_2}R_2)^{2q}}{\varphi_{4,0}^{2q}(X,Z)}\frac{dXdY}{|X|\lY}  \lesssim e^{-qs}.
$$

\noindent \textbf{Step 3} \emph{Estimate for $R_3$}: one first notices that on the support of this term $de^{s/2}\leq |Y|\leq 2de^{s/2}$, which will then be assumed throughout this step, and that for $j_1,j_1'\in \mathbb N$:
\be \lab{main:intermediaire step3 R}
|(Y\pa_Y)^{j_1}\pa_Z^{j_1'}(\pa_s \chi_d +\frac 12 Y \pa_Y \chi_d-  \pa_{YY}\chi_d)|\lesssim 1.
\ee
Also, $\pa_X^j \tilde \Theta=\pa_X^j \Theta_e$ for $j=0,...,4$ on the axis $\{X=0\}$. From this, the formulas \fref{def:tildeTheta} and \fref{def:Thetae}, and the estimate \fref{eq:def tildeF} and \fref{eq:def tildeG} one obtains that if $j_1+j_1'\leq J$:
$$
|(Y\pa_Y)^{j_1}\pa_Z^{j_1'}(X\pa_X)^{j_2}(\tilde \Theta-\Theta_e)|\lesssim (1+|Z|)^{k-j_1'} |\tilde X|^5 (1+|\tilde X|)^{\frac 13 -5}
$$
giving using \fref{main:sizephi} the estimate:
$$
\frac{|(Y\pa_Y)^{j_1}\pa_Z^{j_1'}(X\pa_X)^{j_2}(\tilde \Theta-\Theta_e)|}{|\varphi_{j,0}(X,Z)|}\lesssim (1+|Z|)^{-k(j-3)} |\tilde X|^{5-j} (1+|\tilde X|)^{-4+j-\frac j3}
$$
The above estimate and \fref{main:intermediaire step3 R} therefore imply, since $|Z|\sim e^{\frac{1}{2k}s}$ and since $dX/|X|=d\tilde X/|\tilde X|$:
\bee
&&\int \frac{((Y\pa_Y)^{j_1}\pa_Z^{j_1'}(X\pa_X)^{j_2}R_3)^{2q}}{\varphi_{j,0}^{2q}(X,Z)}\frac{dXdY}{|X|\lY} \\
&\lesssim &\int_{e^{\frac{1}{2k}}\leq |Z|\leq 2e^{\frac{1}{2k}}} \left((1+|Z|)^{-k(j-3)} |\tilde X|^{5-j} (1+|\tilde X|)^{-4+j-\frac j3}\right)^{2q}\frac{d\tilde XdZ}{|\tilde X||Z|} \lesssim e^{-(j-3)qs}
\eee
provided that $3<j<5$. Taking $j=4$ and using \fref{an:equivalence5}, \fref{an:equivalence1} and \fref{an:equivalence2} this gives:
$$
\sum_{0\leq j_1+j_2\leq 1} \int_{\mathbb R^2} \frac{((Y\pa_Y)^{j_1}A^{j_2}R_3)^{2q}+(\pa_Z^{j_1}A^{j_2}R_3)^{2q}}{\varphi_{4,0}^{2q}(X,Z)}\frac{dXdY}{|X|\lY}  \lesssim e^{-qs}.
$$

\noindent \textbf{Step 4} \emph{Estimate for $R_4$ and $R_5$}: These estimates can be proved  along the very same lines as we just estimated $R_1$, $R_2$ and $R_3$. We leave the proof to the reader in order to keep the present article short.

\end{proof}

We now estimate the lower order linear term in \fref{main:evolutione}.

\begin{lemma} \lab{main:lem:tildeL}

There holds on $[s_0,s_1]$, for $j_1\leq J$ and $j_2\in \mathbb N^*$:
\be \lab{main:Q-Thetapointwise1}
| \pa_Z^{j_1}(Q-\Theta)| \lesssim e^{-\frac{1}{4k}s} |X|(1+|Z|)^{-j_1} \ \ \ \text{and} \ \ \ | \pa_Z^{j_1}\pa_X^{j_2}(Q-\Theta)| \lesssim e^{-\frac{1}{4k}s} (1+|X|)^{1-j_2}(1+|Z|)^{-j_1}.
\ee

\end{lemma}

\begin{proof}

\noindent \textbf{Step 1} \emph{Inner estimate}. We first consider the zone $|Y|\leq de^{s/2}$, or equivalently $|Z|\leq de^{s/2k}$. From \fref{def:e} and \fref{eq:def Theta}:
$$
Q-\Theta = \tilde \Theta-\Theta = b^{-1}f^{-\frac 12} \Psi_1 \left(bf^{\frac 32}X \right)-F_k^{-\frac 12}(Z)\Psi_1 \left(F_k^{\frac 32}(Z)X \right).
$$
First, we have using \eqref{eq:def tildeF} and \eqref{eq:def tildeG}:
$$|\frac{f}{F_k}-1|=|\frac{\tilde{f}}{F_k}|\lesssim e^{-\frac{1}{4k}s}(1+|Z|)^{\frac 12}, \qquad |\frac{g}{G_k}-1|=|\frac{\tilde{g}}{G_k}|\lesssim e^{-\frac{1}{4k}s}(1+|Z|)^{\frac 12}.$$
Since the above right hand sides are also $\leq \frac 12$ from the choice of $d$ \fref{id:choiced}, and since $G_k/6F_k=1$, $\mu$ defined by \fref{eq:def b} satisfies:
$$
|\mu-1|\lesssim e^{-\frac{1}{4k}s}(1+|Z|)^{\frac 12}.
$$
Therefore, using \fref{eq:def tildeF}, \fref{eq:def tildeG}, the above inequalities and \fref{main:sizeFkPsi1}:
\bee
|Q-\Theta| &=& \Big| \mu^{-1}f^{-\frac 12}\Psi_1 (\mu f^{\frac 32}X)-f^{-\frac 12}\Psi_1 (f^{\frac 32}X)+f^{-\frac 12} \Psi_1 (f^{\frac 32}X)-F_k^{-\frac12}(Z) \Psi_1 (F_k^{\frac 32}(Z)X)\Big| \\
&=& \left| f^{-\frac 12}\int_{\tilde \mu=1}^\mu \tilde \mu^{-2}(-\Psi_1+\tX\pa_{\tX}\Psi_1)(\tilde \mu f^{\frac 32}X)d\tilde \mu+F_k^{-\frac 12}\int_{\lb=1}^{f/F_k} \lb^{-\frac 32}(-\frac 12 \Psi_1 +\frac 32 \tX \pa_{\tX} \Psi_1)(F_k^{\frac 32}\lb^{\frac 32}X)d\lb \right| \\
&\lesssim & |f^{-\frac 12}| |\mu-1| \sup_{ |\tilde \mu|\in [1,\mu] }\left| (-\Psi_1+\tX\pa_{\tX}\Psi_1)(\tilde \mu f^{\frac 32}X)\right|\\
&&+F_k^{-\frac 12} \left| \frac{f}{F_k}-1\right| \sup_{\lb \in [1, f/F_k]} \left|(-\frac 12 \Psi_1 +\frac 32 \tX \pa_{\tX} \Psi_1)(F_k^{\frac 32}\lb^{\frac 32}X)\right| \\
&\lesssim & \left(|\mu-1|+|\frac{f}{F_k}-1|\right)(1+|Z|)^k|\tilde X|(1+|\tilde X|)^{\frac 13-1} \lesssim e^{-\frac{1}{4k}s}(1+|Z|)^{k+\frac 12 }|\tilde X|(1+|\tilde X|)^{\frac 13-1} \\
&\lesssim & e^{-\frac{1}{4k}s}(1+|Z|)^{-2k+\frac 12 }| X|(1+|\tilde X|)^{\frac 13-1} \lesssim e^{-\frac{1}{4k}s} | X|. \\
\eee
One computes similarly that
\bee
|\pa_X(Q-\Theta)| &=& \Big| f\pa_X\Psi_1 (\mu f^{\frac 32}X)-f\pa_X \Psi_1 (f^{\frac 32}X)+f\pa_X \Psi_1 (f^{\frac 32}X)-F_k(Z)\pa_X \Psi_1 (F_k^{\frac 32}(Z)X)\Big| \\
&=& \left| f \int_{\tilde \mu=1}^\mu \tilde \mu^{-1} (\tX\pa_{\tX\tX}\Psi_1)(\tilde \mu f^{\frac 32}X)d\tilde \mu+F_k \int_{\lb=1}^{f/F_k} \lb^{-1}(\pa_{\tX} \Psi_1 +\frac 32 \tX \pa_{\tX}^2 \Psi_1)(F_k^{\frac 32}\lb^{\frac 32}X)d\lb \right| \\
&\lesssim & \left(|\mu-1|+|\frac{f}{F_k}-1|\right)(1+|Z|)^{-2k}(1+|\tilde X|)^{-\frac 23} \lesssim e^{-\frac{1}{4k}s}(1+|Z|)^{\frac 12 -2k} (1+|\tilde X|)^{-\frac 23} \\
&\lesssim & e^{-\frac{1}{4k}s}.\\
\eee
The proof for higher order derivatives is a direct generalisation of the above computations, that we omit here, giving \fref{main:Q-Thetapointwise1} in this zone.\\

\noindent \textbf{Step 2} \emph{Outer estimate}. Let $0<\gamma \ll 1$ be a small constant whose value can change from one line to another. We now turn to the zone $ de^{s/2}\leq |Y|\leq 2de^{s/2}$ or equivalently $de^{s/(2k)}\leq |Z|\leq 2de^{s/(2k)}$. We perform brute force estimates on the identity \fref{def:e} using \fref{eq:def tildeF} and \fref{eq:def tildeG}:
\bee
|Q-\Theta| &=& |\chi_d \tilde \Theta +(1-\chi_d)\Theta_e-\Theta| \leq |\Theta|+|\tilde \Theta|+|\Theta_e| \\
& \lesssim & (1+|Z|)^{k }|\tilde X|(1+|\tilde X|)^{\frac 13-1}+ (1+|Z|)^{k }|\tilde X|(1+|\tilde X|)^{\frac 13-1}+(1+|Z|)^{k }|\tilde X|e^{-\gamma \tilde X^4} \\
&\lesssim & e^{-\frac{1}{4k}s}(1+|Z|)^{k+\frac 12 }|\tilde X|(1+|\tilde X|)^{\frac 13-1} \lesssim e^{-\frac{1}{4k}s}(1+|Z|)^{-2k+\frac 12 }| X|(1+|\tilde X|)^{\frac 13-1} \\
&\lesssim & e^{-\frac{1}{4k}s} | X|. \\
\eee
and similarly
\bee
|\pa_X(Q-\Theta )| &\leq &|\pa_X \Theta|+|\pa_X\tilde \Theta|+|\pa_X\Theta_e|\\
 & \lesssim & (1+|Z|)^{-2k } (1+|\tilde X|)^{-\frac 23}+ (1+|Z|)^{-2k } (1+|\tilde X|)^{-\frac 23}+(1+|Z|)^{-2k }e^{-\gamma \tilde X^4} \\
&\lesssim & e^{-\frac{1}{4k}s}(1+|Z|)^{-2k+\frac 12 }(1+|\tilde X|)^{-\frac 23}\lesssim e^{-\frac{1}{4k}s}. \\
\eee
Again, the generalisation of this argument for higher order derivatives is direct, yielding \fref{main:Q-Thetapointwise1} in this zone.\\

\noindent \textbf{Step 3} \emph{Outer estimate}. We now turn to the zone $ |Z|\geq 2de^{\frac{1}{2k}s}$ where $Q-\Theta=\Theta_e-\Theta$. We perform the very same computations as in Step 2, estimating $\Theta$ and $\Theta_e$ separately, giving \fref{main:Q-Thetapointwise1} in this zone and ending the proof of the Lemma.

\end{proof}

We can now perform energy estimates in the bootstrap regime of Proposition \ref{main:pr:bootstrap} and improve the bootstrap bounds.

\begin{lemma} \lab{main:lem:energy0}

There exists $K_{0,0}^*>0$ such that for any positive constants $(K_{j_1,j_2})_{0\leq j_1+j_2\leq 2}$ and $(\tilde K_{j_1,j_2})_{0\leq j_1+j_2\leq 2, 1\leq j_1}$ with $K_{0,0}\geq K^*_{0,0}$, there exists $s^*$ such that if $u$ is trapped on $[s_0,s_1]$ with $s_0\geq s^*$, then:
\be \lab{main:weighteddecayes1}
\left( \int_{\mathbb R^2} \frac{\e^{2q}(s_1)}{\varphi_{4,0}^{2q}(X,Z)}\frac{dXdY}{|X|\lY}\right)^{\frac{1}{2q}}\leq \frac{K_{0,0}}{2} e^{-\left(\frac 12 -\kappa \right)s_1}.
\ee

\end{lemma}

\begin{proof}

We compute from \fref{main:evolutione} and \fref{id:estimationlineaire} (note that this estimate is valid for Schwartz functions so that we implicitly use here a density argument):
\bee
&&\frac{d}{ds} \left( \frac{1}{2q} \int_{\mathbb R^2} \frac{\e^{2q}}{\varphi_{4,0}^{2q}(X,Z)}\frac{dXdY}{|X| \lY} \right)\\
\non &\leq & -\left( \frac{1}{2}-\frac{\kappa}{2}\right) \int_{\mathbb R^2} \frac{\e^{2q}}{\varphi_{4,0}^{2q}(X,Z)}\frac{dXdY}{|X| \lY}-\frac{2q-1}{q^2}\int \frac{|\pa_Y (\e^q)|^2}{\varphi_{4,0}^{2q}}\frac{dXdY}{|X|\lY} \\
\non &&+\frac{1}{2q}\int \e^{2q}\pa_X \left(\frac{Q-\Theta}{\varphi_{4,0}^{2q}|X|} \right)\frac{dXdY}{\lY}- \int \frac{\e^{2q-1}}{\varphi_{4,0}^{2q}}\left(\e \e_X+R+\pa_X(Q-\Theta)\e \right)\frac{dXdY}{|X|\lY}.
\eee
We now estimate the last terms. First,
$$
\pa_X \left(\frac{Q-\Theta}{\varphi_{4,0}^{2q}|X|}\right)=\frac{\pa_X(Q-\Theta)}{\varphi_{4,0}^{2q}|X|}-\frac{Q-\Theta }{\varphi_{4,0}^{2q}X|X|}+\frac{Q-\Theta}{|X|}\pa_X\left(\frac{1}{{\varphi_{4,0}^{2q}}} \right).
$$
One has from \fref{id:phij0} that $|\pa_X \varphi_{4,0}(Z,X)|\lesssim |\varphi_{4,0}(X,Z)|/|X|$. From this fact, from \fref{main:Q-Thetapointwise1} and \fref{main:Q-Thetapointwise1} one infers that:
\be \lab{main:bdpaxQTheta}
\left| \pa_X \left(\frac{Q-\Theta}{\varphi_{4,0}^{2q}|X|}\right)\right| \lesssim q e^{-\frac{1}{4k}s}\frac{1}{\varphi_{4,0}^{2q}(X,Z)|X|}
\ee
From the above estimate, \fref{main:bde2} and \fref{main:Q-Thetapointwise1} we infer that:
\bee
&&\left| \frac{1}{2q}\int \e^{2q}\pa_X \left(\frac{Q-\Theta}{\varphi_{4,0}^{2q}|X|} \right)\frac{dXdY}{\lY}- \int \frac{\e^{2q-1}}{\varphi_{4,0}^{2q}}\left(\e \e_X+\pa_X(Q-\Theta)\e \right)\frac{dXdY}{|X|\lY}\right|\\
&\leq & C \left(e^{-\frac{1}{4k}s}+e^{-\left(\frac 12 -\kappa \right)s} \right) \int \frac{\e^{2q}}{\varphi_{4,0}^{2q}}\frac{dXdY}{|X|\lY} \leq \frac{\kappa}{2} \int \frac{\e^{2q}}{\varphi_{4,0}^{2q}}\frac{dXdY}{|X|\lY}, \\
\eee
where above $C$ depends on the boostrap constants except $s^*$, so that the last inequality is obtained upon taking $s^*$ large enough. Applying H\"older, using \fref{main:weighteddecaye} and \fref{main:estimationR}:
\bee
&& \left| \int \frac{\e^{2q-1}}{\varphi_{4,0}^{2q}} R \frac{dXdY}{|X|\lY} \right| \leq \left| \int \frac{\e^{2q}}{\varphi_{4,0}^{2q}} \frac{dXdY}{|X|\lY} \right|^{\frac{2q-1}{2q}}\left| \int \frac{R^{2q}}{\varphi_{4,0}^{2q}} \frac{dXdY}{|X|\lY} \right|^{\frac{1}{2q}} \leq CK_{0,0}^{2q-1} s^{\frac{1}{2q}} e^{-2q(\frac 12 -\kappa)s-\kappa s}.
\eee
We obtain
\bee
&&\frac{d}{ds} \left( \frac{1}{2q} \int_{\mathbb R^2} \frac{\e^{2q}}{\varphi_{4,0}^{2q}(X,Z)}\frac{dXdY}{|X| \lY} \right)\\
\non &\leq & -\left(\frac{1}{2}-\kappa \right) \int_{\mathbb R^2} \frac{\e^{2q}}{\varphi_{4,0}^{2q}(X,Z)}\frac{dXdY}{|X| \lY}-\frac{2q-1}{q^2}\int \frac{|\pa_Y (\e^q)|^2}{\varphi_{4,0}^{2q}}\frac{dXdY}{|X|\lY} +CK_{0,0}^{2q-1}s^{\frac{1}{2q}}e^{-2q(\frac 12 -\kappa)s-\kappa s}.
\eee
We now reintegrate until the time $s_1$ the above estimate, yielding from \fref{main:weighteddecayeinit}:
\bee
\int_{\mathbb R^2} \frac{\e^{2q}}{\varphi_{4,0}^{2q}(X,Z)}\frac{dXdY}{|X| \lY}& \leq &e^{-2q \left(\frac 12 -\kappa \right)(s_1-s_0)} \int_{\mathbb R^2} \frac{\e_0^{2q}}{\varphi_{4,0}^{2q}(X,Z)}\frac{dXdY}{|X| \lY} +CK_{0,0}^{2q-1}e^{-2q \left(\frac 12 -\kappa \right)s_1} \int_{s_0}^s \tilde s^{\frac{1}{2q}} e^{-\kappa \tilde s}d\tilde s \\
&\leq &  \frac{K_{0,0}^{2q}}{2^{2q}} e^{-2q \left(\frac 12 -\kappa \right)s_1}, \\
\eee
as $C$ is independent of $(K_{j_1,j_2})_{0\leq j_1+j_2\leq 2}$, $(\tilde K_{j_1,j_2})_{0\leq j_1+j_2\leq 2, 1\leq j_1}$ and $s^*$, upon choosing $K_{0,0}$ large enough. This ends the proof of the Lemma.

\end{proof}

We now perform similar weighted energy estimates for the derivatives of $\e$.

\begin{lemma} \lab{main:lem:energy1}

There exists a choice of constants $K_{j_1,j_2}\gg 1$ for $0\leq j_1+j_2\leq 2$ and $(j_1,j_2)\neq (0,0)$, and $\tilde K_{j_1,j_2}\gg 1$ for $0\leq j_1+j_2\leq 2$ and $j_1\geq 1$, such that if $u$ is trapped on $[s_0,s_1]$, then at time $s_1$, for $0\leq j_1+j_2\leq 2$ and $(j_1,j_2)\neq (0,0)$:
\be \lab{main:weighteddecayeimproved}
\left( \int_{\mathbb R^2} \frac{((\pa_Z^{j_1}A^{j_2}\e(s_1))^{2q}}{\varphi_{4,0}^{2q}(X,Z)}\frac{dXdY}{|X|\lY}\right)^{\frac{1}{2q}}\leq \frac{K_{j_1,j_2}}{2} e^{-\left(\frac 12 -\kappa \right)s_1},
\ee
and for $0\leq j_1+j_2\leq 2$ and $j_2\geq 1$:
\be \lab{main:weighteddecayeimproved2}
\left( \int_{\mathbb R^2} \frac{(((Y\pa_Y)^{j_1}A^{j_2}\e(s_1))^{2q}}{\varphi_{4,0}^{2q}(X,Z)}\frac{dXdY}{|X|\lY}\right)^{\frac{1}{2q}}\leq \frac{\tilde K_{j_1,j_2}}{2} e^{-\left(\frac 12 -\kappa \right)s_1}.
\ee

\end{lemma}

\begin{proof}

\noindent \textbf{Step 1} \emph{Proof for $A\e$.} We claim that for any $K_{0,0}>0$, there exists $K_{0,1}^*>0$ such that for any positive constants $(K_{j_1,j_2})_{1\leq j_1+j_2\leq 2}$ and $(\tilde K_{j_1,j_2})_{0\leq j_1+j_2\leq 2, 1\leq j_1}$ with $K_{0,1}\geq K^*_{0,1}$, there exists $s^*$ such that if $s_0\geq s^*$ then \fref{main:weighteddecayeimproved} holds true for $j_1=0$ and $j_2=1$. We now prove this claim. Recall \fref{main:def:A} and let $\e_1:=A\e$. Then $A$ commutes with the transport part of the flow:
\be \label{id:commutationA}
\left[A, \pa_s+\left(\frac 32 X +\Theta \right)\pa_X+\frac 12 Z\pa_Z \right]=0.
\ee
Indeed, we compute the commutator using \fref{NLH:eq:Fk}:
\bee
&&\left[(\frac{3}{2}X+F_k^{-\frac 32}(Z)\Theta )\pa_X, \pa_s+\left(\frac 32 X +\Theta \right)\pa_X+\frac 12 Z\pa_Z \right]\\
&=& \Big(-\pa_s (F_k^{-\frac 32}(Z)\Psi_1(F_k^{\frac 32}X))+\left(\frac 32 X +F_k^{-\frac 32}(Z)\Psi_1(F_k^{\frac 32}X) \right)\pa_X \left(\frac 32 X +F_k^{-\frac 12}(Z)\Psi_1(F_k^{\frac 32}X) \right)\\
&& -\left(\frac 32 X +F_k^{-\frac 12}(Z)\Psi_1(F_k^{\frac 32}X) \right)\pa_X \left(\frac 32 X +F_k^{-\frac 32}(Z)\Psi_1(F_k^{\frac 32}X) \right) -\frac 12 Z \pa_Z \left(F_k^{-\frac 32}(Z)\Psi_1(F_k^{\frac 32}X) \right) \Big)\pa_X \\
&=& \frac 32 \left( -\frac{1}{2k}Z\pa_Z F_kF_k^{-\frac 52}+F_k^{-\frac 12}-F_k^{-\frac 32}\right)(-\Psi_1+\tX\pa_{\tX} \Psi_1)(F_k^{\frac 32}X)\pa_X = 0.
\eee
We compute from \fref{main:def:A}, \fref{main:evolutione} and the above cancellation the evolution of $\e_1$:
$$
(\e_1)_s+\mathcal L \e_1 +\tL \e_1+ (A\pa_X \Theta)\e+AR+[A,\tL ]\e+\e_1\e_X+ \e \pa_X \e_1 +\e[A,\pa_X] \e-  [A,\pa_{YY}]\e=0.
$$
Using the linear energy identity \fref{id:estimationlineaire} we infer that:
\bee
&&\frac{d}{ds} \left( \frac{1}{2q} \int_{\mathbb R^2} \frac{\e_1^{2q}}{\varphi_{4,0}^{2q}(X,Z)}\frac{dXdY}{|X| \lY} \right)\\
\non &\leq & -\left(\frac{1}{2}-\frac{\kappa}{2}\right) \int_{\mathbb R^2} \frac{\e_1^{2q}}{\varphi_{4,0}^{2q}(X,Z)}\frac{dXdY}{|X| \lY}-\frac{2q-1}{q^2}\int \frac{|\pa_Y (w^q)|^2}{\varphi_{4,0}^{2q}}\frac{dXdY}{|X|\lY} \\
\non &&+\frac{1}{2q}\int \e_1^{2q}\pa_X \left(\frac{Q-\Theta}{\varphi_{4,0}^{2q}|X|} \right)\frac{dXdY}{\lY}\\
&&- \int \frac{\e_1^{2q-1}}{\varphi_{4,0}^{2q}}\left((A\pa_X \Theta)\e+[A,\tL]\e+AR+\e_1\e_X+ \e \pa_X\e_1 +\e[A,\pa_X ]\e -  [A,\pa_{YY}]\e+\pa_X(Q-\Theta)\e_1 \right)\frac{dXdY}{|X|\lY}.
\eee
From \fref{main:bdpaxQTheta}, and \fref{main:Q-Thetapointwise1}, one has:
$$
\left| \frac{1}{2q}\int \e_1^{2q}\pa_X \left(\frac{Q-\Theta}{\varphi_{4,0}^{2q}|X|} \right)\frac{dXdY}{\lY}- \int \frac{\e_1^{2q-1}}{\varphi_{4,0}^{2q}}\pa_X(Q-\Theta)\e_1 \frac{dXdY}{|X|\lY}\right| \lesssim e^{-\frac{1}{4k}s} \int \frac{\e_1^{2q}}{\varphi_{4,0}^{2q}} \frac{dXdY}{|X|\lY}.
$$
Moreover, since
$$
A\pa_X\Theta=\frac 32 F_k(Z) (\tX\pa_{\tX}^2\Psi_1)(F_k^{\frac 32}(Z)X)+F_k(Z) (\Psi_1 \pa_{\tX}^2 \Psi_1)(F_k^{\frac 32}(Z)X)
$$
and since both $F_k$ and $X\pa_X^2\Psi_1$ and $\Psi_1 \pa_X^2 \Psi_1$ are bounded, one has that
$$
|A\pa_X\Theta|\lesssim 1
$$
and therefore using the bootstrap bound \fref{main:weighteddecaye} on $\e$ one deduces from H\"older:
$$
\left| \int \frac{\e_1^{2q-1}}{\varphi_{4,0}^{2q}} (A\pa_X \Theta)\e \frac{dXdY}{|X|\lY} \right|\leq C \left| \int \frac{\e_1^{2q}}{\varphi_{4,0}^{2q}} \frac{dXdY}{|X|\lY} \right|^{\frac{2q-1}{2q}}\left| \int \frac{\e^{2q}}{\varphi_{4,0}^{2q}} \frac{dXdY}{|X|\lY} \right|^{\frac{1}{2q}}\leq C K_{0,1}^{2q-1}K_{0,0}e^{-2q\left(\frac 12 -\kappa \right)s}.
$$
We compute from \fref{main:def:A} and \fref{main:def:tL} that
\bee
[A,\tL]\e &=& \left( A (Q-\Theta)-(Q-\Theta)(\frac 32 +\pa_{\tX} \Psi_1 (F_k^{\frac 32}X))\right)\pa_X\e +(A\pa_X(Q-\Theta))\e \\
&=&\frac{A (Q-\Theta)-(Q-\Theta)(\frac 32 +\pa_{\tX} \Psi_1 (F_k^{\frac 32}X))}{\frac 32 X +F_k^{-\frac 32}(Z)\Psi_1(F_k^{\frac 32}(Z)X)}\e_1+(A\pa_X(Q-\Theta))\e .
\eee
From \fref{main:def:A}, \fref{main:eq:estimationA}, \fref{main:sizeFkPsi1} and \fref{main:Q-Thetapointwise1} we then obtain
$$
|[A,\tL]\e |\leq C e^{-\frac{1}{4k}s} (|\e_1|+|\e|)
$$
for $C$ independent of the bootstrap bounds, which implies using H\"older and \fref{main:weighteddecaye}:
\bee
&&\left| \int \frac{\e_1^{2q-1}}{\varphi_{4,0}^{2q}} [A,\tL]\e \frac{dXdY}{|X|\lY} \right| \leq C e^{-\frac{1}{4k}s}  \int \frac{|\e_1|^{2q-1}}{\varphi_{4,0}^{2q}} (|\e_1|+|\e|) \frac{dXdY}{|X|\lY} \\
& \leq & C e^{-\frac{1}{4k}s}  \int \frac{\e_1^{2q}}{\varphi_{4,0}^{2q}} \frac{dXdY}{|X|\lY} +C\left( \int \frac{\e_1^{2q}}{\varphi_{4,0}^{2q}} \frac{dXdY}{|X|\lY}\right)^{\frac{2q-1}{2q}} \left( \int \frac{\e^{2q}}{\varphi_{4,0}^{2q}} \frac{dXdY}{|X|\lY}\right)^{\frac{1}{2q}}\\
&\leq &C e^{-\frac{1}{4k}s}  \int \frac{\e_1^{2q}}{\varphi_{4,0}^{2q}} \frac{dXdY}{|X|\lY}+C K_{0,1}^{2q-1}K_{0,0}e^{-2q\left(\frac 12 -\kappa \right)s}.
\eee
One then deduces from \fref{main:eq:estimationA}, \fref{main:def:A}, \fref{main:estimationR} and H\"older:
\bee
&&\left| \int \frac{\e_1^{2q-1}}{\varphi_{4,0}^{2q}} AR \frac{dXdY}{|X|\lY} \right| \leq \left| \int \frac{\e_1^{2q}}{\varphi_{4,0}^{2q}} \frac{dXdY}{|X|\lY} \right|^{\frac{2q-1}{2q}}\left| \int \frac{(AR)^{2q}}{\varphi_{4,0}^{2q}} \frac{dXdY}{|X|\lY} \right|^{\frac{1}{2q}}\\
&\leq & K_{0,1}^{\frac{2q-1}{2q}}e^{-\frac{2q-1}{2q}(\frac 12-\kappa)s} C \left| \int \frac{(X\pa_XR)^{2q}}{\varphi_{4,0}^{2q}} \frac{dXdY}{|X|\lY} \right|^{\frac{1}{2q}} \leq C K_{0,1}^{2q-1} s^{\frac{1}{2q}} e^{-2q(\frac 12-\kappa)s-\kappa s}.
\eee
Using \fref{main:bde2} we infer that for a constant $C$ depending on $(K_{j_1,j_2})_{1\leq j_1+j_2\leq 2}$ and $(\tilde K_{j_1,j_2})_{0\leq j_1+j_2\leq 2, 1\leq j_1}$:
$$
\left| \int \frac{\e_1^{2q-1}}{\varphi_{4,0}^{2q}} \e_1\e_X\frac{dXdY}{|X|\lY} \right| \leq C e^{-\left(\frac 12-\kappa\right)s} \int \frac{\e_1^{2q}}{\varphi_{4,0}^{2q}}\frac{dXdY}{|X|\lY}\leq \frac{\kappa}{16} \int \frac{\e_1^{2q}}{\varphi_{4,0}^{2q}}\frac{dXdY}{|X|\lY}
$$
upon choosing $s^*$ large enough. Integrating by parts, one has the identity:
$$
\int \frac{\e_1^{2q-1}}{\varphi_{4,0}^{2q}(X,Z)}\e \pa_X \e_1\frac{dXdY}{|X|\lY} = -\frac{1}{2q} \int \frac{\e_1^{2q}}{\varphi_{4,0}^{2q}(X,Z)}\pa_X\e -\frac{1}{2q}\int \e_1^{2q}\e \pa_X \left(\frac{1}{\varphi_{4,0}^{2q}(X,Z)|X|} \right) \frac{dXdY}{\lY}.
$$
From \fref{main:sizephi} one has that $|\pa_X \varphi_{4,0}(X,Z)/\varphi_{4,0}(X,Z)|\lesssim |X|^{-1}$. Therefore, using \fref{main:bde} we obtain that, again for $s^*$ large enough:
\bee
\left| \int \frac{\e_1^{2q-1}}{\varphi_{4,0}^{2q}(X,Z)}\e \pa_X \e_1\frac{dXdY}{|X|\lY} \right| & \leq & C \left( \| \pa_X \e\|_{L^{\infty}}+\| \frac{\e}{|X|} \|_{L^{\infty}}\right) \int \frac{\e_1^{2q}}{\varphi_{4,0}^{2q}(X,Z)}\frac{dXdY}{|X|\lY} \\
&\leq & \frac{\kappa}{16} \int \frac{\e_1^{2q}}{\varphi_{4,0}^{2q}(X,Z)}\frac{dXdY}{|X|\lY} .
\eee
Next, from the identity
$$
[A,\pa_X]=-\left(\frac 32 +\pa_{\tX} \Psi_1(F_k^{\frac 32}(Z)X)\right)\pa_X =- \frac{\frac 32 +\pa_{\tX} \Psi_1(F_k^{\frac 32}(Z)X)}{\frac 32 X +F_k^{-\frac 32}(Z)\Psi_1(F_k^{\frac 32}(Z)X)}A,
$$
since $F_k$ and $\pa_X \Psi_1$ are uniformly bounded, from \fref{main:eq:estimationA} and \fref{main:bde} we obtain that:
\bee
&& \left| \int \frac{\e_1^{2q-1}}{\varphi_{4,0}^{2q}} \e[A,\pa_X ]\e \frac{dXdY}{|X|\lY} \right| = \left| \int \frac{\e_1^{2q}}{\varphi_{4,0}^{2q}} \e \frac{\frac 32 +F_k(Z)\pa_{\tX} \Psi_1(F_k^{\frac 32}(Z)X)}{\frac 32 X +F_k^{-\frac 12}(Z)\Psi_1(F_k^{\frac 32}(Z)X)} \frac{dXdY}{|X|\lY} \right| \\
&\lesssim & \| \frac{\e}{X} \|_{L^{\infty}} \int \frac{\e_1^{2q}}{\varphi_{4,0}^{2q}} \frac{dXdY}{|X|\lY} \lesssim \frac{\kappa}{16} \int \frac{\e_1^{2q}}{\varphi_{4,0}^{2q}} \frac{dXdY}{|X|\lY}.
\eee
Finally, one computes that:
\bee
&&[A,\pa_{YY}] \e = -2\pa_Y\left(F_k^{-\frac 32}(Z)\Psi_1(F_k^{\frac 32}(Z)X)\right)\pa_{YX}\e-\pa_{YY}\left(F_k^{-\frac 32}(Z)\Psi_1(F_k^{\frac 32}(Z)X)\right)\pa_{X}\e \\
 &=& -2\pa_Y\left(F_k^{-\frac 32}(Z)\Psi_1(F_k^{\frac 32}(Z)X)\right)\pa_{Y}\left(\frac{\e_1}{\frac 32 X +F_k^{-\frac 32}(Z)\Psi_1(F_k^{\frac 32}(Z)X)} \right) -\frac{\pa_{YY}(F_k^{-\frac 32}(Z)\Psi_1(F_k^{\frac 32}(Z)X))}{\frac 32 X +F_k^{-\frac 32}(Z)\Psi_1(F_k^{\frac 32}(Z)X)}\e_1\\
&=& F_1 \e_1+F_2\pa_{Y}\e_1 \\
\eee
where
$$
F_1:= \left( 2\frac{\left(\pa_Y(F_k^{-\frac 32}(Z)\Psi_1(F_k^{\frac 32}(Z)X))\right)^2}{\left( \frac 32 X +F_k^{-\frac 32}(Z) \Psi_1(F_k^{\frac 32}(Z)X)\right)^2}-\frac{\pa_{YY}(F_k^{-\frac 32}(Z)\Psi_1(F_k^{\frac 32}(Z)X))}{\frac 32 X +F_k^{-\frac 32}(Z)\Psi_1(F_k^{\frac 32}(Z)X)}\right)
$$
and
$$
F_2:=-2\frac{\pa_Y(F_k^{-\frac 32}(Z)\Psi_1(F_k^{\frac 32}(Z)X))}{\frac 32 X +F_k^{-\frac 32}(Z)\Psi_1(F_k^{\frac 32}(Z)X))}.
$$
One has that $\pa_Z F_k/F_k$ is bounded, and that $|\Psi_1(\tX)|+|\tX\pa_{\tX} \Psi_1 (\tX)|\lesssim |\tX|$. Therefore,
$$
\left| \pa_Y(F_k^{-\frac 32}(Z)\Psi_1(F_k^{\frac 32}(Z)X))\right|=e^{-\frac{k-1}{2k}s}\left|\pa_Z F_k F_k^{-\frac 52}(-\frac 12 \Psi_1+\frac 32 \tX\pa_{\tX} \Psi_1)(F_k^{\frac 32}(Z)X)\right|\lesssim e^{-\frac{k-1}{2k}s}|X|.
$$
The same computation can be performed for the second term in $F_1$, giving from \fref{main:eq:estimationA}:
$$
\left| \frac{(\pa_Y(F_k^{-\frac 32}(Z)\Psi_1(F_k^{\frac 32}(Z)X)))^2}{\left( \frac 32 X +F_k^{-\frac 32}(Z) \Psi_1(F_k^{\frac 32}(Z)X)\right)^2}\right|+\left| \frac{\pa_{YY}(F_k^{-\frac 32}(Z)\Psi_1(F_k^{\frac 32}(Z)X))}{\frac 32 X +F_k^{-\frac 32}(Z)\Psi_1(F_k^{\frac 32}(Z)X)}\right| \lesssim e^{-\frac{k-1}{k}s}
$$
and hence for $s^*$ large enough:
$$
\left| \int \frac{\e_1^{2q-1}}{\varphi_{4,0}^{2q}} F_1 \e_1 \frac{dXdY}{|X|\lY} \right| \lesssim e^{-\frac{k-1}{k}s}\int \frac{\e_1^{2q}}{\varphi_{4,0}^{2q}} \frac{dXdY}{|X|\lY}\leq \frac{\kappa}{32} \int \frac{\e_1^{2q}}{\varphi_{4,0}^{2q}} \frac{dXdY}{|X|\lY} .
$$
From the above estimate one obtains similarly that:
$$
|F_2|\lesssim e^{-\frac{k-1}{2k}s}, \ \ |\pa_Y F_2|\lesssim e^{-\frac{k-1}{k}s}
$$
so that integrating by parts and using \fref{main:sizephi}, for $s^*$ large enough:
$$
\left| \int \frac{\e_1^{2q-1}}{\varphi_{4,0}^{2q}} F_2 \pa_Y w \frac{dXdY}{|X|\lY} \right|=\frac{1}{2q} \left| \int \e_1^{2q} \pa_Y \left(\frac{F_2}{\varphi_{4,0}^{2q}(X,Z)\lY}\right) \frac{dXdY}{|X|} \right| \lesssim \frac{\kappa}{32}\int \frac{\e_1^{2q}}{\varphi_{4,0}^{2q}} \frac{dXdY}{|X|\lY}.
$$
One has then proven that
$$
\left| \int \frac{\e_1^{2q-1}}{\varphi_{4,0}^{2q}} [A,\pa_{YY}] \e \frac{dXdY}{|X|\lY} \right| \lesssim \frac{\kappa}{16}\int \frac{\e_1^{2q}}{\varphi_{4,0}^{2q}} \frac{dXdY}{|X|\lY}.
$$
From the collection of the above estimates, one infers that:
\bee
&&\frac{d}{ds} \left( \frac{1}{2q} \int_{\mathbb R^2} \frac{\e_1^{2q}}{\varphi_{4,0}^{2q}(X,Z)}\frac{dXdY}{|X| \lY} \right)\\
\non &\leq & -\left(\frac{1}{2}-\frac 34\kappa \right) \int_{\mathbb R^2} \frac{\e_1^{2q}}{\varphi_{4,0}^{2q}(X,Z)}\frac{dXdY}{|X| \lY} +CK_{0,1}^{2q-1}K_{0,0}e^{-2q(\frac 12 -\kappa) s}+CK_{0,1}^{2q-1}s^{\frac{1}{2q}}e^{-2q(\frac 12 -\kappa +\frac{\kappa}{2q})s} \\
\non &\leq & -\left(\frac{1}{2}-\frac 34 \kappa \right) \int_{\mathbb R^2} \frac{\e^{2q}}{\varphi_{4,0}^{2q}(X,Z)}\frac{dXdY}{|X| \lY} +CK_{0,1}^{2q-1}K_{0,0}e^{-2q(\frac 12 -\kappa) s}.
\eee
We reintegrate with time the above estimate, yielding from \fref{main:weighteddecayeinit}:
\begin{align}\label{ChoiceK01}
 \int_{\mathbb R^2} \frac{\e_1^{2q}}{\varphi_{4,0}^{2q}(X,Z)}\frac{dXdY}{|X| \lY} \leq& e^{-2q \left(\frac 12 -\frac 34 \kappa \right)(s-s_0)} \int_{\mathbb R^2} \frac{\e_1(s_0)^{2q}}{\varphi_{4,0}^{2q}(X,Z)}\frac{dXdY}{|X| \lY} +CK_{0,1}^{2q-1}K_{0,0}e^{-2q \left(\frac 12 -\frac 34\kappa \right)s} \int_{s_0}^s e^{\frac{q\kappa}{2}\tilde s}d\tilde s \nonumber\\
\leq &  C(1+K_{0,1}^{2q-1}K_{0,0}) e^{-2q \left(\frac 12 -\kappa \right)s} \leq  \frac{K_{0,1}^{2q}}{2^{2q}} e^{-2q \left(\frac 12 -\kappa \right)s}
\end{align}
where the last inequality holds if $K_{0,1}$ has been chosen large enough depending on $K_{0,0}$.\\

\noindent \textbf{Step 2} \emph{Proof for $\pa_Z$}. We claim that for any $K_{0,0},K_{0,1}>0$, there exists $K_{1,0}^*>0$ such that for any choice of the remaining constants $K_{j_1,j_2}$ and $\tilde K_{j_1,j_2}$ with $K_{1,0}\geq K^*_{1,0}$, there exists $s^*$ such that if $s_0\geq s^*$ then \fref{main:weighteddecayeimproved} holds true for $j_1=1$ and $j_2=0$. To prove this claim, define $\tilde{\e}_1=\pa_Z \e=e^{(k-1)s/(2k)}\pa_Y\e$, which from \fref{main:evolutione} solves:
\bee
0&=& (\tilde{\e}_1)_s+\frac{1}{2k} \tilde{\e}_1+\mathcal L \e-  \pa_{YY}\tilde{\e}_1+\tL \tilde{\e}_1+\pa_Z \Theta\pa_X\e+\pa_{ZX}\Theta \e+\pa_Z (Q-\Theta)\pa_X\e+\pa_{XZ} (Q-\Theta)\e\\
&&+\pa_Z R +\tilde{\e}_1\pa_X \e +\e \pa_X \tilde{\e}_1,
\eee
and hence obeys the energy identity from \fref{id:estimationlineaire}:
\bee
&&\frac{d}{ds} \left( \frac{1}{2q} \int_{\mathbb R^2} \frac{\tilde{\e}_1^{2q}}{\varphi_{4,0}^{2q}(X,Z)}\frac{dXdY}{|X| \lY} \right)\\
\non &\leq & -\left(\frac{1}{2}+\frac{1}{2k}-\frac{\kappa}{2}\right) \int_{\mathbb R^2} \frac{\tilde{\e}_1^{2q}}{\varphi_{4,0}^{2q}(X,Z)}\frac{dXdY}{|X| \lY}-\frac{2q-1}{q^2}\int \frac{|\pa_Y (\tilde{\e}_1^q)|^2}{\varphi_{4,0}^{2q}}\frac{dXdY}{|X|\lY} \\
\non &&+\frac{1}{2q}\int \tilde{\e}_1^{2q}\pa_X \left(\frac{Q-\Theta}{\varphi_{4,0}^{2q}|X|} \right)\frac{dXdY}{\lY}\\
&&- \int \frac{\tilde{\e}_1^{2q-1}}{\varphi_{4,0}^{2q}}\Big(\pa_Z \Theta\pa_X\e+\pa_{ZX}\Theta \e+\pa_Z (Q-\Theta)\pa_X\e+\pa_{XZ} (Q-\Theta)\e\\
&&+\pa_Z R +\tilde{\e}_1\pa_X \e +\e \pa_X \tilde{\e}_1+\pa_X(Q-\Theta)\tilde{\e}_1\Big)\frac{dXdY}{|X|\lY}.
\eee
Using \fref{main:bdpaxQTheta}, and \fref{main:Q-Thetapointwise1} we infer that for $s^*$ large enough:
$$
\left| \frac{1}{2q}\int \tilde{\e}_1^{2q}\pa_X \left(\frac{Q-\Theta}{\varphi_{4,0}^{2q}|X|} \right)\frac{dXdY}{\lY}- \int \frac{\tilde{\e}_1^{2q-1}}{\varphi_{4,0}^{2q}}\pa_X(Q-\Theta)\tilde{\e}_1 \frac{dXdY}{|X|\lY}\right| \leq \frac{\kappa}{12} \int \frac{\tilde{\e}_1^{2q}}{\varphi_{4,0}^{2q}} \frac{dXdY}{|X|\lY}.
$$
Next one computes that
$$
\pa_Z \Theta \pa_X\e=\frac{\pa_Z F_k(Z)}{F_k(Z)}\frac{F_k^{-\frac 12}(Z)(-\frac 12 \Psi_1 +\frac 32\tX \pa_{\tX} \Psi_1)(F_k^{\frac 32}(Z)X)}{\frac 32 X +F_k^{-\frac 12}(Z)(\Psi_1(F_k^{\frac 32}(Z)X))}A\e.
$$
In the above formula, $\pa_Z F_k(Z)/F_k(Z)$ is uniformly bounded, and as $F_k$ is bounded and $|(-1/2 \Psi_1 +3/2\tX \pa_{\tX} \Psi_1)(\tX)|\lesssim |\tX|$ we obtain from \fref{main:eq:estimationA} that:
$$
\left| \frac{\pa_Z F_k(Z)}{F_k(Z)}\frac{F_k^{-\frac 12}(Z)(-\frac 12 \Psi_1 +\frac 32\tX \pa_{\tX} \Psi_1)(F_k^{\frac 32}(Z)X)}{\frac 32 X +F_k^{-\frac 12}(Z)(\Psi_1(F_k^{\frac 32}(Z)X))}\right| \lesssim 1.
$$
From H\"older, \fref{main:weighteddecaye} and \fref{main:weighteddecaye2} we then infer that:
\bee
&& \left| \int \frac{\tilde{\e}_1^{2q-1}}{\varphi_{4,0}^{2q}} \pa_Z \Theta \pa_X\e \frac{dXdY}{|X|\lY}\right| \leq  C\int \frac{|\tilde{\e}_1|^{2q-1}}{\varphi_{4,0}^{2q}} |A\e| \frac{dXdY}{|X|\lY}\\
&\leq & C \left| \int \frac{\tilde{\e}_1^{2q}}{\varphi_{4,0}^{2q}} \frac{dXdY}{|X|\lY}\right|^{\frac{2q-1}{2q}} \left| \int \frac{(A\e)^{2q}}{\varphi_{4,0}^{2q}} \frac{dXdY}{|X|\lY}\right|^{\frac{1}{2q}} \leq  C K_{1,0} K_{0,1}^{2q-1}e^{-2q(\frac 12-\kappa)s}.
\eee
Similarly, since $F_k$, $\pa_ZF_k/F_k$, $\pa_{\tX} \Psi_1$ and $\tX\pa_{\tX}^2\Psi_1$ are uniformly bounded,
$$
\left| \pa_{ZX}\Theta \right|= \left|\frac{\pa_Z F_k(Z)}{F_k(Z)}F_k (\pa_{\tX} \Psi_1+\frac 32\tX\pa_{\tX}^2\Psi_1)(F_k^{\frac 32}(Z)X) \right| \lesssim 1,
$$
and from H\"older, \fref{main:weighteddecaye} and \fref{main:weighteddecaye2} we then infer that:
\bee
&& \left| \int \frac{\tilde{\e}_1^{2q-1}}{\varphi_{4,0}^{2q}} \pa_{ZX} \Theta \e \frac{dXdY}{|X|\lY}\right| \leq C \left| \int \frac{|\tilde{\e}_1|^{2q-1}}{\varphi_{4,0}^{2q}} |\e| \frac{dXdY}{|X|\lY}\right| \\
&\leq& C \left| \int \frac{\tilde{\e}_1^{2q}}{\varphi_{4,0}^{2q}} \frac{dXdY}{|X|\lY}\right|^{\frac{2q-1}{2q}} \left| \int \frac{\e^{2q}}{\varphi_{4,0}^{2q}} \frac{dXdY}{|X|\lY}\right|^{\frac{1}{2q}} \leq  C K_{0,0} K_{1,0}^{2q-1}e^{-2q(\frac 12-\kappa)s}.
\eee
Then, from \fref{main:Q-Thetapointwise1} and \fref{main:eq:estimationA} we infer that:
$$
|\pa_Z (Q-\Theta)\pa_X\e|= \left|\frac{\pa_Z (Q-\Theta)}{\frac 32 X +F_k^{-\frac 12}(Z)\Psi_1(F_k^{\frac 32}(Z)X)}\right| |A\e|\lesssim e^{-\frac{1}{4k}s}|A\e|
$$
and therefore from H\"older, \fref{main:weighteddecaye} and \fref{main:weighteddecaye2}:
\bee
\left| \int \frac{\tilde{\e}_1^{2q-1}}{\varphi_{4,0}^{2q}} \pa_Z (Q-\Theta) \pa_X\e \frac{dXdY}{|X|\lY}\right| & \lesssim & e^{-\frac{1}{4k}s} \left| \int \frac{\tilde{\e}_1^{2q}}{\varphi_{4,0}^{2q}} \frac{dXdY}{|X|\lY}\right|^{\frac{2q-1}{2q}} \left| \int \frac{(A\e)^{2q}}{\varphi_{4,0}^{2q}} \frac{dXdY}{|X|\lY}\right|^{\frac{1}{2q}} \\
&\lesssim & C(K_{1,0},K_{0,1})e^{-2q(\frac 12-\kappa)s}e^{-\frac{1}{4k}s}.
\eee
Using \fref{main:Q-Thetapointwise1} one has that
$$
|\pa_{ZX} (Q-\Theta)| \lesssim e^{-\frac{1}{4k}s}.
$$
Therefore, one infers by H\"older,  \fref{main:weighteddecaye} and \fref{main:weighteddecaye2}:
\bee
\left| \int \frac{\tilde{\e}_1^{2q-1}}{\varphi_{4,0}^{2q}} \pa_{ZX} (Q-\Theta) \e \frac{dXdY}{|X|\lY}\right| & \lesssim & e^{- \frac{1}{4k}s} \left| \int \frac{\tilde{\e}_1^{2q}}{\varphi_{4,0}^{2q}} \frac{dXdY}{|X|\lY}\right|^{\frac{2q-1}{2q}} \left| \int \frac{\e^{2q}}{\varphi_{4,0}^{2q}} \frac{dXdY}{|X|\lY}\right|^{\frac{1}{2q}} \\
&\lesssim & C(K_{0,0},K_{1,0})e^{-2q(\frac 12-\kappa)s}e^{-\frac{1}{4k}s}.
\eee
Next, from H\"older, \fref{main:estimationR} and \fref{main:weighteddecaye}:
$$
\left| \int \frac{\tilde{\e}_1^{2q-1}}{\varphi_{4,0}^{2q}} \pa_Z R \frac{dXdY}{|X|\lY}\right| \leq C  \left| \int \frac{\tilde{\e}_1^{2q}}{\varphi_{4,0}^{2q}} \frac{dXdY}{|X|\lY}\right|^{\frac{2q-1}{2q}} \left| \int \frac{(\pa_Z R)^{2q}}{\varphi_{4,0}^{2q}} \frac{dXdY}{|X|\lY}\right|^{\frac{1}{2q}} \leq C K_{1,0}^{2q-1} s^{\frac{1}{2q}}e^{-2q(\frac 12-\kappa+\frac{\kappa}{2q})s}.
$$
From \fref{main:bde2} one has that for $s^*$ large enough:
$$
\left| \int \frac{\tilde{\e}_1^{2q-1}}{\varphi_{4,0}^{2q}} \tilde{\e}_1\pa_X \e \frac{dXdY}{|X|\lY}\right| \leq \|\pa_X \e \|_{L^{\infty}} \int \frac{\tilde{\e}_1^{2q}}{\varphi_{4,0}^{2q}}\frac{dXdY}{|X|\lY}\leq \frac{\kappa}{12} \int \frac{\tilde{\e}_1^{2q}}{\varphi_{4,0}^{2q}}\frac{dXdY}{|X|\lY}.
$$
We finally perform an integration by parts to obtain:
$$
\int \frac{\tilde{\e}_1^{2q-1}}{\varphi_{4,0}^{2q}(X,Z)} \e \pa_X \tilde{\e}_1 \frac{dXdY}{|X|\lY}=\frac{1}{2q}\int \tilde{\e}_1^{2q} \pa_X \left(\frac{\e}{\varphi_{4,0}^{2q}(X,Z)|X|}\right)\frac{dXdY}{\lY}.
$$
From \fref{main:bde}, \fref{main:bde2}, \fref{main:sizephi} one has for $C$ depending on the bootstrap constants except $s^*$:
\be \lab{bd:pointwisefatigue}
\left|  \pa_X \left(\frac{\e}{\varphi_{4,0}^{2q}(X,Z)|X|}\right) \right|\leq \frac{Ce^{-(\frac 12 -\kappa)s}}{\varphi_{4,0}^{2q}(X,Z)|X|}
\ee
so that for $s^*$ large enough:
$$
\left| \int \frac{\tilde{\e}_1^{2q-1}}{\varphi_{4,0}^{2q}(X,Z)} \e \pa_X \tilde{\e}_1 \frac{dXdY}{|X|\lY} \right| \leq \frac{\kappa}{12}\int \frac{\tilde{\e}_1^{2q}}{\varphi_{4,0}^{2q}}\frac{dXdY}{|X|\lY}.
$$
From the collection of the above estimates, the energy identity becomes:
\bee
&&\frac{d}{ds} \left( \frac{1}{2q} \int_{\mathbb R^2} \frac{\tilde{\e}_1^{2q}}{\varphi_{4,0}^{2q}(X,Z)}\frac{dXdY}{|X| \lY} \right)\\
\non &\leq & -\left(\frac{1}{2}+\frac{1}{2k}-\frac 34 \kappa \right) \int_{\mathbb R^2} \frac{\tilde{\e}_1^{2q}}{\varphi_{4,0}^{2q}(X,Z)}\frac{dXdY}{|X| \lY}  \\
&&+C(K_{0,1}+K_{0,0})K_{1,0}^{2q-1}e^{-2q(\frac 12-\kappa)s}+CK_{1,0}^{2q-1}s^{\frac{1}{2q}}e^{-2q\left(\frac 12 -\kappa +\frac{\kappa}{2q}\right)s}+C(K_{0,0},K_{1,0},K_{0,1})e^{-2q(\frac 12-\kappa)s}e^{-\frac{1}{4k}s}\\
\non &\leq & -\left(\frac{1}{2}-\frac 34 \kappa \right) \int_{\mathbb R^2} \frac{\tilde{\e}_1^{2q}}{\varphi_{4,0}^{2q}(X,Z)}\frac{dXdY}{|X| \lY} +C(K_{0,1}+K_{0,0})K_{1,0}^{2q-1}e^{-2q(\frac 12-\kappa)s}
\eee
for $s^*$ large enough. From the initial size \fref{main:weighteddecayeinit} the above differential inequality yields:
\bee
&& \int_{\mathbb R^2} \frac{\tilde{\e}_1^{2q}(s_1)}{\varphi_{4,0}^{2q}(X,Z)}\frac{dXdY}{|X| \lY}\\
  &\leq& e^{-2q(\frac 12-\frac 34 \kappa)(s_1-s_0)} \int_{\mathbb R^2} \frac{\tilde{\e}_1(s=s_0)^{2q}}{\varphi_{4,0}^{2q}(X,Z)}\frac{dXdY}{|X| \lY} +CK_{1,0}^{2q-1} e^{-2q\left(\frac{1}{2}-\frac 34 \kappa\right)s} \int_{s_0}^s (K_{0,0}+K_{0,1})e^{\frac{q\kappa}{2} \tilde s}d\tilde s \\
&\leq &  C(1+K_{1,0}^{2q-1}(K_{0,0}+K_{0,1})) e^{-2q \left(\frac 12 -\kappa \right)s} \leq  \frac{K_{1,0}^{2q}}{2^{2q}} e^{-2q \left(\frac 12 -\kappa \right)s} \\
\eee
if $K_{1,0}$ has been chosen large enough depending on $K_{0,0}$ and $K_{0,1}$.\\

\noindent \textbf{Step 3} \emph{Proof for $Y\pa_Y$}. We here claim that for any $K_{0,0},K_{0,1}>0$, there exists $\tilde{K}_{1,0}^*>0$ such that for any choice of the remaining constants $K_{j_1,j_2}$ and $\tilde K_{j_1,j_2}$ with $\tilde{K}_{1,0}\geq \tilde{K}^*_{1,0}$, there exists $s^*$ such that if $s_0\geq s^*$ then \fref{main:weighteddecayeimproved2} holds true for $j_1=1$ and $j_2=0$. To prove this claim, Let $\hat{\e}_1=Z\pa_Z\e=Y\pa_Y \e$. From \fref{main:evolutione} one obtain the evolution of $w$:
\bee
0&=& (\hat{\e}_1)_s+\mathcal L \e-  \pa_{YY}\hat{\e}_1+\tL \hat{\e}_1+Z\pa_Z \Theta \pa_X\e+Z\pa_{ZX}\Theta \e\\
&&+2 \pa_{YY}\e+Z\pa_Z(Q-\Theta)\pa_X\e+Z\pa_{ZX}(Q-\Theta)+Z\pa_Z R+\hat{\e}_1\pa_X \e+\e\pa_X \hat{\e}_1.
\eee
Using \fref{id:estimationlineaire} yields the energy identity:
\bee
&&\frac{d}{ds} \left( \frac{1}{2q} \int_{\mathbb R^2} \frac{\hat{\e}_1^{2q}}{\varphi_{4,0}^{2q}(X,Z)}\frac{dXdY}{|X| \lY} \right)\\
\non &\leq & -\left(\frac{1}{2}-\frac{\kappa}{2}\right) \int_{\mathbb R^2} \frac{\hat{\e}_1^{2q}}{\varphi_{4,0}^{2q}(X,Z)}\frac{dXdY}{|X| \lY}-\frac{2q-1}{q^2}\int \frac{|\pa_Y (\hat{\e}_1^q)|^2}{\varphi_{4,0}^{2q}}\frac{dXdY}{|X|\lY} \\
\non &&+\frac{1}{2q}\int \hat{\e}_1^{2q}\pa_X \left(\frac{Q-\Theta}{\varphi_{4,0}^{2q}|X|} \right)\frac{dXdY}{\lY}\\
&&- \int \frac{\hat{\e}_1^{2q-1}}{\varphi_{4,0}^{2q}}\Big(Z\pa_Z \Theta \pa_X\e+Z\pa_{ZX}\Theta \e+2\pa_{YY}\e+Z\pa_Z(Q-\Theta)\pa_X\e+\\
&&Z\pa_{ZX}(Q-\Theta)\e+Z\pa_Z R+\hat{\e}_1\pa_X \e+\e\pa_X \hat{\e}_1+\pa_X(Q-\Theta)\hat{\e}_1 \Big)\frac{dXdY}{|X|\lY}.
\eee
Using \fref{main:bdpaxQTheta}, and \fref{main:Q-Thetapointwise1} we infer that for $s^*$ large enough:
$$
\left| \frac{1}{2q}\int \hat{\e}_1^{2q}\pa_X \left(\frac{Q-\Theta}{\varphi_{4,0}^{2q}|X|} \right)\frac{dXdY}{\lY}- \int \frac{\hat{\e}_1^{2q-1}}{\varphi_{4,0}^{2q}}\pa_X(Q-\Theta)\hat{\e}_1 \frac{dXdY}{|X|\lY}\right| \leq \frac{\kappa}{8}\int \frac{\hat{\e}_1^{2q}}{\varphi_{4,0}^{2q}} \frac{dXdY}{|X|\lY}.
$$
Next,
$$
Z\pa_Z \Theta \pa_X\e=\frac{Z\pa_Z F_k(Z)}{F_k(Z)}\frac{F_k^{-\frac 12}(Z)(-\frac 12 \Psi_1 +\frac 32\tX \pa_{\tX} \Psi_1)(F_k^{\frac 32}(Z)X)}{\frac 32 X +F_k^{-\frac 12}(Z)(\Psi_1(F_k^{\frac 32}(Z)X))}A\e.
$$
In the above formula, $Z\pa_Z F_k(Z)/F_k(Z)$ is uniformly bounded, and as $F_k$ is bounded and $|(-1/2 \Psi_1 +3/2\tX \pa_{\tX} \Psi_1)(\tX)|\lesssim |\tX|$ we obtain from \fref{main:eq:estimationA} that:
$$
\left| \frac{Z\pa_Z F_k(Z)}{F_k(Z)}\frac{F_k^{-\frac 12}(Z)(-\frac 12 \Psi_1 +\frac 32\tX \pa_{\tX} \Psi_1)(F_k^{\frac 32}(Z)X)}{\frac 32 X +F_k^{-\frac 12}(Z)(\Psi_1(F_k^{\frac 32}(Z)X))}\right| \lesssim 1.
$$
From H\"older, \fref{main:weighteddecaye} and \fref{main:weighteddecaye2} we then infer that:
\bee
\left| \int \frac{\hat{\e}_1^{2q-1}}{\varphi_{4,0}^{2q}} Z\pa_Z \Theta \pa_X\e \frac{dXdY}{|X|\lY}\right| &\leq &C \left| \int \frac{\hat{\e}_1^{2q-1}}{\varphi_{4,0}^{2q}} A\e \frac{dXdY}{|X|\lY}\right|\leq C\left| \int \frac{\hat{\e}_1^{2q}}{\varphi_{4,0}^{2q}} \frac{dXdY}{|X|\lY}\right|^{\frac{2q-1}{2q}} \left| \int \frac{(A\e)^{2q}}{\varphi_{4,0}^{2q}} \frac{dXdY}{|X|\lY}\right|^{\frac{1}{2q}} \\
&\leq & C \tilde K_{1,0}^{2q-1}K_{0,1} e^{-2q(\frac 12-\kappa)s}.
\eee
Similarly, since $F_k$, $\pa_ZF_k/F_k$, $\pa_{\tX} \Psi_1$ and $\tX\pa_{\tX}^2\Psi_1$ are uniformly bounded,
$$
\left| Z\pa_{ZX}\Theta \right|=\left|\frac{Z\pa_Z F_k(Z)}{F_k(Z)}F_k (\pa_{\tX} \Psi_1+\frac 32\tX\pa_{\tX}^2\Psi_1)(F_k^{\frac 32}(Z)X) \right| \lesssim 1,
$$
and from H\"older, \fref{main:weighteddecaye} and \fref{main:weighteddecaye2} we then infer that:
\bee
\left| \int \frac{\hat{\e}_1^{2q-1}}{\varphi_{4,0}^{2q}} Z\pa_{ZX} \Theta \e \frac{dXdY}{|X|\lY}\right| &\leq &C \left| \int \frac{\hat{\e}_1^{2q-1}}{\varphi_{4,0}^{2q}} \e \frac{dXdY}{|X|\lY}\right|\leq C\left| \int \frac{\hat{\e}_1^{2q}}{\varphi_{4,0}^{2q}} \frac{dXdY}{|X|\lY}\right|^{\frac{2q-1}{2q}} \left| \int \frac{\e^{2q}}{\varphi_{4,0}^{2q}} \frac{dXdY}{|X|\lY}\right|^{\frac{1}{2q}} \\
&\leq & C \tilde K_{1,0}^{2q-1}K_{0,0} e^{-2q(\frac 12-\kappa)s}.
\eee
We then integrate by parts:
$$
\int \frac{\hat{\e}_1^{2q-1}}{\varphi_{4,0}^{2q}(X,Z)}\pa_{YY}\e\frac{dXdY}{|X|\lY}=-\frac{2q-1}{q} \int \frac{\pa_Y (\hat{\e}_1^q) \hat{\e}_1^{q-1}}{\varphi_{4,0}^{2q}(X,Z)}\pa_Y \e\frac{dXdY}{|X|\lY}-\int \hat{\e}_1^{2q-1}\pa_Y \e \pa_Y \left(\frac{1}{\varphi_{4,0}^{2q}(X,Z)\lY} \right)\frac{dXdY}{|X|}.
$$
For the first term we use the generalised H\"older inequality, \fref{main:weighteddecaye} and \fref{main:weighteddecaye2}:
\bee
&&\left| \int \frac{\pa_Y (\hat{\e}_1^q) \hat{\e}_1^{q-1}}{\varphi_{4,0}^{2q}(X,Z)}\pa_Y \e\frac{dXdY}{|X|\lY} \right| \\
& \leq & e^{-\frac{k-1}{2k}s} \left| \int \frac{|\pa_Y (\hat{\e}_1^q)|^2}{\varphi_{4,0}^{2q}(X,Z)}\frac{dXdY}{|X|\lY} \right|^{\frac 12} \left| \int \frac{\hat{\e}_1^{2q}}{\varphi_{4,0}^{2q}(X,Z)}\frac{dXdY}{|X|\lY} \right|^{\frac{q-1}{2q}} \left| \int \frac{(\pa_Z\e)^{2q}}{\varphi_{4,0}^{2q}(X,Z)}\frac{dXdY}{|X|\lY} \right|^{\frac{1}{2q}} \\
& \leq & \nu \int \frac{|\pa_Y (\hat{\e}_1^q)|^2}{\varphi_{4,0}^{2q}(X,Z)}\frac{dXdY}{|X|\lY} +C(\tilde K_{1,0},K_{1,0})e^{-2q(\frac 12-\kappa)s-\frac{k-1}{2k}s}. 
\eee
for any $\nu$ small enough to be chosen later on. For the second term, from \fref{main:sizephi} we infer that
$$
\left| \pa_Y \left(\frac{1}{\varphi_{4,0}^{2q}(X,Z)\lY} \right)\right| \lesssim \frac{1}{\varphi_{4,0}^{2q}(X,Z)\lY},
$$
and therefore, from H\"older,  \fref{main:weighteddecaye} and \fref{main:weighteddecaye2}:
\bee
&&\left| \int \hat{\e}_1^{2q-1}\pa_Y \e \pa_Y \left(\frac{1}{\varphi_{4,0}^{2q}(X,Z)\lY} \right)\frac{dXdY}{|X|}\right| \\
& \lesssim & e^{-\frac{k-1}{2k}s} \left| \int \frac{\hat{\e}_1^{2q}}{\varphi_{4,0}^{2q}(X,Z)}\frac{dXdY}{|X|\lY}\right|^{\frac{2q-1}{2q}}\left| \int \frac{(\pa_Z \e)^{2q}}{\varphi_{4,0}^{2q}(X,Z)}\frac{dXdY}{|X|\lY}\right|^{\frac{1}{2q}} \leq C(\tilde K_{1,0},K_{1,0})e^{-2q(\frac 12-\kappa)s-\frac{k-1}{2k}s}.
\eee
From \fref{main:eq:estimationA} and \fref{main:Q-Thetapointwise1}:
\bee
&& \left| Z\pa_Z(Q-\Theta)\pa_X\e \right| = \left|\frac{Z\pa_Z (Q-\Theta)}{\frac 32 X +F_k^{-\frac 12}(Z)\Psi_1(F_k^{\frac 32}(Z)X)} \right| |A\e|\\
& \lesssim & \frac{e^{-\frac{1}{4k}s}|Z|(1+|Z|)^{-2k-\frac 12}|X|(1+|\tilde X|)^{\frac 13 -1}}{|X|}|A\e| \lesssim e^{-\frac{1}{4k}s}|A\e|
\eee
and therefore from H\"older, \fref{main:weighteddecaye} and \fref{main:weighteddecaye2}:
\bee
\left| \int \frac{\hat{\e}_1^{2q-1}}{\varphi_{4,0}^{2q}} Z\pa_Z(Q-\Theta)\pa_X\e \frac{dXdY}{|X|\lY}\right| & \lesssim & e^{-\frac{1}{4k}s} \left| \int \frac{\hat{\e}_1^{2q}}{\varphi_{4,0}^{2q}(X,Z)}\frac{dXdY}{|X|\lY}\right|^{\frac{2q-1}{2q}} \left| \int \frac{(A \e)^{2q}}{\varphi_{4,0}^{2q}(X,Z)}\frac{dXdY}{|X|\lY}\right|^{\frac{1}{2q}} \\
&\lesssim & C(K_{0,1},\tilde K_{1,0}) e^{-\frac{1}{4k}s} e^{-2q(\frac 12-\kappa)s}.
\eee
Similarly, from \fref{main:Q-Thetapointwise1}:
$$
|Z\pa_{XZ}(Q-\Theta)|\lesssim e^{-\frac{1}{4k}s}|Z|(1+|Z|)^{-2k-\frac 12}(1+|\tilde X|)^{\frac 13 -1}\lesssim e^{-\frac{1}{4k}s} 
$$
and from H\"older, \fref{main:weighteddecaye} and \fref{main:weighteddecaye2}:
\bee
\left| \int \frac{\hat{\e}_1^{2q-1}}{\varphi_{4,0}^{2q}} Z\pa_{ZX}(Q-\Theta) \e \frac{dXdY}{|X|\lY}\right| & \lesssim & e^{-\frac{1}{4k}s} \left| \int \frac{\hat{\e}_1^{2q}}{\varphi_{4,0}^{2q}(X,Z)}\frac{dXdY}{|X|\lY}\right|^{\frac{2q-1}{2q}} \left| \int \frac{(\e)^{2q}}{\varphi_{4,0}^{2q}(X,Z)}\frac{dXdY}{|X|\lY}\right|^{\frac{1}{2q}} \\
&\lesssim & C(K_{0,0},\tilde K_{1,0}) e^{-\frac{1}{4k}s} e^{-2q(\frac 12-\kappa)s}.
\eee
Next, from H\"older, \fref{main:estimationR} and \fref{main:weighteddecaye2}:
\bee
\left| \int \frac{\hat{\e}_1^{2q-1}}{\varphi_{4,0}^{2q}} Z\pa_{Z}R \frac{dXdY}{|X|\lY}\right| & \leq & \left| \int \frac{\hat{\e}_1^{2q}}{\varphi_{4,0}^{2q}(X,Z)}\frac{dXdY}{|X|\lY}\right|^{\frac{2q-1}{2q}} \left| \int \frac{(Z\pa_Z R)^{2q}}{\varphi_{4,0}^{2q}(X,Z)}\frac{dXdY}{|X|\lY}\right|^{\frac{1}{2q}} \\
&\leq & C\tilde K_{1,0}^{2q-1} s^{\frac{1}{2q}}e^{-2q(\frac 12-\kappa+\frac{\kappa}{2q})s}.
\eee
Performing an integration by parts, and then using \fref{main:bde2} and \fref{bd:pointwisefatigue} we finally obtain:
\bee
&& \left| \int \frac{\hat{\e}_1^{2q-1}}{\varphi_{4,0}^{2q}}(\hat{\e}_1\pa_X \e+\e\pa_X \hat{\e}_1)\frac{dXdY}{|X|\lY}\right| =\left| \int \frac{\hat{\e}_1^{2q}}{\varphi_{4,0}^{2q}} \pa_X \e\frac{dXdY}{|X|\lY}-\frac{1}{2q} \int \hat{\e}_1^{2q} \pa_X \left( \frac{\e}{\varphi_{4,0}^{2q}|X|}\right) \frac{dXdY}{\lY}\right| \\
&\lesssim & \left(\|\pa_X \e\|_{L^{\infty}}+\| X\varphi_{4,0}^{2q}(X,Z)\pa_X \left( \frac{\e}{\varphi_{4,0}^{2q}|X|} \right)\|_{L^{\infty}} \right)\int \frac{\hat{\e}_1^{2q}}{\varphi_{4,0}^{2q}}\frac{dXdY}{|X|\lY}\\
&\leq & \frac{\kappa}{8}\int \frac{\hat{\e}_1^{2q}}{\varphi_{4,0}^{2q}}\frac{dXdY}{|X|\lY}
\eee
for $s^*$ large enough. From the collection of the above estimates, as $(k-1)/(2k)\geq 1/(4k)$ one deduces that:
\bee
&&\frac{d}{ds} \left( \frac{1}{2q} \int_{\mathbb R^2} \frac{\hat{\e}_1^{2q}}{\varphi_{4,0}^{2q}(X,Z)}\frac{dXdY}{|X| \lY} \right)\\
\non &\leq & -\left(\frac{1}{2}-\frac 34 \kappa \right) \int_{\mathbb R^2} \frac{\hat{\e}_1^{2q}}{\varphi_{4,0}^{2q}(X,Z)}\frac{dXdY}{|X| \lY} -\left(\frac{2q-1}{q^2}-\nu\right) \int \frac{|\pa_Y (\hat{\e}_1^q)|^2}{\varphi_{4,0}^{2q}}\frac{dXdY}{|X|\lY}\\
&&+C( K_{0,0}, K_{1,0}, K_{0,1},\tilde  K_{1,0})e^{-2q(\frac 12-\kappa)s}e^{-\frac{1}{4k}s} +C\tilde  K_{1,0}^{2q-1}( K_{0,0}+ K_{0,1})e^{-2q\left(\frac 12 -\kappa \right)s}\\
&&+C\tilde  K_{1,0}^{2q-1}s^{\frac{1}{2q}}e^{-2q\left(\frac 12 -\kappa +\frac{\kappa}{2q}\right)s}\\
\non &\leq & -\left(\frac{1}{2}-\frac{3\kappa}{4} \right) \int_{\mathbb R^2} \frac{\e^{2q}}{\varphi_{4,0}^{2q}(X,Z)}\frac{dXdY}{|X| \lY} +C \tilde  K_{1,0}^{2q-1}( K_{0,0}+ K_{0,1})e^{-2q\left(\frac 12 -\kappa \right)s}\\
\eee
if $\nu$ has been chosen small enough depending only on $q$, and then $s^*$ has been chosen large enough. From the initial size \fref{main:weighteddecayeinit} the above differential inequality yields:
\bee
 \int_{\mathbb R^2} \frac{\hat{\e}_1^{2q}(s_1)}{\varphi_{4,0}^{2q}(X,Z)}\frac{dXdY}{|X| \lY} &\leq& e^{-2q \left(\frac 12 -\frac 34\kappa \right)(s-s_0)} \int_{\mathbb R^2} \frac{\hat{\e}_1(s=s_0)^{2q}}{\varphi_{4,0}^{2q}(X,Z)}\frac{dXdY}{|X| \lY}  \\
&&+C \tilde  K_{1,0}^{2q-1}( K_{0,0}+ K_{0,1}) e^{-2q \left(\frac 12 -\frac 34 \kappa  \right)s} \int_{s_0}^s e^{\frac{q\kappa}{2} \tilde s}d\tilde s \\
&\leq &  C (1+\tilde K_{1,0}^{2q-1}( K_{0,0}+ K_{0,1})) e^{-2q \left(\frac 12 -\kappa \right)s} \leq  \frac{ \tilde K_{1,0}^{2q}}{2^{2q}} e^{-2q \left(\frac 12 -\kappa \right)s} \\
\eee
if $ \tilde K_{1,0}$ has been chosen large enough depending on $K_{0,1}$ and $K_{0,0}$.\\

\noindent \textbf{Step 4} \emph{Proof for higher order derivatives}. The proof for higher order derivatives works the very same way and we leave it to the reader.

\end{proof}

We can now end the proof of Proposition \ref{main:pr:bootstrap}.

\begin{proof}[Proof of Proposition \ref{main:pr:bootstrap}]

We reason by contradiction. Let $v$ be a solution to \fref{main:eqvautosim} with initial value $v(s_0)$ at time $s_0$ that satisfies \fref{main:weighteddecayeinit} when decomposed according to \fref{def:e}. Let $s_1$ denote the supremum of times $\tilde s\geq s_0$ such that $v$ is well defined and that the bounds \fref{main:weighteddecaye} and \fref{main:weighteddecaye2} hold on $[s_0,\tilde s]$. From the initial bounds \fref{main:weighteddecayeinit} and a continuity argument, one has that $s_1>s_0$ is well defined. We then prove Proposition \ref{main:pr:bootstrap} by contradiction and assume that $s_1$ is finite. If it is the case, then the bounds \fref{main:weighteddecaye} and \fref{main:weighteddecaye2} are strict at time $s_1$ from \fref{main:weighteddecayes1}, \fref{main:weighteddecayeimproved} and \fref{main:weighteddecayeimproved2}. Therefore, by a continuity argument there exists $\delta>0$ such that $v$ is well defined and satisfies \fref{main:weighteddecaye} and \fref{main:weighteddecaye2} on $[s_1,s_1+\delta]$, contradicting the definition of $s_1$.

\end{proof}


\section{Analysis of the vertical axis} \lab{sec:NLH}

This section is devoted to the proof of Theorem \ref{th:NLHinstable} and Proposition \ref{pr:LHinstable}. The proof of Theorem \ref{th:NLHinstable} follows and refines the works \cite{BK,HV,MZ}, and differs in particular in the way we deal with the problem outside the origin, see Lemma \ref{NLH:lem:improvedoustide}. For more comparisons with these works, see the comments after Theorem \ref{th:NLHinstable}. The proof of Proposition \ref{pr:LHinstable} then uses a very similar analytical framework.


\subsection{Flat blow-up for the semi-linear heat equation}

We prove in this subsection the result in Theorem \ref{th:NLHinstable} concerning the solution $\xi$ to \fref{eq:NLH}. The strategy is the following. We construct an approximate blow-up profile in self-similar variables and show the existence of a true solution staying in its neighbourhood via a bootstrap argument. This existence result relies on the control of the difference of the two functions via a spectral decomposition at the origin and energy estimates far away, showing the existence of a finite number of instabilities only allowing for the use of a topological argument to control them.\\

\noindent The unstable blow-ups are related to unstable analytic backward self-similar solutions of the quadratic equation
\be \lab{eq:quadra}
\xi_t-\xi^2=0.
\ee
Their properties are the following.

\begin{proposition}[unstable self-similar blow-ups for the quadratic equation] \lab{pr:Fk}

For $k\in \mathbb N$, $k\geq 2$, the functions $F_k(Z):=(1+Z^{2k})^{-1}$ are such that
$$
\xi(t,y)=\frac{1}{T-t}F_k \left(\frac{ay}{(T-t)^{\frac{1}{2k}}} \right)
$$
is a solution of \fref{eq:quadra} for any $T\in \mathbb R$ and $a>0$. For any $a>0$, $Z\mapsto F_k(aZ)$ is a solution of the stationary self-similar equation
\be \lab{NLH:eq:Fk}
F_k+\frac{1}{2k}Z\pa_Z F_k-F_k^2=0.
\ee
The linearised transport operator $H_Z := 1+\frac{1}{2k}Z\pa_Z -2F_k(aZ)$ acting on $C^{\infty}(\mathbb R)$ has the point spectrum
$$
\Upsilon (H_Z)= \left\{ \frac{\ell-2k}{2k}, \ \ \ell \in \mathbb N \right\} .
$$
The associated eigenfunctions are
\be \lab{eq:def phiXell}
H_Z\phi_{Z,\ell}=\frac{\ell-2k}{2k} \phi_{Z,\ell}, \ \ \phi_{Z,\ell}=\frac{Z^{\ell}}{(1+(aZ)^{2k})^2} .
\ee
Two of them\footnote{In fact there is a third one due to translation invariance which is absent here due to even symmetry.} are linked to the invariances of the flow:
$$
\phi_{0}=F_k(aZ)+\frac{1}{2k}Z\pa_ZF_k(aZ)=\frac{\pa}{\pa_\lb}\left(\lb F_k(\lb^{\frac{1}{2k}}bZ) \right)_{|\lb=1}, \ \ \phi_{2k}=Z\pa_ZF_k(aZ)=\frac{\pa}{\pa_{\tilde a}}\left( F_k(\tilde a aZ) \right)_{|\tilde b=1}.
$$

\end{proposition}

\begin{proof}

The proof is made of straightforward computations that we do not write here.

\end{proof}

We now introduce for $\xi$ a solution to \fref{eq:NLH} for $a>0$ and $T>0$ the self-similar variables following \cite{GK} 
\be \lab{NLH:id:def a}
Y:=\frac{y}{\sqrt{T-t}}, \ \ s:=-\log (T-t), \ \ Z:=\frac{aY}{e^{\frac{k-1}{2k}s}}, \ \ f(s,Y)=(T-t)\xi (t,y),
\ee
to zoom at the blow-up location, and $f$ solves the first equation in \fref{eq:f}. The function that we want to construct here, from \fref{id:decomposition NLH} and \fref{bd:remainder NLH}, should converge to $1$ in compact sets of the variable $Y$. Therefore close to the origin the linearised operator is
$$
H_\rho:=-1 +\frac Y2 \pa_Y -\pa_{YY}.
$$
Its spectral structure is well-known on the weighted $L^2$-based Sobolev spaces
$$
H^k_\rho:=\left\{f\in H^k_{\text{loc}}(\mathbb R), \ \ \sum_{k'=0}^k \int_{\mathbb R} |\pa_Y^k f|^2e^{-\frac{Y^2}{4}}dY<+\infty \right\}
$$
with norm and scalar product
\be \lab{eq:def L2rho}
\| f\|_{H^k_\rho}^2:= \sum_{k'=0}^k \int_{\mathbb R} |\pa_Y^k f|^2e^{-\frac{Y^2}{4}}dY, \ \ \langle f,g \rangle_\rho:=\int_{\mathbb R} fge^{-\frac{Y^2}{4}}dY, \ \ \rho(Y):=e^{-\frac{Y^2}{4}}.
\ee

\begin{proposition}[Linear structure at the origin (see e.g. \cite{MZ})] \lab{pr:Lrho}

The operator $H_{\rho}$ is essentially self-adjoint on $C^2_0(\mathbb R)\subset L^2(\rho)$ with compact resolvant. The space $H^2_\rho$ is included in the domain of its unique self-adjoint extension. Its spectrum is
$$
\Upsilon (H_\rho) = \left\{\frac{\ell-2}{2}, \ \ell \in \mathbb N \right\}.
$$
The eigenvalues are all simple and the associated orthonormal basis of eigenfunctions is given by Hermite polynomials:
\be \lab{def:hell}
h_{\ell}(Y):= c_{\ell} \sum_{n=0}^{\left[\frac{\ell}{2}\right]} \frac{\ell!}{n!(\ell-2n)!}(-1)^n Y^{\ell-2n}.
\ee
$h_\ell$ is orthogonal to any polynomial of degree lower than $\ell-1$ for the $L^2_\rho$ scalar product.

\end{proposition}

To construct the blow-up solution of Theorem \ref{th:NLHinstable}, we will use an approximate solution to \fref{eq:f} close to $F_k(bZ)$ that is adapted to the linearised dynamics both at the origin and far away.

\begin{proposition}[Approximate blow-up profile] \lab{NLH:pr:profile}

Let $k\in \mathbb N$ with $k\geq 2$. There exists universal constants $(\bar c_{2\ell})_{0\leq \ell\leq k-1}$ with: 
$$
\bar c _{2k-2}:=-2k(2k-1), \ \ \bar c_{2\ell}:=-\frac{(2\ell+2) (2\ell+1)}{k-\ell}\bar c_{2\ell+2} \ \ \text{for} \ \ell=0,...,k-2,
$$
such that for any $0\leq s_0< s_1$ and $a\in C^1([s_0,s_1],(\frac12,\frac32))$, the profile
\be \lab{def:Fb}
F[a](s,Y):= F_k(Z)+\sum_{\ell=0}^{k-1}\bar c_{2\ell}(ae^{-\frac{k-1}{2k}s})^{2k-2\ell} \phi_{2\ell}(Z)
\ee
satisfies the following identity:
$$
\pa_s F [a]+F [a]+\frac Y2 \pa_Y F [a] -F^2 [a]-\pa_{YY}F[a]= a_s \pa_a F[a]+\Psi,
$$
with the error $\Psi$ satisfying for any $j\geq 0$:
\be \lab{bd:psirho}
\| \pa_Z^j\Psi \|_{L^2_\rho}\lesssim e^{-\left(2(k-1)-j\frac{k-1}{2k}\right)s} \ \ \text{for} \ \ j=0,...,2k, \ \ \| \pa_Z^j\Psi \|_{L^2_\rho}\lesssim e^{-\frac{k-1}{k}s} \ \ \text{for} \ \ j\geq 2k+1,
\ee
and for $|Y|\geq 1$:
\be \lab{bd:Psi2}
|(Z\pa_Z)^{j}\Psi |\lesssim e^{-\frac{k-1}{k}s} |Z|^{4k-2}(1+|Z|)^{-6k}.
\ee
and
\be \lab{bd:Psi3}
| \pa_Z^j \Psi |\lesssim e^{-\frac{k-1}{k}s}|Z|^{4k-2-j}(1+|Z|)^{-6k} \ \ \text{for} \ \ j=0,...,2k, \ \ |\pa_Z^j \Psi|\lesssim e^{-\frac{k-1}{k}s}(1+|Z|)^{-2k-2-j} \ \ \text{for} \ \ j\geq 2k+1,
\ee
The variation with respect to $a$ enjoys the following properties:
\be \lab{id:pabFprojection}
\langle \pa_a F[a],h_{2k} \rangle_\rho =ca^{2k-1}e^{-(k-1)s}\left(1+O(e^{-(k-1)s})\right), \ \ c\neq 0,
\ee
and for $j\in \mathbb N$:
\be \lab{bd:pabFrho}
\| \pa_Z^j\pa_aF[a]\|_{L^2_\rho}\lesssim e^{-\left((k-1)-j\frac{k-1}{2k}\right)s} \ \ \text{for} \ \ j=0,...,2k, \ \ \| \pa_Z^j\pa_a F[a] \|_{L^2_\rho}\lesssim 1 \ \ \text{for} \ \ j\geq 2k+1,
\ee
and for $|Y|\geq 1$:
\be \lab{bd:pabF2}
| (Z\pa_Z)^j\pa_a F[a]|\lesssim |Z|^{2k} (1+|Z|)^{-4k} .
\ee
and
\be \lab{bd:pabF3}
| \pa_Z^j \pa_a F[a] |\lesssim |Z|^{2k-j}(1+|Z|)^{-4k} \ \ \text{for} \ \ j=0,...,2k, \ \ |\pa_Z^j \pa_a F[a]|\lesssim (1+|Z|)^{-2k-j} \ \ \text{for} \ \ j\geq 2k+1,
\ee
In all the above estimates, the implicit constants in the $\lesssim$ notation depend solely on $k$ and $j$.
\end{proposition}

\begin{proof}

This is a brute force computation.\\

\noindent \textbf{Step 1:} \emph{Estimates for $\Psi$}. We first decompose from \fref{NLH:eq:Fk} and \fref{eq:def phiXell}:
\bee
&& \Psi =\pa_s F [a]+F[a]+\frac Y2 \pa_Y F[a] -F^2[a]-\pa_{YY}F[a]-a_s\pa_a F[a] \\
&=& F_k(Z)+\frac{1}{2k}Z\pa_Z F_k(Z)-F_k^2(Z) -\sum_{\ell=0}^{k-1} \bar c_{2\ell}\frac{k-1}{2k}(2k-2\ell)(ae^{-\frac{k-1}{2k}s})^{2k-2\ell}\phi_{2\ell}(Z) \\
&&+\sum_{\ell=0}^{k-1} \bar c_{2\ell} (ae^{-\frac{k-1}{2k}s})^{2k-2\ell}(H_Z \phi_{2\ell})(Z)-\sum_{\ell=0}^{k-1} \bar c_{2\ell} (ae^{-\frac{k-1}{2k}s})^{2k-2\ell+2}(\pa_{ZZ} \phi_{2\ell})(Z)-(ae^{-\frac{k-1}{2k}s})^2\pa_{ZZ}F_k(Z) \\
&&-\left(\sum_{\ell=0}^{k-1}\bar c_{2\ell}(ae^{-\frac{k-1}{2k}s})^{2k-2\ell} \phi_{2\ell}(Z)\right)^2\\
&=& -\sum_{\ell=0}^{k-1} \bar c_{2\ell}(k-\ell)(ae^{-\frac{k-1}{2k}s})^{2k-2\ell}\phi_{2\ell}(Z) -\sum_{\ell=0}^{k-1} \bar c_{2\ell} (ae^{-\frac{k-1}{2k}s})^{2k-2\ell+2}(\pa_{ZZ} \phi_{2\ell})(Z)-(ae^{-\frac{k-1}{2k}s})^2\pa_{ZZ}F_k(Z)\\
&&-\left(\sum_{\ell=0}^{k-1}\bar c_{2\ell}(ae^{-\frac{k-1}{2k}s})^{2k-2\ell} \phi_{2\ell}(Z)\right)^2\\
&=&-\sum_{\ell=0}^{k-2} (ae^{-\frac{k-1}{2k}s})^{2k-2\ell}\left( \bar c_{2\ell}(k-\ell)\phi_{2\ell}(Z)+\bar c_{2\ell+2}\pa_{ZZ}\phi_{2\ell+2}(Z)  \right)\\
&&- (ae^{-\frac{k-1}{2k}s})^{2}(\bar c_{2k-2}\phi_{2k-2}-\pa_{ZZ}F_k)-\bar c _0 (ae^{-\frac{k-1}{2k}s})^{2k+2} \pa_{ZZ}\phi_0(Z)-\left(\sum_{\ell=0}^{k-1}\bar c_{2\ell}(ae^{-\frac{k-1}{2k}s})^{2k-2\ell} \phi_{2\ell}(Z)\right)^2\\\\
&=&-\sum_{\ell=0}^{k-2} (ae^{-\frac{k-1}{2k}s})^{2k-2\ell}\bar c_{2\ell+2} \left(\pa_{ZZ}\phi_{2\ell+2}(Z)- (2\ell+2)(2\ell+1) \phi_{2\ell}(Z)\right)\\
&&+ (ae^{-\frac{k-1}{2k}s})^{2}( 2k(2k-1)\phi_{2k-2}+\pa_{ZZ}F_k)-\bar c _0 (ae^{-\frac{k-1}{2k}s})^{2k+2} \pa_{ZZ}\phi_0(Z)-\left(\sum_{\ell=0}^{k-1}\bar c_{2\ell}(ae^{-\frac{k-1}{2k}s})^{2k-2\ell} \phi_{2\ell}(Z)\right)^2\\
\eee
As
$$
\pa_{ZZ}\phi_{2\ell}=\frac{2\ell (2\ell-1)Z^{2\ell-2}}{(1+Z^{2k})^2}-\frac{4k(4\ell+2k-1) Z^{2\ell+2k-2}}{(1+Z^{2k})^3}+\frac{24 k^2 Z^{2\ell+4k-2}}{(1+Z^{2k})^4}
$$
one deduces that for $\ell=0,...,k-2$:
\bee
&& \left| (ae^{-\frac{k-1}{2k}s})^{2k-2\ell} \left(\pa_{ZZ}\phi_{2\ell+2}(Z)- (2\ell+2)(2\ell+1) \phi_{2\ell}(Z)\right)\right| \\
&=&\left|(ae^{-\frac{k-1}{2k}s})^{2k-2\ell} \left(-\frac{4k(4\ell+2k+3) kZ^{2\ell+2k}}{(1+Z^{2k})^3}+\frac{24 k^2 Z^{2\ell+4k}}{(1+Z^{2k})^4}\right)\right|\\
&\lesssim & (e^{-\frac{k-1}{2k}s})^{2k-2\ell}  Z^{2\ell+2k} (1+|Z|)^{-6k}.
\eee
Similarly, as
$$
\pa_{ZZ}F_k=-\frac{2k(2k-1)Z^{2k-2}}{(1+Z^{2k})^2}+\frac{8k^2Z^{4k-2}}{(1+Z^{2k})^3}
$$
one deduces that
\bee
\left| (ae^{-\frac{k-1}{2k}s})^{2}( 2k(2k-1)\phi_{2k-2}+\pa_{ZZ}F_k)\right|&=& \left|(ae^{-\frac{k-1}{2k}s})^{2}\frac{8k^2Z^{4k-2}}{(1+Z^{2k})^3}\right| \\
&\lesssim &  (e^{-\frac{k-1}{2k}s})^{2}  Z^{4k-2} (1+|Z|)^{-6k}.
\eee
Eventually,
\bee
\left| -(ae^{-\frac{k-1}{2k}s})^{2k+2} \pa_{ZZ}\phi_0(Z) \right| &=& \left| (ae^{-\frac{k-1}{2k}s})^{2k+2}\left( \frac{4k(2k-1) Z^{2k-2}}{(1+Z^{2k})^3}+\frac{24 k^2 Z^{4k-2}}{(1+Z^{2k})^4}\right)\right|  \\
&\lesssim &  (e^{-\frac{k-1}{2k}s})^{2k+2} Z^{2k-2}(1+|Z|)^{-6k}
\eee
and
$$
\left(\sum_{\ell=0}^{k-1}\bar c_{2\ell}(ae^{-\frac{k-1}{2k}s})^{2k-2\ell} \phi_{2\ell}(Z)\right)^2\lesssim \sum_{\ell=0}^{k-1}(e^{-\frac{k-1}{2k}s})^{4k-4\ell} Z^{4\ell}(1+|Z|)^{-8k}.
$$
From the above identities one deduces that:
\be \lab{eq:boundpointwisePsi}
\left|\Psi \right| \lesssim \sum_{\ell=1}^{k+1}  (e^{-\frac{k-1}{2k}s})^{2\ell} |Z|^{4k-2\ell}(1+|Z|)^{-6k}+\sum_{\ell=0}^{k-1}(e^{-\frac{k-1}{2k}s})^{4k-4\ell} |Z|^{4\ell}(1+|Z|)^{-8k}.
\ee
For $\ell=1,...,k+1$ one computes that
$$
\int  \left((e^{-\frac{k-1}{2k}s})^{2\ell} Z^{4k-2\ell}(1+|Z|)^{-6k}\right)^2e^{-\frac{Y^2}{4}}dY \lesssim   \int  (e^{-\frac{k-1}{2k}s})^{4\ell} (Ye^{-\frac{k-1}{2k}s})^{8k-4\ell}e^{-\frac{Y^2}{4}}dY  \lesssim e^{-4(k-1)s}
$$
and similarly for $\ell=0,...,k-1$:
$$
\int  \left((e^{-\frac{k-1}{2k}s})^{4k-4\ell} Z^{4\ell}(1+|Z|)^{-8k}\right)^2e^{-\frac{Y^2}{4}}dY \lesssim e^{-4(k-1)s}.
$$
The above two bounds imply \fref{bd:psirho} for $j=0$. For $|Y|\geq 1$ one has that for $\ell=1,...,k+1$:
$$
(e^{-\frac{k-1}{2k}s})^{2\ell} Z^{4k-2\ell}(1+|Z|)^{-6k}\lesssim e^{-\frac{k-1}{k}s} (aYe^{-\frac{k-1}{2k}s})^{2\ell-2} Z^{4k-2\ell}(1+|Z|)^{-6k}=e^{-\frac{k-1}{k}s} Z^{4k-2}(1+|Z|)^{-6k}
$$
and similarly for $\ell=0,...,k-1$:
$$
(e^{-\frac{k-1}{2k}s})^{4k-4\ell} Z^{4\ell}(1+|Z|)^{-8k}\lesssim e^{-\frac{k-1}{k}s} Z^{4k-2}(1+|Z|)^{-8k}.
$$
The above two bounds yield \fref{bd:Psi2} for $j=0$. We claim the the other bounds for $|\Psi|$ can be proved the same way as \fref{eq:boundpointwisePsi} naturally extends to derivatives.\\

\noindent \textbf{Step 2} \emph{Estimate for $\pa_a F$}. We compute two identities:
\bea
\nonumber a\pa_a F[a]&=& -2k \phi_{2k}(Z)+\sum_{\ell=0}^{k-1}\bar c_{2\ell}(ae^{-\frac{k-1}{2k}s})^{2k-2\ell} \left((2k-2\ell)\phi_{2\ell}(Z)+ Z\pa_Z\phi_{2\ell}(Z)\right)\\
\label{id:apaF1}&=& -2k \frac{Z^{2k}}{(1+Z^{2k})^{2}}+\sum_{\ell=0}^{k-1} \bar c_{2\ell} (ae^{-\frac{k-1}{2k}s})^{2k-2\ell}\left(\frac{2kZ^{2\ell}}{(1+Z^{2k})^2}-\frac{4kZ^{2\ell+2k}}{(1+Z^{2k})^3} \right) \\
\label{id:apaF2}&=&-2k Z^{2k}+\sum_{\ell=0}^{k-1} \bar c_{2\ell} (ae^{-\frac{k-1}{2k}s})^{2k-2\ell}2k Z^{2\ell}  \\
\nonumber&&+2k\frac{2Z^{4k}-Z^{6k}}{(1+Z^{2k})^2}-\sum_{\ell=0}^{k-1} \bar c_{2\ell} (ae^{-\frac{k-1}{2k}s})^{2k-2\ell}\left(\frac{2kZ^{2\ell+2k}}{(1+Z^{2k})^2}+\frac{4kZ^{2\ell+4k}}{(1+Z^{2k})^3} \right)
\eea
From the first identity \fref{id:apaF1} we infer that:
\be \lab{bd:pointwisepabF}
\left| \pa_a F[a]\right| \lesssim \sum_{\ell=0}^{k} (e^{-\frac{k-1}{2k}s})^{2k-2\ell}Z^{2\ell}(1+Z^{2k})^{-2}.
\ee
This implies that
$$
\left| \pa_a F[a]\right|\lesssim e^{-(k-1)s}\sum_{\ell=0}^{k} Y^{2\ell}
$$
which yields \fref{bd:pabFrho} for $j=0$. For $|Y|\geq 1$ one therefore estimates:
$$
\left| \pa_a F[a]\right| \lesssim \sum_{\ell=0}^{k} (bYe^{-\frac{k-1}{2k}s})^{2k-2\ell}Z^{2\ell}(1+Z^{2k})^{-2} \lesssim Z^{2k}(1+Z^{2k})^{-2}
$$
which proves \fref{bd:pabF2} for $j=0$. Again, we claim that the other bounds concerning $\pa_a F[a]$ can be proved along the same lines since \fref{bd:pointwisepabF} naturally extends to derivatives. We now use the second identity \fref{id:apaF2} and compute since $h_{2k}(Y)$ is orthogonal to any polynomial of degree less or equal than $2k-1$ and $Z=e^{-(k-1)s/(2k)}Y$:
$$
\langle -2k Z^{2k}+\sum_{\ell=0}^{k-1} \bar c_{2\ell} (ae^{-\frac{k-1}{2k}s})^{2k-2\ell}2k Z^{2\ell},h_{2k} \rangle_{\rho}=\langle -2k Z^{2k},h_{2k} \rangle_{\rho}= b^{2k}ce^{-(k-1)s}, \ \ c\neq 0.
$$
We then get the desired nondegeneracy \fref{id:pabFprojection} since
$$
\| 2k\frac{2Z^{4k}-Z^{6k}}{(1+Z^{2k})^2}-\sum_{\ell=0}^{k-1} \bar c_{2\ell} (ae^{-\frac{k-1}{2k}s})^{2k-2\ell}\left(\frac{2kZ^{2\ell+2k}}{(1+Z^{2k})^2}+\frac{4kZ^{2\ell+4k}}{(1+Z^{2k})^3} \right)\|_{L^2_{\rho}}\lesssim e^{-2(k-1)s}.
$$

\end{proof}

We now fix $k\geq 1 $ for the remaining part of the Subsection, and prove Theorem \ref{th:NLHinstable} by showing that there exists a global solution to \fref{eq:f} converging to $F_k(Z)$ as $s\rightarrow +\infty$. To this end, we perform a bootstrap argument near the approximate profile $F[a]$. We decompose the solution in self-similar variables according to (using \fref{bd:pabF3}):
\be \lab{id:decomposition e}
f=F[a]+\e, \ \ \e=\sum_{\ell=0}^{k-1} c_{2\ell} h_{2\ell}(Y)+\tilde \e, \ \ \tilde \e \perp_\rho h_{2\ell} \ \ \text{for} \ \ \ell=0,...,k,
\ee
where the $\perp_\rho $ is the orthogonality with respect to the $L^2_{\rho}$ scalar product. Such a decomposition in the vicinity of $F[a]$ is a consequence of \fref{id:pabFprojection} and of the implicit function Theorem.

\begin{definition}[Trapped solutions] \label{def:trappedNLH}

Fix $J\in \mathbb N$, and let constants $\tilde K\geq K_J\geq ...\geq K_0>0$. We say a solution to (NLH) is trapped on $[s_0,s_1]$ if it is even, can be decomposed according to\footnote{Note that for $s_0$ large enough, due to the bounds for trapped solutions, this decomposition is unique and the associated parameters are $C^{\infty}$ from a standard application of the implicit function Theorem and of parabolic regularity} \fref{id:decomposition e} and if it satisfies for $j=0,...,J$:
\be \lab{bd:eweightlossy}
 \int_{|Y|\geq 1} \frac{|(Y\pa_Y)^j \e|^2}{\phi_{2k+1}^2(Z)}\frac{dY}{|Y|}\leq K_je^{-\frac{3}{4k}s}, 
\ee
\be \lab{bd:eweightsharp}
 \int_{|Y|\geq 1} \frac{|(Y\pa_Y)^j \e|^2}{\phi_{2k+1/2}^2(Z)}\frac{dY}{|Y|}\leq K_{j}e^{-\frac{1}{2k}s},
\ee
\be \lab{bd:eweightjgeq2k+1}
\int_{|Y|\geq 1} \frac{|\pa_Z^j \e|^2}{\phi_{0}^2(Z)}\frac{dY}{|Y|}\leq K_j  e^{-\frac{1}{2k}s} \ \ \text{for} \ j\geq 2k+1,
\ee
\be \lab{bd:eloc}
\| \pa_Z^j \tilde \e \|_{L^2_{\rho}}\leq \sqrt{K_j}e^{-\left(k-\frac12 +\frac{j}{2k}-\frac j2 \right)s} \ \ \text{for} \ \ j=0,...,2k, \ \ \| \pa_Z^j \tilde \e \|_{L^2_{\rho}}\leq \sqrt{K_j}e^{-\frac{1}{2k} s} \ \ \text{for} \ j\geq 2k+1,
\ee
\be \lab{bd:cell bootstrap}
\left( \sum_{\ell=0}^{k-1} |c_{2\ell} |^2\right)^{\frac 12}\leq \sqrt{\tilde K} e^{-(k-\frac 12+\frac{1}{4k})s}
\ee

\end{definition}

We will show that one solutions remained trapped forever, and our argument necessitates to adjust the initial values of the projection of the error on the modes $(h_{2\ell})_{0\leq \ell\leq k-1}$. These functions are unbounded but the following technical Lemma provides a harmless localisation.

\begin{lemma} \label{lem:implicit}
There exists $M^*,C>0$ such that for all, $M\geq M^*$, $\tilde K>0$, there exists $s^*\geq 0$ such that for $s\geq s^*$ the following holds true introducing $\chi_M(Y):=\chi(Y/M)$. For any $\frac 12<\bar a<\frac 32$, $(\bar{c}_{2\ell})_{0\leq \ell \leq k-1}\in B(\tilde Ke^{-(k-\frac12+\frac{1}{4k})s})$ (the Euclidean ball) there exist unique parameters $(\bar{c}_{2\ell})_{0\leq \ell \leq k-1}\in B(2\tilde Ke^{-(k-\frac12+\frac{1}{4k})s})$ and $a $ with $|a-\bar a|\leq C\tilde Ke^{-(\frac 12+\frac{1}{2k})s}$ such that
$$
F[\bar a]+\chi_M \sum_{\ell=0}^{k-1} \bar c_{2\ell} h_{2\ell}(Y)=F[a]+ \sum_{\ell=0}^{k-1} c_{2\ell} h_{2\ell}(Y) +\tilde \e
$$
where $\tilde \e$ satisfies the orthogonality conditions in \fref{id:decomposition e}. Moreover, with $\bar a$ fixed, the mapping that to $(\bar c_{2\ell})_{0\leq \ell\leq k-1}$ associates $(c_{2\ell})_{0\leq \ell\leq k-1}$ is a smooth diffeomorphism and the preimage of $B(\tilde Ke^{-(k-\frac12+\frac{1}{4k})s})$ is contained in $B(2\tilde Ke^{-(k-\frac12+\frac{1}{4k})s})$.
\end{lemma}

\begin{proof}
We fix $\bar a\in [\frac 12,\frac32]$. We write $a=\bar a +\tilde a e^{(k-1)s}$ and define the mappings $\omega:\mathbb R^{k}\rightarrow \mathbb R^{k+1}$ and $\tilde \omega:\mathbb R^{k+1}\rightarrow \mathbb R^{k+1}$ by:
$$
\omega \left(\bar c_0,...,\bar c_{2k-2}\right)=( \langle \chi_M\sum_{\ell=0}^{k-1} \bar c_{2\ell} h_{2\ell},h_{2n}\rangle_\rho )_{0\leq n\leq k},
$$
$$
\tilde \omega (c_0,...,c_{2k-2},\tilde a)=( \langle F[a]-F[\bar a]+\sum_{\ell=0}^{k-1} \bar c_{2\ell} h_{2\ell},h_{2n}\rangle_\rho )_{0\leq n\leq k}.
$$
One has the identities $\omega \left(\bar c_0,...,\bar c_{2k-2}\right)=\left(\bar c_0,...,\bar c_{2k-2},0\right)+o_{M\rightarrow \infty}(|(\bar c_0,...,\bar c_{2k-2})|)$ as the $h_{\ell}$ form an orthonormal basis of $L^2_\rho$ of polynomials, and since the weight $\rho$ decreases exponentially fast. One similarly has that $\tilde \omega (c_0,...,c_{2k-2},\tilde a)=(c_0,...,c_{2k-2},ca^{2k-1}\tilde a)+O(\tilde a^2)$ where $c>0$ is defined in \fref{id:pabFprojection}. The Lemma is obtained choosing $(c_0,...,c_{2k-2},(a-\bar a)e^{-(k-1)s})=\tilde \omega^{-1}\circ \omega (\bar c_0,...,\bar c_{2k-2})$. The inversion of $\tilde \omega$, and the desired estimates, are consequences of the aforementioned identities and of the implicit function Theorem. 
\end{proof}

\begin{proposition} \lab{pr:bootstrap}

Fix $J\geq 1$ and $M>0$ so that Lemma \ref{lem:implicit} holds true. There exists $\tilde K\geq K_J\geq ...\geq K_0>0$ and $s^*\geq 0$ such that the following holds for any $s_0\geq s^*$. For any $\bar \e_0=\bar \e(s_0)$ that is even, satisfies the orthogonality conditions in \fref{id:decomposition e} and
\be \lab{bd:decayextinit} 
\sum_{j=0}^J  \int_{|Y|\geq 1} \frac{|(Y\pa_Y)^j \bar \e_0|^2}{\phi_{2k+1}^2(Z)}\frac{dY}{|Y|}\leq e^{-\frac{3}{4k}s_0}, \ \ \sum_{j=0}^J \int_{|Y|\geq 1} \frac{|(Y\pa_Y)^j \bar \e_0|^2}{\phi_{2k+1/2}^2(Z)}\frac{dY}{|Y|}\leq e^{-\frac{1}{2k}s_0},
\ee
\be \lab{bd:decayextinit2}
\sum_{j=2k+1}^J \int_{|Y|\geq 1} \frac{|\pa_Z^j \bar \e_0|^2}{\phi_{0}^2(Z)}\frac{dY}{|Y|}\leq e^{-\frac{1}{2k}s_0},
\ee
\be \lab{bd:eL2rhoinit}
\| \pa_Z^j \bar \e_0 \|_{L^2_{\rho}}\leq e^{-\left(k-\frac12 +\frac{j}{2k}-\frac j2 \right)s_0} \ \ \text{for} \ \ j=0,...,2k, \ \ \| \pa_Z^j \bar \e_0 \|_{L^2_{\rho}}\leq e^{-\frac{1}{2k} s_0} \ \ \text{for} \ j\geq 2k+1,
\ee
and $3/4\leq a(s_0)\leq 5/4$, there exist $\bar c_0(s_0),...,\bar c_{2\ell-2}(s_0)$ with
\be \lab{NLH:eq:bootstrap instable init}
\left(\sum_{\ell=0}^{k-1} |\bar c_{2\ell}(s_0)|^2\right)^{\frac 12}\leq 2\sqrt{\tilde K} e^{-(k-\frac 12+\frac{1}{4k})s_0}
\ee
such that the solution $f$ to \fref{eq:f} with initial datum
\be \label{bd:initialdecompoNLH}
f(s_0)=F[a](s_0)+\chi_M(Y)\sum_{\ell=0}^{k-1} \bar c_{2\ell}(s_0) h_{\ell}(Y)+\bar \e_0,
\ee
is a global solution to \fref{eq:f}, which is trapped on $[s_0,+\infty)$. Moreover, there exists an asymptotic limit $a^*\in \mathbb R$, with $1/2\leq a^*\leq 3/2$ such that
\be \lab{NLH:bd:a}
|a-a^*|\lesssim (K_1^2+K_2^2) e^{-(\frac 12+\frac{1}{4k})s}. 
\ee
 
\end{proposition}

All parameters are fixed, except the constants $\tilde K$, $K_j$ and $s^*$ that will be fixed in the forthcoming Lemmas. Hence generic constant $C$ will be used when independent of $\tilde K$, $K_j$ and $s^*$, with associated notation $\lesssim$. We prove Proposition \ref{main:pr:bootstrap} with a classical bootstrap reasoning. The results below will specify the dynamics in the trapped regime and allow to prove Proposition \ref{pr:bootstrap} at the end of the subsection.

\begin{lemma} \lab{NLH:lem:pointwise}

For any constants $\tilde K\geq K_J\geq ...\geq K_0>0$, there exists $s^*$ large enough such that if $f$ is trapped on $[s_0,s_1]$ with $s_0\geq s^*$, there holds for $j=0,...,J-1$, for any $Z\in \mathbb R$:
\be \lab{bd:eLinfty}
|\pa_Z^j\e|\lesssim \sqrt{K_J} e^{-\frac{1}{4k}s} (1+|Z|)^{\frac{1}{2}-2k-j},
\ee
and for any $C>0$, for all $|Y|\leq C$:
\be \lab{bd:eLinftyloc}
|\pa_Z^j\e|\lesssim \sqrt{\tilde K}e^{-\left(k-\frac12 +\frac{j}{2k}-\frac j2 \right)s} \ \ \text{for} \ \ j=0,...,2k, \ \ |\pa_Z^j\e|\lesssim \sqrt{\tilde K}e^{-\frac{1}{2k}s} \ \ \text{for} \ j\geq 2k+1.
\ee

\end{lemma}

\begin{proof}

\noindent \textbf{Step 1:} \emph{Proof of \fref{bd:eLinftyloc}}. From \fref{bd:eloc}, and Sobolev embedding one deduces that for $|Y|\leq C$:
$$
|\pa_Z^j \tilde \e|\lesssim \sqrt{K_J} e^{-\left(k-\frac12 +\frac{j}{2k}-\frac j2 \right)s}.
$$
From \fref{bd:cell bootstrap} and \fref{def:hell} for $j=0,...,2k-1$ and $|Y|\leq C$ one has:
$$
|\pa_Z^j (c_\ell h_\ell(Y))|=|e^{\frac{k-1}{2k}js}\pa_Y^j (c_\ell h_\ell(Y))|\lesssim e^{\frac{k-1}{2k}js} |c_\ell |\lesssim \sqrt{\tilde K}e^{-\left(k-\frac 12 +\frac{1}{4k}-\frac j2 +\frac{j}{2k}\right)}\lesssim \sqrt{\tilde K}e^{-\left(k-\frac 12 -\frac j2 +\frac{j}{2k}\right)}
$$
for $s_0$ large enough. For $j\geq 2k$ and $\ell\leq 2k-1$ one notices that $\pa_Z^j h_\ell=0$. Therefore, one obtains \fref{bd:eLinftyloc} from the decomposition \fref{id:decomposition e} and the two above bounds.\\

\noindent \textbf{Step 2:} \emph{Proof of \fref{bd:eLinfty}}. We apply \fref{bd:Sobolev} and use the facts that $Y\pa_Y=Z\pa_Z$ and $|Z\pa_Z \phi_{2k+1/2}|\lesssim |\phi_{2k+1/2}|\lesssim |Z\pa_Z \phi_{2k+1/2}|$:
\bee
\| \frac{Z^j\pa_Z^j\e}{\phi_{2k+1/2}(Z)} \|_{L^{\infty}(\{|Y|\geq 1 \})}^2 &\lesssim& \| \frac{Z^j\pa_Z^j\e}{\phi_{2k+1/2}(Z)} \|_{L^2\left(\{|Y|\geq 1 \},\frac{dY}{|Y|}\right)}^2 +\| Z\pa_Z(\frac{Z^j\pa_Z^j\e}{\phi_{2k+1/2}(Z)}) \|_{L^2\left(\{|Y|\geq 1 \},\frac{dY}{|Y|}\right)}^2 \\
&\lesssim & \sum_{k=0}^{j+1} \| \frac{(Z\pa_Z)^j\e}{\phi_{2k+1/2}(Z)} \|_{L^2\left(\{|Y|\geq 1 \},\frac{dY}{|Y|}\right)}^2 \lesssim K_Je^{-\frac{1}{2k}s}.
\eee
Since $|\phi_{2k+1/2}|\lesssim |Z|^{2k+1/2}(1+|Z|)^{-4k}$ one obtains that for $|Y|\geq 1$
\be \lab{bd:pointwise1}
|\pa_Z^j \e|\lesssim \sqrt{K_J}e^{-\frac{1}{4k}s}|Z|^{2k+1/2-j}(1+|Z|)^{-4k}.
\ee
For $j\geq 2k+1$, the fact that $|\phi_{2k+1/2}|(Z)=|Z|^{2k+1}|\phi_0(Z)|$ yields the inequality
$$
\frac{|Z|}{|\phi_0(Z)|}\lesssim \frac{1}{|\phi_0(Z)|}+\frac{|Z|^{j+1}}{|\phi_{2k+1/2}|}
$$
from which we infer that
$$
\int_{|Y|\geq 1} \frac{|Z\pa_Z (\pa_Z^j \e)|^2}{\phi_{0}^2(Z)}\frac{dY}{|Y|} \lesssim \int_{|Y|\geq 1} \frac{| \pa_Z^{j+1} \e |^2}{\phi_{0}^2(Z)}\frac{dY}{|Y|}+\int_{|Y|\geq 1} \frac{| Z^{j+1}\pa_Z^{j+1} \e)|^2}{\phi_{0}^2(Z)}\frac{dY}{|Y|} \lesssim K_J e^{-\frac{1}{2k}s}.
$$
The very same reasoning using the above bound and \fref{bd:eweightjgeq2k+1} gives
$$
|\pa_Z^j \e|\lesssim \sqrt{K_J}e^{-\frac{1}{4k}s}(1+|Z|)^{-4k}.
$$
From the two above bounds one infers that for $|Y|\geq 1$:
\be \lab{bd:pointwise2}
|\pa_Z^j \e|\lesssim \sqrt{K_J}e^{-\frac{1}{4k}s}(1+|Z|)^{\frac 12-2k-j}.
\ee
Combining \fref{bd:pointwise1}, \fref{bd:pointwise2} and \fref{bd:eLinftyloc} yields \fref{bd:eLinfty}.

\end{proof}

The evolution of $\e$ is given by:
\be \lab{eq:evolutione}
\e_s+\frac{Y}{2}\pa_Y \e+\e -2F_k(Z) \e+2(F_k(Z)-F[a])\e- \pa_{YY}\e+a_s \pa_a F[a]+\Psi-\e^2=0
\ee
and that of $\tilde \e$ by:
\be \lab{eq:evolutiontildee}
\tilde \e_s+\mathcal L_{\rho}\tilde \e+\sum_{\ell=0}^{k-1}(c_{2\ell,s}+\frac{2\ell-2}{2}c_{2\ell})h_{2\ell}+2(1-F[a])\e+a_s \pa_a F[a]+\Psi-\e^2=0
\ee

\begin{lemma}[Modulation equations] \lab{NLH:lem:modulation}

For any constants $\tilde K\geq K_J\geq ...\geq K_0>0$, there exists $s^*$ large enough such that if $f$ is trapped on $[s_0,s_1]$ with $s_0\geq s^*$, the following identities hold for $\ell=0,1,...,k-1$:
\be \lab{eq:modulation}
|a_s|\lesssim K_J e^{-(\frac 12+\frac{1}{4k}) s} \ \ \text{and} \ \ \left|c_{2\ell,s}+\frac{2\ell-2}{2}c_{\ell}\right|\lesssim K_J e^{-(k-\frac 12 +\frac{1}{4k})s}
\ee

\end{lemma}

\begin{proof}

\noindent \textbf{Step 1} \emph{Law for $a$}. We take the scalar product between \fref{eq:evolutiontildee} and $h_{2k}$, yielding, using \fref{id:pabFprojection} and \fref{id:decomposition e}:
$$
a_s c a^{2k-1}e^{-(k-1)s}(1+O(e^{-(k-1)s}))=\langle -2(1-F[a])\e-\Psi+\e^2,h_{2k}\rangle_{\rho}.
$$
We now estimate the right hand side. First, since $|F[a]-1|\lesssim e^{-(k-1)s}\sum_0^{k}Y^{2\ell}$, using \fref{bd:eloc} and \fref{bd:cell bootstrap} for $s^*$ large enough:
$$
\left| \langle -2(1-F[a])\e,h_{2k}\rangle_{\rho} \right|\lesssim \| \e\|_{L^2_{\rho}} \|(1-F[a])h_{2k} \|_{L^{2}_\rho}\lesssim \sqrt{K_0} e^{-(k-\frac 12)s}e^{-(k-1)s}= \sqrt{K_0} e^{-(2k-\frac 32)s}.
$$
Using the bound on the error \fref{bd:psirho}:
$$
\left| \langle \Psi,h_{2k}\rangle_{\rho} \right|\lesssim e^{-2(k-1)s}.
$$
Finally, using the bounds \fref{bd:eloc}, \fref{bd:cell bootstrap} and \fref{bd:eLinfty} for the nonlinear term:
$$
\left| \langle \e^2,h_{2k}\rangle_{\rho} \right| \lesssim \| \e\|_{L^2_\rho} \| \e\|_{L^{\infty}}\lesssim (\sqrt{K_0} e^{-(k-\frac 12)s}+\sqrt{\tilde K} e^{-(k-\frac 12+\frac{1}{4k})s}) \sqrt{K_J}e^{-\frac{1}{4k}s}\lesssim K_J e^{-(k-\frac 12+\frac{1}{4k})s},
$$
for $s^*$ large enough. Summing the above identities yields \fref{eq:modulation} for $a_s$.\\

\noindent \textbf{Step 2} \emph{Law for $c_{\ell}$}. We take the scalar product between \fref{eq:evolutiontildee} and $h_{2\ell}$ for $\ell=0,1,...,k-1$, yielding, using \fref{id:pabFprojection} and \fref{id:decomposition e}:
$$
(c_{2\ell,s}+\frac{2\ell-2}{2}c_{2\ell}) \| h_{2\ell}\|_{L^2_\rho}=-\langle 2(1-F[a])\e+a_s \frac{\pa}{\pa_a} F[a]+\Psi-\e^2,h_{2\ell}\rangle
$$
Performing the same computations as in Step 1 gives for $s^*$ large enough:
$$
\left| \langle 2(1-F[a])\e+\Psi-\e^2,h_{2\ell}\rangle \right|\lesssim K_J e^{-(k-\frac 12 +\frac{1}{4k})s}.
$$
Using the bound for $a_s$ \fref{eq:modulation} obtained in Step 1 and \fref{bd:pabFrho} one obtains:
$$
\left| \langle a_s \pa_a F[a] ,h_{2\ell}\rangle \right| \lesssim K_J e^{-(\frac 12+\frac{1}{4k})s}e^{-(k-1)s}=K_J e^{-(k-\frac 12 +\frac{1}{4k})s}.
$$
The three previous identities then yield \fref{eq:modulation} for $c_{2\ell}$.

\end{proof}

\begin{lemma} \lab{NLH:lem:energyrho}

There exists a choice of constants $ K_J\geq ...\geq K_0>0$, such that for any $\tilde K\geq K_{J}$, there exists $s^*$ large enough such that if $f$ is trapped on $[s_0,s_1]$ with $s_0\geq s^*$ and satisfies \fref{bd:eL2rhoinit}, at time $s_1$ there holds for $j=0,...,J$:
\be \lab{eq:boundfatigue}
\| \pa_Z^j \tilde \e (s_1) \|_{L^2_{\rho}}\leq \frac{\sqrt{K_j}}{2}e^{-\left(k-\frac12 +\frac{j}{2k}-\frac j2 \right)s_1} \ \ \text{for} \ \ j=0,...,2k, \ \ \| \pa_Z^j \tilde \e(s_1) \|_{L^2_{\rho}}\leq \frac{\sqrt{K_j}}{2}e^{-\frac{1}{2k} s_1} \ \ \text{for} \ j\geq 2k+1.
\ee

\end{lemma}

\begin{proof}

Set $\tilde \e_j=\pa_Z^j \tilde \e$. Then $\tilde \e_j$ solves from \fref{eq:evolutiontildee}:
$$
\pa_s \tilde \e_j+\frac{j}{2k} \tilde \e_j +H_\rho \tilde \e_j+\sum_{\ell=0}^{k-1}(c_{2\ell,s}+\frac{2\ell-2}{2}c_{2\ell})\pa_Z^j(h_{2\ell})+2\pa_Z^j((1-F[a])\e)+a_s\pa_Z^j \pa_a F[a]+\pa_Z^j \Psi-\pa_Z^j \e^2=0
$$
which yields the following expression for the energy identity:
\bee
\frac{d}{ds}\left(\frac{1}{2}\int \tilde \e_j^2e^{-\frac{Y^2}{4}}dY \right) &=& -\langle \tilde \e_j,(H_\rho+\frac{j}{2k}) \tilde \e_j\rangle_\rho-\sum_{\ell=0}^{k-1}(c_{2\ell,s}+\frac{2\ell-2}{2}c_{2\ell}) \langle \tilde \e_j, \pa_Z^j(h_{2\ell})\rangle_\rho \\
&&-2\langle \tilde \e_j,\pa_Z^j((1-F[a])\e)\rangle_\rho+\langle \tilde \e_j, a_s\pa_Z^j \pa_a F[a]+\pa_Z^j \Psi\rangle_\rho-\langle \tilde \e_j,\pa_Z^j \e^2\rangle_\rho.
\eee
For $j=0,...,2k$, by integrating by parts one obtains from \fref{id:decomposition e} that $\tilde \e_j$ is orthogonal to $h_{\ell}$ for $\ell=0,...,2k-j$ for the $L^2_\rho$ scalar product and therefore to any polynomial of degree less or equal to $2k-j$. Therefore, from Proposition \ref{pr:Lrho} there holds:
$$
-\langle \tilde \e_j,(H_\rho +\frac{j}{2k})\tilde \e_j\rangle_\rho \leq \left\{ \begin{array}{l l} -\left(k-\frac j2-\frac 12+\frac{j}{2k}\right)\int \tilde \e_j^2e^{-\frac{Y^2}{4}}dY \ \ \text{if} \ j=0,...,2k, \\
-(\frac{j}{2k}-1)\int \tilde \e_j^2e^{-\frac{Y^2}{4}}dY \ \ \text{if} \ j=2k+1,...,J. \end{array} \right.
$$
Let $0<\nu \ll 1$ be a small constant to be fixed later on. Integrating by parts yields:
$$
-\int \tilde \e_j H_\rho \tilde \e_j e^{-\frac{Y^2}{4}}dY \leq -c\int Y^2\tilde \e_j^2 e^{-\frac{Y^2}{4}}dY +c'\int \tilde \e_j^2 e^{-\frac{Y^2}{4}}dY
$$
for some constants $c,c'>0$. Combining the above estimates one writes:
\bee
-\langle \tilde \e_j,(H_\rho +\frac{j}{2k})\tilde \e_j\rangle_\rho   &= & - (1-e^{-2\nu s})\langle \tilde \e_j,(H_\rho +\frac{j}{2k})\tilde \e_j\rangle_\rho -e^{-2\nu s}\langle \tilde \e_j,(H_\rho +\frac{j}{2k})\tilde \e_j\rangle_\rho  \\
 &\leq & - ce^{-2\nu s}\int Y^2\tilde \e_j^2 e^{-\frac{Y^2}{4}}dY+O(e^{-2\nu s}\int \tilde \e_j^2 e^{-\frac{Y^2}{4}}dY)\\
 &&+\left\{ \begin{array}{l l} -\left(k-\frac j2-\frac 12+\frac{j}{2k}\right)\int \tilde \e_j^2e^{-\frac{Y^2}{4}}dY \ \ \text{if} \ j=0,...,2k, \\
-(\frac{j}{2k}-1)\int \tilde \e_j^2e^{-\frac{Y^2}{4}}dY \ \ \text{if} \ j=2k+1,...,J. \end{array} \right..
\eee
As we said earlier, $\tilde \e_j$ is orthogonal to any polynomial of degree less or equal to $2k-j$ for $j=0,...,2k$. For $j\geq 2k+1$ one notices that $\pa_Z^j h_{2\ell}=0$ for any $\ell=0,...,k-1$. Hence the cancellation for $j=0,...,J$:
$$
\sum_{\ell=0}^{k-1}(c_{2\ell,s}+\frac{2\ell-2}{2}c_{2\ell})\langle \tilde \e_j, \pa_Z^j(h_{2\ell})\rangle_\rho=0
$$
One has from \fref{def:Fb} that the lower order linear potential satisfies (as $Z=Ye^{-\frac{k-1}{2k}s}$):
$$
\left| (1-F[a]) \right| \lesssim |Z|^2(1+|Z|)^{-2-2k}\lesssim \min(1,e^{-\frac{k-1}{k}s}Y^2)
$$
which adapts to derivatives. Therefore, applying Leibnitz rule and Cauchy-Schwarz one gets that for $j=0,...,J$ using \fref{bd:eloc}:
\begin{align*}
& \left|\langle \tilde \e_j,\pa_Z^j((1-F[a])\e)\rangle_\rho \right|  \leq C \left(\int \tilde \e_j^2e^{-\frac{Y^2}{4}}dY \right)^{\frac 12} \sum_{i=0}^{j-1}\left(\int |\pa_Z^i \tilde \e|^2e^{-\frac{Y^2}{4}}dY \right)^{\frac 12}+Ce^{-\frac{k-1}{k}s} \int Y^2 \tilde \e_j^2e^{-\frac{Y^2}{4}}dY\\
& \qquad \leq Ce^{-\frac{k-1}{k}s} \int Y^2 \tilde \e_j^2e^{-\frac{Y^2}{4}}dY+\left\{ \begin{array}{l l} Ce^{-2\left(k-\frac{1}{2}+\frac{j}{2k}+\nu \right)s} K_j \qquad \mbox{for }j=0,...,2k+1,\\
 C \sqrt{K_j}\sqrt{K_{j-1}}e^{-\frac{1}{k}s}\qquad \mbox{for }j=2k+2,...,J,
\end{array} \right.
\end{align*}
where we used that for $i<j$ one has $K_i<K_j$ and $\frac{i}{2k}-\frac i2\geq \frac{j}{2k}-\frac{j}{2}+2\nu$ for $0<2\nu \leq \frac{1}{2}-\frac{1}{2k}$. From \fref{eq:modulation}, \fref{bd:pabFrho}, \fref{bd:psirho}, \fref{bd:eloc} and Cauchy-Schwarz we estimate for $0<\nu<\frac{1}{8k}$:
$$
\left| \int \tilde \e_j \pa_Z^j \left( a_s \pa_a F[a]+\Psi \right)e^{-\frac{Y^2}{4}}dY  \right| \lesssim \sqrt{K_j} \left\{ \begin{array}{l l} e^{-2\left(k-\frac 12 -\frac j2 +\frac{j}{2k}+\nu \right)s} \ \ \text{for} \ j=0,...,2k, \\ e^{-2\left(\frac{1}{2k}+\nu\right)s} \ \ \text{for} \ j=2k+1,...,J. \end{array} \right.
$$
Using Leibnitz rule and Cauchy Schwarz, from \fref{bd:eLinfty}, \fref{id:decomposition e}, \fref{bd:eloc} and \fref{bd:cell bootstrap} we infer for the nonlinear term, as $\tilde K>K_J>...>K_0$:
\begin{align*}
&\left| \int \tilde \e_j \pa_Z^j (\e^2)e^{-\frac{Y^2}{4}}dY  \right|  \lesssim  \left(\sum_{i=0}^{J-1} \| \pa_Z^i\e \|_{L^{\infty}} \right)\left(\sum_{i=0}^{j} \int |\pa_Z^j \e|^2e^{-\frac{Y^2}{4}}dY \right)^{\frac 12}\left( \int \tilde \e_j^2e^{-\frac{Y^2}{4}}dY \right)^{\frac 12}\\
&\qquad \qquad \lesssim  e^{-\frac{1}{4k}s} \left( \sqrt{\tilde K}e^{-(k-\frac 12 +\frac{1}{4k}-j\frac{k-1}{2k})s}+\left(\sum_{i=0}^{j} \int |\pa_Z^j \tilde \e|^2e^{-\frac{Y^2}{4}}dY \right)^{\frac 12}\right)\left( \int \tilde \e_j^2e^{-\frac{Y^2}{4}}dY \right)^{\frac 12} \\
&\qquad \qquad \lesssim C(\tilde K) \left\{ \begin{array}{l l} e^{-2\left(k-\frac 12 -\frac j2 +\frac{j}{2k}+\nu \right)s} \ \ \text{for} \ j=0,...,2k, \\ e^{-2\left(\frac{1}{2k}+\nu\right)s} \ \ \text{for} \ j=2k+1,...,J. \end{array} \right.
\end{align*}
for $0<2\nu <\frac{1}{4k}$. Putting all the previous estimates together one obtains that there exists $\nu>0$ depending only $k$ such that for $s^*$ large enough, after some signs inspection:
$$
\frac{d}{ds} \left(\int \tilde \e_j^2 e^{-\frac{Y^2}{4}} \right)\leq \left\{ \begin{array}{l l} -2\left(k-\frac j2 -\frac 12 +\frac{j}{2k} \right)\int \tilde \e_j^2 e^{-\frac{Y^2}{4}}+C(\tilde K)e^{-2\left(k-\frac j2 -\frac 12 +\frac{j}{2k}+\nu \right)s} \ \ \text{for} \ j=0,...,2k \\
-\frac{j-2k}{k}\int \tilde \e_j^2 e^{-\frac{Y^2}{4}}+C(\tilde K)e^{-2\left(\frac{1}{2k} +\nu\right) s}+C(K_{j-1})e^{-\frac{1}{k}s} \ \ \text{for} \ j=2k+2,...,J \end{array} \right.
$$
where we used that $\tilde K>K_J>...>K_0$. We claim that integrating over time the above differential inequality shows \fref{eq:boundfatigue}. We only show that this is the case for $j=2k+2,...,J$, the proof being the same for $0\leq j \leq  2k+1$. For $j=2k+2,...,J$, one notices that $\frac{j-2k}{k}\geq \frac{2}{k}$, so that from \fref{bd:eL2rhoinit} and the above inequality one deduces that at time $s_1$:
\bee
\int \tilde \e_j^2 e^{-\frac{Y^2}{4}}dY &\leq &e^{-\frac 2k s_1}\left( \int \tilde \e_j(s_0)^2 e^{-\frac{Y^2}{4}}dY+ e^{-\frac{2}{k}s_1}\int_{s_0}^{s_1}\left( C(\tilde K)e^{\left(\frac{1}{k}-2\nu\right)\tilde s}+\sqrt{K_j}\sqrt{K_{j-1}}e^{\frac 1 k \tilde s} \right)d\tilde s\right)\\
&\leq & e^{-\frac 2k s_1}+C(\tilde K)e^{-\frac 1k s}e^{-2\nu s_1}+C\sqrt{K_{j-1}}\sqrt{K_j}e^{-\frac 1k s_1} \leq \frac{K_j}{2}e^{-\frac 1 k s_1}
\eee
where the last inequality holds true provided $K_j\geq 4$ and $K_{j}>CK_{j-1}$ for some universal constant $C$, and that $s^*$ has been chosen large enough depending on $\tilde K$. The same inequality holds true for $j=0,...,2k+1$. Therefore, one can choose inductively the constants $K_j$ one after another to satisfy these conditions, ending the proof of the Lemma.

\end{proof}

\begin{lemma} \lab{NLH:lem:improvedoustide}

There exists a choice of constants $ K_J\geq ...\geq K_0>0$, such that for any $\tilde K\geq K_{J}$, there exists $s^*$ large enough such that if $f$ is trapped on $[s_0,s_1]$ with $s_0\geq s^*$ and satisfies \fref{bd:decayextinit} and \fref{bd:decayextinit2}, at time $s_1$ there holds for $j=0,...,J$:
\be \lab{bd:lossy improved}
 \int_{|Y|\geq 1} \frac{|(Y\pa_Y)^j \e (s_1)|^2}{\phi_{2k+1}^2(Z)}\frac{dY}{|Y|}\leq \frac{K_j}{2}e^{-\frac{3}{4k}s_1}, \ \ \
 \int_{|Y|\geq 1} \frac{|(Y\pa_Y)^j \e (s_1)|^2}{\phi_{2k+1/2}^2(Z)}\frac{dY}{|Y|}\leq \frac{K_{j}}{2}e^{-\frac{1}{2k}s_1},
\ee
\be \lab{bd:lossy improved3}
\int_{|Y|\geq 1} \frac{|\pa_Z^j \e(s_1)|^2}{\phi_{0}^2(Z)}\frac{dY}{|Y|}\leq \frac{K_j}{2}  e^{-\frac{1}{2k}s_1} \ \ \text{for} \ j\geq 2k+1.
\ee

\end{lemma}

\begin{proof}

Note that by even symmetry it suffices to perform the estimates for $Y\geq 1$. We shall then use the identity $Y=|Y|$ in the following.\\

\noindent \textbf{Step 1:} \emph{Proof of \fref{bd:lossy improved}}. Let $\chi $ be a smooth cut-off function, with $\chi=1$ for $Y\geq 2$ and $\chi=0$ for $Y\leq 1$. Let $\ell=2k+1$ or $\ell=2k+1/2$. Set $\hat \e_j=(Y\pa_Y)^j \e$. Then $\hat \e_j$ solves from \fref{eq:evolutione} since $[\pa_{YY},Y\pa_Y]=2\pa_{YY}$:
\bee
0&=&\pa_s\hat \e_j+\frac Y2 \pa_Y \hat \e_j +\hat \e_j -2F_k(Z)\hat \e_j-\pa_{YY}\hat \e_j+2(F_k\hat \e_j-(Y\pa_Y)^j (F_k \e))+2(Y\pa_Y)^j ((F_k(Z)-F[a])\e)\\
&&+2\sum_{n=0}^{j-1}(Y\pa_Y)^{j-1-n}\pa_{YY}(Y\pa_Y)^n \e+a_s (Y\pa_Y)^j \pa_a F[a]+(Y\pa_Y)^j\Psi-(Y\pa_Y)^j \e^2
\eee
From the above identity we compute, performing integrations by parts, that
\bee
&&\frac{d}{ds}\left(\frac 12 \int \chi \frac{ \hat \e_j^2}{\phi_{\ell}^2(Z)}\frac{dY}{Y} \right)\\
&=& \int \chi \frac{\hat \e_j}{\phi_{\ell}(Z)} \frac{1}{Y} \frac{-\hat \e_j+2F_k\hat \e_j-\frac{Y}{2}\pa_Y \hat \e_j+\pa_{YY}\hat \e_j-2(F_k\hat \e_j-(Y\pa_Y)^j(F_k\e))+(Y\pa_Y)^j \e^2}{\phi_{\ell}(Z)}dY\\
&&+ \int \chi \frac{\hat \e_j}{\phi_{\ell}(Z)} \frac{1}{Y} \frac{-2(Y\pa_{Y})^j((F_k(Z)-F[a])\e)-2\sum_{n=0}^{j-1}(Y\pa_Y)^{j-1-n}\pa_{YY}(Y\pa_Y)^n \e-(Y\pa_Y)^j(a_s\pa_aF[a]+\Psi)}{\phi_{\ell}(Z)}dY  \\
&&- \int \chi \frac{\hat \e_j^2}{\phi_{\ell}^2(Z)} \frac{1}{Y}\left(\frac{a_s}{a}\frac{(Z\pa_Z\phi_{\ell})(Z)}{\phi_{\ell}(Z)}-\frac{k-1}{2k}\frac{(Z\pa_Z\phi_{\ell})(Z)}{\phi_{\ell}(Z)} \right)dY \\
&=& \int \chi \frac{\hat \e_j^2}{\phi^2_{2k+1}(Z)}\frac{-\phi_{\ell}(Z)+2F_k(Z)\phi_{\ell}(Z)-\frac{Z}{2k}\pa_Z \phi_{\ell}(Z)}{\phi_{\ell}(Z)}\frac{dY}{Y}-\int \chi \frac{|\pa_Y \hat \e_j|^2}{\phi_{\ell}^2(Z)}\frac{dY}{Y} \\
&&+\frac 14 \int \pa_Y \chi \frac{\hat \e_j^2}{\phi_{\ell}^2(Z)}dY+\frac 12 \int \hat \e_j^2 \pa_{YY} \left( \frac{\chi}{\phi_{\ell}^2(Z)}\frac{1}{Y}\right)dY \\
&&-2\int \chi \frac{\hat \e_j(F_k(Z)\hat \e_j-(Y\pa_Y) ^j(F[a]\e)) }{\phi^2_{\ell}(Z)} \frac{dY}{Y}-2\int \chi \frac{\hat \e_j(Y\pa_Y)^j((F_k(Z)-F[a])\e)}{\phi_{\ell}^2(Z)}\frac{dY}{Y}\\
&&-2\int \chi \frac{\hat \e_j\sum_{n=0}^{j-1}(Y\pa_Y)^{j-1-n}\pa_{YY}(Y\pa_Y)^n\e}{\phi_{\ell}^2(Z)}\frac{dY}{Y}-\frac{a_s}{a}\int \frac{\hat \e_j^2}{\phi_{\ell}^2(Z)}\frac{(Z\pa_Z \phi_{\ell})(Z)}{\phi_{\ell}(Z)}\frac{dY}{Y}\\
&&- \int \chi \frac{\hat \e_j}{\phi_{\ell}(Z)} \frac{(Y\pa_Y)^j(a_s \pa_aF[a]+\Psi)}{\phi_{\ell}(Z)}\frac{dY}{Y}+ \int \chi \frac{ \hat \e_j(Y\pa_Y)^j(\e^2)}{\phi_{\ell}^2(Z)}\frac{dY}{Y}\\
\eee
where in the last equality, on the first line one has the main order linear effects, on the second their associated boundary terms, one the third and fourth the lower order linear effects, and on the last line the influence of the forcing and of the nonlinear effects. We now estimate all terms. For the first term from \fref{eq:def phiXell}:
$$
 \int \chi \frac{\hat \e_j^2}{\phi^2_{\ell}(Z)}\frac{-\phi_{\ell}(Z)+2F_k[b]\phi_{\ell}(Z)-\frac{Z}{2k}\pa_Z \phi_{\ell}(Z)}{\phi_{\ell}(Z)}\frac{dY}{Y}=-\frac{\ell-2k}{2k} \int \chi \frac{\hat \e_j^2}{\phi_{\ell}^2(Z)}\frac{dY}{Y}.
$$
The second term is dissipative and has a negative sign since $\chi$ is positive. For the third term, using \fref{id:decomposition e}, \fref{bd:eloc} and \fref{bd:cell bootstrap}:
\bea
\non \left|\frac 12 \int \pa_Y \chi \frac{\hat \e_j^2}{\phi_{\ell}^2(Z)}dY \right|& \lesssim & \| \frac{1}{\phi_{\ell}^2(Z)}\|_{L^{\infty}(1\leq Y\leq 2)}\| \hat \e_j \|_{L^2(1\leq Y\leq 2)}^2\\
\non & \lesssim & \| |Z|^{-2 \ell } \|_{L^{\infty}(1\leq Y \leq 2)} \left(\sum_{i=0}^j \| Z^i\pa_Z^i \e \|_{L^2(1\leq Y\leq 2)}^2 \right)\\
\non & \lesssim &  \| |e^{-\frac{k-1}{2k}s}Y|^{-2 \ell } \|_{L^{\infty}(1\leq Y \leq 2)} \left(\sum_{i=0}^j \| Z^i\pa_Z^i \tilde \e \|_{L^2(1\leq Y\leq 2)}^2+\sum_{\ell=0}^{k-1}|c_{2\ell} |^2 \right)\\
\lab{eq:linboundary}  & \lesssim & e^{\frac{k-1}{2k} 2 \ell s}\left( K_j e^{-2(k-\frac 12)s}+\tilde K e^{-2(k-\frac 12 +\frac{1}{4k})s}\right) \lesssim K_j e^{-\frac 1 k s}
\eea
for $s^*$ large enough. For the fourth term, we first decompose:
$$
\pa_{YY}\left(\frac{\chi}{\phi_\ell^2(Z)Y} \right)=\pa_{YY}\chi \left(\frac{1}{\phi_\ell^2(Z)Y}\right)+2\pa_Y \chi \pa_Y \left(\frac{1}{\phi_\ell^2(Z)Y}\right)+\chi \pa_{YY} \left(\frac{1}{\phi_\ell^2(Z)Y}\right).
$$
Since one has 
$$
\left| \pa_{YY}\chi \left(\frac{1}{\phi_\ell^2(Z)Y}\right)+2\pa_Y \chi \pa_Y \left(\frac{1}{\phi_\ell^2(Z)Y}\right) \right|\lesssim  |Z|^{-2 \ell } 1_{1\leq Y\leq 2} \lesssim e^{\frac{k-1}{k}\ell s} 1_{1\leq Y\leq 2} 
$$
we claim that one can perform the very same estimate for the first two terms as \fref{eq:linboundary}, giving:
$$
\left| \int \hat \e_j^2 \left(\pa_{YY}\chi \left(\frac{1}{\phi_\ell^2(Z)Y}\right)+2\pa_Y \chi \pa_Y \left(\frac{1}{\phi_\ell^2(Z)Y}\right)\right)dY \right| \lesssim K_je^{-\frac 1 k s}.
$$
For the last term, from a direct computation, for $|Y|\geq 1$, one has that:
$$
\left| \pa_{YY} \left( \frac{1}{\phi_{\ell}^2(Z)}\frac{1}{Y}\right)\right| \lesssim \frac{1}{\phi_{\ell}^2(Z)}\frac{1}{Y^3}.
$$
Therefore, if $\ell=2k+1$, we take some $0<\kappa \ll 1$ small enough and split the integral using some $Y^*\gg 1$ large enough, and use \fref{id:decomposition e}, \fref{bd:eloc} and \fref{bd:cell bootstrap}:
\bee
&& \left| \int \chi \hat \e_j^2\pa_{YY} \left( \frac{1}{\phi_{2k+1}^2(Z)}\frac{1}{Y}\right) dY\right| \lesssim \int_{1}^{Y^*} \frac{\hat \e_j^2}{\phi_{2k+1}^2(Z)}dY+\int_{Y^*}^{+\infty} \hat \e_j^2  \frac{1}{\phi_{2k+1}^2(Z)} \frac{1}{|Y|^2}\frac{1}{Y}dY\\
&\lesssim &  \| \frac{1}{\phi_{2k+1}^2(Z)}\|_{L^{\infty}(1\leq Y\leq Y^*)}\| \hat \e_j \|_{L^2(1\leq Y\leq Y^*)}^2+\kappa  \int \chi \frac{ \hat \e_j^2}{\phi_{2k+1}^2(Z)}\frac{dY}{Y}\\
&\lesssim & K_j e^{\frac{k-1}{k}(2k+1) s}e^{-2(k-\frac 12)s}+\kappa  \int \chi \frac{ \hat \e_j^2}{\phi_{2k+1}^2(Z)}\frac{dY}{Y} \lesssim K_j e^{-\frac 1k s}+\kappa  \int \chi \frac{ \hat \e_j^2}{\phi_{2k+1}^2(Z)}\frac{dY}{Y}.
\eee
If $\ell =2k+1/2$, one uses the fact that $\phi_{2k+1}(Z)=Z^{1/2}\phi_{2k+1/2}$, so that
\be \lab{eq:estimationlowerlin}
 \frac{1}{\phi_{2k+1/2}^2(Z)}\frac{1}{Y^3}=\frac{1}{\phi_{2k+1}^2(Z)Y^2}e^{-\frac{k-1}{2k}s}
\ee
to estimate using \fref{bd:eweightlossy}:
$$
\left| \int \hat \e_j^2\pa_{YY} \left( \frac{1}{\phi_{2k+1/2}^2(Z)}\frac{1}{Y}\right) dY\right| \lesssim e^{-\frac{k-1}{2k}s} \int_{Y\geq 1} \frac{\hat \e_j^2}{\phi_{2k+1}^2(Z)}dY\lesssim K_Je^{-(\frac{1}{2}+\frac{1}{4k})s}\lesssim K_Je^{-\frac 1k s}.
$$
Collecting the above bounds one has proven that:
$$
\left|  \int \hat \e_j^2 \pa_{YY} \left( \frac{\chi}{\phi_{\ell}^2(Z)}\frac{1}{Y}\right)dY \right| \lesssim \left\{ \begin{array}{l l} K_Je^{-\frac 1k s}+\kappa  \int \chi \frac{ \hat \e_j^2}{\phi_{2k+1}^2(Z)}\frac{dY}{Y} \ \ \text{for} \ \ell=2k+1, \\ K_Je^{-\frac 1k s} \ \ \text{for} \ \ell=2k+1/2. \end{array} \right.
$$
We turn to the fifth term. We first estimate using Leibniz rule:
$$
| F_k(Z)\hat \e_j-(Y\pa_Y) ^j(F[a]\e)|\lesssim \sum_{n=0}^{j-1} |(Z\pa_Z)^{j-n}F[a]||(Y\pa_Y)^n \e|\lesssim |Z|^{2k}(1+|Z|)^{-4k}\sum_{n=0}^{j-1} |(Y\pa_Y)^n \e|.
$$
If $\ell=2k+1$ we then use Cauchy-Schwarz and \fref{bd:eweightlossy} to obtain, as $K_J>...>K_{0}$:
\bee
\left| \int \chi \frac{\hat \e_j(F_k(Z)\hat \e_j-(Y\pa_Y) ^j(F[a]\e)) }{\phi^2_{2k+1}(Z)} \frac{dY}{Y} \right| & \leq & C \left(\int_{Y\geq 1} \frac{\hat \e_j^2}{\phi^2_{2k+1}(Z)}\frac{dY}{Y} \right)^{\frac 12}\sum_{n=0}^{j-1} \left(\int_{Y\geq 1} \frac{((Y\pa_Y)^n \e)^2}{\phi^2_{2k+1}(Z)}\frac{dY}{Y} \right)^{\frac 12} \\
&\leq & \sqrt{K_j}\sqrt{K_{j-1}}e^{-\frac{3}{4k}s}.
\eee
If $\ell=2k+1/2$ we use the fact that
$$
\frac{|Z|^{2k}}{\phi_{2k+1/2}^2(Z)(1+|Z|)^{4k}}\leq \frac{1}{\phi_{2k+1}^2(Z)},
$$
which, combined with Cauchy-Schwarz and \fref{bd:eweightlossy}, gives:
\bee
\left| \int \chi \frac{\hat \e_j(F_k(Z)\hat \e_j-(Y\pa_Y) ^j(F[a]\e)) }{\phi^2_{2k+1/2}(Z)} \frac{dY}{Y} \right| & \lesssim & \left(\int_{Y\geq 1} \frac{\hat \e_j^2}{\phi^2_{2k+1}(Z)}\frac{dY}{Y} \right)^{\frac 12}\sum_{n=0}^{j-1} \left(\int_{Y\geq 1} \frac{((Y\pa_Y)^n \e)^2}{\phi^2_{2k+1}(Z)}\frac{dY}{Y} \right)^{\frac 12} \\
&\lesssim & \sqrt{K_j}\sqrt{K_{j-1}}e^{-\frac{3}{4k}s}.
\eee
One has then proven that:
$$
\left| \int \chi \frac{\hat \e_j(F_k(Z)\hat \e_j-(Y\pa_Y) ^j(F[a]\e)) }{\phi^2_{2k+1/2}(Z)} \frac{dY}{Y} \right| \lesssim \left\{ \begin{array}{l l} \sqrt{K_{j-1}}\sqrt{K_j}e^{-\frac{3}{4k}s} \ \ \text{for} \ \ell=2k+1, \\ \sqrt{K_{j-1}}\sqrt{K_j}e^{-\frac{3}{4k} s} \ \ \text{for} \ \ell=2k+1/2. \end{array} \right.
$$
We turn to the sixth term. Since from \fref{def:Fb} for any $j\in \mathbb N$:
$$
|(Z\pa_Z)^j (F_k(Z)-F[a])|\lesssim e^{-\frac{k-1}{k}s},
$$
one has using Cauchy-Schwarz that
$$
\left| \int \chi \frac{\hat \e_j(Y\pa_Y)^j((F_k(Z)-F[a])\e)}{\phi_{\ell}^2(Z)}\frac{dY}{Y} \right| \lesssim e^{-\frac{k-1}{k}s} \sum_{n=0}^j \int_{Y\geq 1} \frac{|(Y\pa_Y)^j \e|^2}{\phi_{\ell}^2(Z)}\frac{dY}{Y}
$$
which using \fref{bd:eweightlossy} implies that for some $0<\nu \ll 1$ be small enough:
$$
\left| \int \chi \frac{\hat \e_j(Y\pa_Y)^j((F_k(Z)-F[a])\e)}{\phi_{\ell}^2(Z)}\frac{dY}{Y} \right| \lesssim  \left\{ \begin{array}{l l} K_Je^{-\left(\frac{3}{4k}+\nu\right)s} \ \ \text{for} \ \ell=2k+1, \\ K_{j}e^{-\left(\frac{1}{2k}+\nu \right) s} \ \ \text{for} \ \ell=2k+1/2. \end{array} \right.
$$
For the seventh term, we use the fact that $\pa_{YY}=((Y\pa_Y)^2-Y\pa_Y)/Y^2$ to decompose:
\bee
\sum_{n=0}^{j-1}(Y\pa_Y)^{j-1-n} \pa_{YY} (Y\pa_Y)^n \e & = & Y^{-1}\pa_Y \hat \e_j  -\sum_{n=0}^{j-1} (Y\pa_Y)^{j-1-n}Y^{-2} (Y\pa_Y)^{n+1}\e\\
&&+\sum_{n=0}^{j-1} (Y\pa_Y)^{j-1-n}Y^{-2} (Y\pa_Y)^{n+2}\e- Y^{-2} (Y\pa_Y)^{j+1}\e .
\eee
For the first term, we integrate by parts and find:
$$
\int \chi \frac{\hat \e_j Y^{-1}\pa_Y \hat \e_j}{\phi_{\ell}^2(Z)}\frac{dY}{Y}= \left|-\frac 12 \int \chi\hat \e_j^2 \pa_Y \left( \frac{1}{\phi_{\ell}^2(Z)Y^2}\right)dY-\frac 12 \int \pa_Y \chi \frac{\hat \e_j^2}{\phi_\ell^2(Z)}\frac{dY}{Y^2}\right|\lesssim \int_{Y\geq 1} \frac{\hat \e_j^2}{\phi_{\ell}^2(Z)Y^3} dY.
$$
If $\ell=2k+1$, we take $0<\kappa \ll 1$ small enough and $Y^*$ large enough so that, using \fref{id:decomposition e}, \fref{bd:eloc} and \fref{bd:cell bootstrap}:
\bee
&& \int_{Y\geq 1} \frac{\hat \e_j^2}{\phi_{\ell}^2(Z)Y^3} dY \lesssim \int_1^{Y^*} \frac{\hat \e_j^2}{\phi_{\ell}^2(Z)Y^3} dY+ \int_{Y^*}^{+\infty} \frac{\hat \e_j^2}{\phi_{\ell}^2(Z)Y^3} dY\\
& \lesssim & \| \frac{1}{\phi_{\ell}^2(Z)}\|_{L^{\infty}(1\leq Y\leq Y^*)} \| \hat \e_j\|_{L^2_\rho}^2+\kappa \int \chi \frac{\hat \e_j^2}{\phi_\ell^2(Z)}\frac{dY}{Y} \lesssim K_Je^{-\frac 1k s}+\kappa \int \chi \frac{\hat \e_j^2}{\phi_\ell^2(Z)}\frac{dY}{Y}.\\
\eee
If $\ell=2k+1/2$, we use \fref{eq:estimationlowerlin}, \fref{bd:eweightlossy} and obtain for $s^*$ large enough:
$$
\int_{Y\geq 1} \frac{\hat \e_j^2}{\phi_{\ell}^2(Z)Y^3} dY \lesssim e^{-\frac{k-1}{2k}s} \int_{Y\geq 1} \frac{\hat \e_j^2}{\phi_{2k+1}^2(Z)}dY\lesssim K_Je^{-(\frac{1}{2}+\frac{1}{4k})s}\leq e^{-\frac 1k s},
$$
If $\ell=2k+1$, one estimates using \fref{bd:eweightlossy} that
\bee
&&\int \chi \frac{\left|\sum_{n=0}^{j-1} (Y\pa_Y)^{j-1-n}Y^{-2} (Y\pa_Y)^{n+1}\e+\sum_{n=0}^{j-1} (Y\pa_Y)^{j-1-n}Y^{-2} (Y\pa_Y)^{n+2}\e- Y^{-2} (Y\pa_Y)^{j+1}\e\right|^2}{\phi_{2k+1}^2}\frac{dY}{Y} \\
&\lesssim & \sum_{n=0}^{j-1} \int_{Y\geq 1} \frac{((Y\pa_Y)^j \e)^2}{Y^2\phi_{2k+1}^2}\frac{dY}{Y} \lesssim K_{j-1} e^{-\frac{3}{4k}s}.
\eee
If $\ell=2k+1/2$, one estimates using \fref{bd:eweightlossy} and \fref{eq:estimationlowerlin} that
\bee
&&\int \chi \frac{\left|\sum_{n=0}^{j-1} (Y\pa_Y)^{j-1-n}Y^{-2} (Y\pa_Y)^{n+1}\e+\sum_{n=0}^{j-1} (Y\pa_Y)^{j-1-n}Y^{-2} (Y\pa_Y)^{n+2}\e- Y^{-2} (Y\pa_Y)^{j+1}\e\right|^2}{\phi_{2k+1/2}^2}\frac{dY}{Y} \\
&\lesssim & e^{-(\frac{1}{2}+\frac{1}{4k})s} \sum_{n=0}^{j-1} \int_{Y\geq 1} \frac{((Y\pa_Y)^j \e)^2}{\phi_{2k+1}^2}\frac{dY}{Y} \lesssim K_{j-1}e^{-\left(\frac 12 +\frac{1}{k}\right)s}.
\eee
One has therefore proven that, taking $s^*$ large enough:
\bee
&& \left| \int \chi \frac{\hat \e_j\sum_{n=0}^{j-1}(Y\pa_Y)^{j-1-n}\pa_{YY}(Y\pa_Y)^n\e}{\phi_{\ell}^2(Z)}\frac{dY}{Y} \right| \\
&\leq &  \left\{ \begin{array}{l l} CK_{j-1} e^{-\left(\frac{3}{4k}\right)s}+ \kappa  \int \chi \frac{ \hat \e_j^2}{\phi_{2k+1}^2(Z)}\frac{dY}{Y}\ \ \text{for} \ \ell=2k+1, \\ e^{-\frac{1}{k} s} \ \ \text{for} \ \ell=2k+1/2. \end{array} \right.
\eee
Since $|Z\pa_Z \phi_{\ell}/\phi_{\ell}|$ is bounded, one infers for the eighth term from \fref{eq:modulation}, \fref{bd:eweightlossy} and \fref{bd:eweightsharp} that for $\nu>0$ small enough:
$$
\left| \frac{a_s}{a}\int \chi \frac{\hat \e_j^2}{\phi_{\ell}^2(Z)}\frac{(Z\pa_Z \phi_{\ell})(Z)}{\phi_{\ell}(Z)}\frac{dY}{Y}\right| \lesssim  \left\{ \begin{array}{l l} K_J e^{-\left(\frac{3}{4k}+\nu \right)s} \ \ \text{for} \ \ell=2k+1, \\ K_{j}e^{-\left(\frac{1}{2k}+\nu \right) s} \ \ \text{for} \ \ell=2k+1/2. \end{array} \right.
$$
For the ninth term, from \fref{bd:Psi2}:
\bee
&& \int_{Y\geq 1} \frac{((Z\pa_Z)^j\Psi)^2}{\phi_{\ell}^2}\frac{dY}{Y} \lesssim e^{-\frac{k-1}{k}2s} \int_{Y\geq 1} \frac{|Z|^{8k-4}(1+|Z|)^{-12k}}{|Z|^{2\ell}(1+|Z|)^{-8k}}\frac{dY}{Y}\\
&\lesssim & e^{-\frac{k-1}{k}2s} \int_{e^{-\frac{k-1}{2k}s}}^{+\infty} |Z|^{8k-4-2\ell}(1+|Z|)^{-4k}\frac{dZ}{Z}  \lesssim  e^{-\frac{k-1}{k}2s} \lesssim e^{-\frac{3}{2k} s}
\eee
and from \fref{eq:modulation} and \fref{bd:pabF2}:
\bee
&&\int_{Y\geq 1} \frac{ |a_s (Z\pa_Z)^j\pa_a F[a](Z)|^2}{\phi_{\ell}^2(Z)}\frac{dY}{Y} \lesssim  |a_s|^2 \int_{Y\geq 1} \chi \frac{|Z|^{4k}(1+|Z|)^{-8k}}{|Z|^{2\ell}(1+|Z|)^{-8k}}\frac{dY}{Y}\\
&\lesssim & e^{-(1+\frac{1}{2k})s}\int_{e^{-\frac{k-1}{2k}s}}^{+\infty} |Z|^{4k-2\ell} \frac{dZ}{Z} \lesssim e^{-\frac{3}{2k}s}.
\eee
Therefore, using Cauchy Schwarz, \fref{bd:eweightlossy} and \fref{bd:eweightsharp}, for $\nu>0$ small enough
$$
\left| \int \chi \frac{\hat \e_j}{\phi_{\ell}(Z)} \frac{(Y\pa_Y)^j(a_s \pa_aF[a]+\Psi)}{\phi_{\ell}(Z)}\frac{dY}{Y} \right| \lesssim  \left\{ \begin{array}{l l} K_J e^{-\left(\frac{3}{4k}+\nu \right)s} \ \ \text{for} \ \ell=2k+1, \\ K_{j}e^{-\left(\frac{1}{2k}+\nu \right) s} \ \ \text{for} \ \ell=2k+1/2. \end{array} \right.
$$
Finally, for the last term, using \fref{bd:eLinfty}, \fref{bd:eweightlossy} and \fref{bd:eweightsharp} for $\nu>0$ small enough.
\bee
\left| \int \chi \frac{ \hat \e_j(Y\pa_Y)^j(\e^2)}{\phi_{\ell}^2(Z)}\frac{dY}{Y} \right| &\lesssim& \left(\int_{Y\geq 1} \frac{\hat \e_j^2}{\phi_{\ell}^2}\frac{dY}{Y} \right)^{\frac 12} \left( \sum_{n=0}^{j-1} \| (Y\pa_Y)^n \e\|_{L^{\infty}}\right)\left( \sum_{n=0}^{j} \left(\int_{Y\geq 1} \frac{((Y\pa_Y)^j \e)^2}{\phi_{\ell}^2}\frac{dY}{Y} \right)^{\frac 12} \right) \\
& \lesssim & \left\{ \begin{array}{l l} K_J e^{-\left(\frac{3}{4k}+\nu \right)s} \ \ \text{for} \ \ell=2k+1, \\ K_{j}e^{-\left(\frac{1}{2k}+\nu \right) s} \ \ \text{for} \ \ell=2k+1/2. \end{array} \right.
\eee
Combining all the above estimates, for any $\kappa>0$, for some $\nu>0$ small enough depending only on $k$, and for $s^*$ large enough, then gives the two identities:
$$
\frac{d}{ds}\left( \int \chi \frac{ ((Y\pa_Y)^j \e)^2}{\phi_{2k+1}^2(Z)}\frac{dY}{Y} \right) \leq -\left(\frac 1k -\kappa \right) \int \chi \frac{ ((Y\pa_Y)^j \e)^2}{\phi_{2k+1}^2(Z)}\frac{dY}{Y}+C\sqrt{K_{j-1}}\sqrt{K_j}e^{-\frac{3}{4k}s}
$$
$$
\frac{d}{ds}\left( \int \chi \frac{ ((Y\pa_Y)^j \e)^2}{\phi_{2k+1/2}^2(Z)}\frac{dY}{Y} \right) \leq - \frac{1}{2k} \int \chi \frac{ ((Y\pa_Y)^j \e)^2}{\phi_{2k+1}^2(Z)}\frac{dY}{Y} +CK_je^{-(\frac{1}{2k}+\nu)s}.
$$
with the convention $K_{-1}=1$. After integration in time using \fref{bd:decayextinit}, and \fref{bd:eloc} for the zone $1\leq Y\leq 2$, the above differential inequality shows that, upon choosing the constants $K_j$ inductively one after another, \fref{bd:lossy improved} holds true.\\

\noindent \textbf{Step 2:} \emph{Proof of \fref{bd:lossy improved3}}. Let $j\geq 2k+1$. We claim that this bound can be proved the very same way as in step 1. The main argument is the following. $\bar \e_j:=\pa_Z^j \e$ solves from \fref{eq:evolutione}:
\bee
&&\pa_s \bar \e_j+\frac{j+2k}{2k}\bar \e_j+\frac{Y}{2}\pa_Y \bar \e_j-2F_k(Z)\bar \e_j-\pa_{YY}\bar \e_j \\
&=& 2(\pa_Z^j(F_k(Z)\e)-F_k(Z)\bar \e_j)-2\pa_Z^j((F_k(Z)-F[a])\e)-\pa_Z^j(a_s\pa_a F[a]+\Psi).
\eee
The function $\phi_0(Z)$ is a stable eigenfunction of the operator without dissipation in the left hand side:
$$
\left(\frac{j+2k}{2k}+\frac{Y}{2}\pa_Y -2F_k(Z)\right)\phi_0(Z)=\frac{j-2k}{2k}\phi_0
$$
and in particular $(j-2k)/2k\geq 1/(2k)$ since $j\geq 2k+1$. As $1/(2k)>1/(4k)$, one can then prove \fref{bd:lossy improved3} as in Step 1, checking that all the terms in the right hand side of the equation for $w$ are lower order, and that the boundary terms at the origin are controlled by \fref{bd:eloc}.

\end{proof}

We can now end the proof of Proposition \ref{pr:bootstrap}.

\begin{proof}[Proof of Proposition \ref{pr:bootstrap}]

We assume that $\e_0$ is fixed satisfying the orthogonality conditions \fref{id:decomposition e} and the initial bounds \fref{bd:eL2rhoinit}, \fref{bd:decayextinit2} and \fref{bd:decayextinit}. We assume $K_J\geq ...\geq K_0>0$ are fixed so that the Lemmas \ref{NLH:lem:energyrho} and \ref{NLH:lem:improvedoustide} hold true. We let $\tilde K\geq K_J$ to be fixed at the end of the proof. We use Lemma \ref{lem:implicit} to relate the initial coefficients $(\bar c_{2\ell}(s_0))_{0\leq \ell \leq k-1}$ and $(c_{2\ell}(s_0))_{0\leq \ell \leq k-1}$. We consider for all possible initial values $(c_{2\ell}(s_0))_{0\leq \ell \leq k-1}\in B(\tilde K e^{-(k-\frac 12+\frac{1}{4k})s_0})$ the corresponding solution to \fref{eq:f} with initial datum \fref{bd:initialdecompoNLH}.

We let the exit time $s_e\in [s_0,+\infty]$ be the supremum of times $s_1\geq s_0$ such that the solution is trapped on $[s_0,s_1]$. From Definition \ref{def:trappedNLH} and Lemmas \ref{NLH:lem:energyrho} and \ref{NLH:lem:improvedoustide} and a continuity argument, if $s_e<+\infty$ then necessarily the inequality \fref{bd:cell bootstrap} is saturated at time $s_e$:
$$
|(c_{2\ell}(s_e))_{0\leq \ell \leq k-1}|=\tilde K e^{-(k-\frac 12+\frac{1}{4k})s_e}.
$$
Hence the exit mapping $\Phi$ which to $(c_{2\ell}(s_0))_{0\leq \ell \leq k-1}\in B(\tilde K e^{-(k-\frac 12+\frac{1}{4k})s_0})$ such that $s_e<+\infty $ associates
$$
\Phi ((c_{2\ell}(s_0))_{0\leq \ell \leq k-1})= e^{-(k-\frac 12+\frac{1}{4k})(s_0-s_e)}(c_{2\ell}(s_e))_{0\leq \ell \leq k-1}
$$
is a mapping whose domain is a subset of the ball $B(\tilde K e^{-(k-\frac 12+\frac{1}{4k})s_0})$ and whose range lies in its boundary the sphere $S(\tilde K e^{-(k-\frac 12+\frac{1}{4k})s_0})$.

We claim that one can choose $\tilde K\geq K_J$ such that for $s^*$ large enough, $\Phi$ is the identity map on the sphere $S(\tilde K e^{-(k-\frac 12+\frac{1}{4k})s_0})$. Indeed, for any $\tilde K\geq K_J$ if initially $|(c_{2\ell}(s_0))_{0\leq \ell \leq k-1}|=\tilde K e^{-(k-\frac 12+\frac{1}{4k})s_0}$ then for $s^*$ large enough one computes from \fref{eq:modulation} the outgoing flux condition
\bee
&&\pa_s \left( \frac{\sum_{\ell=0}^{2k-1}|c_{\ell}(s)|^2}{\tilde K^2 e^{-2(k-\frac 12+\frac{1}{4k})s}} \right)(s_0)=2\left(k-\frac 12 +\frac{1}{4k} \right)\frac{\sum_{\ell=0}^{2k-1}|c_{\ell}(s_0)|^2}{\tilde K^2 e^{-2(k-\frac 12+\frac{1}{4k})s_0}}+2\frac{\sum_{\ell=0}^{2k-1}c_{\ell}(s_0)\pa_s c_{\ell}(s_0)}{\tilde K^2 e^{-2(k-\frac 12+\frac{1}{4k})s_0}}\\
&=&2\left(k-\frac 12 +\frac{1}{4k} \right)\frac{\sum_{\ell=0}^{2k-1}|c_{\ell}(s_0)|^2}{\tilde K^2 e^{-2(k-\frac 12+\frac{1}{4k})s_0}}+2\frac{\sum_{\ell=0}^{2k-1}c_{\ell}(s_0)\left(-\frac{\ell-2}{2}c_{\ell}(s_0)+O(K_Je^{-(k-\frac 12 +\frac{1}{4k})s_0}) \right)}{\tilde K^2 e^{-2(k-\frac 12+\frac{1}{4k})s_0}} \\
&=&\frac{2}{\tilde K^2 e^{-2(k-\frac 12+\frac{1}{4k})s_0}}\sum_{\ell=0}^{2k-1}\left(k-\frac 12 +\frac{1}{4k}-\frac{\ell-2}{2}\right) |c_{\ell}(s_0)|^2+O\left(\frac{K_J}{\tilde K} \right) \\
&\geq & 1+\frac{1}{4k}+O\left(\frac{K_J}{\tilde K} \right)
\eee
One then chooses $\tilde K\geq K_J$ depending on $K_J$ such that the above quantity is $\geq 1$. Then in that case the solution leaves the trapped regime immediately at time $s_0$ since \fref{bd:cell bootstrap} fails right after $s_0$. Hence $s_e=s_0$, and $\Phi((c_{2\ell}(s_0))_{0\leq \ell \leq k-1})=((c_{2\ell}(s_0))_{0\leq \ell \leq k-1})$ is indeed the identity map on the sphere $S(\tilde K e^{-(k-\frac 12+\frac{1}{4k})s_0})$.

From the above outgoing flux condition and a standard continuity argument, $\Phi$ is continuous on its domain. $\Phi$ is thus a continuous mapping from the ball onto its boundary, that is the identity on its boundary. The domain of $\Phi$ cannot be the entire ball, as this would contradict Brouwer's fixed point Theorem. This means that there exists one $(c_{2\ell}(s_0))_{0\leq \ell \leq k-1}\in B(\tilde K e^{-(k-\frac 12+\frac{1}{4k})s_0})$ such that $s_e=+\infty$, which ends the proof of the Proposition.

\end{proof}


\subsection{The coupled linear heat equation} \lab{sec:LFH}

We now turn to the proof of Proposition \ref{pr:LHinstable}. We keep the notations of the previous section. The proof is similar and simpler to the one of the related Theorem \ref{th:NLHinstable} concerning $\xi$. Indeed, there are no nonlinear effects and instabilities. For a solution $\zeta$ to $(LFH)$ we start by going again to self-similar variables
$$
Y=\frac{y}{\sqrt{T-t}}, \ \ s=-\log (T-t), \ \ g(s,y)=(T-t)^4\zeta(x,t), \ \ Z:=\frac{a^*Y}{e^{\frac{k-1}{2k}s}},
$$
Then $g$ solves the second equation in \fref{eq:f}. Throughout this section, we assume that $k$ and $J$ are fixed, and that $f$ is the solution to the first equation in \fref{eq:f} satisfying the properties of Proposition \ref{pr:bootstrap}. In particular, in the current subsection, all the constants appearing in the previous subsection are considered as fixed and universal. Without loss of generality for the argument, since its exact value will never play a role, we fix:
$$
a^*=1.
$$
In particular since $f=F_k(a^*e^{-(k-1)/(2k)}Y)$ at leading order, and since the dissipation is lower order, the main order equation reads in $Z$ variable:
$$
g_s+\mathcal M_Z g=0, \ \ \mathcal M_Z:=4-4F_k(aZ)+\frac{Z}{2k} \pa_Z.
$$

\begin{proposition}[Spectral structure for $\mathcal M_Z$] \lab{pr:mathcalMZ}

The operator $\mathcal M_Z$ acting on $\mathcal C^{\infty}(\mathbb R)$ has point spectrum $\Upsilon (\mathcal M_Z)=\left\{\ell/(2k), \ \ \ell\in \mathbb N \right\}$ and the associated eigenfunctions are
$$
\psi_{\ell}:=\frac{Z^{\ell}}{(1+(aZ)^{2k})^4}, \ \ \mathcal M_Z \psi_\ell=\psi_\ell .
$$

\end{proposition}

\begin{proof}

The result comes from a direct computation.

\end{proof}

$\mathcal M_Z$ having a nontrivial kernel and non-negative spectrum, we expect formally the solution to approach an element of its kernel as $s\rightarrow +\infty$. Near the origin, as $f=F_k(a^*e^{-(k-1)/(2k)}Y)$ at leading order and since $F_k(0)=1$, the main order equation for $g$ in the zone $|Y|\lesssim 1$ is:
$$
g_s+\mathcal M_\rho g=0, \ \ \mathcal M_\rho:=\frac{Y}{2} \pa_Y-\pa_{YY}.
$$

\begin{proposition}[Linear structure at the origin, (see e.g. \cite{MZ}] \lab{pr:Mrho}

The operator $\mathcal M_{\rho}$ is essentially self-adjoint on $C^2_0(\mathbb R)\subset L^2(\rho)$ with compact resolvant. The space $H^2_\rho$ is included in the domain of its unique self-adjoint extension. Its spectrum is $ \Upsilon (\mathcal M_\rho) = \left\{\ell/2, \ \ell \in \mathbb N \right\}$. The eigenvalues are all simple and the associated orthonormal basis of eigenfunctions is given by the family of Hermite polynomials $(h_{\ell})_{\ell \in \mathbb N}$ defined by \fref{def:hell}.

\end{proposition}

We now perform a bootstrap analysis and decompose the solution according to:
\be \lab{eq:decompo v}
g=b(s)F^4_k\left(Z^*\right)+\e
\ee
where $Z^*=a^*e^{-\frac{k-1}{2k}s}Y$ and without loss of generality since the value of $a$ never plays a role we take $a^*=1$ for simplicity, which fixes $
Z^*=e^{-\frac{k-1}{2k}s} Y=Z$, and where $b$ is fixed through the orthogonality condition for the $\langle \cdot,\cdot \rangle_\rho$ scalar product:
\be \lab{id:ortho2}
\e \perp h_0.
\ee

\begin{proposition} \lab{pr:bootstrap2}

There exist $L_{J} \geq ... \geq L_0>0$ and $s_0\gg 1$ large enough, such that for $\e (s_0)=\e_0$ satisfying the orthogonality condition \fref{id:ortho2} and
\be \lab{bd:bootstrap 2 psi1 init}
\sum_{j=0}^J \int_{|Y|\geq 1} \frac{|(Z\pa_Z)^j \e_0|^2}{\psi_{1}^2(Z)}\frac{dY}{|Y|}\leq e^{-\frac{3}{4k}s_0}, \ \ \ \sum_{j=0}^J  \int_{|Y|\geq 1} \frac{|(Z\pa_Z)^j \e_0|^2}{\psi_{1/2}^2(Z)}\frac{dY}{|Y|}\leq e^{-\frac{1}{2k}s_0},
\ee
\be \lab{bd:bootstrap 2 Z}
\sum_{j=1}^J  \int_{|Y|\geq 1} \frac{| \pa_Z^j \e_0|^2}{\psi_0^2(Z)}\frac{dY}{|Y|}\leq e^{-\frac{1}{2k}s_0},
\ee
\be \lab{bd:bootstrap 2 rho init}
\|  \e_0 \|_{L^2_{\rho}}\leq e^{-\frac 12s_0}, \ \ \|  \pa_Z^j \e_0 \|_{L^2_{\rho}}\leq e^{-\frac{1}{2k}s_0}, \ \ \text{for} \ j=1,...,J,
\ee
and an initial parameter $3/4\leq b \leq 5/4$ the solution $g$ to the second equation in \fref{eq:f} then satisfies for all $s\geq s_0$, for $j=0,...,J$:
\be \lab{bd:bootstrap 1 Z}
\int_{|Y|\geq 1} \frac{|(Z\pa_Z)^j \e|^2}{\psi_{1}^2(Z)}\frac{dY}{|Y|}\leq L_j e^{-\frac{3}{4k}s}, \ \ \  \int_{|Y|\geq 1} \frac{|(Z\pa_Z)^j \e|^2}{\psi_{1/2}^2(Z)}\frac{dY}{|Y|}\leq L_{j} e^{-\frac{1}{2k}s},
\ee
\be \lab{bd:bootstrap 3 Z}
 \int_{|Y|\geq 1} \frac{| \pa_Z^j \e|^2}{\psi_0^2(Z)}\frac{dY}{|Y|}\leq L_j e^{-\frac{1}{2k}s},
\ee
\be \lab{bd:bootstrap 2 rho2}
\|  \e \|_{L^2_{\rho}}\leq \sqrt{L_0} e^{-\frac 12s}, \ \ \|  \pa_Z^j \e \|_{L^2_{\rho}}\leq L_J e^{-\frac{1}{2k}s}, \ \ \text{for} \ j=1,...,J,
\ee
and there exists an asymptotic parameter $1/2\leq b^*\leq 3/2$ such that $ |b-b^*|\lesssim e^{-(k-1)s}$.
 
\end{proposition}

The rest of the subsection is devoted to the proof of Proposition \ref{pr:bootstrap2}. In what follows we assume that $g$ solves \fref{eq:f} and satisfies the bounds of Proposition \ref{pr:bootstrap2} on some time interval $[s_0,s_1]$, and perform modulation and energy estimates to improve those bounds. Proposition \ref{pr:bootstrap2} is then proved at the end of the subsection.

\begin{lemma}

There holds on $[s_0,s_1]$ for $j=0,...,J-1$:
\be \lab{2bd:Linfty}
|\pa_Z^j \e |\lesssim L_J e^{-\frac{1}{4k}s}(1+|Z|)^{\frac 12 -8k}.
\ee

\end{lemma}

\begin{proof}

First, since $\pa_Z=e^{\frac{k-1}{2k}s}\pa_Y$, one deduces from \fref{bd:bootstrap 2 rho2} that $\| \e \|_{H^j_\rho}\lesssim e^{-s/(2k)}$. Therefore, from Sobolev,
\be \lab{bd:2Linftyorigine}
|\pa_Z^j \e |\lesssim L_Je^{-\frac{1}{2k}s} \ \ \text{for} \ \ |Y|\leq 1.
\ee
Then, applying \fref{bd:Sobolev}, as $|Z\pa_Z \psi_{1/2}|\lesssim |\psi_{1/2}|\lesssim |Z\pa_Z \psi_{1/2}|$:
\bee
\| \frac{Z^j\pa_Z^j\e}{\psi_{1/2}(Z)} \|_{L^{\infty}(\{|Y|\geq 1 \})}^2 &\lesssim& \| \frac{Z^j\pa_Z^j\e}{\psi_{1/2}(Z)} \|_{L^2\left(\{|Y|\geq 1 \},\frac{dY}{|Y|}\right)}^2 +\| Z\pa_Z(\frac{Z^j\pa_Z^j\e}{\psi_{1/2}(Z)}) \|_{L^2\left(\{|Y|\geq 1 \},\frac{dY}{|Y|}\right)}^2 \\
&\lesssim & \sum_{k=0}^{j+1} \| \frac{(Z\pa_Z)^j\e}{\psi_{1/2}(Z)} \|_{L^2\left(\{|Y|\geq 1 \},\frac{dY}{|Y|}\right)}^2 \lesssim L_Je^{-\frac{1}{2k}s}
\eee
from \fref{bd:bootstrap 1 Z}. From the definition of $\psi$ this implies that
\be \lab{bd:2Linftyaway1}
|\pa_Z^j \e |\lesssim L_Je^{-\frac{1}{2k}s}|Z|^{\frac 12 -j}(1+|Z|)^{-8k} \ \ \text{for} \ \ |Y|\geq 1.
\ee
Finally, since $\psi_{1/2}(Z)=|Z|^{1/2}\psi_0(Z)$, for $j\geq 1$ one has the inequality
$$
\frac{|Z|}{\psi_0(Z)}\lesssim \frac{1}{\psi_0(Z)}+\frac{|Z|^{j+1}}{\psi_{1/2}(Z)}.
$$
This implies from \fref{bd:bootstrap 1 Z} and \fref{bd:bootstrap 3 Z} the estimate for $j\geq 1$:
$$
\int_{|Y|\geq 1} \frac{|Z\pa_Z (\pa_Z^j \e)|^2}{\psi_{0}^2(Z)}\frac{dY}{|Y|} \lesssim \int_{|Y|\geq 1} \frac{| \pa_Z^{j+1} \e |^2}{\psi_{0}^2(Z)}\frac{dY}{|Y|}+\int_{|Y|\geq 1} \frac{| Z^{j+1}\pa_Z^{j+1} \e)|^2}{\psi_{0}^2(Z)}\frac{dY}{|Y|} \lesssim L_Je^{-\frac{1}{2k}s}.
$$
Thus, from \fref{bd:Sobolev} one obtains for $j\geq 1$:
\be \lab{bd:2Linftyaway2}
|\pa_Z^j \e|\lesssim L_J e^{-\frac{1}{4k}s}(1+|Z|)^{-8k} \ \ \text{for} \ |Y|\geq 1.
\ee
The bounds \fref{bd:2Linftyorigine}, \fref{bd:2Linftyaway1} and \fref{bd:2Linftyaway2} then imply the desired result \fref{2bd:Linfty}.

\end{proof}

From \fref{eq:f} and \fref{eq:decompo v}, the evolution of $\e$ is given by the following equation:
\be \lab{eq:evo e2}
b_s F^4_k(Z)+\e_s+\mathcal M \e +\tilde{\mathcal M} \e+\Psi=0,
\ee
where
$$
\mathcal M := 4-4F_k(Z)+\frac Y2 \pa_Y -\pa_{YY}, \ \ \tilde{\mathcal M}:=-4(f-F_k(Z)),
$$
and
$$
\Psi:= -4b F_k^4(Z)(f-F_k(Z))-b( e^{-\frac{k-1}{2k}s})^2 \pa_{ZZ}(F_k^4)(Z)
$$
From the various bounds of Proposition \ref{pr:bootstrap} and Lemma \ref{NLH:lem:pointwise} we infer the following estimates for the above objects. We recall that the constants of the previous Subsection are fixed and considered as universal.

\begin{lemma}

One has the following bounds:
\be \lab{bd:tildemathcalMLinfty}
|\pa_Z^j(f-F_k(Z))|\lesssim e^{-\frac{1}{4k}s}(1+|Z|)^{\frac 12-2k-j},
\ee
\be \lab{bd:tildemathcalMrho}
\| f-F_k(Z)\|_{L^2_\rho}\lesssim e^{-(k-1)s},
\ee
$$
|(Z\pa_Z)^j(\Psi)|\lesssim e^{-\frac{1}{4k}s}(1+|Z|)^{\frac 12-10k}, \ \ j=0,1,2
$$
\be \lab{2bd:Psirho}
\| \Psi \|_{L^2_\rho}\lesssim  e^{-(k-1)s}, \ \ \| \pa_Z\Psi \|_{L^2_\rho}\lesssim e^{-\left(k-\frac 32 +\frac{1}{2k}\right)s} \ \ \text{and} \ \ \| \pa_Z^j\Psi \|_{L^2_\rho}\lesssim e^{-\frac{1}{2k}s} \ \ \text{for} \ j\geq 2,
\ee
\be \lab{2bd:PsiZ}
\int_{|Y|\geq 1} \frac{|(Y\pa_Y) ^j\Psi |^2}{\psi_{1}^2(Z)}\frac{dY}{|Y|} +\int_{|Y|\geq 1} \frac{|(Y\pa_Y) ^j \Psi |^2}{\psi_{1/2}^2(Z)}\frac{dY}{|Y|}   \lesssim e^{-\frac{3}{4k}s}, \ \ j=0,...,J,
\ee
\be \lab{2bd:PsiZ2}
\int_{|Y|\geq 1} \frac{|\pa_Z ^j\Psi |^2}{\psi_0^2(Z)}\frac{dY}{|Y|}  \lesssim e^{-\frac{1}{2k}s}, \ \ j=1,...,J.
\ee

\end{lemma}

\begin{proof}

These are direct bounds implied by the estimates of Proposition \ref{pr:bootstrap}, and the behaviour of the corresponding eigenfunctions given by Propositions \ref{pr:Fk} and \ref{pr:mathcalMZ}.

\end{proof}

We start by computing the evolution of the modulation parameter.

\begin{lemma} \lab{LH:lem:modulation}

There exists $C>0$ independent of $L_1,...,L_J$ such that for $s_0$ large enough on $[s_0,s_1]$:
\be \lab{2bd:mod}
|b_s|\leq Ce^{-(k-1)s}.
\ee

\end{lemma}

\begin{proof}

One takes the scalar product between \fref{eq:evo e2} and $h_0=1$ in $L^2_\rho$, yielding using \fref{id:ortho2}:
$$
b_s \langle F^4(Z^*),h_0 \rangle_\rho = \langle -\mathcal M \e -\tilde{\mathcal M} \e-\Psi,h_0\rangle_\rho .
$$
First, since $|F_k(Z)-1|\lesssim Z^{2k}(1+|Z|)^{-2k}$ one has the projection in the left hand side is non-degenerate:
$$
\langle F^4(Z^*),h_0 \rangle_\rho =1+O(e^{-(k-1)s})
$$
and that, since $\mathcal M =\mathcal M_\rho+4(1-F(Z^*))$:
$$
 \langle \mathcal M \e ,h_0\rangle_\rho =0+  \langle \e ,4(1-F(Z^*))h_0\rangle_\rho =O(\sqrt{L_0}e^{-(k-\frac 12)s}).
$$
using \fref{id:ortho2}, \fref{bd:bootstrap 2 rho2} and the fact that $|Z^{2k}|\lesssim e^{-(k-1)s}|Y|^{2k}$. Then, from \fref{bd:tildemathcalMrho} one computes:
$$
\langle \tilde{\mathcal M} \e,h_0\rangle_\rho =- 4 \langle \e, (f-F(Z^*)) h_0 \rangle_\rho=O(\sqrt{L_0}e^{-(k-\frac 12)s}).
$$
Finally, using \fref{2bd:Psirho}:
$$
\langle \Psi,h_0 \rangle_\rho =O(e^{-(k-1)s}).
$$
From the above identities one gets the desired result \fref{2bd:mod}.

\end{proof}

We then perform an energy estimate in the zone $|Y|\lesssim 1$.

\begin{lemma} \lab{LH:lem:energyrho}

There exist constants $L_J\geq ...\geq L_0>0$ such that for $s_0$ large enough, at time $s_1$:
\be \lab{bd:bootstrap 2 rho}
\|  \e \|_{L^2_{\rho}}\leq \frac{\sqrt{L_0}}{2} e^{-\frac 12s_1}, \ \ \|  \pa_Z^j \e \|_{L^2_{\rho}}\leq \frac{L_J}{2} e^{-\frac{1}{2k}s_1}, \ \ \text{for} \ j=1,...,J.
\ee

\end{lemma}

\begin{proof}

\noindent \textbf{Step 1} \emph{Estimate for $\e$}. Let $0<\kappa \ll 1$ be an arbitrarily small constant. Then from \fref{eq:evo e2} we infer that:
$$
\frac{d}{ds} \frac 12 \int \e^2 e^{-\frac{Y^2}{4}}dY= \int \e \left(-b_s F^4(Z^*)-\mathcal M \e -\tilde{\mathcal M} \e-\Psi \right)e^{-\frac{Y^2}{4}}dY.
$$
For the first term, from \fref{id:ortho2} and \fref{2bd:mod}, using Cauchy-Schwarz:
$$
b_s \int \e F^4(Z^*)e^{-\frac{Y^2}{4}}dY=b_s \int \e (1-F^4(Z^*))e^{-\frac{Y^2}{4}}dY =O( \sqrt{L_0}e^{-(2(k-1)+\frac 12 )s}).
$$
For the second, as $\mathcal M =\mathcal M_\rho+4(1-F_k(Z))$, using \fref{2bd:Linfty}:
\bee
&&- \int \e \mathcal M \e e^{-\frac{Y^2}{4}}dY = -\int \e \mathcal M_\rho \e e^{-\frac{Y^2}{4}}dY+4 \int (F_k(Z)-1)\e^2 e^{-\frac{Y^2}{4}}dY \\
&\leq& -\frac 12 \int \e^2 e^{-\frac{Y^2}{4}}dY +O(\| \e \|_{L^2_\rho}\| \e \|_{L^{\infty}}\|F_k(Z)-1 \|_{L^2_\rho})\leq -\frac 12 \int \e^2 e^{-\frac{Y^2}{4}}dY +O( \sqrt{L_0}L_J e^{-(k-\frac 12 +\frac{1}{4k})s}).
\eee
For the third, from \fref{bd:tildemathcalMLinfty} and \fref{bd:bootstrap 2 rho2}:
$$
\left| \int \e \tilde{\mathcal M} \e e^{-\frac{Y^2}{4}}dY\right|\lesssim \| \e \|_{L^2_\rho}^2 \| f-F(Z^*) \|_{L^{\infty}}\lesssim e^{-(1+\frac{1}{4k})s}.
$$
Finally, for the fourth, from \fref{2bd:Psirho}:
$$
\left| \int \e \Psi e^{-\frac{Y^2}{4}}dY  \right|\lesssim \| \e \|_{L^2_\rho}\| \Psi \|_{L^2_\rho}\lesssim \sqrt{L_0}e^{-(k-\frac 12)s}.
$$
Combining the above expressions one obtains as $L_J\geq ...\geq L_0>0$, assuming $L_J\geq 1$ without loss of generality:
$$
\frac{d}{ds}  \left( \int \e^2 e^{-\frac{Y^2}{4}}dY\right)\leq - \int \e^2 e^{-\frac{Y^2}{4}}dY+C L_j e^{-(1+\frac{1}{4k})s}.
$$
When reintegrated in time, using \fref{bd:bootstrap 2 rho init} this gives:
$$
 \int \e^2 e^{-\frac{Y^2}{4}}dY\leq e^{-s}+CL_Je^{-\frac{1}{4k}s_0}e^{-s}\leq \frac{L_0^2}{4}e^{-s}
$$
provided $L_0>2$ and $s_0$ has been taken large enough.\\

\noindent \textbf{Step 2:} \emph{Higher order derivatives.} Let $1\leq j \leq J$ and define $w:=\pa_Z^j \e$. Then from \fref{eq:evo e2} the evolution of $w$ is
$$
w_s+\frac{j}{2k}+\mathcal M_\rho w+4(1-F_k(Z))w+4(F_k(Z)w-\pa_Z^j(F_k(Z)\e))-\pa_Z^j((f-F_k(Z))\e)+\pa_Z^j (\Psi+b_sF_k(Z))=0.
$$
From the above equation, we infer that:
\bee
\frac{d}{ds}\frac 12 \| w\|_{L^2_\rho}^2 & = & -\frac{j}{2k} \| w\|_{L^2_\rho}^2- \| \pa_Y w\|_{L^2_\rho}^2+4\langle (F_k(Z)-1)w,w\rangle_\rho+4\langle \pa_Z^j(F_k(Z)\e)-F_k(Z)w,w\rangle_\rho \\
&&+\langle \pa_Z^j((f-F_k(Z))\e),w\rangle_\rho-\langle \pa_Z^j (\Psi+b_sF_k(Z)),w\rangle_\rho.
\eee
Let $0<\nu\ll1$ be a small constant to be fixed later on. We estimate all terms in the right hand side. First, from \fref{eq:Poincare} and \fref{bd:bootstrap 2 rho2} one has:
$$
- \| \pa_Y w\|_{L^2_\rho}^2 \leq -c e^{-\nu s} \| Y w\|_{L^2_\rho}^2+e^{-\nu s} \| w\|_{L^2_\rho}^2 \leq -c e^{-\nu s} \| Y w\|_{L^2_\rho}^2+CL_j e^{-\left(\frac{1}{k}+\nu \right) s} , \ \ c>0.
$$
Next, since $|1-F_k(Z)|\lesssim |Z^{2k}|(1+|Z|)^{-2k}\lesssim e^{-(k-1)s/k}|Y|^2$ we infer that
$$
\left| \langle (F_k(Z)-1)w,w\rangle_\rho \right| \lesssim e^{-\frac{k-1}{k}s} \| Yw\|_{L^2_\rho}^2.
$$
From \fref{bd:bootstrap 2 rho2}, as $\pa_Z^j F_k$ is bounded, we infer using Cauchy-Schwarz that:
$$
\left| \langle \pa_Z^j(F_k(Z)\e)-F_k(Z)w,w\rangle_\rho \right| \lesssim \left\{ \begin{array}{l l} \sqrt{L_{0}}\sqrt{L_1} e^{-\left(\frac 12 +\frac{1}{2k} \right)s} \ \ \text{for} \ j=1, \\ \sqrt{ L_{j-1}}L_J e^{-\frac{1}{k} s} \ \ \text{for} \ j\geq 2.  \end{array} \right.
$$
Using \fref{bd:bootstrap 2 rho2} and \fref{bd:tildemathcalMLinfty} we infer:
$$
\left| \langle \pa_Z^j((f-F_k(Z))\e),w\rangle_\rho \right| \lesssim \left(\sum_{i=0}^J \| \pa_Z^i \e \|_{L^2_\rho}^2 \right)\left(\sum_{i=0}^J \| \pa_Z^i (f-F_k(Z)) \|_{L^{\infty}}\right)\lesssim L_je^{-\left(\frac{1}{k}+\frac{1}{4k}\right)s}.
$$
Finally, from \fref{bd:bootstrap 2 rho2} and \fref{2bd:mod}, \fref{2bd:Psirho} and Cauchy-Schwarz:
$$
\left| \langle \pa_Z^j (\Psi+b_sF_k(Z)),w\rangle_\rho \right| \lesssim  \left\{ \begin{array}{l l} \sqrt{L_1}e^{-\left(\frac 12 +\frac{1}{k} \right)s} \ \ \text{for} \ j=1, \\ L_J e^{-\frac{1}{k} s} \ \ \text{for} \ j\geq 2.  \end{array} \right.
$$
Collecting the above estimates, one finds finally that there exists $\nu>0$ depending on $k$, such that for  $s_0$ large enough:
$$
\frac{d}{ds} \| w\|_{L^2_\rho} \leq  \left\{ \begin{array}{l l} -\frac{1}{k}\| w\|_{L^2_\rho} +O(L_je^{-\left(\frac{1}{k}+\nu\right)s}) \ \ \text{for} \ j=1, \\ -\frac{j}{k}\| w\|_{L^2_\rho} +\sqrt{L_{j-1}}L_J e^{-\frac{1}{k} s} \ \ \text{for} \ j\geq 2.  \end{array} \right.
$$
Reintegrated in time using \fref{bd:bootstrap 2 rho init}, this yields the desired bound \fref{bd:bootstrap 2 rho2} for $j\geq 1$, upon choosing the constants $L_1$,...,$L_J$ inductively.

\end{proof}

\begin{lemma} \lab{LH:lem:energyoutside}

There exist constants $K_J\geq ...\geq L_0>0$ such that for $s_0$ large enough, at time $s_1$, for $j=0,...,J$:
\be \lab{bd:2eweightimproved1}
\int_{|Y|\geq 1} \frac{|(Z\pa_Z)^j \e|^2}{\psi_{1}^2(Z)}\frac{dY}{|Y|}\leq K_j e^{-\frac{3}{4k}s}, \ \ \  \int_{|Y|\geq 1} \frac{|(Z\pa_Z)^j \e (s_1)|^2}{\psi_{1/2}^2(Z)}\frac{dY}{|Y|}\leq \frac{K_{j+1}}{2} e^{-\frac{1}{2k}s},
\ee
\be \lab{bd:2eweightimproved3}
 \int_{|Y|\geq 1} \frac{| \pa_Z^j \e|^2}{\psi_0^2(Z)}\frac{dY}{|Y|}\leq \frac{K_j}{2} e^{-\frac{1}{2k}s} \ \ \text{for} \ \ j\geq 1.
\ee

\end{lemma}

\begin{proof}

We only perform the analysis for $Y\geq 1$ since it is exactly the same for $Y\leq -1$, thus writing $|Y|=Y$.\\

\noindent \textbf{Step 1:} \emph{Bounds for $\e$.} Let $0<\kappa \ll 1$ be an arbitrarily small constant. Let $\chi$ be a smooth and positive cut-off function, $\chi=1$ for $ Y \geq 2$ and $\chi=0$ for $Y\leq 1$. Let $\ell=1$ or $\ell=1/2$. We compute first the identity by integrating by parts and using Proposition \ref{pr:mathcalMZ}:
\bee
\frac{d}{ds} \frac 12 \left( \int \chi \frac{\e^2}{\psi_\ell^2(Z)}\frac{dY}{|Y|} \right)&=& -\frac{\ell}{2k}\int \chi \frac{\e^2}{\psi_\ell^2(Z)}\frac{dY}{|Y|}-\int \chi \frac{|\pa_Y \e|^2}{\psi_\ell^2(Z)}\frac{dY}{|Y|} +\frac 14 \int \pa_Y \chi \frac{\e^2}{\psi_{\ell}^2(Z)}dY \\
&&+\frac 12 \int \e^2\left(\frac{\pa_{YY}\chi}{\psi_\ell^2(Z)Y}+2\pa_Y \chi \pa_Y \left(\frac{1}{\psi_\ell^2(Z)Y}\right)+\chi\pa_{YY}\left(\frac{1}{\psi_\ell^2(Z)Y}\right)  \right)dY\\
&&+\int \frac{\e^2}{\psi_\ell^2(Z)}4(u-F_k(Z))\frac{dY}{|Y|}-\int \chi \frac{\e}{\psi_\ell^2(Z)}\left(b_s F^4_k(Z)+\Psi \right)\frac{dY}{|Y|}.
\eee
We treat the boundary terms using \fref{bd:bootstrap 2 rho2}:
\bee
&&\left|\frac 14 \int \pa_Y \chi \frac{\e^2}{\psi_{\ell}^2(Z)}dY+\frac 12 \int \e^2\left(\frac{\pa_{YY}\chi}{\psi_\ell^2(Z)Y}+2\pa_Y \chi \pa_Y \left(\frac{1}{\psi_\ell^2(Z)Y}\right) \right)dY \right| \\
&\lesssim & \| \e \|_{L^2_\rho}^2 \left(\| \frac{1}{\psi_\ell^2(Z)}\|_{L^{\infty}(1\leq Y\leq 2)}+\| \pa_Y\left(\frac{1}{\psi_\ell^2(Z)}\right)\|_{L^{\infty}(1\leq Y\leq 2)}\right) \lesssim e^{-s+\frac{k-1}{k}\ell s}\\
&\lesssim & \left\{ \begin{array}{l l} L_0 e^{-\left(\frac{3}{4k}+\frac{1}{4k} \right)s} \ \ \text{if} \ \ell=1, \\ L_0e^{-\left(\frac{1}{2k}+\frac 12 \right)s} \ \ \text{if} \ \ell=\frac 12. \end{array} \right.
\eee
Next, notice that for $Y\geq 1$,
$$
\left| \pa_{YY}\left(\frac{1}{\psi_\ell^2(Z)Y}\right)\right|\lesssim \frac{1}{Y^3\psi_\ell^2(Z)}.
$$
If $\ell=1$, we then decompose for $Y^*$ large enough depending on $\kappa$:
\bee
\left| \int \e^2 \chi\pa_{YY}\left(\frac{1}{\psi_1^2(Z)Y}\right) dY\right| & \leq & C \int \chi \frac{\e^2}{\psi_1^2(Z)Y^2}\frac{dY}{|Y|} \\
&\leq &  C \int_{Y\leq Y^*} \chi \frac{\e^2}{\psi_1^2(Z)Y^2}\frac{dY}{|Y|} +C \int_{Y\geq Y^*} \chi \frac{\e^2}{\psi_1^2(Z)Y^2}\frac{dY}{|Y|} \\
&\leq &  C \| \e \|_{L^2_\rho}^2 \| \frac{1}{\psi_1^2(Z)}\|_{L^{\infty}(1\leq Y \leq Y^*)}+\kappa \int_{Y\geq Y^*} \chi \frac{\e^2}{\psi_1^2(Z)}\frac{dY}{|Y|}\\
&\leq & C L_0 e^{-\left(\frac{3}{4k}+\frac{1}{4k}\right)s}+\kappa \int \chi \frac{\e^2}{\psi_1^2(Z)Y^2}\frac{dY}{|Y|}.
\eee
If $\ell=1/2$, we use the fact that $1/(Y\psi_{1/2}^2(Z))=e^{-(k-1)/(2k)s}/\psi_1^2(Z)$ to obtain from \fref{bd:bootstrap 1 Z}:
$$
\left| \int \e^2 \chi\pa_{YY}\left(\frac{1}{\psi_{1/2}^2(Z)Y}\right) dY\right| \lesssim \int \chi \frac{\e^2}{\psi_{1/2}^2(Z)Y^2}\frac{dY}{|Y|} \lesssim e^{-\frac{k-1}{2k}s}\int \chi \frac{\e^2}{|Y|\psi_1^2(Z)}\frac{dY}{|Y|} \lesssim e^{-\left(\frac{1}{2k}+\frac{2k-1}{4k}\right)s}.
$$
The lower order linear term is estimated via \fref{bd:tildemathcalMLinfty} and \fref{bd:bootstrap 1 Z}:
$$
\left| \int \frac{\e^2}{\psi_\ell^2(Z)}4(f-F_k(Z))\frac{dY}{|Y|} \right|\lesssim \| f-F_k(Z)\|_{L^{\infty}} \int \frac{\e^2}{\psi_\ell^2(Z)}\frac{dY}{|Y|}\lesssim  \left\{ \begin{array}{l l} L_0 e^{-\left(\frac{3}{4k}+\frac{1}{4k} \right)s} \ \ \text{if} \ \ell=1, \\ L_0e^{-\left(\frac{1}{2k}+\frac{1}{4k} \right)s} \ \ \text{if} \ \ell=\frac 12. \end{array} \right.
$$
The error terms are estimated via \fref{2bd:mod}, \fref{2bd:PsiZ} and \fref{bd:bootstrap 1 Z}:
\bee
&& \left| \int \chi \frac{\e}{\psi_\ell^2(Z)} b_s F^4_k(Z) \frac{dY}{|Y|} \right| \lesssim |b_s | \left| \int \chi \frac{\e^2}{\psi_\ell^2(Z)}\frac{dY}{|Y|} \right|^{\frac 12}\left| \int \chi \frac{F^8_k(Z)}{\psi_\ell^2(Z)}\frac{dY}{|Y|} \right|^{\frac 12} \\
&\lesssim & e^{-(k-1)s}  \left| \int \chi \frac{\e^2}{\psi_\ell^2(Z)}\frac{dY}{|Y|} \right|^{\frac 12} \int_{e^{-\frac{k-1}{2k}s}}^{+\infty} \frac{1}{|Z|^{2\ell}}\frac{dZ}{Z} \lesssim  \left\{ \begin{array}{l l} \sqrt{L_0} e^{-\left(\frac{3}{4k}+k-2\frac{5}{8k} \right)s} \ \ \text{if} \ \ell=1, \\ \sqrt{L_0} e^{-\left(\frac{1}{2k}+k-\frac 32 +\frac{1}{4k} \right)s} \ \ \text{if} \ \ell=\frac 12, \end{array} \right.
\eee
and via Cauchy-Schwarz using \fref{2bd:PsiZ}:
$$
\left| \int \chi \frac{\e}{\psi_\ell^2(Z)}\Psi \frac{dY}{|Y|} \right| \leq  \left| \int \chi \frac{\e^2}{\psi_\ell^2(Z)}\frac{dY}{|Y|} \right|^{\frac 12}\left| \int \chi \frac{|\Psi|^2}{\psi_\ell^2(Z)}\frac{dY}{|Y|} \right|^{\frac 12}\leq \left\{ \begin{array}{l l} C\sqrt{L_0} e^{-\left(\frac{3}{4k} \right)s} \ \ \text{if} \ \ell=1, \\ C\sqrt{L_0}e^{-\left(\frac{1}{2k}+\frac{1}{8k} \right)s} \ \ \text{if} \ \ell=\frac 12. \end{array} \right.
$$
We now collect all the previous estimates and obtain:
\bee
\frac{d}{ds} \left( \int \chi \frac{\e^2}{\psi_\ell^2(Z)}\frac{dY}{|Y|} \right)& \leq & \left\{ \begin{array}{l l} -\left(\frac{1}{k}-\kappa\right) \int \chi \frac{\e^2}{\psi_1^2(Z)}\frac{dY}{|Y|}  +C\sqrt{L_0} e^{-\left(\frac{3}{4k} \right)s} \ \ \text{if} \ \ell=1, \\ -\frac{1}{2k} \int \chi \frac{\e^2}{\psi_1^2(Z)}\frac{dY}{|Y|}  +C L_0 e^{-\left(\frac{1}{2k}+\frac{1}{8k} \right)s} \ \ \text{if} \ \ell=\frac 12 \end{array} \right.
\eee
if $\kappa$ has been chosen small enough, and $s_0$ large enough. The two above differential inequalities yield the desired results \fref{bd:2eweightimproved1} when reintegrated in time using \fref{bd:bootstrap 2 psi1 init} and \fref{bd:bootstrap 2 rho2}, if $L_0$ has been chosen large enough independently of the other constants in the bootstrap.\\

\noindent \textbf{Step 2:} \emph{Proof of \fref{bd:2eweightimproved1} for $Z\pa_Z \e$.} Let $\ell=1$ or $\ell=1/2$ and define $w:=Z\pa_Z \e$. It solves from \fref{eq:evo e2}:
$$
w_s+\mathcal M_Z w -\pa_{YY}w-4Z\pa_Z(F_k(Z))\e+2\pa_{YY}\e-Z\pa_Z((f-F_k(Z))\e)+Z\pa_Z (b_sF_k^4(Z)+\Psi)=0.
$$
Therefore, one infers that
\bee
\frac{d}{ds} \frac 12 \left( \int \chi \frac{w^2}{\psi_\ell^2(Z)}\frac{dY}{|Y|} \right)&=& -\frac{\ell}{2k}\int \chi \frac{w^2}{\psi_\ell^2(Z)}\frac{dY}{|Y|}-\int \chi \frac{|\pa_Y w|^2}{\psi_\ell^2(Z)}\frac{dY}{|Y|} +\frac 14 \int \pa_Y \chi \frac{w^2}{\psi_{\ell}^2(Z)}dY \\
&&+\frac 12 \int w^2\left(\frac{\pa_{YY}\chi}{\psi_\ell^2(Z)Y}+2\pa_Y \chi \pa_Y \left(\frac{1}{\psi_\ell^2(Z)Y}\right)+\chi\pa_{YY}\left(\frac{1}{\psi_\ell^2(Z)Y}\right)  \right)dY\\
&&+4\int \chi \frac{wZ\pa_Z(F_k(Z))\e}{\psi_\ell^2(Z)}\frac{dY}{|Y|}-\int w^2\frac{1}{Y}\left( \frac{\pa_Y \chi}{\psi_\ell^2(Z)} +\chi \pa_Y\left(\frac{1}{\psi_\ell^2(Z)}\right)\right)\frac{dY}{|Y|} \\
&&+4\int \frac{w}{\psi_\ell^2(Z)}Z\pa_Z((f-F_k(Z))\e)\frac{dY}{|Y|}-\int \chi \frac{w}{\psi_\ell^2(Z)}Z\pa_Z\left(b_s F^4_k(Z)+\Psi \right)\frac{dY}{|Y|}.
\eee
We treat the boundary terms using \fref{bd:bootstrap 2 rho2} and the fact that $|w|=|Z\pa_Z \e|\leq e^{-(k-1)s/(2k)}|\pa_Z\e|$ for $1\leq Y\leq 2$:
\bee
&&\left|\frac 14 \int \pa_Y \chi \frac{w^2}{\psi_{\ell}^2(Z)}dY+\frac 12 \int w^2\left(\frac{\pa_{YY}\chi}{\psi_\ell^2(Z)Y}+2\pa_Y \chi \pa_Y \left(\frac{1}{\psi_\ell^2(Z)Y}\right) \right)dY-\int w^2 \frac{\pa_Y \chi}{\psi_\ell^2(Z)}\frac{dY}{Y^2} \right| \\
&\lesssim & e^{-\frac{k-1}{k}s}\| \pa_Z \e \|_{L^2_\rho}^2 \left(\| \frac{1}{\psi_\ell^2(Z)}\|_{L^{\infty}(1\leq Y \leq 2)}+\| \pa_Y\left(\frac{1}{\psi_\ell^2(Z)}\right)\|_{L^{\infty}(1\leq Y\leq 2)}\right)\\
&\lesssim & L_1 e^{-\frac 1k s +\frac{k-1}{k}(\ell-1)s}\lesssim \left\{ \begin{array}{l l} L_1 e^{-\left(\frac{3}{4k}+\frac{1}{4k} \right)s} \ \ \text{if} \ \ell=1, \\ L_1 e^{-\left(\frac{1}{2k}+\frac 12 \right)s} \ \ \text{if} \ \ell=\frac 12. \end{array} \right.
\eee
As in Step 1, since for $Y\geq 1$, $| \pa_Y^i ( 1/(\psi_\ell^{2}(Z)))\lesssim 1/(\psi_\ell^{2}(Z)|Y|^\ell))$ one deduces that if $\ell=1$, for some $Y^*\gg 1$ large enough:
\bee
&&\left| \frac 12 \int w^2 \chi\pa_{YY}\left(\frac{1}{\psi_1^2(Z)Y}\right) dY-\int w^2\frac{1}{Y}\chi \pa_Y\left(\frac{1}{\psi_1^2(Z)}\right)\frac{dY}{|Y|} \right| \leq C \int \chi \frac{w^2}{\psi_1^2(Z)Y^2}\frac{dY}{|Y|} \\
&\leq &  C \int_{Y\leq Y^*} \chi \frac{w^2}{\psi_1^2(Z)Y^2}\frac{dY}{|Y|} +C \int_{Y\geq Y^*} \chi \frac{w^2}{\psi_1^2(Z)Y^2}\frac{dY}{|Y|} \\
&\leq &  C \| w \|_{L^2_\rho}^2 \| \frac{1}{\psi_1^2(Z)}\|_{L^{\infty}(1\leq Y \leq Y^*)}+\kappa \int_{Y\geq Y^*} \chi \frac{w^2}{\psi_1^2(Z)}\frac{dY}{|Y|}\\
&\leq & C L_1 e^{-\left(\frac{3}{4k}+\frac{1}{4k}\right)s}+\kappa \int \chi \frac{\e^2}{\psi_1^2(Z)Y^2}\frac{dY}{|Y|}
\eee
using \fref{bd:bootstrap 2 rho2}. If $\ell=1/2$, we use the fact that $1/(Y\psi_{1/2}^2(Z))=e^{-(k-1)/(2k)s}/\psi_1^2(Z)$ to obtain from \fref{bd:bootstrap 1 Z}:
\bee
&& \left| \frac 12 \int w^2 \chi\pa_{YY}\left(\frac{1}{\psi_{1/2}^2(Z)Y}\right) dY-\int w^2\frac{1}{Y}\chi \pa_Y\left(\frac{1}{\psi_{1/2}^2(Z)}\right)\frac{dY}{|Y|} \right| \\
& \leq & C \int \chi \frac{w^2}{\psi_{1/2}^2(Z)Y^2}\frac{dY}{|Y|} \lesssim L_1 e^{-\frac{k-1}{2k}s}\int \chi \frac{w^2}{|Y|\psi_1^2(Z)}\frac{dY}{|Y|} \lesssim e^{-\left(\frac{1}{2k}+\frac{2k-1}{4k}\right)s}.
\eee
The linear term coming from the commutator between $Z\pa_Z$ and $\mathcal M_Z$ is estimated by Cauchy-Schwarz and \fref{bd:bootstrap 1 Z} for $\ell=1$:
$$
\left| \int \chi \frac{wZ\pa_Z(F_k(Z))\e}{\psi_1^2(Z)}\frac{dY}{|Y|} \right| \leq C \left| \int \chi \frac{w^2}{\psi_1^2(Z)}\frac{dY}{|Y|} \right|^{\frac 12}  \left| \int \chi \frac{\e^2}{\psi_1^2(Z)}\frac{dY}{|Y|} \right|^{\frac 12} \leq C\sqrt{K_1}\sqrt{L_0} e^{-\frac{3}{4k}s}
$$
as $|Z\pa_Z F_k(Z)|\lesssim |Z|^{2k}(1+|Z|)^{-4k}\lesssim 1$. For $\ell=1/2$, since
$$
\frac{|Z|^{2k}(1+|Z|)^{-4k}}{\psi_{1/2}^2(Z)}\lesssim \frac{1}{\psi_{1}^2(Z)}
$$
as $|\psi_{1}(Z)|=|Z|^{1/2}|\psi_{1/2}(Z)|$, one uses Cauchy-Schwarz, \fref{bd:bootstrap 1 Z}:
$$
\left| \int \chi \frac{wZ\pa_Z(F_k(Z))\e}{\psi_{1/2}^2(Z)}\frac{dY}{|Y|} \right| \lesssim \left| \int \chi \frac{w^2}{\psi_{1}^2(Z)}\frac{dY}{|Y|} \right|^{\frac 12}  \left| \int \chi \frac{\e^2}{\psi_1^2(Z)}\frac{dY}{|Y|} \right|^{\frac 12} \lesssim \sqrt{L_0K_1}e^{-\frac{3}{4k}s}.
$$
The lower order linear term is estimated via \fref{bd:tildemathcalMLinfty}, \fref{bd:bootstrap 1 Z}:
\bee
&&\left| \int \frac{w}{\psi_\ell^2(Z)}Z\pa_Z((f-F_k(Z))\e)\frac{dY}{|Y|} \right|\\
&\lesssim &\left(\| f-F_k(Z)\|_{L^{\infty}}+\| Z\pa_Z(f-F_k(Z))\|_{L^{\infty}}\right) \left(\int \chi \frac{\e^2}{\psi_\ell^2(Z)}\frac{dY}{|Y|}+\int \chi \frac{w^2}{\psi_\ell^2(Z)}\frac{dY}{|Y|}\right)\\
&\lesssim & \left\{ \begin{array}{l l} L_1 e^{-\left(\frac{3}{4k}+\frac{1}{4k} \right)s} \ \ \text{if} \ \ell=1, \\ L_1 e^{-\left(\frac{1}{2k}+\frac{1}{4k} \right)s} \ \ \text{if} \ \ell=\frac 12. \end{array} \right.
\eee
The error terms are estimated via \fref{2bd:mod}, \fref{2bd:PsiZ}, \fref{bd:bootstrap 1 Z}:
\bee
&& \left| \int \chi \frac{w}{\psi_\ell^2(Z)} b_s Z\pa_Z(F^4_k(Z)) \frac{dY}{|Y|} \right| \lesssim |b_s | \left| \int \chi \frac{w^2}{\psi_\ell^2(Z)}\frac{dY}{|Y|} \right|^{\frac 12}\left| \int \chi \frac{|Z|^{4k}(1+|Z|)^{-20 k}}{\psi_\ell^2(Z)}\frac{dY}{|Y|} \right|^{\frac 12} \\
&\lesssim & e^{-(k-1)s}  \left| \int \chi \frac{w^2}{\psi_\ell^2(Z)}\frac{dY}{|Y|} \right|^{\frac 12} \int_{e^{-\frac{k-1}{2k}s}}^{+\infty} |Z|^{4k-\ell}(1+|Z|)^{-4k}\frac{dZ}{Z} \\
&\lesssim &  \left\{ \begin{array}{l l} \sqrt{L_1}e^{-\left(\frac{3}{4k}+k-1-\frac{3}{8k} \right)s} \ \ \text{if} \ \ell=1, \\ \sqrt{L_1}e^{-\left(\frac{1}{2k}+k-1-\frac{1}{4k}\right)s} \ \ \text{if} \ \ell=\frac 12, \end{array} \right.
\eee
and via Cauchy-Schwarz using \fref{2bd:PsiZ}:
\bee
\left| \int \chi \frac{w}{\psi_\ell^2(Z)} Z\pa_Z\Psi \frac{dY}{|Y|} \right| &\leq &  \left| \int \chi \frac{w^2}{\psi_\ell^2(Z)}\frac{dY}{|Y|} \right|^{\frac 12}\left| \int \chi \frac{|Z\pa_Z\Psi|^2}{\psi_\ell^2(Z)}\frac{dY}{|Y|} \right|^{\frac 12}\leq \left\{ \begin{array}{l l} C\sqrt{L_1} e^{-\left(\frac{3}{4k} \right)s} \ \ \text{if} \ \ell=1, \\ \sqrt{L_1} e^{-\left(\frac{1}{2k}+\frac{1}{8k} \right)s} \ \ \text{if} \ \ell=\frac 12. \end{array} \right.
\eee
We now collect all the previous estimates and obtain that for any $\kappa>0$, for $s_0$ large enough:
\bee
\frac{d}{ds} \left( \int \chi \frac{w^2}{\psi_\ell^2(Z)}\frac{dY}{|Y|} \right)& \leq & \left\{ \begin{array}{l l} -\left(\frac{1}{k}-\kappa\right) \int \chi \frac{w^2}{\psi_1^2(Z)}\frac{dY}{|Y|}  +C(\sqrt{L_0}\sqrt{L_1}+\sqrt{L_1}) e^{-\left(\frac{3}{4k} \right)s} \ \ \text{if} \ \ell=1, \\ -\frac{1}{2k} \int \chi \frac{\e^2}{\psi_1^2(Z)}\frac{dY}{|Y|}  +C L_1e^{-\left(\frac{1}{2k}+\frac{1}{8k} \right)s} \ \ \text{if} \ \ell=\frac 12, \end{array} \right.
\eee
The two above differential inequalities yield the desired results \fref{bd:2eweightimproved1} when reintegrated in time using \fref{bd:bootstrap 2 psi1 init} and \fref{bd:bootstrap 2 rho2}, if $L_1$ has been chosen large enough depending on $L_0$, and $s_0$ has been chosen large enough.\\

\noindent \textbf{Step 3:} \emph{End of the proof}. We claim that the bounds \fref{bd:2eweightimproved1} for higher order derivatives, as well as the bound \fref{bd:2eweightimproved3}, can be proved with verbatim the same argument that were used in Step 1 and Step 2. We leave the details to the reader.

\end{proof}

\begin{proof}[Proof of Proposition \ref{pr:bootstrap2}]

We use a bootstrap argument. Let $s_1\geq s_0$ be the supremum of times $\tilde s\geq s_0$ such that all the bounds of Proposition \ref{pr:bootstrap2} hold on some time interval $[s_0,\tilde s]$. Then \fref{bd:bootstrap 2 psi1 init}, \fref{bd:bootstrap 1 Z} and \fref{bd:bootstrap 2 rho init} imply $s_1>s_0$ by a continuity argument. Assume by contradiction that $s_1<+\infty$. Then the bounds \fref{bd:bootstrap 1 Z}, \fref{bd:bootstrap 3 Z} and \fref{bd:bootstrap 2 rho2} are strict at time $s_1$ from \fref{bd:bootstrap 2 rho2}, \fref{bd:2eweightimproved1} and \fref{bd:2eweightimproved3}. From a continuity argument there exists $\delta>0$ such that \fref{bd:bootstrap 1 Z}, \fref{bd:bootstrap 3 Z} and \fref{bd:bootstrap 2 rho2} hold on $[s_1,s_1+\delta]$, contradicting the definition of $s_1$. Thus $s_1=+\infty$ and Proposition \ref{pr:bootstrap2} is proved.

\end{proof}


\section{Proof of Theorem \ref{th:mainstable}} \lab{sec:stable}

In this section we prove Theorem \ref{th:mainstable}. The proof is the same as the one of Theorem \ref{th:main}, only few details change. Namely, the analysis is now based on the stable blow-up of the self-similar heat equation $\xi_t-\xi^2-\xi_{yy}=0$ whose properties are classical \cite{BK,BK2,HV,MZ}. We therefore just sketch the proof, with an emphasise on the differences between this proof and that of Theorem \ref{th:main}. We consider only the case $i=1$ with the profile $\Psi_1$ for Burgers equation, the proof being the same for $i\geq 2$. We define the self-similar variables
\be \label{def:selfsimvarstable}
X:=\sqrt{\frac b6} \frac{x}{(T-t)^{\frac 32}}, \ \ Y:=\frac{y}{\sqrt{T-t}}, \ \ s:=-\log (T-t), \ \ Z:=\frac{Y}{8\sqrt s},
\ee
and
$$
u(t,x,y)=\sqrt{\frac 6b} (T-t)^{\frac 12} v\left(s,X,Y \right),
$$
The first step is to obtain precise information for the behaviour of the derivatives of the solution on the  transverse axis.

\subsection{Analysis on the  transverse axis $\{x=0 \}$} \label{sec:1dstable} We start by showing the first part of Theorem \ref{th:NLHinstable} and the analogue of Proposition \ref{pr:LHinstable}. Define for a solution $u$ to \fref{eq:burgers}:
$$
\xi (t,y)=-u_x (t,0,y), \ \ \xi(t,y)=(T-t)^{-1}f(s,Y), \ \ \zeta (t,y)=\pa_x^3u (t,0,y), \ \ \zeta(t,y)=(T-t)^{-4}g(s,Y).
$$
Then $(f,g)$ solve the system \fref{eq:f}.\\

\noindent \textbf{Claim}: For $0<T\ll 1$ small enough, for any $b>0$ and $J\in \mathbb N$, there exists a solution to \fref{eq:f} such that for $0\leq j \leq J$:
\be \lab{stable:bd:tildef}
f(s,Y)=\frac{1}{1+Z^2}+\tilde f , \ \ |\pa_Z^j \tilde f|\lesssim s^{-\frac 14} (1+|Z|)^{-\frac 32 -j},
\ee
\be \lab{stable:bd:tildeg}
g(s,Y)=\frac{b}{(1+Z^2)^4}+\tilde g , \ \ |\pa_Z^j \tilde g|\lesssim s^{-\frac 14} (1+|Z|)^{-8+\frac 12 -j}.
\ee
These estimates are very similar to \fref{eq:def tildeF} and \fref{eq:def tildeG}, but the smallness of the error is in powers of $s^{-1}$ and not of $e^{-s}$ anymore. This loss will however not be a problem for the sequel.\\

\noindent \textbf{Sketch of proof of the claim}: We adapt the strategy of Section \ref{sec:NLH}. We first construct a solution $f$ to $(NLH)$ satisfying \fref{stable:bd:tildef}. We take as an approximate solution to $(NLH)$ the profile
$$
F[a](s,Y):= \frac{1}{1+\left(\frac{1}{8s}+a\right)Y^2}+\left(\frac{1}{4s}+2a\right) \frac{1}{\left(1+\left(\frac{1}{8s}+a\right)Y^2\right)^2}.
$$
Note that here the corrective parameter $a$ will satisfy $|a|\lesssim |\log (s)s^{-2}|$. The main difference between the stable blow-up and the flat blow-ups for $(NLH)$ is then the following. The scaling parameter in the flat case \fref{NLH:id:def a} corresponds to leading order to that of the inviscid case and is not affected by the dissipation \fref{NLH:bd:a}, whereas in the stable blow-up case the dissipation has a modulation effect on this parameter, and forces it to tend to $0$ through a logarithmic correction.\\

\noindent Following the proof of Proposition \ref{NLH:pr:profile}, the approximate profile satisfies the following identity:
\be \lab{stable:Fa}
\pa_s F [a]+F [a]+\frac Y2 \pa_Y F [a] -F^2 [a]-\pa_{YY}F[a] = -\left(\frac{1}{4s^2}+\frac{4a}{s}\right)+\left(-a_s-\frac{2}{s}a\right)h_2+\Psi =R\\
\ee
where $h_2$ is defined by \fref{def:hell}, and where for a corrective modulation parameter $a$ satisfying the a priori bound $|a|\lesssim |\log (s)|s^{-2}$ and $|a_s|\lesssim |\log (s)|s^{-3}$ the errors $\Psi$ and $R$ satisfy:
$$
\| \Psi \|_{L^2_\rho}\lesssim s^{-3}, \ \ \| \pa_Z^j R \|_{L^2_\rho}\lesssim s^{-3+\frac j2} \ \ \text{for} \ j=0,1,2 \ \ \text{and} \ \ \| \pa_Z^j R \|_{L^2_\rho}\lesssim s^{-1} \  \ \text{for} \ j\geq 3,
$$
$$
\int_{|Y|\geq 1} \frac{|(Z\pa_Z)^jR|^2}{|\phi_{\frac 52}(Z)|^2} \frac{dY}{|Y|}\lesssim s^{-1} \ \ \text{and} \ \ \int_{|Y|\geq 1} \frac{|\pa_Z^jR|^2}{|\phi_0(Z)|^2} \frac{dY}{|Y|}\lesssim s^{-1} \ \ j\geq 3,
$$
where $\phi_j$ denotes the eigenfunction
$$
\phi_j(Z)=\frac{Z^j}{(1+Z^2)^2}, \ \ H_Z\phi_j=\frac{j-2}{2}\phi_j, \ \ H_Z=1-\frac{2}{1+Z^2}+\frac Z2 \pa_Z.
$$
We then show the existence of a global solution to $(NLH)$ close to $F[a]$ by a bootstrap argument following Proposition \ref{pr:bootstrap}. We decompose $f$ as
$$
f=F[a]+\e=F[a]+c_1+\tilde \e, \ \ \langle \e,h_2 \rangle_\rho=0, \ \ \langle \tilde \e, 1 \rangle_\rho=\langle \tilde \e, h_2 \rangle_\rho=0
$$
where the orthogonality conditions fix the value of $a$ and $c_1$ in a unique way. We claim that there exists a global solution to the first equation in \fref{eq:f} satisfying for $j=0,...,J+1$:
$$
|a(s)|\lesssim s^{-2}, \ \ |c_1(s)|\lesssim s^{-2}, \ \ \| \tilde \e \|_{L^2_\rho}\lesssim s^{-3}, \ \ \| \pa_Z^j \e \|_{L^2_\rho}\lesssim s^{-3+\frac{j}{2}} \ j=1,2, \ \ \| \pa_Z^j \e \|_{L^2_\rho}\lesssim s^{-1},  \ j\geq 3,
$$
$$
\int_{|Y|\geq 1} \frac{|(Z\pa_Z)^j\e|^2}{|\phi_{\frac 52}(Z)|^2} \frac{dY}{|Y|}\lesssim s^{-\frac 12} \ \  \text{and} \ \ \int_{|Y|\geq 1} \frac{|\pa_Z^j\e|^2}{|\phi_0(Z)|^2} \frac{dY}{|Y|}\lesssim s^{-\frac 12} \ \ \text{for} \ j\geq 3.
$$
To prove this fact, one first performs modulation estimates, then energy estimates at the origin with the $\rho$ weight, and then energy estimates outside the origin as in Lemmas \ref{NLH:lem:modulation}, \ref{NLH:lem:energyrho} and \ref{NLH:lem:improvedoustide}. The evolution equation near the origin reads from \fref{stable:Fa}
$$
c_{1,s}-c_1-\left(\frac{1}{4s^2}+\frac{4a}{s}\right)-\left(a_s+\frac{2}{s}a\right)h_2+\tilde \e_s+\mathcal L \tilde \e-2(F[a]-1)\e-\e^2+\Psi=0.
$$
The modulation estimates are therefore a consequence of the spectral structure of $\mathcal L$ in Proposition \ref{pr:Lrho}, giving in the bootstrap regime when projecting the above equation on $1$ and $h_2$: 
$$
\left| a_s+\frac{2}{s} a \right|\lesssim s^{-3}, \ \ |c_{1,s}-c_1|\lesssim s^{-2}.
$$
The first inequality, when reintegrated in time, gives $|a|\lesssim |\log (s)|s^{-2}$. The second inequality shows an instability, and the use of Brouwer's fixed point theorem then implies the existence of a trajectory such that $|c_1(s)|\lesssim s^{-2}$. The orthogonality conditions for $\tilde \e$ imply the spectral damping $\langle \tilde \e,\mathcal L \tilde \e \rangle \geq \| \tilde \e \|_{L^2_\rho}^2$ since $\tilde \e$ is even, implying the energy identity
$$
\frac{d}{ds}\left(\frac 12 \| \tilde \e \|_{L^2_\rho}^2 \right)\leq -(1-Cs^{-1}-C\| \e \|_{L^{\infty}}) \| \tilde \e \|_{L^2_\rho}^2+Cs^{-6}.
$$
This yields the desired estimates for $\tilde \e$ when reintegrated with time, and the same technique applies to control its derivatives. In the far field, the analysis is the same as in Lemma \ref{NLH:lem:improvedoustide}, the equation for $\e$ reads
$$
\e_s+H_Z \e -\pa_{YY}\e-2\left( F[a]-\frac{1}{1+Z^2} \right)\e-\e^2+R=0, \ \ H_Z=1-\frac{2}{1+Z^2}+\frac Y2 \pa_Y.
$$
Let $\chi$ be a non-negative smooth cut-off function with $\chi=0$ for $|Y|\leq 1$ and $\chi=1$ for $|Y|\geq 2$. One obtains from this equation the following energy estimate:
\bee
\frac{d}{ds} \left(\frac{1}{2}\int \chi \frac{\e^2}{|\phi_{\frac 52 }(Z)|^2}\frac{dY}{|Y|} \right) & \leq & -\left(\frac 14 -\kappa-\| \e\|_{L^{\infty}} \right)\int \chi \frac{\e^2}{|\phi_{\frac 52 }(Z)|^2}\frac{dY}{|Y|}-\int \chi \frac{(\pa_Y\e)^2}{|\phi_{\frac 52 }(Z)|^2}\frac{dY}{|Y|}\\
&&+O\left(s^{\frac 52}\| \e \|_{L^2_\rho}^2 \right)+O(s^{-1})
\eee
where $0<\kappa \ll 1$ is an arbitrary small positive number, since $|\phi_{5/2}|\sim |Z|^{5/2}\sim |Y|^{5/2}s^{-\frac 54}$ on compact sets. Thanks to the damping, this estimate is reintegrated in time and shows the weighted decay outside the origin for $\e$. The analogous estimates for the derivatives are showed similarly. The strategy that we just explained allows one to close the bootstrap estimates. Using the Sobolev embedding \fref{bd:Sobolev}, this concludes the proof of the existence of a solution $f$ to $(NLH)$ satisfying \fref{stable:bd:tildef}. \\

Once the properties of $f$ are known, the analysis of $(LFH)$ follows very closely the one performed in Subsection \ref{sec:LFH}. We decompose $g$ solution to the second equation in \fref{eq:f} according to
$$
g(s,Y)=b(s) f^4(s,Y)+\bar \e, \ \ \langle \bar \e,1\rangle_\rho=0
$$
where $f$ is the solution $(NLH)$ we just constructed. The evolution equation then reads
$$
b_sf^4+\bar \e_s+4\bar \e -4f\bar \e+\frac Y2 \pa_Y \bar \e-\pa_{YY}\bar \e+\bar R=0, \ \ \bar R=-12b (\pa_Y f)^2f^2.
$$
For $|b|\lesssim 1$ the error satisfies from the properties of $f$ already showed:
$$
\| \bar R \|_{L^2_\rho}\lesssim s^{-2}, \ \ \| \pa_Z^j\bar R \|_{L^2_\rho}\lesssim s^{-1} \ \text{for} \ j\geq 1,
$$
$$
\int_{|Y|\geq 1} \frac{|(Z\pa_Z)^j \bar R|}{|\psi_{\frac 12}(Z)|^2}\frac{dY}{|Y|} \lesssim s^{-1}, \ \ \int_{|Y|\geq 1} \frac{|(Z\pa_Z)^j \bar R|}{|\psi_0(Z)|^2}\frac{dY}{|Y|} \lesssim s^{-1}
$$
where $\psi_j$ denotes the eigenfunction
$$
\psi_j(Z)=\frac{Z^j}{(1+Z^2)^4}, \ \ \mathcal M_Z\psi_j=\frac{j}{2}\psi_j, \ \ \mathcal M_Z=4-\frac{4}{1+Z^2}+\frac Z2 \pa_Z.
$$
We claim that there exists a solution satisfying the estimates
$$
|b_s|\lesssim s^{-2}, \ \ \| \tilde \e \|_{L^2_\rho}\lesssim s^{-2}, \ \ \| \pa_Z^j \e \|_{L^2_\rho}\lesssim s^{-1} \ \text{for}\ j\geq 1,
$$
$$
\int_{|Y|\geq 1} \frac{|(Z\pa_Z)^j\e|^2}{|\psi_{\frac 12}(Z)|^2} \frac{dY}{|Y|}\lesssim s^{-\frac 12} \ \  \text{and} \ \ \int_{|Y|\geq 1} \frac{|\pa_Z^j\e|^2}{|\psi_0(Z)|^2} \frac{dY}{|Y|}\lesssim s^{-\frac 12} \ \ \text{for} \ j\geq 1.
$$
Similarly, we prove this property by a bootstrap analysis, following closely the analysis of Lemmas \ref{LH:lem:modulation}, \ref{LH:lem:energyrho} and \ref{LH:lem:energyoutside}. The equation close to the origin reads
$$
b_sf^4+\bar \e_s+\frac Y2 \pa_Y \bar \e-\pa_{YY}\bar \e+4(1-f)\bar \e+\bar R=0.
$$
Taking the $L^2_\rho$ scalar product against the constant $1$ then yields indeed the modulation equation
$$
|b_s|\lesssim s^{-2}.
$$
Similarly, from the spectral gap $\| \pa_Y \bar \e\|_{L^2_\rho}^2\geq \| \bar \e\|_{L^2_\rho}^2 $ one deduces the energy identity
$$
\frac{d}{ds}\left(\frac 12 \|  \bar \e \|_{L^2_\rho}^2 \right)\leq -(1-Cs^{-\frac 14}) \| \tilde \e \|_{L^2_\rho}^2+Cs^{-4}
$$
which yields the corresponding estimate $\| \e \|_{L^2_\rho}\lesssim s^{-2}$ when reintegrated with time. The corresponding estimates for higher order derivatives are showed the same way. In the far field the evolution equation reads
$$
b_sf^4+\bar \e_s+4\bar \e -\frac{4}{1+Z^2}\bar \e+\frac Y2 \pa_Y \bar \e-\pa_{YY}\bar \e+4\tilde f \bar \e+\bar R=0.
$$
This equation enjoys the following energy estimate for an arbitrary constant $0<\kappa \ll 1$:
\bee
\frac{d}{ds} \left(\frac{1}{2}\int \chi \frac{\bar \e^2}{|\psi_{\frac 12 }(Z)|^2}\frac{dY}{|Y|} \right) & \leq & -\left(\frac 14 -\kappa \right)\int \chi \frac{\bar \e^2}{|\psi_{\frac 12 }(Z)|^2}\frac{dY}{|Y|}-\int \chi \frac{(\pa_Y\bar \e)^2}{|\psi_{\frac 12 }(Z)|^2}\frac{dY}{|Y|}\\
&&+O\left(s^{\frac 12}\| \bar \e \|_{L^2_\rho}^2 \right)+O(s^{-1})
\eee
This estimate is reintegrated in time and shows the expected weighted decay outside the origin for $\bar \e$. The estimates for the derivatives are showed the same way. Using the Sobolev embedding \fref{bd:Sobolev}, one then obtained the existence of a solution $g$ to $(LFH)$ satisfying \fref{stable:bd:tildeg}.

\subsection{Analysis of the full $2$-d problem}

We now follow the analysis of Section \ref{sec:main}. Let $f$ and $g$ be the solutions to $(NLH)$ and $(LFH)$ satisfying \fref{stable:bd:tildef} and \fref{stable:bd:tildeg}. For simplicity we fix $b=6$, so that $X=x/(T-t)^{3/2}$. We take the same blow-up profile as in the proof of Theorem \ref{th:main}, adjusting the cut-off between the inner and outer zones. We set for $0<d\ll 1$ a cut-off function $ \chi_d(s,Y):=\chi \left(Y/(ds)\right)$ where $\chi$ is a smooth nonnegative function with $\chi(Y)=1$ for $|Y|\leq 1$ and $\chi(Y)=0$ for $|Y|\geq 2$. We decompose our solution to \fref{main:eqvautosim} according to:
\be \lab{stable:decomposition v}
v(s,X,Y)= Q+\e, \ \ Q=\chi_d(s,Y) \tilde \Theta +(1-\chi_d(s,Y))\Theta_{e} \\
\ee
where (for $d$ small enough $f$ and $g$ do not vanish from \fref{stable:bd:tildef} and \fref{stable:bd:tildeg})
$$
\tilde \Theta(s,X,Y)= \sqrt 6 g^{-\frac 12}f^{\frac 32} \Psi_1 \left( \frac{g^{\frac 12}f^{-\frac 12}}{\sqrt 6} X \right)
$$
and where $\Theta_e$ is the exterior profile (recall that $\tilde X=X/(1+Z^2)^{3/2}$):
$$
\Theta_e(s,X,Y)=  \left(-X f(s,Y)+X^3\frac{g(s,Y)}{6} \right)e^{- \tilde X^4}.
$$
We adjust our initial datum $v(s_0)$ such that $-\pa_X v(s_0,0,Y)=f(s_0,Y)$ and $\pa_X^3 v(s_0,0,Y)=g(s_0,Y)$. This way, since $v$ odd in $x$ and even $y$, one has that $\pa_X^j \e (s,0,Y)=0$ on the  transverse axis for $j=0,1,2,3,4$ for all times $s\geq s_0$. The time evolution of the remainder $\e$ is:
$$
\e_s+\mathcal L \e+\tilde{\mathcal L}\e+R+\e \pa_X \e=0
$$
where
$$
\mathcal L= -\frac 12+\pa_X \Theta +\left(\frac 32 X +\Theta \right)\pa_X +\frac 12 Y\pa_Y-\pa_{YY},
$$
$$
\Theta (s,X,Y)=(1+Z^2)^{\frac 12}\Psi_1 \left(\frac{X}{(1+Z^2)^{\frac 32}} \right), \ \ \ \tL \e=(Q-\Theta)\pa_X \e +(\pa_X Q-\pa_X \Theta)\e,
$$
and
$$
R=Q_s-\frac 12 Q +\frac 32 X \pa_X Q+\frac 12 Y \pa_Y Q +Q\pa_X Q-  \pa_{YY}Q.
$$
From Proposition \ref{pr:mathcalL} the inviscid linearised operator has eigenvalues of the form $(j+\ell-3)/2$ for $(j,\ell)\in \mathbb N$ with associated eigenfunction
$$
\varphi_{j,\ell}(X,Z) = Z^\ell F_k^{1-\frac j2}(Z) \times \frac{(-1)^j \Psi_1^j \left( F_k^{\frac 32}(Z)X\right)}{1+3\Psi_1^2\left( F_k^{\frac 32}(Z)X\right)}.
$$
The sizes of the important objects are 
$$
|\Theta (X,Z)|\approx |X| \left((1+|Z|)^{3}+|X| \right)^{\frac 13 -1} \approx (1+|Z|) |\tilde X|(1+|\tilde X|)^{\frac 13 -1}
$$
$$
|\varphi_{\frac 72,0}(X,Z)|\approx |X|^{\frac 72} \left((1+|Z|)^{3}+|X|\right)^{\frac 12-\frac 72} \approx (1+|Z|)^{\frac 32} |\tilde X|^{\frac 72}(1+|\tilde X|)^{\frac 12 -\frac 72}.
$$
In the previous Section, the weight $\varphi_{4,0}$ was used. Any weight $\varphi_{\alpha,0}$ with $\alpha\in (3,4]$ is suitable since it provides linear decay and suitable weighted Sobolev estimates for the error term $R$ as well. We take $\alpha=\frac 72$ here to make the weight slightly less singular at the origin, so that the integrals below are well defined for $C^4$ function. We claim that thanks to \fref{stable:bd:tildef} and \fref{stable:bd:tildeg} one has the following estimates for the error, which can be proved as in the proof of Lemma \ref{main:lem:R}:
\be \label{bd:stableR}
\sum_{0\leq j_1+j_2\leq 2} \int_{\mathbb R^2} \frac{ (\pa_Z^{j_1}A^{j_2}R)^{2q}+ ((Y\pa_Y)^{j_1}A^{j_2}R)^{2q}}{\varphi_{\frac 72,0}^{2q}(X,Z)}\frac{dXdY}{|X|\lY}  \lesssim s^{-q}
\ee
where $A$ is given by \fref{main:def:A}. From \fref{stable:bd:tildef} and \fref{stable:bd:tildeg} one also deduces the following estimates for the lower order linear term
$$
| \pa_Z^{j_1}(Q-\Theta)| \lesssim s^{-\frac 14} |X|(1+|Z|)^{-j_1},
$$
$$
| \pa_Z^{j_1}\pa_X^{j_2}(Q-\Theta)| \lesssim s^{-\frac 14} (1+|X|)^{1-j_2}(1+|Z|)^{-j_1}.
$$
which can be proved as in the proof of Lemma \ref{main:lem:tildeL}. We can therefore perform the same energy estimates as in Lemmas \ref{main:lem:energy0} and \ref{main:lem:energy1}. Indeed, the leading order linear estimate \fref{id:estimationlineaire} holds also true in that case, and with the above control on the lower order linear term and of the error $R$ one obtains that for any $\kappa >0$ for $q\geq 2$ large enough:
\bee
&&\frac{d}{ds} \left( \frac{1}{2q} \int_{\mathbb R^2} \frac{\e^{2q}}{\varphi_{\frac 72,0}^{2q}(X,Z)}\frac{dXdY}{|X| \lY} \right)\\
 &\leq & -\left(\frac 14-\kappa-Cs^{-\frac 14}-C\| \pa_X \e \|_{L^{\infty}}\right) \int_{\mathbb R^2} \frac{\e^{2q}}{\varphi_{\frac 72,0}^{2q}}\frac{dXdY}{|X| \lY}-\frac{2q-1}{q^2}\int \frac{|\pa_Y (\e^q)|^2}{\varphi_{\frac 72,0}^{2q}}\frac{dXdY}{|X|\lY}+Cs^{-q}.
\eee
The same type of energy estimates hold when applying $A$, $Z$ or $Z\pa_Z$, up to terms involving lower order derivatives. For exemple, one can derive the following estimate:
\bee
&&\frac{d}{ds} \left( \frac{1}{2q} \int_{\mathbb R^2} \frac{(A\e)^{2q}}{\varphi_{\frac 72,0}^{2q}(X,Z)}\frac{dXdY}{|X| \lY} \right)\\
 &\leq & -\left(\frac 14 -\kappa-Cs^{-\frac 14}-C\| \pa_X \e \|_{L^{\infty}}-C\| \frac{\e}{|X|} \|_{L^{\infty}}\right) \int_{\mathbb R^2} \frac{(A\e)^{2q}}{\varphi_{\frac 72,0}^{2q}}\frac{dXdY}{|X| \lY}-\frac{2q-1}{q^2}\int \frac{|\pa_Y (\e^q)|^2}{\varphi_{\frac 72,0}^{2q}}\frac{dXdY}{|X|\lY}\\
 &&+Cs^{-q}+C \int_{\mathbb R^2} \frac{\e^{2q}}{\varphi_{\frac 72,0}^{2q}(X,Z)}\frac{dXdY}{|X| \lY}.
\eee
This implies that the analogue of Proposition \ref{main:pr:bootstrap} holds, i.e. that we can bootstrap the following estimates for the remainder $\e$:
\be \label{bd:varepsilonstable}
\sum_{0\leq j_1+j_2\leq 2}\left( \int_{\mathbb R^2} \frac{((\pa_Z^{j_1}A^{j_2}\e)^{2q}}{\varphi_{\frac 72,0}^{2q}}\frac{dXdY}{|X|\lY}\right)^{\frac{1}{2q}}+\left( \int_{\mathbb R^2} \frac{(((Y\pa_Y)^{j_1}A^{j_2}\e)^{2q}}{\varphi_{\frac 72,0}^{2q}}\frac{dXdY}{|X|\lY}\right)^{\frac{1}{2q}}\lesssim \frac{1}{\sqrt s}.
\ee
This estimate, together with the weighted Sobolev embedding \fref{eq:weightedSobo}, gives the following pointwise estimates on $\e$ as in Lemma \ref{main:lem:pointwise e}:
\bee
|\e|\lesssim s^{-\frac 12} (1+|Z|)^{\frac 32} |\tilde X|^{\frac 72} (1+|\tilde X|)^{\frac 12 -\frac 72}\lesssim s^{-\frac 12} |X|, \\
|\pa_X \e|\lesssim s^{-\frac 12} (1+|Z|)^{-\frac 32} |\tilde X|^{\frac 52} (1+|\tilde X|)^{\frac 12 -\frac 72}\lesssim s^{-\frac 12},\\
|\pa_Z\e|\lesssim s^{-\frac 12} (1+|Z|)^{\frac 12} |\tilde X|^{\frac 72} (1+|\tilde X|)^{\frac 12 -\frac 72}.
\eee
By using the above estimate in the decomposition \ref{stable:decomposition v}, combined with \fref{stable:bd:tildef} and \fref{stable:bd:tildeg}, we get that on compact sets in the variables $X$ and $Z$ there holds the estimate:
$$
v=\Theta +O_{C^1}(s^{-\frac 14}).
$$
This then ends the existence part of the proof of Theorem \ref{th:mainstable}.\\

\noindent The stability of this blow-up pattern in $\mathcal B$ defined by \fref{bd:def mathcal B} is a direct consequence of the stability of the underlying blow-up for (NLH) proved by Merle and Zaag \cite{MZ}. Indeed, assume $u$ with initial datum $u_0$ is a solution that is constructed in this section belonging in addition to the Schwartz class, and $u'_0$ is such that $\| u'_0-u_0\|_{\mathcal B}\leq \delta$ for some $\delta >0$. It is proved in \cite{MZ} that there exists $T'$ with $T'\rightarrow T$ as $\delta\rightarrow 0$ such that $\xi'=-\pa_x u'_{|x=0}$ blows up at time $T'$, and that in the self-similar variables \fref{def:selfsimvarstable} with $s'=-\log ( T'-t)$ the renormalised solution $f'$ is global and satisfies all the bounds of Subsection \fref{sec:1dstable} on $[s_0',+\infty)$.

Let now $t_0>0$. For $t_0$ small enough, and then for $\delta$ small enough, via standard parabolic regularising effects, $f'$ satisfies all the bounds of Subsection \fref{sec:1dstable} on $[s'(t_0),+\infty)$, in particular \fref{stable:bd:tildef} for derivatives up to order $J=3$. The same argument applies for $\zeta'=\pa_x^3u'_{|x=0}$, as we proved that the regime described in the corresponding part of Subsection \fref{sec:1dstable} is stable: $g'$ satisfies all the bounds of this Subsection on $[s'(t_0),+\infty)$, in particular \fref{stable:bd:tildef} for $J=3$.\\

On $[s'(t_0),+\infty)$, the bounds \fref{stable:bd:tildef} and \fref{stable:bd:tildeg} with $J=3$ for $f'$ and $g'$ implies that the bound \fref{bd:stableR} on $R$ holds true. Moreover, the bound \fref{bd:varepsilonstable} holds true at time $s'(t_0)$ for $t_0$ small enough, because of the continuity of the flow of the equation in $\mathcal B$. Thus, from the analysis of the current Subsection, the solution $v$ remains in the bootstrap regime on $[s'(t_0),+\infty)$. The solution $u'$ in original variables thus blows up with the same behaviour than $u$ at time $T'$.


\appendix

\section{One-dimensional functional analysis results}

\begin{lemma}[Poincar\'e inequality in $L^2_\rho$]

For any $f\in H^1_\rho$ defined by \fref{eq:def L2rho} one has that $Yf\in L^2_\rho$ with
\be \lab{eq:Poincare}
\| Yf \|_{L^2_\rho} \lesssim \| f \|_{H^1_\rho}.
\ee

\end{lemma}

\begin{proof}

We prove \fref{eq:Poincare} for smooth and compactly supported functions, and its extension to $H^1_\rho$ follows by a density argument. Performing an integration by parts one first finds
$$
\int Y \e \pa_Y \e e^{-\frac{Y^2}{4}}dY=\frac 14 \int Y^2 \e^2 e^{-\frac{Y^2}{4}}dY-\frac 12 \int \e^2 e^{-\frac{Y^2}{4}}dY.
$$
Therefore, using Cauchy-Schwarz and Young's inequalities one obtains:
\bee
\int Y^2 \e^2 e^{-\frac{Y^2}{4}}dY & = & 4 \int Y \e \pa_Y \e e^{-\frac{Y^2}{4}}dY+2 \int \e^2 e^{-\frac{Y^2}{4}}dY\\
& \leq & 2\epsilon \int Y^2\e^2 e^{-\frac{Y^2}{4}}ddY+\frac{2}{\epsilon} \int |\pa_Y \e|^2e^{-\frac{Y^2}{4}}dY+2 \int \e^2 e^{-\frac{Y^2}{4}}dY.
\eee
Taking $0<\epsilon <1/2$ yields the desired result.

\end{proof}

\begin{lemma}

If $\e\in H^1_{\text{loc}}\{|Y|\geq 1 \}$ is such that $ \int_{|Y|\geq 1} (\e^2+(Y\pa_Y\e)^2)|Y|^{-1}dY$ is finite, then $\e$ is bounded with:
\be \lab{bd:Sobolev}
\| \e \|_{L^{\infty}(\{|Y|\geq 1 \})} \lesssim \| \e \|_{L^2\left(\{|Y|\geq 1 \},\frac{dY}{|Y|}\right)} +\| Y\pa_Y\e \|_{L^2\left(\{|Y|\geq 1 \},\frac{dY}{|Y|}\right)}
\ee

\end{lemma}

\begin{proof}

Assume that the right hand side of \fref{bd:Sobolev} is finite. Let $A\geq 1$ and $v(Z)=\e (AZ)$. Then, changing variables and using Sobolev embedding gives for some $C$ independent on $A$:
\bee
&&\| \e\|_{L^{\infty}([A,2A])}^2 = \| v \|_{L^{\infty}([1,2])}^2 \leq C \left( \int_1^2 v^2(Z)dZ+ \int_1^2 |\pa_Zv|^2(Z)dZ\right) \\
&\leq & C \left(\int_A^{2A} v^2(Y)\frac{dY}{A}+ \int_1^2 A|\pa_Yv|^2(Y)dY\right) \leq C\left( \int_A^{2A} v^2(Y)\frac{dY}{|Y|}+ \int_A^{2A} |Y\pa_Yv|^2(Y)\frac{dY}{|Y|}\right) \\
&\leq & C\left(\int_{|Y|\geq 1} v^2(Y)\frac{dY}{|Y|}+ \int_{|Y|\geq 1} |Y\pa_Yv|^2(Y)\frac{dY}{|Y|}\right) \\
\eee
Taking the supremum with respect to $A$ in the above estimate yields \fref{bd:Sobolev}.

\end{proof}

\section{Two-dimensional functional analysis results}

\begin{lemma}

Let $q\in \mathbb N^*$. Then for any $u\in W^{1,2q}_{\text{loc}}(\mathbb R^2)$ one has:
\be \lab{eq:Sobo}
\| u \|_{L^{\infty}(\mathbb R^2)}^{2q} \leq C(q)\left( \int_{\mathbb R^2} u^{2q}\frac{dXdY}{|X| \lY}+\int_{\mathbb R^2} (X \pa_X u)^{2q}\frac{dXdY}{|X| \lY}+\int_{\mathbb R^2} (\lY \pa_Yu)^{2q}\frac{dXdY}{|X| \lY} \right)
\ee

\end{lemma}

\begin{proof}

The result follows from the classical Sobolev embedding and a scaling argument. Assume that the right hand side of \fref{eq:Sobo} is finite. Let $A\in \mathbb R$ and change variables $\tilde X = X/A$ and $u(X,Y)=v(\tilde X,Y)$. From Sobolev embedding one has that
\bee
&& \| u \|_{L^{\infty}(\{A\leq |X|\leq 2A, \ |Y|\leq 1 \})}^{2q}  = \| v \|_{L^{\infty}(\{1\leq |\tilde X|\leq 2, \ |Y|\leq 1 \})}^{2q} \\
&\leq & C(q) \int_{1\leq |\tilde X|\leq 2, \ |Y|\leq 1} \left( v^{2q}+(\pa_{\tilde X}v)^{2q}+ (\pa_{Y}v)^{2q}\right)d\tilde XdY\\
&\leq & C(q) \int_{A\leq |X|\leq 2A, \ |Y|\leq 1} \left( u^{2q}+(A\pa_{X}u)^{2q} +(\pa_{Y}u)^{2q}\right)\frac{dXdY}{A}\\
&\leq &C(q) \left( \int_{\mathbb R^2} u^{2q}\frac{dXdY}{|X| \lY}+\int_{\mathbb R^2} (X \pa_X u)^{2q}\frac{dXdY}{|X| \lY}+\int_{\mathbb R^2} (\lY \pa_Yu)^{2q}\frac{dXdY}{|X| \lY}\right).
\eee
Now let $A>0$ and $B\geq 1$ and change variables $\tilde X = X/A$, $\tilde Y = Y/B$ and $u(X,Y)=v(\tilde X,\tilde Y)$. Then again from Sobolev:
\bee
&&\| u \|_{L^{\infty}(\{A\leq |X|\leq 2A, \ B\leq |Y|\leq 2B \})}^{2q} = \| v \|_{L^{\infty}(\{1\leq |\tilde X|\leq 2, \ 1\leq |\tilde Y|\leq 2 \})}^{2q} \\
&\leq & C(q) \int_{1\leq |\tilde X|\leq 2, \ 1\leq |\tilde Y|\leq 2} \left(v^{2q}+(\pa_{\tilde X}v)^{2q}+(\pa_{\tilde Y}v)^{2q} \right)d\tilde Xd\tilde Y\\
&\leq & C(q) \Big( \int_{A\leq |X|\leq 2A, \ B\leq |Y|\leq 2B} \left( u^{2q}+ (A\pa_{X}u)^{2q} + (B\pa_{Y}u)^{2q} \right)\frac{dXdY}{AB}\\
&\leq &C(q) \left( \int_{\mathbb R^2} u^{2q}\frac{dXdY}{|X| \lY}+\int_{\mathbb R^2} (X \pa_X u)^{2q}\frac{dXdY}{|X| \lY}+\int_{\mathbb R^2} (\lY \pa_Yu)^{2q}\frac{dXdY}{|X| \lY}\right).
\eee
Combing the two above inequalities, as the constant in the second one does not depend on $A$ and $B$, yields \fref{eq:Sobo} but for the quantity $\| u \|_{L^{\infty}(\mathbb R^2 \backslash \{X=0 \})}^{2q}$. This in turn yields \fref{eq:Sobo} by continuity.

\end{proof}

\begin{corollary}

Let $q\in \mathbb N^*$. Then for any $u\in W^{1,2q}_{\text{loc}}(\mathbb R^2)$ one has:
\be \lab{eq:weightedSobo}
\left\| \frac{u}{\phi_{j,0}(X,Z)} \right\|_{L^{\infty}(\mathbb R^2)}^{2q} \leq C(q)\left( \int_{\mathbb R^2} \frac{u^{2q}}{\phi_{j,0}^{2q}(X,Z)}\frac{dXdY}{|X| \lY}+\int_{\mathbb R^2} \frac{(X \pa_X u)^{2q}}{\phi_{j,0}^{2q}(X,Z)}\frac{dXdY}{|X| \lY}+\int_{\mathbb R^2} \frac{(\lY \pa_Yu)^{2q}}{\phi_{j,0}^{2q}(X,Z)}\frac{dXdY}{|X| \lY} \right)
\ee

\end{corollary}

\begin{proof}

Assume that the right hand side of \fref{eq:weightedSobo} is finite. First, notice from \fref{main:sizephi} that
$$
|X\pa_X \phi_{X,j}|\sim |\phi_{X,j}|
$$
From this and \fref{id:phij0} one deduces that:
\bee
&& |X\pa_X \phi_{j,0}(X,Z)| = \left| X\pa_X \left(\left(1+Z^{2k} \right)^{\frac j2-1}  \phi_{X,j}\left( F_k^{\frac 32}(Z)X\right)  \right)\right| \\
&= &\left| \left(1+Z^{2k} \right)^{\frac j2-1}  (X\pa_X \phi_{X,j})\left( F_k^{\frac 32}(Z)X\right) \right| \sim  \left(1+Z^{2k} \right)^{\frac j2-1}  \left| \phi_{X,j})\left( F_k^{\frac 32}(Z)X\right) \right| = |\phi_{j,0}(X,Z)|.
\eee
This implies that
$$
\left|X\pa_X \left(\frac{1}{\phi_{j,0}(X,Z)}\right) \right|=\left| \frac{X\pa_X\phi_{j,0}(X,Z)  }{\phi_{j,0}^2(X,Z)} \right|  \sim \left| \frac{1}{\phi_{j,0}(X,Z)} \right|.
$$
Assume now that $|Y|\leq 1$. Since $\pa_Y=e^{-(k-1)/(2k)s}\pa_Z$ then
\bee
&& |\pa_Y \phi_{j,0}(X,Z)| = \left| e^{-\frac{k-1}{2k}s}\pa_Z\left(\left(1+Z^{2k} \right)^{\frac j2-1}  \phi_{X,j}\left( F_k^{\frac 32}(Z)X\right)\right) \right| \\
&\leq & \left| \pa_Z\left(\left(1+Z^{2k} \right)^{\frac j2-1}\right)  \phi_{X,j}\left( F_k^{\frac 32}(Z)X\right) \right| + \frac 32 \left|  \left(1+Z^{2k} \right)^{\frac j2-1} \frac{\pa_Z F_k(Z)}{F_k(Z)} (X\phi_{X,j}) \left( F_k^{\frac 32}(Z)X\right) \right| \\
&\leq & C\left(1+Z^{2k} \right)^{\frac j2-1}  \phi_{X,j}\left( F_k^{\frac 32}(Z)X\right) =  C|\phi_{j,0}(X,Z)|
\eee
since $|\pa_Z F_k|\leq F_k$. If $|Y|\geq 1$. Since $Y\pa_Y=Z\pa_Z$ then similarly
\bee
&& |Y \pa_Y \phi_{j,0}(X,Z)|= \left|Z\pa_Z\left(\left(1+Z^{2k} \right)^{\frac j2-1}  \phi_{X,j}\left( F_k^{\frac 32}(Z)X\right)\right) \right| \\
&\lesssim & \left| Z\pa_Z\left(\left(1+Z^{2k} \right)^{\frac j2-1}\right)  \phi_{X,j}\left( F_k^{\frac 32}(Z)X\right) \right| +  \left|  \left(1+Z^{2k} \right)^{\frac j2-1} Z\frac{\pa_Z F_k(Z)}{F_k(Z)} (X\phi_{X,j}) \left( F_k^{\frac 32}(Z)X\right) \right| \\
&\leq & C\left(1+Z^{2k} \right)^{\frac j2-1}  \phi_{X,j}\left( F_k^{\frac 32}(Z)X\right) =  C|\phi_{j,0}(X,Z)|
\eee
since $|Z\pa_Z F_k|\leq F_k$. From the two above inequalities one deduces that for any $Y\in \mathbb R$:
$$
|\lY \pa_Y  \phi_{j,0} (X,Z)|\leq C |\phi_{j,0}(X,Z)|.
$$
This implies again that
$$
\left| \lY \pa_Y \left(\frac{1}{\phi_{j,0}(X,Z)}\right) \right|\leq C \left| \frac{1}{\phi_{j,0}(X,Z)} \right|.
$$
From this one deduces that
\bee
&&\int_{\mathbb R^2} \left( \frac{u^{2q}}{\phi_{j,0}^{2q}(X,Z)}+\left(X\pa_X \left(\frac{u}{\phi_{j,0}(X,Z)}\right)\right)^{2q}  + \left(\lY\pa_Y \left(\frac{u}{\phi_{j,0}(X,Z)}\right)\right)^{2q}\right) \frac{dXdY}{|X| \lY} \\
&\leq &C(q)\left( \int_{\mathbb R^2} \frac{u^{2q}}{\phi_{j,0}^{2q}(X,Z)}\frac{dXdY}{|X| \lY}+\int_{\mathbb R^2} \frac{(X \pa_X u)^{2q}}{\phi_{j,0}^{2q}(X,Z)}\frac{dXdY}{|X| \lY}+\int_{\mathbb R^2} \frac{(\lY \pa_Yu)^{2q}}{\phi_{j,0}^{2q}(X,Z)}\frac{dXdY}{|X| \lY} \right).
\eee
We apply \fref{eq:Sobo} to $u/\phi_{j,0}(X,Z)$ and use the above inequality to get the desired result \fref{eq:weightedSobo}.

\end{proof}

\begin{lemma} \lab{an:lem:equiv}

Let $j_1,j_2\in \mathbb N$. Then there exists a constant $C>0$ such that for any function $u\in C^{j_1+j_2}(\mathbb R^2)$, for $A$ defined by \fref{main:def:A} there holds:
\be \lab{an:equivalence5}
\frac 1C\sum_{j_1'=0}^{j_1}\sum_{j_2'=0}^{j_2}|Z|^{j_1'}|X|^{j_2'}|\pa_Z^{j_1'}\pa_X^{j_2'}u|\leq \sum_{j_1'=0}^{j_1}\sum_{j_2'=0}^{j_2}|(Z\pa_Z)^{j_1}(X\pa_X)^{j_2}u|\leq C\sum_{j_1'=0}^{j_1}\sum_{j_2'=0}^{j_2}|Z|^{j_1'}|X|^{j_2'}|\pa_Z^{j_1'}\pa_X^{j_2'}u|,
\ee
and for $j_2\geq 1$:
\be \lab{an:equivalence1}
|\pa_Z^{j_1}A^{j_2}u|\leq C \sum_{j_1'=0}^{j_1}\sum_{j_2'=1}^{j_2} (1+|Z|)^{-(j_1-j_1')} |X|^{j_2'}|\pa_Z^{j_1'}\pa_X^{j_2'}u|,
\ee
\be \lab{an:equivalence2}
|(Z\pa_Z)^{j_1}A^{j_2}u|\leq C \sum_{j_1'=0}^{j_1}\sum_{j_2'=1}^{j_2} |Z|^{j_1'} |X|^{j_2'}|\pa_Z^{j_1'}\pa_X^{j_2'}u|,
\ee
\be \lab{an:equivalence3}
|X|^{j_2}|\pa_Z^{j_1}\pa_X^{j_2}u|\leq C  \sum_{j_1'=0}^{j_1}\sum_{j_2'=1}^{j_2} (1+|Z|)^{-(j_1-j_1')} |\pa_Z^{j_1'}A^{j_2'}u|,
\ee
\be \lab{an:equivalence4}
|Z|^{j_1}|X|^{j_2}|\pa_Z^{j_1}\pa_X^{j_2}u|\leq C  \sum_{j_1'=0}^{j_1}\sum_{j_2'=1}^{j_2}  |(Z\pa_Z)^{j_1'}A^{j_2'}u|.
\ee

\end{lemma}

\begin{proof}

\fref{an:equivalence5} follows from an easy induction argument that we leave to the reader.\\

\noindent \textbf{Step 1} \emph{Proof of \fref{an:equivalence1}}. We first claim that there exists a family of profiles $(f_{j_2,j_2'})_{j_2'\leq j_2}$ such that
\be \lab{an:idAj2}
A^{j_2}u=\sum_{j_2'=1}^{j_2} f_{j_2,j_2'}\pa_X^{j_2'}u,
\ee
and satisfying for any $k_1,k_2\in \mathbb N$:
\be \lab{an:bdfj2j2'}
|\pa_Z^{k_1}\pa_X^{k_2}f_{j_2,j_2'}|\lesssim (1+|Z|)^{-k_1}(1+|X|)^{\min(-(k_2-j_2'),0)}|X|^{\max(j_2'-k_2,0)}.
\ee
We prove this fact by induction on $j_2\in \mathbb N^*$. From \fref{main:def:A}, \fref{an:idAj2} holds for $j_2=1$ with $f_{1,1}=3X/2+F_k^{-3/2}(Z)\Psi_1(F_k^{3/2}(Z)X)$. Since from Proposition \ref{prop:smoothselfsim}, $\pa_{X}\Psi_1\leq 0$ and is minimal at the origin we infer that
\be \lab{main:eq:estimationA}
\frac 12 |X| \leq \left| \frac 32 X +F_k^{-\frac 32}(Z)\Psi_1(F_k^{\frac 32}(Z)X) \right| \leq \frac 32 |X|
\ee
and from \fref{main:sizeFkPsi1} we infer that for $k_1,k_2\in \mathbb N$:
$$
\left|\pa_Z^{k_1}\pa_X^{k_2}\left(F_k^{-3/2}(Z)\Psi_1(F_k^{3/2}(Z)X)\right)\right|\lesssim \left\{ \begin{array}{l l} (1+|Z|)^{-k_1}|X| \ \ \text{if} \ k_2=0,\\
(1+|Z|)^{-k_1}(1+|X|)^{-(k_2-1)} \ \ \text{if} \ k_2\geq 1.
\end{array} \right.
$$
which proves \fref{an:bdfj2j2'} for $j_2=1$, and thus the claim is true for $j_2=1$. Assume now that the claim is true for some $j_2\in \mathbb N^*$. Then, using \fref{an:idAj2} for the integers $j_2$ and $1$:
\bee
A^{j_2+1}u&=&f_{1,1}\pa_X \left(\sum_{j_2'=1}^{j_2}f_{j_2,j_2'}\pa_X^{j_2'}u \right)=\sum_{j_2'=1}^{j_2}f_{1,1}\pa_X (f_{j_2,j_2'})\pa_X^{j_2'}u+f_{1,1}f_{j_2,j_2'}\pa_X^{j_2'+1}u \\
&=& \sum_{j_2'=1}^{j_2}f_{1,1}\pa_X (f_{j_2,j_2'})\pa_X^{j_2'}u+ \sum_{j_2'=2}^{j_2+1} f_{1,1}f_{j_2,j_2'-1}\pa_X^{j_2'}u
\eee
and the claim is true from the bounds \fref{an:bdfj2j2'} for the integers $j_2$ and $1$. Hence it is true for all $j_2\in \mathbb N^*$. We now apply Leibniz formula and obtain from \fref{an:idAj2}:
$$
\pa_Z^{j_1}A^{j_2}u=\sum_{j_2'=1}^{j_2}\sum_{j_1'=0}^{j_1} C_{j_1'}^{j_1} (\pa_Z^{j_1-j_1'}f_{j_2,j_2'})\pa_Z^{j_1'}\pa_X^{j_2'}u, \ \ |\pa_Z^{j_1-j_1'}f_{j_2,j_2'}|\lesssim (1+|Z|)^{-(j_1-j_1')}|X|^{j_2'}
$$
where the second estimate comes from \fref{an:bdfj2j2'}, and \fref{an:equivalence1} is proven.\\

\noindent \textbf{Step 2} \emph{Proof of \fref{an:equivalence2}}. This is a direct consequence of \fref{an:equivalence5} and \fref{an:equivalence1}:
\bee
|(Z\pa_Z)^{j_1}A^{j_2}u| & \lesssim & \sum_{j_1'=0}^{j_1} |Z|^{j_1'} |\pa_Z^{j_1'}A^{j_2}u| \lesssim \sum_{j_1'=0}^{j_1} \sum_{\tilde j_1'=0}^{j_1'}\sum_{j_2'=1}^{j_2} |Z|^{j_1'}(1+|Z|)^{-(j_1'-\tilde j_1')} |X|^{j_2'}|\pa_Z^{\tilde j_1'}(X\pa_X)^{j_2'}u|,\\
&\lesssim &\sum_{j_1'=0}^{j_1} \sum_{\tilde j_1'=0}^{j_1'}\sum_{j_2'=1}^{j_2} |Z|^{\tilde j_1'} |X|^{j_2'}|\pa_Z^{\tilde j_1'}(X\pa_X)^{j_2'}u|\lesssim \sum_{j_1'=0}^{j_1} \sum_{j_2'=1}^{j_2} |Z|^{ j_1'} |X|^{j_2'}|\pa_Z^{ j_1'}(X\pa_X)^{j_2'}u|.
\eee

\noindent \textbf{Step 3} \emph{Proof of \fref{an:equivalence3}}. First, from \fref{an:equivalence5} one has:
\be \lab{an:interequiv}
|X|^{j_2}|\pa_Z^{j_1}\pa_X^{j_2}|\lesssim \sum_{j'_2=1}^{j_2} |\pa_Z^{j_1}(X\pa_X)^{j_2'}u|.
\ee
We then claim that there exists a family of profiles $(g_{j_2,j_2'})_{j_2'\leq j_2}$ such that
\be \lab{an:idXpaXj2}
(X\pa_X)^{j_2}u=\sum_{j_2'=1}^{j_2} g_{j_2,j_2'}A^{j_2'}u,
\ee
and satisfying for any $k_1,k_2\in \mathbb N$:
\be \lab{an:bdgj2j2'}
|\pa_Z^{k_1}\pa_X^{k_2}g_{j_2,j_2'}|\lesssim (1+|Z|)^{-k_1}(1+|X|)^{-k_2}.
\ee
From \fref{main:def:A}, \fref{an:idXpaXj2} holds for $j_2=1$ with 
$$
g_{1,1}=\frac{X}{\frac 32 X+F_k^{-\frac 32}(Z)\Psi_1\left(F_k^{\frac 32}(Z)X\right)}=\left(\frac{\tilde X}{\frac 32 \tilde X +\Psi_1}\right)(F_k^{\frac 32}(Z)X).
$$
From \fref{main:sizeFkPsi1} and \fref{main:eq:estimationA} one has that for any $k_2\in \mathbb N$,
$$
\left| \pa_{\tilde X} \left(\frac{\tilde X}{\frac 32 \tilde X +\Psi_1(\tilde X)} \right)\right| \lesssim (1+|\tilde X|)^{-k_2}
$$
and hence from \fref{main:sizeFkPsi1} the estimate \fref{an:bdgj2j2'} holds for $g_{1,1}$. The claim can then be proven by induction on $j_2\in \mathbb N^*$ by the same techniques as in Step 1, and we do not give the details here. Using \fref{an:idXpaXj2}, \fref{an:bdgj2j2'} and Leibniz formula then yields that
$$
|\pa_Z^{j_1}(X\pa_X)^{j_2}u|\lesssim \sum_{j_1'=0}^{j_1} \sum_{j_2'=1}^{j_2} |\pa_{Z}^{j_1-j_1'}g_{j_2,j_2'}\pa_Z^{j_1'}A^{j_2'}u|\lesssim \sum_{j_1'=0}^{j_1} \sum_{j_2'=1}^{j_2} (1+|Z|)^{-(j_1-j_1')}|\pa_Z^{j_1'}A^{j_2'}u|.
$$
This implies \fref{an:equivalence3} using \fref{an:interequiv}.\\

\noindent \textbf{Step 4} \emph{Proof of \fref{an:equivalence4}}. This is a direct consequence of \fref{an:equivalence5} and \fref{an:equivalence3}:
\bee
|Z|^{j_1}|X|^{j_2}|\pa_Z^{j_1}\pa_X^{j_2}u| & \lesssim &  \sum_{j_1'=0}^{j_1}\sum_{j_2'=1}^{j_2} |Z|^{j_1}(1+|Z|)^{-(j_1-j_1')} |\pa_Z^{j_1'}A^{j_2'}u| \lesssim \sum_{j_1'=0}^{j_1}\sum_{j_2'=1}^{j_2} |Z|^{j_1'} |\pa_Z^{j_1'}A^{j_2'}u|,\\
&\lesssim &\sum_{j_1'=0}^{j_1} \sum_{j_2'=1}^{j_2} |(Z\pa_Z)^{j_1'}A^{j_2'}u|.
\eee

\end{proof}



\begin{thebibliography}{10}

\bibitem{B} Barenblatt, G. I. (1996). Scaling, self-similarity, and intermediate asymptotics: dimensional analysis and intermediate asymptotics (Vol. 14). Cambridge University Press.

\bibitem{BK} Berger, M., Kohn, R. V. (1988). A rescaling algorithm for the numerical calculation of blowing-up solutions. Communications on pure and applied mathematics, 41(6), 841-863.

\bibitem{BK2} Bricmont, J., Kupiainen, A. (1994). Universality in blow-up for nonlinear heat equations. Nonlinearity, 7(2), 539.

\bibitem{CP} Chen, G. Q., Perthame, B. (2003, July). Well-posedness for non-isotropic degenerate parabolic-hyperbolic equations. In Annales de l'Institut Henri Poincare (C) Non Linear Analysis (Vol. 20, No. 4, pp. 645-668). Elsevier Masson.

\bibitem{CHKLSY} Choi, K., Hou, T. Y., Kiselev, A., Luo, G., Sverak, V., Yao, Y. (2017). On the Finite-Time Blowup of a One-Dimensional Model for the Three-Dimensional Axisymmetric Euler Equations. Communications on Pure and Applied Mathematics, 70(11), 2218-2243.

\bibitem{CINT} Cao, C., Ibrahim, S., Nakanishi, K., Titi, E. S. (2015). Finite-time blowup for the inviscid primitive equations of oceanic and atmospheric dynamics. Communications in Mathematical Physics, 337(2), 473-482.

\bibitem{CEHL} Caflisch, R. E., Ercolani, N., Hou, T. Y., Landis, Y. (1993). Multi-valued solutions and branch point singularities for nonlinear hyperbolic or elliptic systems. Communications on pure and applied mathematics, 46(4), 453-499.

\bibitem{CS} Caflisch, R. E., Sammartino, M. (2000). Existence and singularities for the Prandtl boundary layer equations. ZAMM-Journal of Applied Mathematics and Mechanics/Zeitschrift f\"ur Angewandte Mathematik und Mechanik, 80(11-12), 733-744.

\bibitem{CSW} Cassel, K. W., Smith, F. T., Walker, J. D. A. (1996). The onset of instability in unsteady boundary-layer separation. Journal of Fluid Mechanics, 315, 223-256.

\bibitem{CMR} Collot, C., Merle, F., Rapha\"el, P. (2020). Strongly anisotropic type II blow up at an isolated point. Journal of the American Mathematical Society, 33(2), 527-607.

\bibitem{D} Dafermos, C. M. (2010). Hyperbolic conservation laws in continuum physics, volume 325 of Grundlehren der Mathematischen Wissenschaften [Fundamental Principles of Mathematical Sciences].

\bibitem{EE} Weinan, E., Engquist, B. (1997). Blowup of solutions of the unsteady Prandtl's equation. Communications on pure and applied mathematics, 50(12), 1287-1293.

\bibitem{EF} Eggers, J., Fontelos, M. A. (2008). The role of self-similarity in singularities of partial differential equations. Nonlinearity, 22(1), R1.

\bibitem{GK} Giga, Y., Kohn, R. V. (1985). Asymptotically self-similar blow-up of semilinear heat equations. Communications on Pure and Applied Mathematics, 38(3), 297-319.

\bibitem{HV2} Herrero, M. A., Velazquez, J. J. L. (1992). Flat blow-up in one-dimensional semilinear heat equations. Differential Integral Equations, 5(5), 973-997.

\bibitem{HV3} Herrero, M. A., Velazquez, J. J. L. (1992). Generic behaviour of one-dimensional blow up patterns. Annali della Scuola Normale Superiore di Pisa-Classe di Scienze, 19(3), 381-450.

\bibitem{HV} Herrero, M. A., Velazquez, J. J. L. (1993). Blow-up behaviour of one-dimensional semilinear parabolic equations. In Annales de l'IHP Analyse non lin\'eaire (Vol. 10, No. 2, pp. 131-189). Gauthier-Villars.

\bibitem{H} Hopf, E. (1950). The partial differential equation $u_t+ uu_x= \mu_{xx}$. Communications on Pure and Applied mathematics, 3(3), 201-230.

\bibitem{I}  Il'in, A. M. (1992). Matching of asymptotic expansions of solutions of boundary value problems (Vol. 102). Providence, RI: American Mathematical Society.

\bibitem{KST} Krieger, J., Schlag, W., Tataru, D. (2008). Renormalization and blow up for charge one equivariant critical wave maps. Inventiones mathematicae, 171(3), 543-615.

\bibitem{KVW} Kukavica, I., Vicol, V., Wang, F. (2017). The van Dommelen and Shen singularity in the Prandtl equations. Advances in Mathematics, 307, 288-311.

\bibitem{LPT} Lions, P. L., Perthame, B., Tadmor, E. (1994). A kinetic formulation of multidimensional scalar conservation laws and related equations. Journal of the American Mathematical Society, 7(1), 169-191.

\bibitem{MRS} Merle, F., Rapha\"el, P., Szeftel, J. (2020). On Strongly Anisotropic Type I Blowup. International Mathematics Research Notices, 2020(2), 541-606.

\bibitem{MZ} Merle, F., Zaag, H. (1997). Stability of the blow-up profile for equations of the type $ut=\Delta u+| u|^{p-1}u$. Duke Math. J, 86(1), 143-195.

\bibitem{PLGG} Pomeau, Y., Le Berre, M., Guyenne, P., Grilli, S. (2008). Wave-breaking and generic singularities of nonlinear hyperbolic equations. Nonlinearity, 21(5), T61.

\bibitem{QS} Quittner, P., Souplet, P. (2007). Superlinear parabolic problems: blow-up, global existence and steady states. Springer Science Business Media.

\bibitem{S}  Serre, D. (1999). Systems of Conservation Laws 1: Hyperbolicity, entropies, shock waves. Cambridge University Press.

\bibitem{W} Wong, T. K. (2015). Blowup of solutions of the hydrostatic Euler equations. Proceedings of the American Mathematical Society, 143(3), 1119-1125.

\end{thebibliography}
\end{document}